\documentclass[10pt, reqno]{amsart}
\usepackage{amsfonts}
\usepackage{mathrsfs}
\usepackage{amscd}
\usepackage{amsmath}
\usepackage{amssymb}
\usepackage{latexsym}
\usepackage{lscape}
\usepackage{xypic}
\usepackage{comment}
\usepackage{amscd}
\usepackage{wasysym}  
\usepackage{tikz} 
\usetikzlibrary{matrix,arrows}
\usepackage{tikz-cd}
\usepackage{appendix}
\usepackage{geometry}
\usepackage[nice]{nicefrac}
\usepackage{dsfont}
\usepackage{color}
\usepackage{graphicx}
\usepackage{todonotes}
\usepackage{enumerate}

\usepackage{theoremref}
\usepackage[english]{babel}
\usepackage{bbm}

\usepackage[hyperindex,linktocpage=true]{hyperref}
\hypersetup{
	colorlinks,
	linkcolor={red!50!black},
	citecolor={blue!50!black},
	urlcolor={blue!80!black}
}

\newtheorem{thm}{Theorem}[section]

\newtheorem{prop}[thm]{Proposition}
\newtheorem{lem}[thm]{Lemma}

\newtheorem{lem-def}[thm]{Lemma-Definition}
\newtheorem{cor}[thm]{Corollary}

\theoremstyle{definition}

\newtheorem{ex}[thm]{Example}

\newtheorem{rmk}[thm]{Remark}
\newtheorem{dfn}[thm]{Definition}
\newtheorem{nota}[thm]{Notation}

\newcommand{\Aa}{\mathcal{A}}
\newcommand{\Bb}{\mathcal{B}}
\newcommand{\Cc}{\mathcal{C}}

\newcommand{\Ff}{\mathcal{F}}
\newcommand{\Gg}{\mathcal{G}}
\newcommand{\Hh}{\mathcal{H}}
\newcommand{\Ii}{\mathcal{I}}
\newcommand{\Jj}{\mathcal{J}}

\newcommand{\Mm}{\mathcal{M}}

\newcommand{\Oo}{\mathcal{O}}
\newcommand{\Pp}{\mathcal{P}}
\newcommand{\Qq}{\mathcal{Q}}

\newcommand{\Ss}{\mathcal{S}}
\newcommand{\Tt}{\mathcal{T}}
\newcommand{\Uu}{\mathcal{U}}

\newcommand{\IA}{\mathbb{A}}

\newcommand{\IF}{\mathbb{F}}
\newcommand{\IG}{\mathbb{G}}

\newcommand{\IN}{\mathbb{N}}

\newcommand{\IP}{\mathbb{P}}
\newcommand{\IQ}{\mathbb{Q}}
\newcommand{\IR}{\mathbb{R}}

\newcommand{\IZ}{\mathbb{Z}}

\newcommand{\CC}{\mathfrak{C}}

\newcommand{\FF}{\mathfrak{F}}

\newcommand{\ff}{\mathbf{f}}
\let\gge\gg
\renewcommand{\gg}{\mathfrak{g}}

\newcommand{\qq}{\mathfrak{q}}

\numberwithin{equation}{section}

\newcommand{\pot}[1]{ [\hspace{-0,5mm}[ {#1} ]\hspace{-0,5mm}] }
\newcommand{\rpot}[1]{ (\hspace{-0,7mm}( {#1} )\hspace{-0,7mm}) }
\DeclareMathOperator{\Gr}{Gr} % Affine Grassmannian
\DeclareMathOperator{\Fl}{Fl} % Affine flag variety
\DeclareMathOperator{\Spec}{Spec} % Spectrum
\DeclareMathOperator{\Spa}{Spa} % Adic spectrum
\DeclareMathOperator{\Spd}{Spd} % Diamond spectrum
\newcommand{\perf}{{\operatorname{perf}}} % perfection
\newcommand{\aff}{{\operatorname{af}}}% affine (in context of Kac-Moody algebras)
\DeclareMathOperator{\identity}{id} % Identity map
\DeclareMathOperator{\DM}{DM} % Motives
\DeclareMathOperator{\DTM}{DTM} % Tate motives
\DeclareMathOperator{\DATM}{DATM} % Artin-Tate motives
\DeclareMathOperator{\MTM}{MTM} % Mixed Tate motives
\DeclareMathOperator{\MATM}{MATM} % Mixed Artin-Tate motives
\DeclareMathOperator{\Sch}{Sch} % Category of schemes
\DeclareMathOperator{\ft}{ft} % Finite type
\DeclareMathOperator{\op}{op} % Opposite category
\DeclareMathOperator{\St}{St} % Stable
\renewcommand{\Pr}{\operatorname{Pr}} % Presentable infinity category
\DeclareMathOperator{\et}{\acute{e}t} % étale motives
\DeclareMathOperator{\PreStk}{PreStk} % Prestacks
\DeclareMathOperator{\Ani}{An} % Anima
\DeclareMathOperator{\Fun}{Fun} % Functors
\DeclareMathOperator{\Gal}{Gal} % Galois group
\DeclareMathOperator{\pr}{pr} % Projection map
\DeclareMathOperator{\CT}{CT} % Constant term
\DeclareMathOperator{\coMod}{coMod} % Comodules
\DeclareMathOperator{\Mod}{Mod} % Modules
\DeclareMathOperator{\coav}{coav} % Coaveraging
\DeclareMathOperator{\Cent}{Cent} % Centralizer
\DeclareMathOperator{\Norm}{Norm} % Normalizer
\DeclareMathOperator{\der}{der} % Derived subgroup
\DeclareMathOperator{\im}{im} % Image
\DeclareMathOperator{\nd}{nd} % Non-divisible
\DeclareMathOperator{\proj}{proj} % Projection
\DeclareMathOperator{\adj}{ad} % Adjoint
\DeclareMathOperator{\Aut}{Aut} % Automorphism group
\DeclareMathOperator{\Rep}{Rep} % Representation category
\DeclareMathOperator{\IC}{IC} % Intersection complex
\DeclareMathOperator{\Stab}{Stab} % Stabilizer
\DeclareMathOperator{\sico}{sc} % Simply connected
\DeclareMathOperator{\Res}{Res} % Weil restriction of scalars
\DeclareMathOperator{\pfp}{pfp} % Perfectly of finite presentation
\DeclareMathOperator{\AffSch}{AffSch} % Affine schemes
\DeclareMathOperator{\res}{res} % Restriction
\DeclareMathOperator{\incl}{incl} % Inclusion
\DeclareMathOperator{\anti}{anti} % Anti-effective
\newcommand{\xanti}{(\operatorname{anti})} % Potentially anti-effective
\DeclareMathOperator{\pos}{pos} % Positive
\DeclareMathOperator{\fil}{fil} % Filtration step
\DeclareMathOperator{\Char}{char} % Character
\DeclareMathOperator{\dom}{dom} % Dominant
\DeclareMathOperator{\sgn}{sgn} % Sign
\DeclareMathOperator{\LS}{LS} % LS galleries
\DeclareMathOperator{\Perf}{Perf} % Perfect schemes
\DeclareMathOperator{\Perfd}{Perfd} % Perfectoid spaces
\DeclareMathOperator{\BD}{BD} % Beilinson-Drinfeld
\DeclareMathOperator{\Grp}{Grp} % Category of groups
\DeclareMathOperator{\Set}{Set} % Category of sets
\DeclareMathOperator{\Sat}{Sat} % Satake category
\DeclareMathOperator{\Hck}{Hck} % Hecke stack
\DeclareMathOperator{\bd}{bd} % Bounded
\DeclareMathOperator{\fd}{fd} % Finite dimensional
\DeclareMathOperator{\fg}{fg} % Finitely generated
\DeclareMathOperator{\Perv}{Perv} % Perverse sheaves
\DeclareMathOperator{\Hom}{Hom} % Hom-groups
\DeclareMathOperator{\Maps}{Maps} % Mapping space
\DeclareMathOperator{\Irr}{Irr} % Irreducible components
\DeclareMathOperator{\tr}{tr} % Trace
\DeclareMathOperator{\cl}{cl} % Classical
\DeclareMathOperator{\Ad}{Ad} % Adjoint action
\DeclareMathOperator{\gr}{gr} % Associated graded
\DeclareMathOperator{\As}{As} % Asymptotic cone
\DeclareMathOperator{\tors}{tors} % Torsion
\DeclareMathOperator{\coker}{coker} % Cokernel

\newcommand{\into}{\hookrightarrow}
\newcommand{\onto}{\twoheadrightarrow}
\newcommand{\conv}{{}^{\mathrm{p}}\star} % Convolution product; might add upper p to emphasize we need to truncate.
\newcommand{\pH}{{}^{\mathrm{p}}\mathrm{H}} % Perverse truncation
\newcommand{\twext}{{}^{\mathrm{p}}\widetilde{\boxtimes}} % Truncated twisted exterior product
\newcommand{\IHom}{\underline{\operatorname{Hom}}} % Internal Hom
\newcommand{\unit}{\mathbbm{1}} % Unit object (as not yet sure about which coefficients to use)
\newcommand{\app}{\mathscr{A}} % Apartment
\newcommand{\buil}{\mathscr{B}} % Bruhat-Tits building
\newcommand{\wall}{\mathbf{H}} % Affine hyperplane
\newcommand{\QVect}{\IQ\text{-}\operatorname{Vect}} % Q-vector spaces
\newcommand{\QlVect}{\IQ_\ell\text{-}\operatorname{Vect}} % Ql-vector spaces
\newcommand{\GL}{\mathrm{GL}} % General linear group
\newcommand{\SL}{\mathrm{SL}} % Special linear group
\newcommand{\PGL}{\mathrm{PGL}} % Projective linear group
\newcommand{\SU}{\mathrm{SU}} % Special unitary group
\newcommand{\PU}{\mathrm{PU}} % Projective unitary group
\newcommand{\FibFunctor}{\mathrm{H}^*} % Fiber functor
\newcommand{\grQVect}{\operatorname{gr}\text{-}\IQ\text{-}\operatorname{Vect}} % Graded vector spaces
\newcommand{\grZMod}{\operatorname{gr}\text{-}\IZ[\frac{1}{p}]\text{-}\operatorname{Mod}} % Graded modules
\newcommand{\CG}{{}^CG} % C-group
\newcommand{\vinberg}{V_{\widehat{G},\rho_{\adj}}} % Vinberg monoid
\newcommand{\restrvinberg}{V_{\widehat{G},\rho_{\adj}\mid d_{\rho_{\adj}}=q}} % Restricted Vinberg monoid
\newcommand{\BdR}{B_{\mathrm{dR}}} % de Rham affine Grassmannian
\newcommand{\dR}{\mathrm{dR}} % Shorter notation
\newcommand{\comp}{\mathrm{c}} % Compact objects
\renewcommand{\phi}{\varphi}

\begin{document}

\title[Ramified motivic Satake]{The integral motivic Satake equivalence \\for ramified groups}
\author[Thibaud van den Hove]{Thibaud van den Hove}
	
\address{Max Planck Insitut für Mathematik, Vivatsgasse 7, 53111 Bonn, Germany}
\email{vandenhove@mpim-bonn.mpg.de}
	
%\subjclass[2010]{Primary 14M15; Secondary 14F42, 20C08}
	
\begin{abstract}
	We construct the geometric Satake equivalence for quasi-split reductive groups over nonarchimedean local fields, using étale Artin-Tate motives with $\mathbb{Z}[\frac{1}{p}]$-coefficients.
	We consider local fields of both equal and mixed characteristic.
	Along the way, we extend the work of Gaussent--Littelmann on the connection between LS galleries and MV cycles to the case of residually split reductive groups.
	As an application, we generalize Zhu's integral Satake isomorphism for spherical Hecke algebras to ramified groups.
	Moreover, for residually split groups, we define generic spherical Hecke algebras, and construct generic Satake and Bernstein isomorphisms.
\end{abstract}

\maketitle

\setcounter{tocdepth}{1}
\tableofcontents
\setcounter{section}{0}

%\pagebreak
	
\thispagestyle{empty}

\section{Introduction}

\subsection{Motivation and main results}

The geometric Satake equivalence is a cornerstone of modern mathematics.
For a reductive group \(G\) over an algebraically closed field \(k\), it provides an equivalence between \(L^+G\)-equivariant perverse sheaves on the affine Grassmannian \(\Gr_G\) with representations of the Langlands dual group \(\widehat{G}\). 
Building on work of Lusztig, Beilinson--Drinfeld, and Ginzburg, the first complete proof was given by Mirkovic--Vilonen in \cite{MirkovicVilonen:Geometric}. 
This bridge between geometric and algebraic worlds has led to many applications, for example in the Langlands program and geometric representation theory.
Since \cite{MirkovicVilonen:Geometric}, several variants have been proven, using different groups or cohomology theories; we will come back to this later in the introduction, and refer to the introduction of \cite{CassvdHScholbach:Geometric} for further references.

The Langlands correspondence, which among other things relates automorphic forms with Galois representations, is conjectured to be of motivic origin.
For this reason, it is desirable to extend the geometric Satake equivalence, by replacing the choice of cohomology theory with motives.
In particular, this resolves certain independence-of-\(\ell\) type questions, arising from the use of e.g.~\(\ell\)-adic cohomology.
For split groups over local function fields, such motivic Satake equivalences have already been constructed, first by Richarz--Scholbach \cite{RicharzScholbach:Satake} for rational coefficients, and afterwards in joint work of the author with Cass and Scholbach \cite{CassvdHScholbach:Geometric} for integral coefficients.
These equivalences are part of ongoing projects aiming to provide motivic enhancements of the (geometric) Langlands program for function fields, and we refer to loc.~cit.~for more details.

Still in the situation of local function fields, Zhu \cite{Zhu:Ramified} and Richarz \cite{Richarz:Affine} have constructed versions of the Satake equivalence, where the group \(G\) is allowed to be ramified (but still assumed quasi-split).
In that case, equivariant sheaves on the affine Grassmannian are not controlled by \(\widehat{G}\), but rather by the inertia-invariants \(\widehat{G}^I\).
While \cite{Zhu:Ramified,Richarz:Affine} work with \(\overline{\IQ}_\ell\)-coefficients, the ramified Satake equivalence has recently been extended to \(\IZ_\ell\)-, and hence also modular, coefficients by Achar--Lourenço--Richarz--Riche \cite{ALRR:Modular}.

In another direction, Zhu \cite{Zhu:Affine} has also constructed a Satake equivalence for (unramified) groups over local fields of mixed characteristic.
This allows for more arithmetic applications, such as the construction of cycles on the special fibers of Shimura varieties, verifying instances of the Langlands and Tate conjectures; we refer to \cite{XiaoZhu:Cycles} for more applications and details.

The current paper is a first step in providing motivic enhancements of such applications in the arithmetic Langlands program.
Although the motivic Satake equivalence in mixed characteristic has already appeared in \cite{RicharzScholbach:Witt} for split groups with rational coefficients, we upgrade it to allow integral coefficients (compare with \cite{Yu:Integral}), and quasi-split reductive groups.
We note that the ramified Satake equivalence for mixed characteristic local fields has not appeared yet, even for \(\ell\)-adic cohomology.
Hence, certain results of this article should be of interest even outside of the motivic setting.
In particular, our main result can be seen as a common generalization of \cite{Zhu:Ramified,Richarz:Affine,Zhu:Affine,Yu:Integral,RicharzScholbach:Satake,RicharzScholbach:Witt, ALRR:Modular}.

\begin{thm}\thlabel{thm.intro-Satake}
	Let \(F\) be a nonarchimedean local field with finite residue field of characteristic \(p\), or the completion of a maximal unramified extension thereof.
	Denote by \(\Oo\subseteq F\) its ring of integers, and by \(k\) its residue field.
	Let \(G/F\) be a connected reductive group, and \(\Gg/\Oo\) a very special parahoric model.
	For any coefficient ring \(\Lambda\) in \(\{\IZ[\frac{1}{p}],\IQ,\IF_\ell\mid \ell\neq p\}\), there is a canonical monoidal equivalence
	\[(\MATM_{L^+\Gg}(\Gr_{\Gg},\Lambda), \conv) \cong (\Rep_{\widehat{G}^I} \MATM(\Spec k,\Lambda),\otimes).\]
\end{thm}

We refer to the main text, and in particular Theorems \ref{Hopf algebra thm} and \ref{identification hopf algebra}, for details, but let us make some preliminary comments.
Recall that, roughly speaking, very special parahorics are exactly those parahorics for which the associated affine flag variety behaves like usual affine Grassmannians for split reductive groups, cf.~\cite[Theorem B]{Richarz:Affine}.
By \cite[Lemma 6.1]{Zhu:Ramified}, they exist if and only if \(G\) is quasi-split, so that the theorem above covers all quasi-split groups.
Second, we will define a full subcategory of Artin-Tate motives \(\DATM\subseteq \DM\) on a scheme (or a more general geometric object), together with a suitable t-structure with heart \(\MATM\), cf.~Section \ref{sect-MATM}.
This is in order to accomodate the fact that the general motivic t-structure is still conjectural.
Our proof can also be used (and simplified) to construct Satake equivalences for étale cohomology with \(\IQ_\ell\)-, or even \(\IZ_\ell\)-coefficients, cf.~\thref{Hopf algebra for ladic sheaves} and \thref{Satake for etale cohomology}.
Note that we only assert the existence of a monoidal equivalence.
Although this theorem equips \(\MATM_{L^+\Gg}(\Gr_{\Gg},\Lambda)\) with a symmetric monoidal structure making this equivalence symmetric monoidal, we currently lack the fusion product to directly equip the convolution product with a commutativity constraint.
Finally, we point out that in the motivic setting, it is subtle to actually define \(\Rep_{\widehat{G}^I} \MATM(\Spec k,\Lambda)\).
We proceed by constructing a monoidal functor from \(\IZ\)-graded \(\Lambda\)-modules equipped with a \(\Gamma\)-action, where \(\Gamma\) is a suitable Galois group, to \(\MATM(\Spec k,\Lambda)\).
By equipping the Hopf algebra of global sections of \(\widehat{G}^I\) with a suitable grading and \(\Gamma\)-action, we can consider its image in \(\MATM(\Spec k,\Lambda)\), which is still a Hopf algebra.
We then denote by \(\Rep_{\widehat{G}^I} \MATM(\Spec k,\Lambda)\) the category of comodules in \(\MATM(\Spec k,\Lambda)\) under this Hopf algebra.
In contrast to \cite{Zhu:Ramified, Richarz:Affine}, but similar to \cite{ALRR:Modular}, there are also subtleties arising from the fact that \(\widehat{G}^I\) is not necessarily reductive when working with integral or modular coefficients.
Most of these issues can be overcome by the results from \cite{ALRR:Fixed}.

\subsection{Galleries and MV cycles}

When using Artin-Tate motives as above, technical difficulties arise since these are in general not preserved by the six functors.
Hence, we have to flesh out the geometry of the affine Grassmannian, in order to show certain functors indeed preserve Artin-Tate motives.
Of particular interest for us are the constant term functors, which are used to reduce question for general groups \(G\) to similar questions for tori, which are easier to handle.
This leads us to study the intersections of Schubert cells and semi-infinite orbits in the affine Grassmannian. 
We will understand these intersections by generalizing methods of Gaussent--Littelmann \cite{GaussentLittelmann:LS}.
Our main result in this direction is the following, cf.~\thref{Intersections of Schubert cells and semi-infinite orbits}.

\begin{thm}\thlabel{thm.intro-decomposition}
	Keep the notation from \thref{thm.intro-Satake}, but now assume that \(G\) is residually split.
	Consider a Schubert cell \(\Gr_{\Gg,\mu}\subseteq \Gr_{\Gg}\) for \(\mu\in X_*(T)_I^+\) and a semi-infinite orbit \(\Ss_{\nu}^+\subseteq \Gr_{\Gg}\) for \(\nu\in X_*(T)_I\), where \(T\subseteq G\) is a suitable maximal torus.
	Then, the intersection \(\Gr_{\Gg,\mu}\cap \Ss_{\nu}^+\) admits a filtrable decomposition into schemes of the form \(\IA_k^{n,\perf}\times \IG_{m,k}^{r,\perf}\).
\end{thm}

Recall that a filtrable decomposition of a scheme \(X\) is a disjoint union \(X=\bigsqcup_w X_w\) such that there exists an increasing sequence of closed subschemes \(\varnothing = X_0 \subsetneq X_1 \subsetneq \ldots \subsetneq X_n = X\), for which each successive complement \(X_i\setminus X_{i-1}\) is one of the \(X_w\)'s.
In particular, any stratification is a filtrable decomposition, although the latter is already enough for the purpose of inductively applying localization.
We also note that the schemes appearing in the theorem above are perfect, since the mixed characteristic affine Grassmannians from \cite{Zhu:Affine} are only defined up to perfection.
In case \(F\) is of equal characteristic and the affine Grassmannians are defined as in \cite{PappasRapoport:Twisted}, the proof of \thref{thm.intro-decomposition} shows that one does indeed get a filtrable decomposition into \(\IA_k^{n}\times \IG_{m,k}^{r}\)'s.

The main strategy to prove \thref{thm.intro-decomposition} is to realize Schubert cells as varieties of certain minimal galleries in the Bruhat-Tits building of \(G\).
This allows us to consider retractions, which make the semi-infinite orbits appear.
On the other hand, considering more general galleries, which are not necessarily minimal, gives rise to certain smooth resolutions of the Schubert varieties \(\Gr_{\Gg,\leq \mu}\).
These resolutions admit a filtrable decomposition by the attractors for a certain \(\IG_m\)-action, refining the preimage stratification defined by the semi-infinite orbits.
Hence, it suffices to understand the minimal galleries appearing in each locally closed subscheme of this decomposition, for which we use results of Deodhar \cite{Deodhar:Geometric}.

Not all of Section \ref{sect-LSMV} is strictly necessary to prove \thref{thm.intro-decomposition}.
However, we take the opportunity to generalize results from \cite{GaussentLittelmann:LS} to the case of non-split groups and mixed characteristic local fields, which is of independent interest to us. 
(Strictly speaking, \cite{GaussentLittelmann:LS} use semisimple complex groups, where \(\mathbb{C}\) is the residue field of the complete discretely valued field that implicitly appears. However, their methods work verbatim when replacing \(\mathbb{C}\) by arbitrary algebraically closed fields, which was already used in \cite{CassvdHScholbach:Geometric}.)
For example, we obtain the following building-theoretic interpretation of the representation theory of \(\widehat{G}^I\), defined over a field of characteristic 0.

\begin{thm}\thlabel{thm.intro-character formula}
	Assume \(G\) is residually split, let \(\mu\in X_*(T)_I^+\), and denote by \(V(\mu)\) the irreducible \(\widehat{G}^I\)-representation of highest weight \(\mu\).
	Then, there is a combinatorially defined set \(\Gamma_{\LS}^+(\gamma_\mu)\) of positively folded combinatorial LS galleries, for which
	\[\Char V(\mu) = \sum_{\gamma\in \Gamma_{\LS}^+(\gamma_\mu)} \exp(e(\gamma)).\]
\end{thm}

We refer to Section \ref{sect-LSMV}, and in particular \thref{character formula 2}, for details and the notation used.
We note that although we work with semisimple and simply connected groups in Section \ref{sect-LSMV}, the (reduced) Bruhat-Tits buildings and Weyl groups remain the same after passing to simply connected covers of derived subgroups. 
So, the theorem above also holds for arbitrary residually split groups.

\subsection{Applications to Hecke algebras}

Although further applications of \thref{thm.intro-Satake} are work in progress, we can already provide some first number-theoretic applications by decategorifying \thref{thm.intro-Satake}.
Let us recall the status of Satake isomorphisms for spherical Hecke algebras.

The original version was proven by Satake in \cite{Satake:Theory}, and Borel \cite{Borel:Automorphic} explained the relation with the Langlands dual group in case \(G\) is unramified.
While both references work with complex coefficients, they can also be used to construct Satake isomorphisms with \(\IZ[q^{\pm \frac{1}{2}}]\)-coefficients, where \(q\) is the cardinality of \(k\), cf.~\cite{Gross:Satake}.
For ramified \(G\), a Satake isomorphism was constructed by Haines--Rostami in \cite{HainesRostami:Satake}, again with complex coefficients.
On the other hand, there are the mod \(p\) Satake isomorphisms, constructed by Herzig \cite{Herzig:Satake} for unramified groups in mixed characteristic, and by Henniart--Vignéras \cite{HenniartVigneras:Satake} in general.
In fact, \cite{HenniartVigneras:Satake} even get isomorphisms with \(\IZ\)-coefficients, but the Langlands dual group (or any related object) does not appear there.
In contrast, \cite{Zhu:Integral} constructed integral Satake isomorphisms for unramified groups, relating spherical Hecke algebras to rings of functions on the Vinberg monoid of the Langlands dual group.

Our next goal is to generalize Zhu's integral Satake isomorphism to quasi-split reductive groups.
For such groups, this will then also recover \cite{HainesRostami:Satake} after a suitable change of coefficients.

\begin{thm}
	Let \(\Gg/\Oo\) be as in \thref{thm.intro-Satake}, and \(q\) the residue cardinality of \(F\).
	Let \(\Hh_\Gg\) be the Hecke algebra of \(\Gg\), consisting of certain \(\IZ\)-valued functions on \(G(F)\).
	Then, there is a canonical isomorphism
	\[\IZ[\restrvinberg^I]^{c_{\sigma}(\widehat{G}^I)} \cong \Hh_{\Gg}.\]
\end{thm}

We refer to Section \ref{sect-Hecke} for the definitions of \(\Hh_\Gg\) and \(V_{\widehat{G}}\).
In \cite{Zhu:Integral}, two proofs are given for the above theorem in the unramified case: one using the classical Satake isomorphism, and one using the geometric Satake equivalence for unramified groups, by taking traces of Frobenii.
Our strategy roughly follows the latter, in the sense that it uses \thref{thm.intro-Satake}.
However, here the advantage of using motives over \(\ell\)-adic cohomology becomes clear.
Namely, in \cite[Proposition 18, Lemma 20]{Zhu:Integral}, certain subcategories on which the trace of Frobenius gives \(\IZ\)-valued functions are carefully singled out.
On the other hand, since we are using motives, it is clear that in our situation the trace of Frobenius gives (at least) \(\IQ\)-valued functions.
It is then easy to determine exactly when we get \(\IZ\)-valued functions.

For residually split groups, we can use \thref{thm.intro-Satake} to define generic spherical Hecke algebras, depending on some parameter \(\qq\).
Namely, there is a full subcategory of \(\MTM_{L^+\Gg}(\Gr_{\Gg},\IQ)\), closed under convolution, given by the anti-effective motives (\thref{def--anti-effective}), which is equivalent to representations (in \(\IQ\)-vector spaces) of \(\vinberg^I\).
We define \(\Hh_{\Gg}(\qq)\) as the Grothendieck ring of (the compact objects in) this full subcategory.
Specializing \(\qq\) to the residue cardinality of \(F\) then recovers the usual spherical Hecke algebra \(\Hh_\Gg\).
Note that this agrees with the definitions from \cite{PepinSchmidt:Generic,CassvdHScholbach:Geometric} for split groups, but it is in general subtle to define \(\Hh_{\Gg}(\qq)\).
We also show how this generic Hecke algebra fits into generic Satake and Bernstein isomorphisms (\thref{generic Bernstein}):

\begin{thm}
	There is a canonical isomorphism
	\[\Hh_{\Gg}(\qq) \cong R(\vinberg^I).\]
	Moreover, there is a morphism
	\[\Hh_{\Gg}(\qq) \to \Hh_{\Ii}(\qq),\]
	realizing the left hand side as the center of the right hand side.
\end{thm}

Here \(\Hh_{\Ii}(\qq)\) is the generic Iwahori-Hecke algebra, defined similarly as in \cite{Vigneras:ProI}.
For a more conceptual discussion about generic Hecke algebras at arbitrary parahoric level in the case of split groups in equal characteristic, we refer to \cite[§6]{CassvdHScholbach:Central}.
Finally, we mention that the generic Satake and Bernstein isomorphisms above open the door to extend the work of Pépin--Schmidt \cite{PepinSchmidt:Generic} to (residually split) ramified groups.

\subsection{Outline and strategy of proof}

Aside from many additional difficulties arising from the use of (integral) motives, our proof of \thref{thm.intro-Satake} follows the usual lines of geometric Satake.
Let us give a brief overview of the paper.

We start in Section \ref{sect-motives} by recalling the motivic theory we will use, and develop it further.
In particular, we define the notion of Artin-Tate motives we will need, explain how to equip it with a t-structure, and show t-exactness of certain realization functors.

Next, we recall the definitions and basic geometry of affine Grassmannians and affine flag varieties in Section \ref{Section--affine flag varieties}.
Most of the material is well known already, but we provide proofs when there is a lack of suitable references.

In Section \ref{sect-LSMV}, we generalize the methods of \cite{GaussentLittelmann:LS} to quasi-split groups over general nonarchimedean local fields. In particular, we prove Theorems \ref{thm.intro-decomposition} and \ref{thm.intro-character formula}.

With most of the necessary geometry of affine flag varieties available, we can now move towards the proof of \thref{thm.intro-Satake}.
We begin in Section \ref{sect-convolution} by considering the convolution, and show it preserves Artin-Tate motives and is t-exact.

Then we construct the constant term and fiber functors in Section \ref{sect-CT and F}. 
We also assert basic properties of these functors, such as preservation of Artin-Tate motives and t-exactness. 

In Section \ref{sect-tannakian}, we use a generalized Tannakian approach to show \(\MATM_{L^+\Gg}(\Gr_{\Gg})\) is equivalent to the category of comodules under some bialgebra in \(\MATM(\Spec k)\).

In order to show the bialgebra above is in fact a Hopf algebra, we would need a symmetry constraint for the convolution product.
In \cite{Zhu:Ramified}, this is obtained by considering a nearby cycles functor and the Satake equivalence for a certain unramified group.
Since we are primarily interested in local fields of mixed characteristic, the only nearby cycles available for our purposes is the one constructed by Anschütz--Gleason--Lourenço--Richarz \cite[§6]{AGLR:Local}, which makes it possible to relate our situation to the Satake equivalence from Fargues--Scholze \cite[§VI]{FarguesScholze:Geometrization}.
We thus describe this nearby cycles functor in the setting of \(\ell\)-adic étale cohomology in Section \ref{Sect--padic}, and we reduce motivic questions, such as commutativity of the bialgebra above, to this situation.
(During the refereeing process of this paper, a motivic version of the above Satake equivalence has appeared in \cite{Scholze:Geometrization}.)

We then conclude \thref{thm.intro-Satake} by relating the Hopf algebra arising from the Tannakian approach to the inertia-invariants of the Langlands dual group in Section \ref{sect-identification}.
Again, we make use of nearby cycles in case of \(\ell\)-adic étale cohomology, which helps us in determining the dual group integrally and motivically.

Finally, we provide number-theoretic applications in Section \ref{sect-Hecke}, by constructing integral and generic versions of the Satake and Bernstein isomorphisms for ramified groups.

\begin{rmk}
	After the first draft of this paper was finished, Bando uploaded a preprint \cite{Bando:Derived}, comparing categories of sheaves on mixed characteristic and equicharacteristic affine flag varieties.
	However, his methods are not immediately helpful for our purposes.
	For example, loc.~cit.~starts with a group over \(\Oo\pot{t}\), but, at least for wildly ramified groups, it is not clear how to extend the parahoric \(\Gg/\Oo\) to a suitable \(\Oo\pot{t}\)-group scheme.
	Namely, it can happen that Lourenço's construction \cite{Lourenco:Grassmanniennes} yields groups whose fiber over \(k\rpot{t}\) is not reductive.
	In fact, for wildly ramified odd unitary groups, it is not even known how to construct such lifts, but these are nevertheless still covered by our \thref{thm.intro-Satake}.
	Moreover, since we are working with motives, we would have to replace the notion of ULA sheaves (as in Hansen--Scholze \cite{HansenScholze:Relative}) by ULA motives (defined by Preis \cite{Preis:Motivic}).
	However, not all necessary properties of universal local acyclicity have been established in the motivic setting yet, as this would require conservativity results of motivic nearby cycles, which is still a conjecture.
\end{rmk}

\subsection{Notation}

Throughout this paper, we fix a prime \(p\), and we let \(k\) be either a finite field of characteristic \(p\), or an algebraic closure thereof.
Moreover, we fix a complete discretely valued field \(F\) with residue field \(k\), ring of integers \(\Oo\), and uniformizer \(\varpi\in \Oo\).
Let \(\breve{F}\) be the completion of the maximal unramified extension of \(F\), and \(\breve{\Oo}\) its ring of integers.
In particular, \(\breve{F}\) has residue field \(\overline{k}\).
We also denote by \(I\) the inertia group of \(F\), or equivalently, the absolute Galois group of \(\breve{F}\).

Since we will mostly work with perfect schemes and \(k\)-algebras, we denote by \(\Perf_k\) the category of all perfect \(k\)-schemes, and refer to \thref{defi-pfp} for the case where we want some finiteness conditions.
For such a perfect \(k\)-algebra \(R\), we have the ring of ramified Witt vectors \(W_\Oo(R):=W(R)\widehat{\otimes}_{W(k)} \Oo\), and its truncated version \(W_{\Oo,n}(R) := W_\Oo(R) \otimes_{\Oo} (\Oo/\varpi^{n+1})\) for \(n\geq 0\). 

In general, \(G\) will be a connected reductive \(F\)-group, in which we choose a maximal \(\breve{F}\)-split \(F\)-torus \(S\), with centralizer \(T\), which is a maximal torus.
If \(G\) is quasi-split, we let \(B=B^+\subseteq G\) be a rational Borel containing \(T\).
We denote the (perfect) pairing \(X^*(T)\times X_*(T)\to \IZ\) by \(\langle -,-\rangle\).
This is invariant for the inertia-action, and hence descends to a pairing \(X^*(T)\times X_*(T)_I\to \IZ\).
In Section \ref{sect-LSMV}, we will also need to use the similar pairing \(X^*(S)\times X_*(S)\to \IZ\), or more specifically, its rationalization.
These two pairings are not immediately compatible.
Although strictly speaking, it should always be clear which pairing is used, we will denote the latter one by \(\langle-,-\rangle_S\) for clarity.

\subsection{Acknowledgments}

First and foremost, I thank my PhD advisor Timo Richarz for all his help, advice, and patience, and for sending me a preliminary version of \cite{ALRR:Modular}.
Furthermore, it is a pleasure to thank Robert Cass, Rızacan Çiloğlu, Stéphane Gaussent, Tom Haines, João Lourenço, Cédric Pépin, Jakob Scholbach, Maarten Solleveld, Fabio Tanania and Xinwen Zhu for helpful discussions and email exchanges.
I also thank Alexis Bouthier, Timo Richarz, and the referee for very helpful comments on earlier versions of this paper.
Finally, I thank the Hausdorff center of Mathematics in Bonn for its hospitality during the trimester program \emph{The Arithmetic of the Langlands program}, where part of this work was carried out.
This work was supported by the European research council (ERC) under the European Union’s Horizon 2020 research
and innovation programme (grant agreement No 101002592), and by the Deutsche Forschungsgemeinschaft (DFG), via the TRR 326 \emph{Geometry and Arithmetic of Uniformized Structures} (project number 444845124), under Germany's Excellence Strategy – EXC-2047/1 – 390685813, and via the LOEWE professorship in Algebra via Timo Richarz.

\section{Motives}\label{sect-motives}

In this paper, we will use the theory of étale motives.
More specifically, we will consider the model defined via the h-topology, developed in \cite{CisinskiDeglise:Etale}.
We start by recalling and introducing the relevant notions.

\subsection{Motives on perfect schemes}

In \cite[§5.1]{CisinskiDeglise:Etale}, the authors construct a category of motives, by using the h-topology.
Although loc.~cit.~uses triangulated categories, it will be helpful for us to use \(\infty\)-categories instead, and we refer to \cite{Robalo:Theorie, Khan:Motivic, Preis:Motivic} for how to transfer the results of \cite{CisinskiDeglise:Etale} to the world of \(\infty\)-categories.
Thus, for a coefficient ring \(\Lambda\), we have a contravariant functor
\[\DM(-,\Lambda):=\DM_h(-,\Lambda)\colon (\Sch_k^{\ft})^{\op}\to \Pr_\Lambda^{\St}\colon \left(f\colon X\to Y\right)\mapsto \left(f^!\colon \DM(Y,\Lambda)\to \DM(X,\Lambda)\right)\]
from the category of finite type \(k\)-schemes to the \(\infty\)-category of stably presentably \(\Lambda\)-linear \(\infty\)-categories; where functors are automatically assumed to preserve colimits (we explain below why this is the case for \(f^!\) in our setting).

The motivic theory \(\DM\) satisfies the full six-functor formalism by \cite[Theorem 5.6.2]{CisinskiDeglise:Etale}.
In other words, for any morphism \(f\colon X\to Y\) of finite type \(k\)-schemes, we have functors \(f^!,f^*\colon \DM(Y,\Lambda)\to \DM(X,\Lambda)\), functors \(f_!,f_*\colon \DM(X,\Lambda)\to \DM(Y,\Lambda)\), and \(\DM(X,\Lambda)\) is closed symmetric monoidal. 
We will denote the monoidal unit by \(\unit\in \DM(X,\Lambda)\) (or by \(\unit_X\) to emphasize the scheme \(X\), or even \(\Lambda\) to emphasize the coefficients), and the internal Hom by \(\underline{\operatorname{Hom}}\) or \(\underline{\operatorname{Hom}}_X\).
These functors satisfy various adjunctions, compatibilities and other properties, such as base change, the projection formula, homotopy invariance, stability and localization; we refer to \cite[Theorem 2.4.50]{CisinskiDeglise:Triangulated} and \cite[Synopsis 2.1.1]{RicharzScholbach:Intersection} for a more detailed list. (We also note that by \cite[Proposition 2.1.14]{RicharzScholbach:Intersection}, we can drop the usual separatedness hypothesis for \(f^!\) and \(f_!\) to exist.)
Moreover, by our hypothesis on \(k\), the residue fields of any finite type \(k\)-scheme \(X\) have uniformly bounded étale cohomological dimension, so that \(\DM(X,\Lambda)\) is compactly generated by constructible objects by \cite[Theorem 5.2.4]{CisinskiDeglise:Etale}.
Hence, as \(f^*\) and \(f_!\) preserve compact objects for any map \(f\colon X\to Y\) of finite type \(k\)-schemes by \cite[Proposition 4.2.4 and Corollary 4.2.12]{CisinskiDeglise:Triangulated}, the four functors \(f^*,f_*,f^!,f_!\) preserve colimits.
Finally, there is the Tate twist \(\unit(1)\in \DM(X,\Lambda)\) as in \cite[§3.2.1]{CisinskiDeglise:Etale}, for which the functor
\[\DM(X,\Lambda)\to \DM(X,\Lambda)\colon M\mapsto M(1):= M\otimes \unit(1)\]
is an equivalence.
Hence, we can define \(\unit(m) := \unit(1)^{\otimes m}\) for any \(m\in \IZ\).

Since the mixed characteristic affine flag varieties we will define in Section \ref{Section--affine flag varieties} are only defined up to perfection, the functor \(\DM\) above is not yet suitable for our purposes; note that perfect schemes are rarely of finite type over \(k\).
Instead, following \cite{RicharzScholbach:Witt}, we will modify \(\DM\) to a functor out of the category of perfect schemes.

\begin{dfn}\thlabel{defi-pfp}
	Let \(f\colon X\to Y\) be a morphism of qcqs perfect \(k\)-schemes.
	We say \(f\) is \emph{perfectly of finite presentation}, or \emph{pfp}, if Zariski locally \(f\) looks like the perfection of a morphism of finite presentation.
	A qcqs \(k\)-perfect scheme \(X\) is said to be pfp if the structure morphism \(X\to \Spec k\) is pfp; note that any morphism between pfp schemes is itself pfp.
	We denote by \(\Sch_k^{\pfp}\) the full subcategory of \(k\)-schemes spanned by the qcqs pfp schemes.
\end{dfn}

Recall that for any \(k\)-scheme \(X\), we can functorially define its (limit) perfection \(X^{\perf}\) as in \cite[A.1.2]{Zhu:Affine}, by taking the limit along the Frobenius of \(X\).
Assume from now on that \(p\) is invertible in \(\Lambda\); by \cite[Corollary A.3.3]{CisinskiDeglise:Etale} this is not a restriction.
The following generalizes \cite[Theorem 2.10]{RicharzScholbach:Witt} to more general coefficients rings.

\begin{prop}
	The functor \(\DM(-,\Lambda)\colon (\Sch_k^{\ft})^{\op}\to \Pr_\Lambda^{\St}\) factors uniquely through the perfection functor.
	This yields a motivic theory
	\[\DM(-,\Lambda)\colon (\Sch_k^{\pfp})^{\op}\to \Pr_\Lambda^{\St},\]
	satisfying a six-functor formalism.
\end{prop}
\begin{proof}
	By \cite[Lemma 2.9]{RicharzScholbach:Witt}, the category \(\Sch_k^{\pfp}\) is equivalent to the localization of \(\Sch_k^{\ft}\) at the class of universal homeomorphisms.
	To get the desired factorization it hence suffices to show that universal homeomorphisms induce equivalences on \(\DM(-,\Lambda)\), which is shown in \cite[Corollary 2.1.5]{ElmantoKhan:Perfection}.
	The fact that \(\DM(-,\Lambda)\) on pfp schemes satisfies a six-functor formalism then follows from the similar fact on finite type schemes, and we refer to the proof of \cite[Theorem 2.10]{RicharzScholbach:Witt} for details.
\end{proof}

From now on, we will mostly consider the case \(\Lambda=\IZ[\frac{1}{p}]\), and write \(\DM(-)= \DM(-,\IZ[\frac{1}{p}])\).

\begin{rmk}\thlabel{Etale realization}
	Let \(\ell\neq p\) be a prime, and \(X\in \Sch_k^{\ft}\) a scheme.
	In \cite[§7.2]{CisinskiDeglise:Etale}, Cisinski--Déglise construct an étale (with \(\IZ_\ell\)-coefficients) realization functor out of \(\DM(X)\).
	A reformulation of this realization in terms of the pro-étale site from \cite{BhattScholze:Proetale} was given in \cite[§1.3]{Ruimy:IntegralI}, and this is the version that we will use.
	Indeed, topological invariance of the pro-étale site \cite[Lemma 5.4.2]{BhattScholze:Proetale} implies that the étale realization factors through a natural transformation \(\rho_{\ell}\colon \DM(-)\to D(-,\IZ_\ell)\) of functors out of \(\Sch_k^{\pfp}\), where \(D(X,\IZ_\ell)\) is the unbounded derived category of \(\IZ_\ell\)-sheaves on the pro-étale site of \(X\).
	This realization functor is compatible with the six functor formalism.
\end{rmk}

\subsection{Motives on perfect prestacks}

In order to define equivariant motives on affine flag varieties, which are only ind-schemes, we need to further extend \(\DM\), as in \cite[§2]{RicharzScholbach:Intersection}.

Let \(\PreStk_k:=\Fun\left((\Sch_k^{\aff})^{\op},\Ani\right)\) be the category of \(k\)-prestacks, i.e., presheaves on affine (but not necessarily of finite type) \(k\)-schemes with values in the \(\infty\)-category of anima.
This category contains the categories of schemes, ind-schemes, or algebraic stacks fully faithfully.
By using Kan extensions as in \cite[Definition 2.2.1]{RicharzScholbach:Intersection}, we can extend \(\DM\) to a functor
\[\DM\colon \PreStk_k^{\op}\to \Pr_{\IZ[\frac{1}{p}]}^{\St}\colon \left(f\colon X\to Y\right) \mapsto \left(f^!\colon \DM(Y)\to \DM(X)\right).\]
We note that in order to consider Kan extensions as in \cite{RicharzScholbach:Intersection}, we have to fix a regular cardinal \(\kappa\), and only consider affine schemes obtained as \(\kappa\)-filtered limits of affine schemes that are of finite type over \(\Spec k\).
However, the choice of \(\kappa\) has no influence on \(\DM\), as long as \(\kappa\) is large enough so that we can consider all affine schemes of interest.
Hence, we fix once and for all such a large enough \(\kappa\), and silently forget about it.

Although for these general prestacks, \(\DM\) does not satisfy the full six-functor formalism, and the functors \(f^*\), \(f_*\), and \(f_!\) do not exist in general, they do exist in certain situations, cf.~\cite[§2.3-2.4]{RicharzScholbach:Intersection}.
For example, if \(f\) is a morphism of ind-schemes, then \(f_!\) and \(f_*\) do exist, and \(f_!\) is left adjoint to \(f^!\). 
If \(f\) is also schematic, then \(f^*\) also exists and is left adjoint to \(f_*\).
Moreover, if \(f\) is a morphism of ind-Artin stacks that are ind-(locally of finite type) then \(f_!\) exists and is left adjoint to \(f^!\) \cite[Proposition 2.3.3]{RicharzScholbach:Intersection}.
Finally, assume \(G\) is a pro-smooth pro-algebraic group acting on ind-schemes \(X\) and \(Y\). 
Then for an equivariant map \(f\colon X\to Y\) giving rise to \(\overline{f}\colon G\backslash X\to G\backslash Y\), we can descend \(f_!\) and \(f_*\) (and possibly \(f^*\)) to functors \(\overline{f}_!\) and \(\overline{f}_*\) (and possibly \(\overline{f}^*\)), cf.~\cite[Lemma 2.2.9]{RicharzScholbach:Intersection}.

Now, as in the previous section, we will only want to define certain prestacks up to perfection.
This leads to the notion of perfect prestacks, as in \cite[§2.2]{RicharzScholbach:Witt}.

\begin{dfn}
	Let \(\AffSch_k^{\perf}\) denote the category of perfect affine \(k\)-schemes, not necessarily pfp.
	Then the category of \emph{perfect prestacks} is the functor category \(\PreStk_k^{\perf}:=\Fun\left( (\AffSch_k^{\perf})^{\op},\Ani\right)\).
\end{dfn}

There is a natural restriction functor \(\res\colon \PreStk_k \to \PreStk_k^{\perf}\), preserving all limits and colimits.
Moreover, by \cite[Lemma 2.2]{RicharzScholbach:Witt}, \(\res\) admits a fully faithful left adjoint \(\incl\colon \PreStk_k^{\perf} \to \PreStk_k\).
Note that the endofunctor \(\sigma\) on \(\AffSch_k\) given by Frobenius induces an endofunctor \(\sigma\) of \(\PreStk_k\), which is an equivalence when restricted to objects in the essential image of \(\incl\).

We can then define the \emph{colimit perfection} of a prestack \(X\) as \(X^{\perf}:= \incl \circ \res (X)\).
On the other hand, there is the \emph{limit perfection} \(\varprojlim_\sigma X := \varprojlim(\ldots \xrightarrow{\sigma} X \xrightarrow{\sigma} X \xrightarrow{\sigma} X)\).
These two notions of perfections do not agree in general, although they do agree for affine schemes, and hence Zariski locally for arbitrary schemes.
However, by \cite[Corollary 2.6]{RicharzScholbach:Witt}, there are canonical equivalences \(\DM(X)\cong \DM(\varprojlim_\sigma X) \cong \DM(X^{\perf})\) for any prestack \(X\).
In particular, as in the previous section, we can view \(\DM\) as a functor of perfect prestacks without ambiguity, compatibly with the extension of \(\DM\) to perfect schemes, and satisfying similar properties as \(\DM\) on all prestacks.

\subsection{Mixed Tate motives}

Later on, to construct a motivic Satake equivalence, we will need to consider ``perverse" motives, arising as the heart of a t-structure.
As the existence of the motivic t-structure is part of the standard conjectures on algebraic cycles, we need to restrict our categories of motives, following \cite{SoergelWendt:Perverse, RicharzScholbach:Intersection}.

Recall from \cite[Definition 3.2]{RicharzScholbach:Witt} that a \emph{perfect cell} is a \(k\)-scheme isomorphic to the perfection of \(\IA^n_k\times \IG_{m,k}^r\), and a pfp scheme is \emph{perfectly cellular} if it admits a smooth model and a stratification into perfect cells.
Moreover, a \emph{perfectly cellular stratified} or \emph{pcs} ind-scheme is a stratified ind-(pfp scheme) \(X=\bigsqcup_{w\in W} X_w\) for which each stratum \(X_w\) is a perfectly cellular \(k\)-scheme.

\begin{dfn}
	\begin{enumerate}
		\item For a perfectly cellular \(X\in \Sch_k^{\pfp}\), the category \(\DTM(X)\subseteq \DM(X)\) of \emph{Tate motives} on \(X\) is the full stable presentable subcategory generated under colimits and extensions by \(\unit(m)\), for \(m\in \IZ\).
		\item Let \(X^\dagger = X_{w\in W}\xrightarrow{\bigsqcup \iota_w} X\) be a pcs ind-scheme. 
		We say this stratification is \emph{Whitney--Tate} if for any \(v,w\), the motive \(\iota_v^!\iota_{w,!}\unit\in \DM(X_v)\) is Tate.
		As in \cite[3.1.11]{RicharzScholbach:Intersection}, this is equivalent to requiring each \(\iota_v^*\iota_{w,*}\unit\) to be Tate.
		\item For a Whitney--Tate pcs ind-scheme \(X\) as above, we define the category \(\DTM(X,X^\dagger)\subseteq \DM(X)\) of \emph{stratified Tate motives} as those motives \(M\in \DM(X)\) for which each \(\iota_w^!(M)\) is Tate.
		If \(M\) is bounded, i.e., it is supported on some subscheme of \(X\), then this is equivalent to \(\iota_w^*(M) \in \DTM(X_w)\) for all \(w\in W\) \cite[3.1.11]{RicharzScholbach:Intersection}.
	\end{enumerate}
\end{dfn}

If the stratification is clear from the context, we will also write \(\DTM(X):=\DTM(X,X^\dagger)\).

\begin{lem}\thlabel{BS-vanishing}
	Let \(X/k\) be a perfectly cellular scheme. Then we have \(\Hom_{\DM(X)}(\unit,\unit(n)[m]) = 0\) if either \(m<0\), or \(m=0\) and \(n\neq 0\).
\end{lem}
\begin{proof}
	First, assume \(X=\Spec k\).
	Rationally, the desired vanishing holds by Quillen's computation that the K-theory of finite fields is torsion \cite{Quillen:Cohomology}, and the relation between algebraic K-theory with rational motivic cohomology (compare \cite[Remark 3.10]{SoergelWendt:Perverse}).
	For torsion coefficients, let \(a\geq 1\). 
	Then we have \(\Hom_{\DTM(\Spec k,\IZ/a)}(\IZ/a,\IZ/a(n)[m]) = \mathrm{H}^m_{\et}(\Spec k,\mu_a^{\otimes n}) = 0\) for \(m<0\).
	But this implies that \(\Hom_{\DTM(\Spec k)}(\unit,\unit(n)[m])\to \Hom_{\DTM(\Spec k,\IQ)}(\IQ,\IQ(n)[m])\) is an isomorphism for \(m<0\), and injective for \(m=0\). 
	Hence, the lemma follows from the rational case.
	
	The claim for general \(X\) then follows as in \cite[Proposition 3.9]{SoergelWendt:Perverse}, using that \(\DTM(X)\) is generated by \(\unit(n)[m]\) under colimits and extension.
\end{proof}

\begin{prop}\thlabel{defi-t-structure}
	\begin{enumerate}
		\item If \(X\) is a perfectly cellular \(k\)-scheme, the category \(\DTM(X)\) of (unstratified) Tate motives on \(X\) admits a t-structure for which \(\DTM^{\leq 0}(X)\) is generated under colimits and extensions by the shifted Tate twists \(\unit(n)[\dim X]\) for \(n\in \IZ\).
		This t-structure is non-degenerate, and its heart is compactly generated by \(\{\unit(n)[\dim X]\mid n\in \IZ\}\).
		\item Let \(X^\dagger=\coprod_{w\in W} X_w\xrightarrow{\iota} X\) be a Whitney--Tate pcs ind-scheme.
		Then we can define a non-degenerate t-structure on the category \(\DTM(X)=\DTM(X,X^\dagger)\) of stratified Tate motives by 
		\[\DTM^{\leq 0} := \{M\in \DTM(X)\mid \iota_w^*\in \DTM^{\leq 0}(X_w) \text{ for all } w\in W\}\]
		\[\DTM^{\geq 0} := \{M\in \DTM(X)\mid \iota_w^!\in \DTM^{\geq 0}(X_w) \text{ for all } w\in W\}.\]
	\end{enumerate}
	Both t-structures are compactly generated, and their connective and coconnective parts are closed under filtered colimits.
	We denote the heart of these t-structures by \(\MTM(X) = \DTM^{\heartsuit}(X)\), and call them the category of \emph{(stratified) mixed Tate motives}.
\end{prop}
\begin{proof}
	(1) The t-structure exists by \cite[Proposition 1.4.4.11]{Lurie:HigherAlgebra}.
	Since \(\DTM^{\leq 0}(X)\) is generated by the compact Tate twists, the t-structure is compactly generated, and \(\DTM^{\geq 0}(X)\) is closed under filtered colimits.
	By \thref{BS-vanishing}, the shifted Tate twists lie in the heart of the t-structure.
	The fact that these compactly generate the heart can be shown as in \cite[Lemma 2.13]{CassvdHScholbach:Geometric}.
	Since the Tate twists (and its shifts) moreover generate the whole category \(\DTM(X)\), this t-structure is non-degenerate.
	
	(2) The existence of the t-structure follows from (1) and recollement \cite[Theorem 1.4.10]{BBD:Faisceaux}, the necessary axioms for which can be verified as in \cite[Theorem 10.3]{SoergelWendt:Perverse}.
	Since pullback along a stratification is conservative, the other properties also follow from similar properties for (1).
\end{proof}

The following result gives us some control on the heart of the above t-structure, and follows immediately from \thref{BS-vanishing}.
We denote by \(\grZMod\) the category of \(\IZ\)-graded \(\IZ[\frac{1}{p}]\)-modules.

\begin{cor}\thlabel{graded to motives}
	There is a natural faithful symmetric monoidal functor \(\grZMod \to \MTM(\Spec k)\), where the grading corresponds to the Tate twist in \(\MTM(\Spec k)\).
	When restricted to ind-free modules, this functor is moreover fully faithful.
\end{cor}

\begin{rmk}\thlabel{conservativity of etale realization}
	One of the reasons we use étale motives instead of Nisnevich motives (which we can also define for pfp schemes, as long as \(p\) is invertible in \(\Lambda\)), is that the realization functors from \thref{Etale realization} are jointly conservative when restricted to Artin-Tate motives.
	Indeed, let \(X\) be a Whitney--Tate pcs scheme. 
	By \cite[Proposition 5.4.12]{CisinskiDeglise:Etale}, it suffices to show the étale realizations \(\DTM(X,\IF_\ell) \to D_{\et}(X,\IF_\ell)\) and \(\DTM(X,\IQ) \to D_{\et}(X,\IQ_\ell)\) are conservative for \(\ell\neq p\).
	The former holds since torsion étale motives agree with étale cohomology \cite[Corollary 5.5.4]{CisinskiDeglise:Etale}, while the latter is shown in \cite[Proposition 3.5]{RicharzScholbach:Witt}.
	Along with the following proposition, this will allow us to deduce t-exactness results by using these realization functors.
\end{rmk}

\begin{prop}\thlabel{texactness of realization}
	For any prime \(\ell\neq p\) and perfectly cellular scheme \(X\), restricting the étale realization functor to Tate motives gives a t-exact functor \(\rho_{\ell}\colon \DTM(X)\to D(X,\IZ_\ell)\), for the perverse t-structure on the target.
\end{prop}
\begin{proof}
	Let us denote the connective and coconnective parts of the perverse t-structure on \(D(X,\IZ_\ell)\) by \(D^{\leq 0}(X,\IZ_\ell)\) and \(D^{\geq 0}(X,\IZ_\ell)\).
	We will deduce the proposition from an analogous result for rational and torsion coefficients respectively, similarly to \cite[Proposition 4.2.5]{Ruimy:IntegralI}.
	Since \(\DTM^{\leq 0}(X)\) is generated by \(\unit(m)[\dim(X)]\), which get mapped to \(\IZ_\ell(m)[\dim(X)]\in D^{\leq 0}(X,\IZ_\ell)\), it is clear that \(\rho_{\ell}\) is right t-exact.
	
	For the left t-exactness, let \(M\in \DTM^{\geq 0}(X)\), so that \(\operatorname{Hom}(\unit(m)[\dim(X)],M)=0\) for \(m\in \IZ\).
	Then \(M\otimes \IQ \in \DTM^{\geq 0}(X,\IQ)\) by \cite[Corollary 5.4.11]{CisinskiDeglise:Etale}.
	Moreover, \(\rho_{\ell}(M\otimes \IQ)\in D^{\geq 0}(X,\IQ_\ell)\) by \cite[Lemma 3.2.8]{RicharzScholbach:Intersection}, hence also \(\rho_{\ell}(M\otimes \IQ)\in D^{\geq 0}(X,\IZ_\ell)\).
	
	On the torsion side, we have an exact triangle \(\IZ[\frac{1}{p}] \to \IQ \to \IQ/\IZ[\frac{1}{p}],\) and hence an exact triangle
	\(M\otimes \IQ/\IZ[\frac{1}{p}][-1] \to M \to M\otimes \IQ\).
	Since \(\DTM^{\ge 0}(X)\) is closed under limits, it contains \(M\otimes \IQ/\IZ[\frac{1}{p}][-1]\).
	Consider the full subcategory of torsion étale motives \(\DM_{\tors}(X)\subseteq \DM(X)\), consisting of those \(N\) such that \(N\otimes \IQ = 0\).
	Similarly, we can define the derived category of torsion (pro-)étale sheaves \(D_{\tors}(X,\IZ_{\ell'})\) for \(\ell'\neq p\) prime, the product of which is equivalent to \(\DM_{\tors}(X)\) by \cite[Corollary 2.1.4]{Ruimy:IntegralI}.
	In particular, \(\DM_{\tors}(X)\) is equipped with a unique t-structure, compatible with the perverse t-structure on each \(D_{\tors}(X,\IZ_{\ell'})\).
	This clearly restricts to a t-structure on \(\DTM_{\tors}(X)\), which agrees with the restriction of the t-structure on \(\DTM(X)\) from \thref{defi-t-structure}.
	
	Hence, \(\rho_{\ell}(M\otimes \IQ/\IZ[\frac{1}{p}][-1]) \in D_{\et,\tors}^{\geq 0}(X,\IZ_{\ell}) \subseteq D_{\et}^{\geq 0}(X,\IZ_\ell)\).
	Since \(\rho_{\ell}(M)\) is an extension of \(\rho_{\ell}(M\otimes \IQ/\IZ[\frac{1}{p}][-1])\) and \(\rho_{\ell}(M\otimes \IQ)\), both of which lie in \(D_{\et}^{\geq 0}(X,\IZ_\ell)\), we conclude that \(\rho_{\ell}(M)\in D_{\et}^{\geq 0}(X,\IZ_\ell)\), so that \(\rho_{\ell}\) is t-exact.
\end{proof}

In order to construct a motivic Satake equivalence, we also need to define equivariant mixed Tate motives.

\begin{dfn}
	Let \(X\) be a Whitney--Tate pcs ind-scheme, equipped with an action of a group prestack \(G\).
	Then the category of \emph{\(G\)-equivariant stratified Tate motives on \(X\)} is the full subcategory \(\DTM_G(X)\subseteq \DM(G\backslash X)\) of objects whose !-pullback to \(\DM(X)\) lies in \(\DTM(X)\).
	In other words, it is the pullback \(\DTM_G(X) = \DTM(X) \times_{\DM(X)} \DM(G\backslash X)\).
\end{dfn}

Note that \(\DTM_G(X)\) really depends on the \(G\)-action on \(X\), and not just on the prestack quotient \(G\backslash X\).
Recall that if \(G\) acts on a stratified ind-scheme \(X^\dagger \to X\), we say the \(G\)-action is \emph{stratified} if it restricts to a \(G\)-action on \(X^\dagger\).
Moreover, as in \cite[Definition A.1]{RicharzScholbach:Intersection}, an action of a pro-algebraic group \(G=\varprojlim_i G_i\) on an ind-scheme \(X\) is said to be \emph{admissible} if there exists a \(G\)-stable presentation of \(X\), for which the G-action factors through a finite type quotient at each step.

\begin{prop}\thlabel{Existence MTM}
	Let \(G=\varprojlim_i G_i\) be the perfection of a pro-algebraic group scheme, and assume each \(G_i\) is perfectly cellular and each \(\ker(G\to G_i)\) is the perfection of a split pro-unipotent group scheme.
	Let \(X^\dagger \to X = \varinjlim_i X_i\) be a Whitney--Tate pcs ind-scheme, equipped with a stratified admissible \(G\)-action.
	Then, the pullbacks
	\[\DTM_G^{\leq 0}(X) = \DTM^{\leq 0}(X) \times_{\DM(X)} \DM(G\backslash X)\]
	\[\DTM_G^{\geq 0}(X) = \DTM^{\geq 0}(X) \times_{\DM(X)} \DM(G\backslash X)\]
	define a non-degenerate t-structure on \(\DTM_G(X)\).
	Its heart is denoted by \(\MTM_G(X)\), and called the category of \emph{\(G\)-equivariant stratified mixed Tate motives} on \(X\).
\end{prop}
\begin{proof}
	This can be proven as \cite[Proposition 3.2.15]{RicharzScholbach:Intersection}, cf.~also \cite[Theorem 4.6]{RicharzScholbach:Witt}.
\end{proof}

\subsection{Artin-Tate motives}\label{sect-MATM}

For split reductive groups, motivic Satake equivalences have already been constructed in \cite{RicharzScholbach:Satake,RicharzScholbach:Witt,CassvdHScholbach:Geometric}, which all use categories of mixed Tate motives.
However, as mentioned in \cite[Remark 5.10]{RicharzScholbach:Witt}, this will not suffice when working with more general reductive groups.
In that case, we need to consider \emph{Artin-Tate} motives, which are those motives that become Tate after a suitable finite étale base change.
For a field extension \(k'/k\) and a (perfect) prestack \(X\), we will write \(X_{k'}:=X\times_{\Spec k} \Spec k'\).

\begin{dfn}\thlabel{defi Artin-Tate}
	Fix some field extension \(k'/k\), which is either finite (automatically étale) or an algebraic closure.
	For a perfect scheme \(X\) such that \(X_{k'}\) is perfectly cellular, the category of Artin-Tate motives on \(X\) is the full subcategory \(\DATM(X)\subseteq \DM(X)\) of motives which become Tate after base change to \(X_{k'}\).
	
	Similarly, for an ind-scheme \(X\) and a Whitney--Tate stratification \(X_{k'}^\dagger = \coprod_w (X_{k'})_w\to X_{k'}\) into perfectly cellular schemes, we define \(\DATM(X, X_{k'}^\dagger)\subseteq \DM(X)\) as the full subcategory of objects whose pullback to \(X_{k'}\) lies in \(\DTM(X_{k'},X_{k'}^\dagger)\).
	This is called the category of stratified Artin-Tate motives.
	
	Finally, for an ind-scheme \(X\) as above equipped with an action of a group prestack \(G\), one defines the category of \(G\)-equivariant stratified Artin-Tate motives \(\DATM_{G}(X,X_{k'}^\dagger)\) as the full subcategory of \(\DM(G\backslash X)\) consisting of motives whose !-pullback to \(X\) is stratified Artin-Tate.
	In other words, it is the homotopy pullback
	\[\DATM_{G}(X,X_{k'}^\dagger):= \DATM(X,X_{k'}^\dagger)\times_{\DM(X)} \DM(G\backslash X).\]
\end{dfn}

If \(X_{k'}^\dagger = \coprod_w (X_{k'})_w\to X_{k'}\) is a Whitney--Tate stratification, we will also call it an \emph{Artin--Whitney--Tate} stratification of \(X\).
Moreover, if the field \(k'\) and the stratification \(X_{k'}^\dagger\to X_{k'}\) are clear, we will usually write \(\DATM(X)=\DATM(X,X_{k'}^\dagger)\), and similarly for equivariant Artin-Tate motives.

\begin{rmk}\thlabel{remarks about Artin-Tate}
	Let \(X\) and \(X_{k'}^\dagger\) be as above.
	\begin{enumerate}
		\item If \(k=k'\), Artin-Tateness agrees with Tateness.
		\item If \(k'/k\) is not finite, one has to be careful about applying the six functor formalism.
		For us, the only functors for non pfp morphisms we will need are the pullback functors along base changes of \(\Spec k'\to \Spec k\).
		Then the *- and !-pullback agree, and they exist by the Kan extension process.
		\item Since \(\DM(X)\) is compactly generated, the same holds for \(\DATM(X)\), as the full subcategory generated by certain compact objects.
		Let us denote the compact objects by \(\DATM(X)^{\comp}\).
		\item By continuity as in \cite[Theorem 6.3.9]{CisinskiDeglise:Etale}, any compact Artin-Tate motive becomes Tate after base change along a finite extension of \(k\).
		Consequently, \(\DATM(X)^{\comp}\) is equivalent to the category \(\DTM_{\Gal(k'/k)}(X_{k'})^{\comp}\) of compact Tate motives on \(X_{k'}\) equipped with a continuous (i.e., which factors through a finite quotient) \(\Gal(k'/k)\)-action.
		\item\label{etale descent} For not necessarily compact objects, we have an equivalence
		\[\DATM(X,X_{k'}^\dagger) \cong\varinjlim_{k''} \DATM(X,X_{k''}^\dagger)\cong \varinjlim_{k''} \DTM_{\Gal(k''/k)}(X_{k''},X_{k''}^\dagger),\]
		where \(k''\) ranges over all finite subextensions \(k\subseteq k''\subseteq k'\) over which \(X_{k'}^\dagger\) is defined, and the colimit is taken in \(\Pr_{\IZ[\frac{1}{p}]}^{\St}\).
		This category is again compactly generated, with compact objects the colimit (in \(\infty\)-categories) \(\varinjlim_{k''} \DATM(X,X_{k''}^\dagger)^{\comp}\).
		\item When working with Nisnevich motives, where pullback along finite étale covers is not necessarily conservative, one should \emph{define} Artin-Tate motives as Galois-equivariant Tate motives on a suitable base change.
	\end{enumerate}
\end{rmk}

It remains to show that for a Whitney--Tate stratified ind-scheme, the category of (equivariant) Artin-Tate motives also admits a t-structure.

\begin{prop}
	Fix some \(k'/k\) as above. 
	Let \(X\) be a Artin--Whitney--Tate pcs ind-scheme equipped with an action of some strictly pro-algebraic group \(G\) satisfying the same assumptions as in \thref{Existence MTM}.
	Then \(\DATM_{G}(X)\) admits a t-structure for which the forgetful functor \(\DATM_{G}(X)\to \DTM(X_{k'})\) is t-exact.
\end{prop}
\begin{proof}
	This follows from \thref{Existence MTM} and \thref{remarks about Artin-Tate} \eqref{etale descent}, as well as \cite[Lemma 3.2.18]{RicharzScholbach:Intersection}, applied to the action of \(G_{k'}\rtimes \Gamma\) on \(X_{k'}\).
\end{proof}

\begin{rmk}
	If \(X\) is an Artin--Whitney--Tate pcs ind-scheme, the tensor product \(\otimes\) on \(\DM(X)\) preserves the subcategories \(\DATM^{\leq 0}(X) \subseteq \DATM(X) \subseteq \DM(X)\).
	However, it does not preserve \(\DATM^{\geq 0}(X)\) in general.
	Hence, to get a symmetric monoidal structure on \(\MATM(X)\), we need to truncate by using \(\pH^0(-\otimes -)\).
	We still denote the resulting functor by \(- \otimes - \colon \MATM(X) \times \MATM(X) \to \MATM(X)\).
\end{rmk}

\subsection{Anti-effective motives}

Let us conclude this section about motives by recalling \emph{anti-effective} motives.
These were used in \cite{CassvdHScholbach:Geometric} to relate the Vinberg monoid to the motivic Satake equivalence for split groups, and we will generalize this relation to the ramified case in \thref{Anti-effective equivalence}.
Roughly, these were the motives that were ``non-positively" Tate twisted, as opposed to the effective motives.
However, since we are using étale motives, for which the Tate twist can be trivialized with torsion coefficients (up to adding certain roots of unity), this definition does not make sense with integral or torsion coefficients.
Hence, we will only use rational coefficients when considering anti-effective motives.
(Although we had only constructed a t-structure for \(\IZ[\frac{1}{p}]\)-coefficients, the construction also works for \(\IQ\)-coefficients.)

\begin{dfn}\thlabel{def--anti-effective}
	For a perfectly cellular \(X\in \Sch_k^{\pfp}\), the category \(\DTM(X)^{\anti}\subseteq \DTM(X,\IQ)\) of anti-effective Tate motives is defined as be the full stable presentable category generated under colimits and extensions by \(\unit(m)=\IQ(m)\), for \(m\leq 0\).
	More generally, if \(k'/k\) is a finite extension or an algebraic closure such that \(X_{k'}\) is perfectly cellular, we define \(\DATM(X)^{\anti}\subseteq \DATM(X,\IQ)\) as those motives which become anti-effective Tate after base change to \(k'\).
	
	If \(X\) is an ind-scheme, a perfectly cellular stratification \(\iota \colon X_{k'}^\dagger = \coprod_w (X_{k'})_w\to X_{k'}\) is said to be \emph{anti-effective Artin--Whitney--Tate} if for any \(v,w\), we have \(\iota_v^*\iota_{w,*} \unit \in \DATM((X_{k'})_v)^{\anti}\).
	
	For a field extension \(k'/k\) as above, and an anti-effective Artin--Whitney--Tate pcs ind-scheme \(X\), we let \(\DATM(X)^{\anti}\subseteq \DATM(X,\IQ)\) and \(\MATM(X)^{\anti}\subseteq \MATM(X,\IQ)\) be the full subcategories consisting of those motives which *-pullback to anti-effective Tate motives on each stratum of \(X_{k'}\).
	
	Similarly, we can define \emph{equivariant} anti-effective (mixed) Artin-Tate motives.
\end{dfn}

\begin{nota}
	Since anti-effective motives are only defined for rational coefficients, we will write \(\DTM(X)^{\xanti}\) to denote either \(\DTM(X,\IZ[\frac{1}{p}])\) or \(\DTM(X)^{\anti} = \DTM(X,\IQ)^{\anti}\).
	We will use similar notation for equivariant, mixed, and/or Artin-Tate motives.
\end{nota}

Let us recall the following alternative characterization of anti-effective motives from \cite[Lemma 2.18]{CassvdHScholbach:Geometric}.

\begin{lem}\thlabel{characterization of anti-effective}
	Let \(X/k\) be a perfectly cellular scheme. Then we have
	\[\DTM(X)^{\anti} = \{\Ff\in \DTM(X,\IQ)\mid \Maps_{\DTM(X)}(\unit(n),\Ff) = 0 \text{ for } n\geq 1\},\]
	\[\MTM(X)^{\anti} = \{\Ff\in \MTM(X,\IQ)\mid \Hom_{\MTM(X)}(\unit(n)[\dim X]),\Ff) = 0 \text{ for } n\geq 1\},\]
	where we consider Tate motives for the trivial stratification of \(X\).
\end{lem}

In particular, if \(X\) and \(k'\) are such that \(X_{k'}\) is perfectly cellular, the t-structure on \(\DATM(X,\IQ)\) restricts to a t-structure on \(\DATM(X)^{\anti}\), with heart \(\MATM(X)^{\anti}\).
The same applies for ind-schemes with an anti-effective Artin--Whitney--Tate perfectly cellular stratification

\section{Geometry of affine flag varieties}
\label{Section--affine flag varieties}

Recall that \(F\) was a complete discretely valued field with residue field \(k\).
In this section, we recall the basic definitions and geometry of partial affine flag varieties associated to parahoric integral models of reductive \(F\)-groups.
As our main interest is the case of a mixed characteristic field \(F\), we will only define these as ind-perfect schemes as in \cite{Zhu:Affine} (also for \(F\) of equal characteristic, for the sake of uniformity).

\subsection{Affine flag varieties}

Before we specialize to the case of reductive groups, let us consider a more general situation.
Let \(\Hh/\Oo\) be a smooth affine group scheme, with generic fiber \(H/F\).

\begin{dfn}
	\begin{enumerate}
		\item The \emph{loop group} of \(H\) is the functor
		\[LH\colon (\AffSch_k^{\perf})^{\op} \to \Grp\colon \Spec R\mapsto H(W_{\Oo}(R) \otimes_{\Oo} F).\]
		\item The \emph{positive loop group} of \(\Hh\) is the functor
		\[L^+\Hh\colon (\AffSch_k^{\perf})^{\op} \to \Grp\colon \Spec R\mapsto \Hh(W_{\Oo}(R)).\]
		\item The \emph{affine Grassmannian} \(\Gr_{\Hh}\) of \(\Hh\) is the étale sheafification of \(LH/L^+\Hh\).
	\end{enumerate}
\end{dfn}

By \cite{PappasRapoport:Twisted, Zhu:Affine, BhattScholze:Projectivity}, \(LH\) is representable by a perfect ind-scheme, \(L^+\Hh\) by a pro-(perfectly smooth) affine scheme (usually not perfectly of finite type), and \(\Gr_{\Hh}\) by an ind-(perfect quasi-projective) scheme.
We denote the structure map by \(\pi_{\Hh}\colon \Gr_{\Hh} \to \Spec k\).

\begin{nota}
	For any \(n\geq 0\), let \(L^n\Hh\colon (\AffSch_k^{\perf})^{\op}\to \Grp\colon \Spec R\mapsto \Hh(W_{\Oo,n}(R))\).
	This is a pfp perfectly smooth quotient of \(L^+\Hh\), and we have \(L^+\Hh \cong \varprojlim_n L^n\Hh\).
	In particular, if \(L^+\Hh\) acts on a pfp scheme \(X\), this actions factors through \(L^n\Hh\) for some \(n\geq 0\).
	Note that \(L^0\Hh\) is the special fiber of \(\Hh\).
	Finally, we let \(L^{>n}\Hh := \ker(L^+\Hh \to L^n\Hh)\).
\end{nota}
 
Now, let \(G/F\) be a reductive group, and \(\Gg/\Oo\) a parahoric model.
In this case we will write \(\Fl_{\Gg} := (LG/L^+\Gg)^{\et}\) and call it the \emph{(partial) affine flag variety}.
Instead, we will reserve the notation \(\Gr_{\Gg}\) for the case where \(\Gg\) is very special parahoric, and call it the \emph{(twisted) affine Grassmannian}.
Recall that \(\Gg\) is \emph{special} if it corresponds to a facet \(\ff\) in the Bruhat-Tits building \(\buil(G,F)\) of \(G\), which is contained in an apartment in which each wall is parallel to a wall containing \(\ff\).
Moreover, \(\Gg\) is called \emph{very special} if it remains special after any unramified base change.
On the other extreme, if \(\Gg\) is an Iwahori model, we call \(\Fl_{\Gg}\) the \emph{full affine flag variety}.
By \cite{PappasRapoport:Twisted, BhattScholze:Projectivity}, the affine flag variety of any parahoric model is ind-perfect projective.
The following result is well known, and we refer to \cite[Proposition 1.21]{Zhu:Affine} for a proof.

\begin{lem}\thlabel{connected components of affine flag varieties}
	There are natural isomorphisms 
	\[\pi_1(G)_{\Gal(\overline{F}/F)} \cong \pi_0(LG) \cong \pi_0(\Fl_{\Gg}).\]
\end{lem}

Now, for any parahoric model \(\Gg'/\Oo\) of \(G\), the positive loop group \(L^+\Gg'\) acts on \(\Fl_\Gg\).
In order to describe the orbits for this action, let us recall some Bruhat-Tits theory, following \cite[§3]{AGLR:Local}.

Let \(A\subseteq G\) be a maximal \(F\)-split torus, and \(S\subseteq G\) a maximal \(\breve{F}\)-split torus containing \(A\), but still defined over \(F\).
Let \(T:=\Cent_G(S)\) be the centralizer of \(S\), which is a maximal torus of \(G\).
For any \(\lambda\colon \IG_{m,F}\to T\), we denote by \(\varpi^\lambda\in G(F)\) the image of \(\varpi\in \IG_m(F)\) under \(\IG_{m,F}\xrightarrow{\lambda} T \to G\).
We will also denote the image of \(\varpi^\lambda\) in \(\Fl_{\Gg}\) the same way.

Let \(\Ss\) and \(\Tt\) be the connected Néron \(\Oo\)-models of \(S\) and \(T\) respectively, which are automatically contained in \(\Gg\).
These are the unique parahoric models of \(S\) and \(T\).
Let \(\app(G,S,\breve{F})\) denote the apartment in the Bruhat-Tits building of \(G\) corresponding to \(S\).

\begin{dfn}{\cite[Definition 7]{HainesRapoport:Parahoric}}
	The \emph{Iwahori-Weyl group} of \(G\) (associated to \(S\)) is \[\tilde{W}:= \Norm_G(S)(\breve{F})/\Tt(\breve{\Oo}).\]
\end{dfn}

By \cite[Lemma 14]{HainesRapoport:Parahoric}, this group sits in a short exact sequence 
\[1\to W_\aff\to \tilde{W} \to \pi_1(G)_I\to 1,\]
where \(W_\aff\) is the affine Weyl group, and \(\pi_1(G)\) is Borovoi's fundamental group.
The choice of an alcove in \(\app(G,S,\breve{F})\) induces a splitting \(\tilde{W} = W_\aff \rtimes \pi_1(G)_I\).
In particular, as \(W_\aff\) is a Coxeter group, \(\tilde{W}\) inherits the structure of a quasi-Coxeter group and a length function \(l\), by setting the length of elements in \(\pi_1(G)_I\) to zero.

Let \(W_\Gg:= \left(\Norm_G(S)(\breve{F})\cap \Gg(\breve{\Oo})\right)/\Tt(\breve{\Oo})\subseteq \tilde{W}\); this group is always finite, and even trivial if \(\Gg\) is an Iwahori model.
If \(\Gg\) is the parahoric associated to a facet \(\ff\) in \(\app(G,S,\breve{F})\), we also write \(\Fl_\ff := \Fl_{\Gg}\) and \(W_\ff:=W_\Gg\).
Then, for a second parahoric model \(\Gg'\) of \(G\), \cite[Proposition 8]{HainesRapoport:Parahoric} gives a bijection
\[\Gg'(\breve{\Oo})\backslash G(\breve{F}) /\Gg(\breve{\Oo}) \cong W_{\Gg'}\backslash \tilde{W}/ W_{\Gg}.\]

In particular, over \(\overline{k}\), the \(L^+\Gg'\)-orbits on \(\Fl_\Gg\) are enumerated by \(W_{\Gg'}\backslash \tilde{W}/ W_{\Gg}\), and this already holds over a finite extension \(k'/k\) corresponding to a splitting of \(S\).
For \(w\in W_{\Gg'}\backslash \tilde{W}/ W_{\Gg}\), we denote the corresponding \(L^+\Gg'\)-orbit in \(\Fl_\Gg\) by \(\Fl_{\Gg,w}\); it is defined over the reflex field of \(w\), and called an (affine) \emph{Schubert cell}.
Its closure \(\Fl_{\Gg,\leq w}\) is called an (affine) \emph{Schubert variety}.
As the notation suggests, we have \(\Fl_{\Gg,\leq w} = \bigsqcup_{w'\leq w} \Fl_{\Gg,w'}\) (\cite[Proposition 2.8]{Richarz:Schubert}), so that the Schubert cells form a stratification of \(\Fl_{\Gg}\).
In order to emphasize the parahorics \(\Gg\) and \(\Gg'\), corresponding to facets \(\mathbf{f}\) and \(\mathbf{f}'\), we will sometimes also denote the Schubert cells by \(\Fl_w(\mathbf{f}',\mathbf{f})\) for \(w\in W_{\mathbf{f}'} \backslash \tilde{W}/ W_{\mathbf{f}}\), and their closures by \(\Fl_{\leq w}(\mathbf{f}',\mathbf{f})\).

\begin{rmk}
	At various points in this paper, it will be useful to assume that \(A=S\), i.e., that the maximal \(\breve{F}\)-split torus of \(G\) is already \(F\)-split.
	Following \cite[1.10.2]{Tits:Reductive}, we call such groups \emph{residually split}; recall that residually split groups are automatically quasi-split.
	Any reductive group over \(F\) is residually split after passing to a finite unramified extension of \(F\).
	On the level of affine flag varieties, this corresponds to base change along the corresponding extension of residue fields.
	Residually split groups have the advantage that all Schubert cells and Schubert varieties are defined over \(k\), and that their affine root systems are reduced.
	This makes the geometry of their affine flag varieties easier to study.
	
	For example, assume \(G\) is residually split, and \(\Gg\) is a very special parahoric. Choose a Borel \(T\subseteq B\subseteq G\), and consider the corresponding positive roots and dominant cocharacters.
	Then there is a natural bijection \(W_{\Gg} \backslash \tilde{W} / W_{\Gg} \cong X_*(T)_I^+\).
	Moreover, the dimension of \(\Gr_{\Gg,\mu}\) is given by \(\langle2\rho,\mu\rangle\), where \(2\rho\in X^*(T)\) is the sum of the absolute roots of \(G\); note that this is well-defined as the pairing \(\langle-,-\rangle\colon X^*(T)\times X_*(T)\) is \(I\)-invariant.
	
	As another example, if \(G\) is residually split, then the action of \(\Gal(\overline{k}/k)\) on \(\pi_1(G)_I\) is trivial, so that \thref{connected components of affine flag varieties} induces a bijection \(\pi_0(\Fl_{\Gg}) \cong \pi_1(G)_I\).
\end{rmk}

\begin{prop}\thlabel{Cellularity of Schubert cells}
	Assume \(G\) is residually split, and fix an Iwahori model \(\Ii\subseteq \Gg\) of \(G\).
	Consider the stratifications \(\Fl_{\Gg}^\dagger = \coprod_{w\in \tilde{W}/W_{\Gg}} \Fl_{\Gg,w} \to \Fl_{\Gg}\) and \(\Fl_{\Ii}^\dagger = \coprod_{w\in \tilde{W}} \Fl_{\Ii,w}\to \Fl_{\Ii}\) by \(L^+\Ii\)-orbits.
	\begin{enumerate}
		\item For any \(w\in \tilde{W}/W_{\Gg}\), the Schubert cell \(\Fl_{\Gg,w}\) is isomorphic to \(\IA^{l(w),\perf}\).
		\item The projection map \(\pi\colon \Fl_{\Ii}\to \Fl_{\Gg}\) is proper, and \(\pi_*\) preserves stratified Tate motives.
		\item The induced map \(\pi^\dagger\colon \Fl_{\Ii}^\dagger \to \Fl_{\Gg}^\dagger\) admits a section which is an open and closed immersion.
	\end{enumerate}
\end{prop}
\begin{proof}
	We adapt the classical proof of (1) to our setting.
	The minimal length representative of \(w\in \tilde{W}/W_{\Gg}\) in \(\tilde{W}\) admits a reduced decomposition \(\dot{w} = \tau s_1s_2\ldots s_{l(w)}\), where each \(s_i\) is a simple reflection, and \(\tau\) has length \(0\).
	To each \(s_i\) there is an associated parahoric \(\Ii\subset \Pp_i\) such that \(L^+\Pp_i/L^+\Ii\cong (\IP^1_k)^{\perf}\) (cf.~\cite[Proposition 8.7]{PappasRapoport:Twisted} in case \(F\) is of equal characteristic with algebraically closed residue field; this also holds in our setting by the assumption that \(G\) is residually split). 
	This yields a birational Demazure resolution \[D(\dot{w}) = \{\tau\}\times L^+\Pp_1 \overset{L^+\Ii}{\times} L^+\Pp_2 \overset{L^+\Ii}{\times}\ldots \overset{L^+\Ii}{\times} L^+\Pp_{l(w)}/L^+\Ii \to \Fl_{\Gg,\leq w},\]
	given by multiplication, where we have fixed a lift of \(\tau\) to \(LG(k) = G(F)\).
	By induction, the fibers of this resolution are iterated \(\IP^{1,\perf}_k\)-fibrations, and it restricts to an isomorphism
	\[\{\tau\} \times (L^+\Pp_1\setminus L^+\Ii) \overset{L^+\Ii}{\times} (L^+\Pp_2\setminus L^+\Ii) \overset{L^+\Ii}{\times}\ldots \overset{L^+\Ii}{\times} (L^+\Pp_{l(w)}\setminus L^+\Ii)/L^+\Ii \to \Fl_{\Gg,w}.\]
	This shows (1).
	
	Similarly, we can use Demazure resolutions to see that the restriction of \(\pi\) to each Iwahori-orbit is a relative affine space over an Iwahori orbit in \(\Fl_{\Gg}\).
	Thus, localization and homotopy invariance imply (2).
	Finally, as any \(w\in \tilde{W}/W_{\Gg}\) has the same length as its minimal representative in \(\tilde{W}\), there is a unique Iwahori-orbit in \(\Fl_{\Ii}\) which under \(\pi\) is a relative affine space over \(\Fl_{G,w}\) of dimension 0, i.e., an isomorphism.
	This gives the desired section for (3).
\end{proof}

In \cite[Theorem 5.1.1]{RicharzScholbach:Intersection}, it was shown that for a split reductive group, stratifications of affine flag varieties by Schubert cells are Whitney Tate.
We now upgrade this to the ramified setting.

\begin{thm}\thlabel{Schubert stratification is WT}
	Assume \(G\) is residually split, and consider two parahoric \(\Oo\)-models \(\Gg,\Gg'\) of \(G\).
	Then the stratification \(\iota\colon \Fl^\dagger = \coprod_{w\in W_{\Gg'}\backslash \tilde{W} / W_\Gg}\Fl_{\Gg,w}\to \Fl_\Gg\) by \(L^+\Gg'\)-orbits is anti-effective Whitney--Tate.
\end{thm}
\begin{proof}
	The proof is similar to \cite[Theorem 5.1.1]{RicharzScholbach:Intersection}, cf.~also \cite[Proposition 3.7]{CassvdHScholbach:Geometric} for the anti-effective part. 
	We start by assuming that \(\Gg'\) and \(\Gg\) are Iwahori subgroups. 
	We will show \(\iota^*(\iota_w)_*\unit \in \DTM^{\xanti}(\Fl^\dagger)\) by induction on the length \(l(w)\), for \(w\in \tilde{W}\). 
	If \(l(w)=0\), then \(\iota_w\) is a closed immersion, so we are done.
	If \(l(w)>0\), we can find a simple reflection \(s\) such that \(w=vs\) for some \(v\in W_{\Gg'}\backslash \tilde{W} / W_\Gg\) with \(l(v)=l(w)-1\).
	Let \(\Gg_s\supset \Gg\) denote the parahoric corresponding to the simple reflection \(s\).
	The projection \(\pi\colon \Fl_\Gg\to \Fl_{\Gg_s}\) is proper and perfectly smooth, and is an étale-locally trivial fibration with general fiber \(\Gg_s/\Gg\cong (\IP^1)^\perf\) (\cite[Lemma 4.9]{HainesRicharz:TestParahoric} and \cite[Proposition 8.7]{PappasRapoport:Twisted}, the proofs also work in mixed characteristic). 
	The induced map \(\Fl_\Gg^\dagger\to \Fl_{\Gg_s}^\dagger\) on \(\Gg'\)-orbits is Tate and admits a section by \thref{Cellularity of Schubert cells}.
	As in \cite[Theorem 5.1.1]{RicharzScholbach:Intersection}, localization gives an exact triangle
	\[(\iota_v)_*\unit(-1)[-2] \to \pi^*\pi_! (\iota_v)_! \unit \to (\iota_w)_*\unit.\]
	Then the leftmost term is (anti-effective) stratified Tate by induction, while the middle term is (anti-effective) stratified Tate by \cite[Lemma 3.1.18]{RicharzScholbach:Intersection} and \cite[Lemma 3.6]{CassvdHScholbach:Geometric}.
	Hence \(\iota^*(\iota_w)_*\unit \in \DTM^{\xanti}(\Fl_\Gg^\dagger)\) as well.
	
	Next, we assume \(\Gg'\) is still an Iwahori, but \(\Gg\) is an arbitrary parahoric. 
	Let \(\Ii\) be an Iwahori such that the facet corresponding to \(\Gg\) is contained in the closure of the facet corresponding to \(\Ii\).
	Using the projection \(\Fl_\Ii\to \Fl_\Gg\) and \thref{Cellularity of Schubert cells}, we see that \(\Fl_\Gg\) is Whitney--Tate by \cite[Lemma 3.1.19]{RicharzScholbach:Intersection} and the Iwahori case.
	For the anti-effectivity, we can then use a section \(\Fl_{\Ii} \to \Fl_{\Gg}\), as well as \cite[Lemma 3.6]{CassvdHScholbach:Geometric} again.
	
	Finally, if both \(\Gg'\) and \(\Gg\) are arbitrary, we can use the previous case and \cite[Proposition 3.1.23]{RicharzScholbach:Intersection} to see the stratification is Whitney--Tate, so we are left to show it is anti-effective.
	For this, it suffices by \thref{characterization of anti-effective} to show that 
	\(\Maps_{\DTM(\Fl_{\Gg}^\dagger)}(\unit(n),\iota^*\iota_*\unit ) =0\) for \(n\geq 1\). 
	Let \(\iota'\) be the stratification of \(\Fl_{\Gg}\) by \(\Ii'\)-orbits, where \(\Ii'\subseteq \Gg'\) is an Iwahori.
	Then by the previous case and \thref{characterization of anti-effective} we have
	\(\Maps((\iota')_!(\iota')^*\unit(n),\iota^*\iota_*\unit) \cong \Maps((\iota')^*\unit(n),(\iota')^!\iota^*\iota_*\unit) =0\).
	We conclude by localization, and \thref{characterization of anti-effective} again.
\end{proof}

\begin{rmk}\thlabel{fix asymmetry}
	Recall that we have \(\DM(L^+\Gg'\backslash \Fl_{\Gg}) \cong \DM(L^+\Gg' \backslash LG / L^+\Gg)\) by descent.
	We can also consider the opposite flag variety \(\Fl_{\Gg'}^{\op} := (L^+\Gg'\backslash LG)^{\et}\), for which the obvious right \(L^+\Gg\)-action also determines a Whitney--Tate stratification.
	The previous proposition, along with \cite[Definition and Lemma 3.11]{RicharzScholbach:Intersection}, shows that \(\DATM_{L^+\Gg'}(\Fl_{\Gg}) \subseteq \DM(L^+\Gg' \backslash LG / L^+\Gg)\) consists of those motives whose pullback along \(k'/k\) lies in the category generated (under colimits, shifts, twists, and extensions) by the \(\iota_{w,*}\unit\) for \(w\in W_{\Gg'}\backslash \tilde{W} / W_{\Gg}\), and similarly for \(\DATM_{L^+\Gg}(\Fl_{\Gg'}^{\op})\).
	Hence, as in \cite[Theorem 5.3.4 (ii)]{RicharzScholbach:Intersection}, the two subcategories \(\DATM_{L^+\Gg'}(\Fl_{\Gg})\) and \(\DATM_{L^+\Gg}(\Fl_{\Gg'}^{\op})\) of \(\DM(L^+\Gg' \backslash LG / L^+\Gg)\) agree.
	
	In case \(\Gg=\Gg'\), since the stratifications by \(L^+\Gg\)-orbits of both \(\Fl_{\Gg}\) and \(\Fl_{\Gg}^{\op}\) are Whitney Tate, we can consider the categories of mixed Artin-Tate motives
	\[\MATM_{L^+\Gg}(\Fl_{\Gg})\subseteq \DM(L^+\Gg \backslash LG / L^+\Gg)\supseteq \MATM_{L^+\Gg}(\Fl_{\Gg}^{\op}).\]
	At least over \(k'\), the connective parts of both t-structures are generated under colimits by \(\iota_{w,!}\unit(n)[l(w)]\) for \(w\in W_{\Gg}\backslash \tilde{W} /W_{\Gg}\).
	Thus, they define the same t-structure, and the two categories of mixed Artin-Tate motives above agree.
\end{rmk}

\subsection{Semi-infinite orbits}\label{subsec--semi-infinite}

Next, we study the semi-infinite orbits in our setting. First introduced for split groups in equal characteristic by \cite[§3]{MirkovicVilonen:Geometric}, they are also well understood in the unramified case \cite{Zhu:Affine}. And while they were also expected to exist and behave well in the ramified case \cite[Remark 0.2 (2)]{Zhu:Ramified}, these only recently first appeared in \cite[§5.2]{AGLR:Local} (or implicitly in \cite[Proposition 4.7]{HainesRicharz:TestParahoric}).

Let \(\lambda\colon \IG_{m,\Oo}\to \Ss\) be a cocharacter defined over \(\Oo\).
Let \(M,P=P^+, P^-\subseteq G\) denote the fixed points, attractor, and repeller respectively for the resulting \(\IG_{m,F}\)-action on \(G\) given by conjugation.
It is well known that \(P^+\) and \(P^-\) are opposite parabolics with Levi \(M\), so that there is a semi-direct product decomposition \(P^\pm = U_{P^\pm} \rtimes M\), where \(U_{P^\pm}\subseteq P^\pm\) is the unipotent radical.
Since \(\lambda\) is defined over \(\Oo\), these groups extend to smooth affine \(\Oo\)-group schemes with connected fibers, also fitting into a semi-direct product \(\Pp^\pm = \Uu_{P^\pm} \rtimes \Mm\), and admitting morphisms \(\Mm \leftarrow \Pp^\pm \rightarrow \Gg\).
(In fact, they are the schematic closures of their generic fiber in \(\Gg\).)
Moreover, \(\Mm\) is a parahoric model of \(M\).
The fixed points, attractor, and repeller of the \(\IG_{m,k}^\perf\)-action on \(\Fl_\Gg\) are related to \(\Mm\) and \(\Pp^\pm\) as follows:

\begin{prop}\thlabel{flag varieties of fixed points and attractor}
	There exists a canonical commutative diagram 
	\[\begin{tikzcd}
		\Fl_{\Mm} \arrow[d, "\iota^0"] & \Fl_{\Pp^\pm} \arrow[l] \arrow[d, "\iota^\pm"] \arrow[r] & \Fl_\Gg \arrow[d, "\identity"]\\
		(\Fl_\Gg)^0 & (\Fl_\Gg)^\pm \arrow[l] \arrow[r] & \Fl_\Gg,
	\end{tikzcd}\]
	such that \(\iota^0\) and \(\iota^\pm\) are open and closed immersions, and surjective if \(\Gg\) is very special.
	Here, the lower horizontal maps are the natural maps out of the attractor and repeller of \(\Fl_\Gg\), and the upper horizontal maps are induced by the maps out of the attractor and repeller of \(\Gg\).
\end{prop}
\begin{proof}
	This follows from \cite[Proposition 4.7]{HainesRicharz:TestParahoric}, cf.~also \cite[Theorem 5.2]{AGLR:Local} in mixed characteristic.
\end{proof}

\begin{nota}
	Assume moreover that \(\Gg\) is very special, so that \(\Fl_{\Pp^\pm}\to \Fl_\Gg\) is bijective on points (by the proposition and properness of \(\Fl_\Gg\)).
	Then we denote the connected components of \(\Fl_{\Pp^\pm}\) by \(\Ss_{P,\nu}^\pm\) for \(\nu\in \pi_0(\Fl_\Mm)\cong \pi_1(M)_{\Gal(\overline{F}/F)}\) (\cite[Corollary 1.12]{Richarz:Spaces} and \thref{connected components of affine flag varieties}).
\end{nota}

Now, assume \(\Gg\) is very special parahoric.
Then \(G\) is quasi-split by \cite[Lemma 6.1]{Zhu:Ramified}, and we fix a Borel \(T\subseteq B = B^+\subseteq G\), along with the corresponding choice of positive roots and dominant cocharacters.
We moreover assume \(\lambda\) is regular dominant, so that \(M=T\) is a minimal Levi of \(G\), and \(P^\pm=B^\pm\) are opposite Borels; we denote the unipotent radicals of the latter by \(U^\pm\), and their closures in \(\Gg\) by \(\Uu^{\pm}\).
Then the \(\Ss_{B,\nu}^\pm\) defined above are usually called the \emph{semi-infinite orbits}; we also denote them by \(\Ss_\nu^\pm\) for simplicity.
For the rest of this section, we assume that \(G\) is residually split.
Then the semi-infinite orbits are in bijection with \(\pi_0(\Fl_{\Tt}) \cong \tilde{W}/W_\Gg \cong X_*(T)_I\).
Moreover, we have \(\Ss^\pm_\nu = LU^\pm \cdot \nu\), making \(LU^\pm\to \Ss_{\nu}^\pm\) an \(L^+\Uu^\pm\)-torsor, compare \cite[(5.12)]{AGLR:Local}.

By \cite[Proposition 5.4]{AGLR:Local}, the semi-infinite orbits form a stratification of \(\Gr_\Gg\).
In particular, as any \(\Ss_{P,\nu}^\pm\) is a union of semi-infinite orbits (where \(P\supseteq B\) is the parabolic attached to a not necessarily regular cocharacter, but \(\Gg\) is still very special), we also have stratifications \(\Gr_\Gg = \bigcup_{\nu\in \pi_1(M)_I} \Ss_{P,\nu}^\pm\).
In order to show the convolution product is t-exact later on, we need to understand how semi-infinite orbits intersect Schubert cells, and the dimension of these intersections.

\begin{prop}\thlabel{Nonemptyness of intersections}
	Let \(\nu\in X_*(T)_I\) and \(\mu\in X_*(T)_I^+\). Then the following are equivalent.
	\begin{enumerate}
		\item \(\mathcal{S}^+_\nu\cap \Gr_{\Gg,\leq \mu}\neq \varnothing\),
		\item \(\mathcal{S}^+_\nu\cap \Gr_{\Gg, \mu}\neq \varnothing\), and
		\item The unique dominant representative \(\nu^+\) of \(W_{\Gg}\cdot \nu\) satisfies \(\nu^+\leq \mu\).
	\end{enumerate}
	Moreover, if these assumptions are satisfied, \(\mathcal{S}^+_\nu\cap \Gr_{\Gg,\leq \mu}\) is affine and equidimensional of dimension \(\langle \rho,\mu+\nu \rangle\).
\end{prop}
\begin{proof}
	This is \cite[Lemma 5.5]{AGLR:Local} (while it is only stated when \(F\) is of mixed characteristic, the proof also works in equal characteristic).
\end{proof}

\section{LS galleries and MV cycles for residually split reductive groups}\label{sect-LSMV}

Next, we need to understand how Schubert cells and semi-infinite orbits intersect.
Although \thref{Nonemptyness of intersections} already tells us which intersections are nonempty, we also need to understand the geometry of these intersections, generalizing \cite{GaussentLittelmann:LS, CassvdHScholbach:Geometric}, which handled the case of split groups in equal characteristic.
This will later be used to show, among other things, that the constant term functors preserve Artin-Tate motives, and that the Tannakian dual group is flat.

Since in this section we will study the geometry of semi-infinite orbits, we will assume throughout that \(G\) is residually split;  in particular \(G\) is quasi-split and its affine root system is reduced.
We denote its absolute Weyl group by \(W\), so that its relative Weyl group is given by the inertia-invariants \(W^I\).
As usual, \(\Gg\) denotes a very special parahoric \(\Oo\)-model of \(G\).
We will also assume \(G\) is semisimple and simply connected throughout most of this section.
We will explain in \thref{reduce to simply connected} why it suffices to handle this case.
Note that this assumption implies that \(X_*(T)_I\) is torsionfree.
Hence, we get inclusions \(X_*(S)\subseteq X_*(T)_I\subseteq X_*(S)\otimes_{\IZ} \IQ\), and we will use these inclusions to extend the pairing \(\langle -,-\rangle_S\colon X^*(S)\times X_*(T)_I\to \IQ\).
Many definitions in this section are adapted from \cite{GaussentLittelmann:LS}.

\begin{nota}
	We now introduce some notation, which will only be used in this section.
	Let \(\buil(G,F)\) be the Bruhat-Tits building of \(G\), and consider the appartment \(\app := \app(G,S,F)\). 
	The (relative) affine Weyl group \(W_\aff\) acts simply transitively on the set of alcoves of \(\app\).
	We fix a fundamental alcove \(\Delta_f\) in \(\app\) whose closure contains the facet \(\mathbf{f}_0\) corresponding to \(\Gg\), and let \(\Ii\subseteq \Gg\) be the corresponding Iwahori.
	We identify \(\app\) with \(X_*(S)\otimes_{\IZ} \IR\), by choosing \(\mathbf{f}_0\) as the origin.
	Recall that \(W_\aff\) is a Coxeter group generated by the simple reflections \(S_\aff\).
	Each such simple reflection \(s\) corresponds to a reflection hyperplane (or wall) \(\wall_s\), which contains a face of \(\Delta_f\).
	We call a proper subset of \(S_\aff\) a \emph{type}.
	Then the facets of \(\Delta_f\) correspond bijectively to the types, by sending a facet \(\Gamma\prec \Delta_f\) to the set \(S_{\aff}(\Gamma)\) of simple reflections whose corresponding reflection hyperplane above contains \(\Gamma\); for example \(S_\aff(\Delta_f) = \varnothing\).
	Conversely, for a type \(t\subset S_\aff\), we denote by \(F_t\) the unique face of \(\Delta_f\) of type \(t\).
	Since, under the \(G(F)\)-action on \(\buil(G,F)\), the \(G(F)\)-orbit of any facet contains a unique face of \(\Delta_f\), we can use this to associate a type to each facet \(\Gamma\) in \(\buil(G,F)\), still denoted by \(S_\aff(\Gamma)\).
	In particular, we can attach a type to each parahoric model \(\Pp\) of \(G\), and we will denote the facet of this type contained in \(\Delta_f\) by \(F_\Pp\).
\end{nota}

We also briefly recall the spherical buildings of reductive groups over \(k\).

\begin{nota}
	Consider the maximal reductive quotient \(\mathsf{G}\) of the special fiber of \(\Gg\), i.e., \(\mathsf{G}\) is the quotient of \(\Gg \times_{\Spec \Oo} \Spec k\) by its unipotent radical.
	By our assumption that \(G\) is residually split, \(\mathsf{G}\) is split, and a maximal torus is given by the maximal reductive quotient \(\mathsf{S}\) of the special fiber of \(\Ss\).
	The \emph{spherical building} \(\buil^s(\mathsf{G},k)\) of \(\mathsf{G}\) is the building whose facets correspond to the proper parabolic subgroups of \(\mathsf{G}\) defined over \(k\), with simplicial structure given by inclusion.
	The apartments are given by the sets of parabolics containing a fixed maximal \(k\)-torus; in particular, there is the apartment \(\app^s:=\app^s(\mathsf{G}, \mathsf{S}, k)\) corresponding to \(\mathsf{S}\).
\end{nota}

\begin{rmk}\thlabel{identification of apartments}
	Since \(\Gg\) is special, the affine Weyl group \(W_{\aff}\) is the semi-direct product of its translation subgroup with the finite Weyl group of \(\mathsf{G}\) \cite[Lemma 1.3.42]{KalethaPrasad:BruhatTits}.
	Hence, the apartments \(\app^s(\mathsf{G},\mathsf{S},k)\) and \(\app(G,S,F)\) can be identified as affine spaces, albeit with different simplicial structures.
\end{rmk}

In order to reserve the term \emph{alcoves} for maximal simplices in Bruhat-Tits buildings, we will call the maximal simplices of spherical buildings \emph{chambers}; in the situation above these correspond bijectively to the set of \(k\)-Borels of \(\mathsf{G}\).
We denote the chamber corresponding to \(\mathsf{B}\) (the image of the special fiber of \(\Ii\) in \(\mathsf{G}\)) by \(\CC_f\), the chamber corresponding to the opposite parabolic \(\mathsf{B}^-\supset \mathsf{S}\) by \(-\CC_f\), and the simplex corresponding to a parabolic \(\mathsf{P}\) by \(\FF_{\mathsf{P}}\).

\begin{dfn}\thlabel{defi-sector}
	Under the identification of \(\app=\app(G,S,F)\) with \(\app^s=\app^s(\mathsf{G}, \mathsf{S}, k)\) as affine spaces, the chambers of \(\app^s\) correspond to the connected components of \(\app\setminus \bigcup_{t\in S_\aff(\mathbf{f}_0)} \wall_t\).
	A \(W_\aff\)-translate of such a chamber in \(\app\) is called a \emph{sector}, and two sectors are called \emph{equivalent} if their intersection contains another sector.
	Note that any sector is equivalent to a unique chamber in \(\app^s\).
\end{dfn}

\subsection{Galleries in the Bruhat-Tits building}

One of the main goals of this section is to identify points in \(\Ss^{\pm}_{\nu}\cap \Gr_{\Gg,\mu}\) with certain minimal galleries in \(\buil(G,F)\), at least after passing to \(\breve{F}\), as in \cite{GaussentLittelmann:LS}.
Here, it will be crucial to consider not only galleries of alcoves, but also smaller facets.

\begin{dfn}
	A \emph{gallery} in \(\buil(G,F)\) is a sequence of facets
	\[\gamma = (\Gamma'_0 \prec \Gamma_0 \succ \Gamma'_1 \prec \ldots \succ \Gamma'_r \prec \Gamma_r\succ \Gamma'_{r+1})\]
	in \(\buil(G,F)\), such that
	\begin{enumerate}
		\item \(\Gamma'_0\) and \(\Gamma'_{r+1}\), called the \emph{source} and \emph{target} of \(\gamma\), are vertices,
		\item the \(\Gamma_j\)'s are facets of the same dimension, and
		\item for \(1\leq j\leq r\), \(\Gamma'_j\) is a codimension 1 face of \(\Gamma_{j-1}\) and \(\Gamma_j\).
	\end{enumerate}
By attaching to each facet its type, we get the associated \emph{gallery of types} of \(\gamma\).

If the large facets of \(\gamma\) are alcoves, we say \(\gamma\) is a \emph{gallery of alcoves}.
It is moreover a \emph{minimal gallery of alcoves}, if its length \(r\) is minimal among the lengths of all galleries of alcoves joining \(\Gamma_0\) and \(\Gamma_r\) (where we silently forget the source and target vertices).
\end{dfn}

For general galleries, the notion of minimality is more subtle.
Recall from \cite[Proposition 2.29]{Tits:Buildings} that for any alcove \(\Delta\) and facet \(\Gamma\) contained in an apartment \(\app'\) of \(\buil(G,F)\), there is a unique alcove \(\proj_\Gamma(\Delta)\) in \(\app'\) such that any face of the convex hull of \(\Gamma\) and \(\Delta\) that contains \(\Gamma\) is itself a face of \(\proj_\Gamma(\Delta)\).

\begin{dfn}
	Let \(\Gamma\) be a facet, and \(\Delta\) an alcove with a face \(\Gamma'\), all contained in an apartment \(\app'\) of \(\buil(G,F)\).
	Then \(\Delta\) is \emph{at maximal distance from \(\Gamma\) (among the alcoves containing \(\Gamma'\))} if the length of a minimal gallery of alcoves joining \(\Delta\) and \(\proj_\Gamma(\Delta)\) is the same as the number of walls of \(\app'\) separating \(\Gamma\) and \(\Gamma'\).
\end{dfn}

Here, separating is meant in the following sense:

\begin{dfn}
	Let \(\app'\subseteq \buil(G,F)\) be an apartment, \(\Omega\subseteq \app'\) any subset, and \(\Gamma\) a face of \(\app'\).
	\begin{enumerate}
		\item An affine reflection hyperplane \(\wall\subset \app'\) is said to \emph{separate} \(\Omega\) and \(\Gamma\) if \(\Omega\) is contained in a closed half-space defined by \(\wall\), and the closure of \(\Gamma\) is contained in the opposite open half-space.
		If \(\Omega\) and \(\Gamma\) are both facets, we denote the set of reflection hyperplanes in \(\app'\) separating \(\Omega\) and \(\Gamma\) by \(\Mm_{\app'}(\Omega,\Gamma)\), which is finite.
		\item We say a hyperplane \(\wall\) as above \emph{separates \(\Gamma\) from \(\CC_{-\infty}\)}, if \(\wall\) separates \(\Gamma\) from a sector equivalent to \(-\CC_f\), as in \thref{defi-sector}.
	\end{enumerate}
\end{dfn}

We can now define minimality for general galleries.

\begin{dfn}
	Consider a gallery \(\gamma = (\Gamma'_0 \prec \Gamma_0 \succ \Gamma'_1 \prec \ldots \succ \Gamma'_r \prec \Gamma_r\succ \Gamma'_{r+1})\) in \(\buil(G,F)\) and an alcove \(\Delta\succ \Gamma'_0\) at maximal distance of \(\Gamma'_{r+1}\).
	Then \(\gamma\) is \emph{minimal} if
	\begin{enumerate}
		\item there exists a minimal gallery
		\[\delta = (\Delta= \Delta_0^1, \ldots \Delta_0^{q_0}, \Delta_1^1, \ldots, \Delta_j^1,\ldots, \Delta_j^{q_j}, \ldots, \Delta_r^1, \ldots, \Delta_p^{q_r} = \proj_{\Gamma'_{r+1}}(\Delta)),\]
		of alcoves such that \(\Gamma_j\preceq \Delta_j^0\) and \(\Gamma'_j\prec \Delta_j^i\) for \(j=0,\ldots, r\) and \(i=0,\ldots, q_j\), and
		\item let \(\app'\) be an apartment containing \(\gamma\), which exists by (1). Then \[\Mm_{\app'}(\Gamma'_0,\Gamma'_{r+1}) = \bigsqcup_{j=0,\ldots, r} \left\{ \wall\subset \app'\mid \Gamma_j\subseteq \wall, \Gamma'_j\nsubseteq \wall\right\}.\]
	\end{enumerate}
\end{dfn}

One easily checks that the notion of minimality is independent of the choice of \(\Delta\) made above.

Let \(T_{\adj}\subseteq G_{\adj}\) be the adjoint torus of \(T\), and fix some \(\mu\in X_*(T_{\adj})_I\). 
We denote \[\wall_{\mu}:= \bigcap_{\alpha\in \Phi^{\nd}\colon \langle\alpha, \mu \rangle_S=0} \wall_{\alpha,0}\] where we let \(\wall_\mu= \app\) if \(\mu\) is regular.
We also let \(\mathbf{f}_\mu\) be the vertex in \(\app\) corresponding to \(\mu\).

\begin{dfn}
	A \emph{gallery joining 0 with \(\mu\)} is a gallery 
	\[\gamma = (\mathbf{f}_0 \prec \Gamma_0 \succ \Gamma'_1 \prec \ldots \succ \Gamma'_r \prec \Gamma_r\succ \mathbf{f}_\mu)\] 
	contained in \(\app\), such that \(\dim \Gamma_j = \dim \wall_\mu\) for any \(j\).
\end{dfn}

Let us now fix some \(\mu\in X_*(T_{\adj})_I^+\), together with a minimal gallery
\begin{equation}
	\gamma_\mu = (\mathbf{f}_0 \prec \Gamma_0 \succ \Gamma'_1 \prec \ldots \succ \Gamma'_r \prec \Gamma_r\succ \mathbf{f}_\mu)
\end{equation}
joining \(0\) with \(\mu\).
We denote its gallery of types by 
\begin{equation}\label{standard gallery of types}
	t_{\gamma_\mu} = (t'_0\supset t_0\subset t'_1\supset \ldots \subset t'_r\supset t_r\subset t_\mu).
\end{equation}

\begin{rmk}\thlabel{first facet is fixed}
	As in \cite[Lemma 4]{GaussentLittelmann:LS}, one can show that minimality of \(\gamma_\mu\) implies that it lies in \(\wall_{\mu}\).
	Since \(\Gamma_0\) also lies in \(\Delta_f\) and has the same dimension as \(\wall_\mu\), it is uniquely determined by \(\mu\).
	Namely, it is the unique parahoric \(\Ii\subseteq \Pp\subseteq \Gg\), whose corresponding parabolic \(\mathsf{P}\subseteq \mathsf{G}\) is generated by \(\mathsf{S}\) and the root groups for those roots \(\alpha\) for which \(\langle \alpha,\mu\rangle_{S} \geq 0\).
\end{rmk}

The following set of galleries will be particularly important for us.

\begin{dfn}\thlabel{defi combinatorial galleries}
	A gallery \(\gamma\) is called \emph{combinatorial of type} \(t_{\gamma_\mu}\) if it has source \(\mathbf{f}_0\), has \(t_{\gamma_\mu}\) as its gallery of types, and is contained in \(\app\).
	We denote the set of combinatorial galleries of type \(t_{\gamma_\mu}\) by \(\Gamma(\gamma_\mu)\), and by \(\Gamma(\gamma_\mu,\nu)\subseteq \Gamma(\gamma_\mu)\) those combinatorial galleries with target \(\ff_\nu\), for \(\nu\in X_*(T_{\adj})_I\).
\end{dfn}

We will denote the subgroups of \(W_\aff\) generated by \(t_j\) and \(t'_j\) by \(W_j\) and \(W'_j\) respectively.
Similarly, we will denote the standard parahorics (containing the Iwahori \(\Ii\)) corresponding to these types by \(\Pp_j\) and \(\Pp'_j\).
As the name suggests, there is a combinatorial way to describe \(\Gamma(\gamma_\mu)\), similar to the unramified case \cite[Proposition 2]{GaussentLittelmann:LS}.

\begin{prop}\thlabel{classification of combinatorial galleries}
	The map 
	\[W_{\mathbf{f}_0} \overset{W_0}{\times} W_1'\overset{W_1}{\times} \ldots \overset{W_{r-1}}{\times} W_r'/W_r\to \Gamma(\gamma_\mu)\]
	\[[\delta_0,\delta_1,\ldots, \delta_r]\mapsto \delta = (\mathbf{f}_0 \prec \Sigma_0 \succ \Sigma'_1 \prec \ldots \succ \Sigma'_r \prec \Sigma_r\succ \Sigma_{r+1})\]
	defined by \(\Sigma_j = \delta_0\delta_1\ldots \delta_j (F_{t_j})\) is a bijection.
\end{prop}
\begin{proof}
	This follows easily from \(W_{\aff}\) acting simply transitively on the alcoves of \(\app\), and the \(W_j\) and \(W'_j\) being the stabilizers of the corresponding facets.
\end{proof}

Using this proposition, we can and will denote combinatorial galleries as \(\delta = [\delta_0,\ldots,\delta_r]\), where each \(\delta_j\in W'_j\) is the minimal length representative of its class in \(W'_j/W_j\).
Clearly, we have \(\gamma_\mu\in \Gamma(\gamma_\mu)\), and \(\gamma_\mu = [1,\tau_1,\ldots,\tau_r]\), where each \(\tau_j\in W'_j\) is the minimal length representative of the longest class in \(W'_j/W_j\).
On the other hand, if \(\delta\) is such that \(\delta_j\neq \tau_j\) for some \(1\leq j \leq r\), we say \(\delta\) is \emph{folded} around \(\Sigma_j'\).
More precisely, consider the combinatorial galleries
\[\gamma^{j-1} = [\delta_0, \ldots, \delta_{j-1}, \tau_j, \tau_{j+1}, \ldots, \tau_r] = (\mathbf{f}_0 \prec\ldots \prec \Sigma_{j-1} \succ \Sigma_{j}' \prec \Omega_j \succ \Omega_{j+1}' \prec \Omega_{j+1} \succ \ldots \Omega_r \succ \Omega_{r+1})\]
and
\[\gamma^j = [\delta_0, \ldots, \delta_{j-1}, \delta_j, \tau_{j+1}, \ldots, \tau_r] = (\mathbf{f}_0 \prec\ldots \prec \Sigma_{j-1} \succ \Sigma_{j}' \prec \Sigma_j \succ \Sigma_{j+1}' \prec \Xi_{j+1} \succ \ldots \Xi_r \succ \Xi_{r+1}).\]
Then the following lemma says we can fold \(\gamma^{j-1}\) around \(\Sigma_j'\) to obtain \(\gamma^j\).

\begin{lem}\thlabel{affine roots for positive folds}
	There exist positive affine roots \(\psi_1, \ldots, \psi_a\) for which \(\wall_{\psi_i}\supset \Sigma_j'\), and such that 
	\[\Sigma_j = s_{\psi_a}\ldots s_{\psi_1}(\Omega_j).\]
	Then, for any \(j<l\leq r\), we also have
	\[\Xi_l = s_{\psi_a}\ldots s_{\psi_1}(\Omega_l).\]
\end{lem}
\begin{proof}
	This is similar to \cite[Lemma 5]{GaussentLittelmann:LS}; let us sketch the proof.
	By \thref{classification of combinatorial galleries}, we know that
	\[\Sigma_j = \delta_0\ldots\delta_j (F_{t_j}) = (\delta_0\ldots \delta_{j-1}) \delta_j \tau_j^{-1} (\delta_0\ldots \delta_{j-1})^{-1}(\Omega_j).\]
	Since \(\Sigma_j'\) is a face of both \(\Sigma_j\) and \(\Omega_j\), it must be fixed by \(\delta_j\tau_j^{-1}\).
	Fix a reduced decomposition \(s_{\zeta_1}\ldots s_{\zeta_a}\) of the minimal representative of the class of \(\delta_j\tau_j^{-1}\) in \(W_j'/W_j\).
	All \(s_{\zeta_i}\) lie in the type \(S_\aff(F_{t_j})\).
	Then one checks that the (unique) positive affine roots corresponding to the hyperplane \(\delta_0\ldots \delta_{j-1} s_{\zeta_1}\ldots s_{\zeta_i} \wall_{s_{\zeta_i}}\), for varying \(i\), satisfy the desired condition for the first statement of the lemma.
	The second statement then follows from the observation that
	\[\Xi_l = \delta_0\ldots\delta_j \tau_{j+1} \ldots \tau_l (F_{t_l}) = (\delta_0\ldots \delta_{j-1}) \delta_j \tau_j^{-1} (\delta_0\ldots \delta_{j-1})^{-1}(\Omega_l).\]
\end{proof}

This allows us to define a special subset of combinatorial galleries.

\begin{dfn}
	Let \(\delta= [\delta_0,\ldots,\delta_r] = (\mathbf{f}_0 \prec \Sigma_0 \succ \Sigma'_1 \prec \ldots \succ \Sigma'_r \prec \Sigma_r\succ \Sigma_{r+1})\) be a combinatorial gallery of type \(\gamma_\mu\).
	Let \(j\geq 1\), and consider the associated positive affine roots \(\psi_1,\ldots,\psi_a\) obtained by the previous lemma.
	We say \(\delta\) is positively folded at \(\Sigma_j'\) if for all \(1\leq i\leq a\), the facet \(\Sigma_j\) is contained in the positive closed half-space determined by \(\wall_{\psi_i}\).
	We denote by \(\Gamma^+(\gamma_\mu)\subseteq \Gamma(\gamma_\mu)\) the subset of positively folded combinatorial galleries, i.e., those galleries which are positively folded everywhere.
	Moreover, for \(\nu\in X_*(T_{\adj})_I\), we write \(\Gamma^+(\gamma_\mu,\nu):=\Gamma^+(\gamma_\mu)\cap \Gamma(\gamma_\mu,\nu)\).
\end{dfn}

It is clear that this definition is independent of the chosen affine roots \(\psi_1,\ldots,\psi_a\), which in turn depend on a choice of reduced decomposition of a certain element in \(W_{\aff}\).

\begin{ex}
	If \(\delta_j = \tau_j\) for all \(j\geq 1\), i.e., when \(\delta\) does not have any folds, then the set of affine roots obtained from \thref{affine roots for positive folds} is empty.
	Hence the condition above is automatically satisfied, so that \(\delta\) is positively folded.
	In particular, all minimal combinatorial galleries are positively folded.
\end{ex}

Now, let us relate the above combinatorics to geometry.
The choice of \(\gamma_\mu\) induces a resolution of the corresponding Schubert variety.

\begin{dfn}\thlabel{defi BS resolution}
	The Bott-Samelson scheme \(\Sigma(\gamma_\mu)\) is the contracted product
	\[\Sigma(\gamma_\mu):= L^+\Pp_{\mathbf{f}_0} \overset{L^+\Pp_0}{\times} L^+\Pp'_1 \overset{L^+\Pp_1}{\times} \ldots \overset{L^+\Pp_{r-1}}{\times} L^+\Pp'_r/L^+\Pp_r.\] 
\end{dfn}

Clearly, \(\Sigma(\gamma_\mu)\) is representable by a smooth projective scheme, as an iterated Zariski-locally trivial fibration with partial flag varieties as fibers.
As we moreover have \(\Pp_r\subseteq \Pp_\mu\), multiplication induces a proper map
\[\pi_\mu\colon \Sigma(\gamma_\mu)\to \Fl_{\mathbf{f}_\mu}.\]
Thanks to the following proposition, we can also call \(\Sigma(\gamma_\mu)\) a Bott-Samelson resolution.

\begin{prop}
	The image of the multiplication map \(\pi\) above is \(\Fl_{\leq \mu}(\mathbf{f}_0,\mathbf{f}_\mu)\), and the induced map \(\pi\colon \Sigma(\gamma_\mu)\to \Fl_{\leq \mu}(\mathbf{f}_0,\mathbf{f}_\mu)\) is \(L^+\Gg\)-equivariant and birational.
\end{prop}
\begin{proof}
	The \(L^+\Gg\)-equivariance follows from the definition, while the fact that \(\pi\) is birational onto \(\Fl_{\leq \mu}(\mathbf{f}_0,\mathbf{f}_\mu)\) can be proven as in \cite[(3.1)]{Richarz:Schubert}
\end{proof}

We can describe the closed points of \(\Sigma(\gamma_\mu)\) as galleries.
Recall that for any parahoric \(\Pp\) of \(G\), the \(k\)-valued points of \(LG/L^+\Pp\) can be identified with the parahorics of the same type as \(\Pp\).

\begin{prop}
	The Bott-Samelson resolution \(\Sigma(\gamma_\mu)\) can be identified with the closed subscheme of
	\[LG/L^+\Pp_0 \times \left(\prod_{j=1,\ldots, r} LG/L^+\Pp'_j\times LG/L^+\Pp_j \right) \times LG/L^+\Pp_{\mu}\]
	corresponding to the sequences of parahorics (i.e., galleries)
	\[(\Pp_{\mathbf{f}_0} \supset \Qq_0 \subset \Qq'_1 \supset \Qq_1 \supset \ldots \subset \Qq'_r \supset \Qq_r \subset \Qq_\mu)\]
	of type \(t_{\gamma_\mu}\) starting at \(\mathbf{f}_0\).
\end{prop}
\begin{proof}
	The proof is the same as in the split case, cf.~\cite[Definition-Proposition 1]{GaussentLittelmann:LS}.
	Specifically, the morphism
	\[\Sigma(\gamma_\mu)\to LG/L^+\Pp_0 \times \left(\prod_{j=1,\ldots, r} LG/L^+\Pp'_j\times LG/L^+\Pp_j \right) \times LG/L^+\Pp_{\mu}\]
	is given by sending \([g_0,\ldots, g_p]\) to the sequence of parahorics corresponding to the gallery
	\[(\mathbf{f}_0 \prec \Gamma_0 \succ \Gamma'_1 \prec \ldots \succ \Gamma'_r \prec \Gamma_r\succ \Gamma'_{r+1}),\]
	where \(\Gamma_j=g_0g_1\ldots g_j F_{t'_j}\). This already determines the small faces uniquely
\end{proof}

As in \cite[Remark 3.28]{CassvdHScholbach:Geometric}, \(\pi\) maps the minimal gallery \(\gamma_\mu\) to the point \(\varpi^{w_0(\mu)}\in \Fl_\mu(\mathbf{f}_0,\mathbf{f}_\mu)\), where \(w_0\) is the longest element in the relative Weyl group of \(G\).
Hence, we see that \(\pi\) restricts to a map \(L^+\Gg_{\mathbf{f}_0}\gamma_\mu\to \Fl_\mu(\mathbf{f}_0,\mathbf{f}_\mu)\).
Since all minimal galleries of type \(t_{\gamma_\mu}\) lie in the same \(L^+\Gg_{\mathbf{f}_0}\)-orbit, birationality of \(\pi\) gives the following generalization of \cite[Lemma 10]{GaussentLittelmann:LS}.

\begin{cor}\thlabel{Schubert cells are the minimal galleries}
	The Bott-Samelson resolution \(\pi\colon \Sigma(\gamma_\mu)\to \Fl_{\leq \mu}(\mathbf{f}_0,\mathbf{f}_\mu)\) induces a bijection between the minimal galleries in \(\Sigma(\gamma_\mu)\) (i.e., sequences of parahorics defined over \(\Oo\)), and \(\Fl_\mu(\mathbf{f}_0,\mathbf{f}_\mu)(k)\).
\end{cor}

\subsection{Root operators}

Next, we introduce root operators on the set of combinatorial galleries \(\Gamma(\gamma_\mu)\). 
Although this is not strictly necessary for \thref{Intersections of Schubert cells and semi-infinite orbits}, it will simplify a step in the proof of this theorem.
While root operators can be defined for arbitrary affine buildings as in \cite{Schwer:Roots}, we will emphasize the connection to the group \(G\) and the affine Grassmannian \(\Gr_{\Gg}\), in order to relate them to representation theory.
In particular, we believe the results below to be of independent interest.

Let \(\alpha\) be a simple nondivisible relative root of \(G\), and set \(u_{\alpha}\in \IQ\) to be as in \cite[Proposition 1.3.49 (3)]{KalethaPrasad:BruhatTits}, i.e., it represents the jumps between the affine roots with derivative \(\alpha\).
We will define three partially defined operators \(f_\alpha\), \(e_\alpha\), and \(\widetilde{e}_\alpha\) on \(\Gamma(\gamma_\mu)\). 
So let us fix some \(\gamma = (\mathbf{f}_0 = \Gamma_0' \prec \Gamma_0 \succ \Gamma'_1 \prec \ldots \succ \Gamma'_r \prec \Gamma_r\succ \Gamma'_{r+1} = \mathbf{f}_\nu)\in \Gamma(\gamma_\mu)\).
Let \(m\in u_{\alpha}\IZ\) be minimal such that \(\Gamma_l'\subseteq \wall_{\alpha,m} = \{x\in \app\mid \langle\alpha,x\rangle_S =m\}\) for some \(0\leq l\leq r+1\); in particular \(m\leq 0\).
We will consider the following (not mutually exclusive) cases:

\begin{enumerate}[(I)]
	\item If \(m<0\), let \(l\) be minimal such that \(\Gamma_l'\subseteq \wall_{\alpha,m}\), and let \(0\leq j\leq l\) be maximal such that \(\Gamma_j'\subseteq \wall_{\alpha,m+u_{\alpha}}\).
	\item If \(m< \langle \alpha,\nu \rangle_S\), let \(j\) be maximal such that \(\Gamma_j'\subseteq \wall_{\alpha,m}\), and let \(j\leq l\leq r+1\) be minimal such that \(\Gamma_l'\subseteq \wall_{\alpha,m+u_{\alpha}}\).
	\item If \(\gamma\) crosses \(\wall_{\alpha,m}\), let \(j\) be minimal such that \(\Gamma_j'\subseteq \wall_{\alpha,m}\) and such that \(\wall_{\alpha,m}\) separates \(\Gamma_{i}'\) from \(\CC_{-\infty}\) for \(i<j\). Moreover, let \(l>j\) be minimal such that \(\Gamma_l'\subseteq \wall_{\alpha,m}\).
\end{enumerate}

\begin{dfn}
	For a nondivisible simple relative root \(\alpha\) as above, we define three root operators as partially defined functions \(\Gamma(\gamma_\mu) \to \Gamma(\gamma_\mu)\).
	Fix some \(\gamma\in \Gamma(\gamma_\mu)\) as above.
	\begin{itemize}
		\item If \(\gamma\) satisfies (I), we define \(e_\alpha(\gamma)=(\mathbf{f}_0 \prec \Sigma_0 \succ \Sigma'_1 \prec \ldots \succ \Sigma'_r \prec \Sigma_r\succ \Sigma_{r+1})\), where
		\[\Sigma_i = \begin{cases}
			\Gamma_i & \text{for } i<j,\\
			s_{\alpha,m+u_{\alpha}}(\Gamma_i) & \text{for } j\leq i< l,\\
			t_{\alpha^\vee}(\Gamma_i) & \text{for } i\geq l.
		\end{cases}\]
		Here, \(s_{\alpha,m+u_{\alpha}}\) denotes the reflection across \(\wall_{\alpha,m+u_{\alpha}}\), while \(t_{\alpha^\vee}\) denotes the translation by \(\alpha^\vee\).
		\item If \(\gamma\) satisfies (II), we define \(f_\alpha(\gamma)=(\mathbf{f}_0 \prec \Sigma_0 \succ \Sigma'_1 \prec \ldots \succ \Sigma'_r \prec \Sigma_r\succ \Sigma_{r+1})\), where
		\[\Sigma_i = \begin{cases}
			\Gamma_i & \text{for } i<j,\\
			s_{\alpha,m}(\Gamma_i) & \text{for } j\leq i< l,\\
			t_{-\alpha^\vee}(\Gamma_i) & \text{for } i\geq l.
		\end{cases}\]
		\item If \(\gamma\) satisfies (III), we define \(\widetilde{e}_\alpha(\gamma)=(\mathbf{f}_0 \prec \Sigma_0 \succ \Sigma'_1 \prec \ldots \succ \Sigma'_r \prec \Sigma_r\succ \Sigma_{r+1})\), where
		\[\Sigma_i = \begin{cases}
			\Gamma_i & \text{for } i<j \text{ and } i\geq l,\\
			s_{\alpha,m}(\Gamma_i) & \text{for } j\leq i< l.
		\end{cases}\]
	\end{itemize}
\end{dfn}

One readily checks that \(e_\alpha\), \(f_\alpha\), and \(\widetilde{e}_\alpha\) take values in \(\Gamma(\gamma_\mu)\).
The following properties of the root operators are standard, compare with \cite[Lemma 6]{GaussentLittelmann:LS}, and follow immediately from the definitions.

\begin{lem}\thlabel{observations about root operators}
	Fix a nondivisible simple relative root \(\alpha\), as well as some \(\gamma\in \Gamma(\gamma_\mu,\nu)\).
	\begin{enumerate}
		\item If \(e_\alpha(\gamma)\) is defined, we have \(e_\alpha(\gamma)\in \Gamma(\gamma_\mu,\nu+\alpha^\vee)\). 
		\item If \(f_\alpha(\gamma)\) is defined, we have \(f_\alpha(\gamma)\in \Gamma(\gamma_\mu,\nu-\alpha^\vee)\).
		\item If \(e_\alpha(\gamma)\) is defined, then \(f_\alpha(e_\alpha(\gamma))\) is defined and equals \(\gamma\).
		\item If \(f_\alpha(\gamma)\) is defined, then \(e_\alpha(f_\alpha(\gamma))\) is defined and equals \(\gamma\).
		\item Let \(p\) and \(q\) be maximal such that \(f_\alpha^p(\gamma)\) and \(e_\alpha^q(\gamma)\) are defined. Then \(p-q = \frac{\langle \alpha,\nu\rangle_S}{u_{\alpha}}\).
	\end{enumerate}
\end{lem}

The result below gives a first indication of the relation between galleries and representation theory.
For a combinatorial gallery \(\delta\), let \(e(\delta)\in X_*(T_{\adj})_I\) correspond to the target of \(\gamma\), so that \(\delta\in \Gamma(\gamma_\mu,e(\delta))\).
We also set \(\Char \Gamma(\gamma_\mu) := \sum_{\gamma\in \Gamma(\gamma_\mu)} \exp(e(\gamma))\).
Finally, let \(\Gamma(\gamma_\mu,\dom)\) be the set of combinatorial galleries \(\delta\in \Gamma(\gamma_\mu)\) for which \(e_\alpha(\delta)\) is not defined for any \(\alpha\).

On the other hand, consider the inertia-invariants \(\widehat{G}^I\) of the Langlands dual group of \(G\); for this section, we consider \(\widehat{G}\) and \(\widehat{G}^I\) over \(\Spec \mathbb{C}\) (or any field of characteristic 0).
In that case, \(\widehat{G}^I\) is a (connected, by the assumption that \(G\) is simply connected) split reductive group with maximal torus \(\widehat{T}^I\) \cite[Theorem 1.1 (4)]{ALRR:Fixed}.
The Weyl group of \(\widehat{G}^I\) is canonically isomorphic to \(W^I\), the relative Weyl group of \(G\).
Then, for \(\mu\in X_*(T)_I^+ \cong X^*(\hat{T}^I)^+\), we let \(V(\mu)\) be the unique irreducible complex representation of \(\widehat{G}^I\) of highest weight \(\mu\).

\begin{cor}\thlabel{character formula 1}
	There is an equality \[\Char \Gamma(\gamma_\mu) = \sum_{\gamma\in \Gamma(\gamma_\mu,\dom)} \Char V(e(\gamma)).\]
\end{cor}
\begin{proof}
	The proof is analogous to \cite[Corollary 1]{GaussentLittelmann:LS}.
	By Weyl's character formula, we have to show
	\[\left(\sum_{w\in W^I} \sgn(w) \exp(w(\rho))\right)\cdot \Char(\Gamma(\gamma_\mu)) = \sum_{\gamma\in \Gamma(\gamma_\mu,\dom)} \left(\sum_{w\in W^I} \sgn(w) \exp\left( w(\rho+e(\gamma))\right) \right).\]
	As \(\Char \Gamma(\gamma_\mu)\) is stable under the \(W^I\)-action by \thref{observations about root operators}, both sides of the above equation are stable under \(W^I\), so it suffices to show the coefficients for the dominant weights agree.
	In other words, let \(\Omega:=\{(w,\delta)\in W^I\times \Gamma(\gamma_\mu)\mid w(\rho)+e(\delta)\in X_*(T)^+\}\), and let \(\Omega'\subseteq \Omega\) be the subset consisting of those \((w,\delta)\) such that either \(w\neq 1\) or \(\delta\notin \Gamma(\gamma_\mu,\dom)\).
	Then we are reduced to showing \(\sum_{(w,\delta)\in \Omega'} \exp(w(\rho) + e(\delta)) = 0\).
	For this, we will construct an involution \(\phi\colon \Omega'\to \Omega'\colon (w,\delta)\mapsto (w',\delta')\), such that \(\sgn(w)= -\sgn(w')\) and \(w(\rho) + e(\delta) = w'(\rho) + e(\delta')\).
	The existence of such an involution then finishes the proof.
	
	Fix some \((w,\delta)\in \Omega\), and let \(w(\rho) + \delta\) be the shifted gallery obtained by translating \(\delta\) facetwise by \(w(\rho)\).
	If moreover \((w,\delta)\in \Omega'\), then \(w(\rho)+\delta\) must meet a proper face \(\FF\) of the dominant Weyl chamber in \(\app\).
	For such a proper face \(\FF\), let \(\Omega'(\FF)\subseteq \Omega'\) be the subset of those \((w,\delta)\), with \(\delta=  (\mathbf{f}_0 \prec \Sigma_0 \succ \Sigma'_1 \prec \ldots \succ \Sigma'_r \prec \Sigma_r\succ \Sigma_{r+1})\), satisfying the following: there exists some \(j\) such that \(w(\rho) + \Gamma_j'\) is contained in the interior of \(\FF\), and \(w(\rho)+\Gamma_l'\) is contained in the interior of the dominant Weyl chamber for \(l>j\).
	
	Then \(\Omega' = \bigsqcup \Omega'(\FF)\), so it suffices to define \(\phi\colon \Omega'(\FF)\to \Omega'(\FF)\) with the required properties, for some fixed \(\FF\).
	Let \(\alpha\) be a nondivisible simple relative root of \(G\) orthogonal to \(\FF\).
	Then \(n:=\frac{\langle \alpha,w(\rho)\rangle_S}{u_{\alpha}} \neq 0\) for any \((w,\delta)\in \Omega'(\FF)\).
	If \(\langle \alpha,w(\rho)\rangle_S<0\), we set \(\phi(w,\delta):=(s_\alpha w,f_\alpha^{-n}(\delta))\), which is well-defined and lies in \(\Omega'(\FF)\).
	Similarly, if \(\langle \alpha,w(\rho)\rangle_S>0\), we set \(\phi(w,\delta):=(s_\alpha w,e_\alpha^{n}(\delta))\in \Omega'(\FF)\).
	\thref{observations about root operators} then implies that \(\phi\) is an involution and satisfies the desired properties.
\end{proof}

The corollary above relates the character of all combinatorial galleries of the same type as \(\gamma_\mu\) with the sum of the characters of certain representations of \(\widehat{G}^I\).
In order to single out the character of \(V(\mu)\), we will introduce the dimension of combinatorial galleries.

\begin{dfn}
	Let \(\gamma = (\mathbf{f}_0 = \Gamma_0' \prec \Gamma_0 \succ \Gamma'_1 \prec \ldots \succ \Gamma'_r \prec \Gamma_r\succ \Gamma'_{r+1})\in \Gamma^+(\gamma_\mu)\) be a positively folded combinatorial gallery.
	\begin{enumerate}
		\item A reflection hyperplane \(\wall\) is a \emph{load-bearing wall} for \(\gamma\) at \(\Gamma_j\) if we have \(\Gamma_j'\subseteq \wall\) but \(\Gamma_j\nsubseteq \wall\), and if \(\wall\) separates \(\Gamma_j\) from \(\CC_{-\infty}\).
		\item The \emph{dimension} \(\dim \gamma\) of \(\gamma\) is the number of pairs \((\wall,\Gamma_j)\) such that the reflection hyperplane \(\wall\) is a load-bearing wall for \(\gamma\) at \(\Gamma_j\).
	\end{enumerate}
\end{dfn}

\begin{rmk}
	In the definition above, \(\gamma\) is assumed to be positively folded.
	Hence, if \(\gamma\) is folded at \(\Gamma_j'\), then all folding hyperplanes containing \(\Gamma_j'\) are load-bearing by definition.
	Note also that a hyperplane \(\wall\) need not be load-bearing for all small facets of \(\gamma\) it contains.
\end{rmk}

\begin{ex}\thlabel{examples of dimension of galleries}
	For any positive nondivisible root \(\beta\), there are \(\frac{\langle \beta,\mu\rangle_S}{u_{\beta}}\) load-bearing walls for \(\gamma_\mu\) which are parallel to \(\wall_{\beta,0}\).
	Hence, we have
	\[\dim \gamma_\mu = \sum_\beta \frac{\langle \beta,\mu\rangle_S}{u_{\beta}} = \langle 2\rho,\mu\rangle.\]
	Here, we are implicitly using that our definition of \(\langle 2\rho,-\rangle\) agrees with the definition from e.g.~\cite[§1]{Richarz:Schubert}.
	For any \(w\in W^I\), we obtain a gallery \(\gamma_{w(\mu)}\) be applying \(w\) facetwise to \(\gamma_\mu\).
	Then \(\gamma_{w(\mu)}\) is minimal without foldings, as the same holds for \(\gamma_\mu\).
	Moreover, it is easy to see that \(\dim \gamma_{w(\mu)} = \langle \rho,\mu+w(\mu)\rangle\).
\end{ex}

\begin{rmk}\thlabel{dual galleries}
	Let \(\gamma = (\mathbf{f}_0 = \Gamma_0' \prec \Gamma_0 \succ \Gamma'_1 \prec \ldots \succ \Gamma'_r \prec \Gamma_r\succ \Gamma'_{r+1} = \mathbf{f}_\nu) \in \Gamma(\gamma_\mu)\) be a combinatorial gallery.
	Consider the gallery \(\gamma^*\), which is obtained by reversing the order of the facets in the shifted gallery \(\gamma-\nu\).
	In particular, \(\gamma^*\in \Gamma(w_0(\gamma_\mu^*))\) is a combinatorial gallery, of the same type as \(\gamma_\mu^*\) and \(w_0(\gamma_\mu^*)\).
	In fact, \(\gamma\mapsto \gamma^*\) defines a bijection \(\Gamma(\gamma_\mu)\cong \Gamma(w_0(\gamma_\mu^*))\), and this bijection identifies the positively folded galleries.
	Moreover, for any nondivisible simple relative root \(\alpha\) we have \(e_\alpha(\gamma) = (f_\alpha(\gamma^*))^*\) (respectively \(f_\alpha(\gamma) = (e_\alpha(\gamma^*))^*\)), as soon as either side of these equalities is defined.
\end{rmk}

Using this notion of \emph{dual} combinatorial galleries, we can determine the effect of root operators on the dimensions and folds of galleries, similar to \cite[Lemma 7]{GaussentLittelmann:LS}.

\begin{lem}\thlabel{effect of root operators}
	Let \(\gamma = (\mathbf{f}_0 = \Gamma_0' \prec \Gamma_0 \succ \Gamma'_1 \prec \ldots \succ \Gamma'_r \prec \Gamma_r\succ \Gamma'_{r+1}) \in \Gamma(\gamma_\mu)\) and \(\alpha\) a nondivisible simple relative root.
	\begin{enumerate}
		\item If \(e_\alpha(\gamma)\) (resp.~\(\widetilde{e}_\alpha(\gamma)\), resp.~\(f_\alpha(\gamma)\)) is defined, we have \(\dim e_\alpha(\gamma) = \dim \gamma +1\) (resp.~\(\dim \widetilde{e}_\alpha(\gamma) = \dim \gamma +1\), resp.~\(\dim f_\alpha(\gamma) = \dim \gamma -1\)).
		\item If \(\gamma\in \Gamma^+(\gamma_\mu)\) and \(\widetilde{e}_\alpha(\gamma)\) is defined, then also \(\widetilde{e}_\alpha(\gamma)\in \Gamma^+(\gamma_\mu)\).
		\item Assume \(\gamma\in \Gamma^+(\gamma_\mu)\) and that \(\widetilde{e}_\alpha(\gamma)\) is not defined. 
		Then, if \(e_\alpha(\gamma)\) or \(f_\alpha(\gamma)\) is defined, it is positively folded.
	\end{enumerate}
\end{lem}
\begin{proof}
	The proof is similar to \cite[Lemma 7]{GaussentLittelmann:LS}; we sketch it for convenience.
	\begin{enumerate}
		\item Consider some facets \(\Gamma_j'\prec \Gamma_j\) of \(\gamma\), and a reflection hyperplane \(\wall\supseteq \Gamma_j'\).
		Then translations preserve the relative position of \(\Gamma_j\) and \(\wall\) with respect to \(\CC_{-\infty}\).
		On the other hand, a reflection reverses the relative position if and only if \(\Gamma_j'\) is contained in a hyperplane parallel to the hyperplane fixed by the reflection.
		From there one deduces the lemma for \(e_{\alpha}\) and \(f_{\alpha}\).
		
		Now, let \(j,l\) be as in (III). Then \(\wall_{\alpha,m}\) is load-bearing at both \(\Gamma_j\) and \(\Gamma_l\) for \(\widetilde{e}_{\alpha}(\gamma)\), but only at \(\Gamma_l\) for \(\gamma\).
		Since \(\widetilde{e}_{\alpha}\) does not affect the relative position at the other facets of the galleries, we conclude that \(\dim \widetilde{e}_{\gamma} = \dim \gamma +1\).
		\item Again, we only have to check the facets for the indices \(j\) and \(l\). 
		But the proof of (1) already implies that all additional folds are positive.
		\item We consider \(e_{\alpha}\), the case of \(f_{\alpha}\) being dual.
		Let \(j,l\) be as in (I); as before, it suffices to consider the folds at \(\Gamma_j'\) and \(\Gamma_l'\).
		By definition, \(\Gamma_j'\) only obtains a positive fold by applying \(e_{\alpha}\), regardless of the assumption on \(\widetilde{e}_{\alpha}(\gamma)\).
		However, if the latter is not defined, this means that \(\wall_{\alpha,m}\) does not separate \(\Gamma_{l-1}\) and \(\Gamma_l\).
		Hence, applying \(e_{\alpha}\) to these two facets gives two facets which are separated by \(\wall_{\alpha,m+2u_{\alpha}}\).
		Since \(\alpha\) is simple, we conclude that all folds of \(e_{\alpha}(\gamma)\) are positive, as required.
	\end{enumerate}
\end{proof}

This allows us to give a bound on the dimension of positively folded combinatorial galleries.

\begin{prop}\thlabel{dimension estimate}
	For any \(\gamma\in \Gamma^+(\gamma_\mu)\), we have \(\dim \gamma \leq \langle \rho,\mu+e(\gamma)\rangle\).
\end{prop}
\begin{proof}
	We proceed by induction on \(\langle \rho,\mu-e(\gamma) \rangle\in \IN\).
	If this number is \(0\), we must have \(e(\gamma)=\mu\).
	But \(\gamma_\mu\) is the only element in \(\Gamma(\gamma_\mu)\) with target \(\mu\), for which the dimension formula holds by \thref{examples of dimension of galleries}.
	So we may assume \(e(\gamma)\neq \mu\), in which case \(\gamma = [\gamma_0,\ldots,\gamma_r]\) with \(\gamma_0\neq \identity\).
	Let \(\alpha\) be a nondivisible simple relative root such that \(s_\alpha \gamma_0< \gamma_0\).
	
	Assume first that \(\widetilde{e}_{\alpha}(\gamma)\) is not defined.
	Then \(\gamma\) does not satisfy condition (III) above, so that it must satisfy condition (I) as \(\gamma_0\neq \identity\).
	In other words, \(e_\alpha(\gamma)\) is defined, and by \thref{effect of root operators} it is positively folded.
	Then the dimension estimate follows by \(\dim e_\alpha(\gamma) = \dim \gamma +1\) and induction.
	
	On the other hand, if \(\widetilde{e}_\alpha(\gamma)\) is defined, we keep applying \(\widetilde{e}_{\alpha'}\) (possibly for different roots), until we get a gallery \(\gamma'\) (necessarily positively folded and with \(e(\gamma') = e(\gamma)\)) and a root \(\beta\) such that \(\widetilde{e}_\beta(\gamma')\) is not defined.
	Then the proposition follows from induction and \thref{effect of root operators}.
\end{proof}

\begin{dfn}
	A positively folded gallery \(\gamma\in \Gamma^+(\gamma_\mu)\) is called an \emph{LS gallery} (or Lakshimibai-Seshadri gallery, in analogy to the terminology from \cite{Littelmann:Paths}) if \(\dim \gamma = \langle \rho, \mu+e(\gamma) \rangle\).
	We denote the set of LS galleries of the same type as \(\gamma_\mu\) by \(\Gamma_{\LS}^+(\gamma_\mu)\).
\end{dfn}

Using these LS galleries, we can now determine the character of a single irreducible representation of \(\widehat{G}^I\), in terms of combinatorial galleries.

\begin{thm}\thlabel{character formula 2}
	\begin{enumerate}
		\item The set \(\Gamma_{\LS}^+(\gamma_\mu)\) is the subset of \(\Gamma(\gamma_\mu)\) generated by \(\gamma_\mu\) under the root operators \(f_\alpha\).
		\item A gallery \(\gamma\in \Gamma^+(\gamma_\mu)\) is an LS gallery if and only if its dual \(\gamma^*\) is an LS gallery.
		\item We have 
		\[\Char V(\mu) = \sum_{\gamma\in \Gamma_{\LS}^+(\gamma_\mu)} \exp(e(\gamma)).\]
	\end{enumerate}
\end{thm}
\begin{proof}
	\begin{enumerate}
		\item If \(\gamma\in \Gamma_{\LS}^+(\gamma_\mu)\), the proof of \thref{dimension estimate} shows that we can obtain \(\gamma_\mu\) from \(\gamma\) by applying root operators of the form \(e_\alpha\).
		By \thref{observations about root operators}, this shows that repeatedly applying the root operators \(f_\alpha\) to \(\gamma_\mu\) generates at least \(\Gamma_{\LS}^+(\gamma_\mu)\).
		To ensure we only get the LS galleries, it suffices to show the root operators \(f_\alpha\) preserve LS galleries, so let again \(\gamma\in \Gamma_{\LS}^+(\gamma_\mu)\).
		Then \(\widetilde{e}_\alpha(\gamma)\) is not defined for dimension reasons, so \(f_\alpha(\gamma)\) is an LS gallery by \thref{effect of root operators} (1) and (3).
		\item follows from (1) and \thref{dual galleries}.
		\item We can repeat the proof of \thref{character formula 1} to the set of LS galleries, rather than all combinatorial galleries. 
		The statement then follows from the observation that \(\gamma_\mu\) is the only LS gallery in \(\Gamma(\gamma_\mu,\dom)\).
	\end{enumerate}
\end{proof}

\subsection{Retractions}

In order to make the semi-infinite orbits enter the picture, we will use the retraction at infinity.
We will also need to relate it to retractions in the spherical building attached to the maximal reductive quotient of the special fiber of \(\Gg\).
Let us briefly recall these notions, and refer to \cite{Brown:Buildings} for more details, in particular for the building at infinity.

\begin{ex}\thlabel{example spherical retraction}
	For a parahoric integral model \(\Pp\) of \(G\) contained in \(\Gg\), the image in \(\mathsf{G}\) of the special fiber of \(\Pp\) is a parabolic subgroup \(\mathsf{P}\).
	In particular, the Iwahori \(\Ii\) gives rise to a Borel \(\mathsf{B}\subseteq \mathsf{G}\).
	Recall that the chambers of \(\buil^s(\mathsf{G},k)\) are in bijection to \(\mathsf{G}/\mathsf{B}(k)\), and similarly for more general facets; we had denoted the simplex corresponding to a parabolic \(\mathsf{P}\) by \(\FF_{\mathsf{P}}\).
	The retraction \(r_{\CC_f, \app^s(\mathsf{G},\mathsf{S},k)}\colon \buil^s(\mathsf{G},k)\to \app^s(\mathsf{G}, \mathsf{S}, k)\) with center \(\CC_f\) can be described explicitly via the Bruhat decomposition, as in \cite[Example 1]{GaussentLittelmann:LS}.
	Namely, consider a proper parabolic \(\mathsf{P}'\subset \mathsf{G}\) of the same type as a standard parabolic \(\mathsf{P}\), along with its associated simplex \(\FF_{\mathsf{P}'}\) in \(\buil^s(\mathsf{G},k)\).
	By the Bruhat decomposition, we can find \(b\in \mathsf{B}(k)\) and a unique \(w\in W_{\mathsf{G}}/W_{\mathsf{P}}\) such that \(\mathsf{P}'= bw\mathsf{P}\) in \(\mathsf{G}/\mathsf{P}(k)\).
	Then \(\FF_{\mathsf{P}'}\) retracts to \(w\FF_{\mathsf{P}}\), i.e., 
	\[r_{\CC_f, \app^s}(\FF_{\mathsf{P}'}) = w\FF_{\mathsf{P}} = \FF_{w\mathsf{P}w^{-1}}.\] 
\end{ex}

\begin{dfn}
	The retraction at \(-\infty\) is the unique map \(r_{-\infty}\colon \buil(G,F)\to \app\) defined via \(r_{-\infty}(\Delta') = r_{\Delta,\app}(\Delta')\), where \(r_{\Delta,\app}\) is the usual retraction of \(\buil(G,F)\) onto \(\app\) with center \(\Delta\) \cite[Definition 1.5.25]{KalethaPrasad:BruhatTits}, and \(\Delta\) is any alcove contained in a common apartment with \(\Delta'\), as well as in a sector equivalent to \(-\CC_f\).
\end{dfn}

Since the fibers of a retraction \(r_{\Delta,\app}\colon \buil(G,F)\to \app\) are exactly the \(U_\Delta\)-orbits on \(\buil(G,F)\), the following proposition can be obtained similarly to \cite[Proposition 1]{GaussentLittelmann:LS}, cf.~also \cite[§2.3]{Schwer:Roots}.

\begin{prop}\thlabel{fibers of retraction}
	The fibers of \(r_{-\infty}\colon \buil(G,F)\to \app\) are the \(U^-(F)\)-orbits on \(\buil(G,F)\).
\end{prop}

We can also facetwise extend \(r_{-\infty}\) to galleries, which gives the following geometric picture.

\begin{prop}\thlabel{retraction fibers are filtrably decomposed}
	The retraction at infinity induces a map \(r_{\gamma_\mu}\colon \Sigma(\gamma_\mu)(\overline{k})\to \Gamma(\gamma_\mu)\).
	The fibers of this map correspond to locally closed subschemes \(C_\delta\subseteq \Sigma(\gamma_\mu)\), and they form a filtrable decomposition.
\end{prop}
\begin{proof}
	As retractions are morphisms of simplicial complexes, the extension of \(r_{-\infty}\) to galleries preserves the types of galleries.
	Since \(r_{-\infty}\) moreover fixes \(\app\), it maps a gallery of type \(t_{\gamma_\mu}\) to a combinatorial gallery.
	This induces the desired map \(r_{\gamma_\mu}\colon \Sigma(\gamma_\mu)(\overline{k})\to \Gamma(\gamma_\mu)\).
	
	In order to describe the fibers of this map, choose some anti-dominant regular cocharacter \(\lambda\colon \IG_{m,\Oo}\to \Ss\), and consider the induced \(\IG_{m,k}\)-action on \(\Sigma(\gamma_\mu)\).
	We claim that the \(C_\delta\) arise from a Bialynicki-Birula decomposition as in \cite{BB:Theorems}.
	As the Bott-Samelson resolution can be \(\IG_m\)-equivariantly embedded in a projective space with linear \(\IG_m\)-action, it follows from \cite[Theorem 3]{BB:Properties} that this decomposition is filtrable.
	The claim can be shown as in \cite[Proposition 6]{GaussentLittelmann:LS}, we sketch the proof for convenience of the reader.
	
	By \cite[Corollary 1.16]{Richarz:Spaces}, we can assume \(k=\overline{k}\).
	Let \(\delta = (\mathbf{f}_0 \prec \Sigma_0 \succ \Sigma'_1 \prec \ldots \prec \Sigma_r\succ \Sigma'_{r+1})\in \Gamma(\gamma_\mu)\) be a combinatorial gallery, and let \(g=(\mathbf{f}_0 \prec \Gamma_0 \succ \Gamma'_1 \prec \ldots \prec \Gamma_r \succ \Gamma'_{r+1})\) be a gallery of type \(t_{\gamma_\mu}\) retracting to \(\delta\).
	By \thref{fibers of retraction}, there exist \(u_j\in U^-(\breve{F})\) such that \(\Gamma_j = u_j\cdot \Sigma_j\), which moreover satisfy \(u_0\cdot \mathbf{f}_0 = \mathbf{f}_0\) and \(u_{j-1}^{-1}u_j\cdot \Sigma'_j = \Sigma'_j\).
	Conversely, any sequence \((u_j)_{0\leq j\leq r}\) satisfying these conditions determines a gallery \((u_0\cdot \mathbf{f}_0 \prec u_0\cdot \Sigma_0 \succ u_1\cdot \Sigma'_1 \prec \ldots \prec \Sigma_r\succ \Sigma'_{r+1})\) retracting to \(\delta\).
	Since \(S\) normalizes \(U^-\) and for any \(t\in S(\breve{F})\), the element \((tu_{j-1}^{-1} t^{-1})(tu_jt{-^1})\) preserves \(\Sigma'_j\) exactly when \(t_{j-1}^{-1}u_j\) does, the fibres of \(r_{-\infty}\) are \(\IG_{m,k}\)-equivariant.
	In particular, the \(\IG_{m,k}\)-fixed points in \(\Sigma(\gamma_\mu)\) are exactly the combinatorial galleries \(\Gamma(\gamma_\mu)\).
	The assumption that \(\lambda\) is anti-dominant then implies that \(\lim_{t\to 0} \lambda(t) U^-(\breve{F}) \lambda(t)^{-1} = 1\), and hence also that \(\lim_{t\to 0} \lambda(s) g = \delta\), as desired.
\end{proof}

\begin{lem}\thlabel{only care about positively folded}
	If \(C_\delta\cap \Fl_\mu(\mathbf{f}_0,\mathbf{f}_\mu)\neq \varnothing\) for \(\delta\in \Gamma(\gamma_\mu)\), then \(\delta\in \Gamma^+(\gamma_\mu)\).
\end{lem}
\begin{proof}
	The proof of \cite[Lemma 11]{GaussentLittelmann:LS} is purely based on building combinatorics, and hence applies verbatim to our setting.
\end{proof}

The fibers of the retraction map describe the intersection of Schubert cells and semi-infinite orbits as follows.

\begin{cor}\thlabel{Cdelta gives semiinfinite orbits}
	The retraction map \(r_{\gamma_\mu}\) restricts to a map \(r^+_{\gamma_\mu}\colon \Fl_\mu(\mathbf{f}_0,\mathbf{f}_\mu)\to \Gamma^+(\gamma_\mu)\).
	Moreover, for \(\nu\in X_*(T_{\adj})_I\) we have a filtrable decomposition
	\[\bigcup_{\delta\in \Gamma^+(\gamma_\mu,\nu)} (C_\delta\cap \Fl_\mu(\mathbf{f}_0,\mathbf{f}_\mu)) = \Ss_{w_0(\nu)}^-\cap \Fl_\mu(\mathbf{f}_0,\mathbf{f}_\mu).\]
\end{cor}
\begin{proof}
	The first statement follows immediately from \thref{only care about positively folded}, while the second statement follows from \thref{fibers of retraction}, and the fact that the Bott-Samelson resolution is \(L^+\Gg\)-equivariant and maps a combinatorial gallery \(\delta\in \Sigma(\gamma_\mu)\) to \(\varpi^{w_0(e(\delta))}\in \Fl_{\leq \mu}(\ff_0,\ff_\mu)\).
\end{proof}

Next, we describe an open affine covering of \(\Sigma(\gamma_\mu)\) as in \cite{GaussentLittelmann:LS}.
They will be built inductively out of the following base case.
\begin{lem}
	For standard parahorics \(\Pp'\supset \Pp\supseteq \Ii\), consider the set \(R\) of affine roots \(\psi\) for which \(\Uu_\psi\subset \Pp'\) and \(\Uu_\psi\nsubseteq \Pp\).
	Then there is an open immersion \(\prod_{\psi \in R} L^+\Uu_{\psi}/L^+\Uu_{\psi+} \to L^+\Pp'/L^+\Pp\), the source of which is isomorphic to (the perfection of) an affine space.
\end{lem}
We denote the image of this open immersion by \(\IA(\Pp'/\Pp)\subseteq L^+\Pp'/L^+\Pp\).
\begin{proof}
	Since \(\Uu_{\psi}\subseteq L^+\Pp'\) implies \(\Uu_{\psi+}\subseteq L^+\Ii\subseteq L^+\Pp\), we have natural maps \[L^+\Uu_{\psi}/L^+\Uu_{\psi+} \to L^+\Pp'/L^+\Pp.\]
	Note that since \(G\) is residually split, its affine root system is reduced, so that \(L^+\Uu_{\psi}/L^+\Uu_{\psi+} \cong \IA_k^{1,\perf}\).
	Now, consider the map \(f\colon \prod_{\psi\in R} L^+\Uu_{\psi}/L^+\Uu_{\psi+} \to L^+\Pp'/L^+\Pp\), induced by the multiplication with respect to a choice of ordering on \(R\).
	Then, under the identification of \(L^+\Pp'/L^+\Pp\) with (the perfection of) a partial flag variety for the maximal reductive quotient of the special fiber of \(\Pp'\), the map \(f\) corresponds exactly to the inclusion of an open cell as in \cite[II.1.9]{Jantzen:Representations}.
\end{proof}

Next, for \(w\in W_{\Pp'}/W_{\Pp}\), consider the finite subsets of affine roots 
\[R^+(w):=\{\psi > 0\mid \Uu_{w^{-1}(\psi)}\nsubseteq \Pp\}\] 
and 
\[R^-(w):=\{\psi<0\mid w(\psi)<0, \quad \Uu_\psi\subseteq \Pp', \quad \Uu_\psi \nsubseteq \Pp\},\] 
where we identify \(w\) with its minimal length representative in \(W_{\Pp'}\).
Then we define \(\IA^+(w):=\prod_{\psi\in R^+(w)} L^+\Uu_\psi/L^+\Uu_{\psi+}\) and \(\IA^-(w):=\prod_{\psi\in R^-(w)} L^+\Uu_\psi/L^+\Uu_{\psi+}\). 
As the notation suggests, these are (perfections of) affine spaces.
Moreover, we have an affine open neighbourhood \(w\IA(\Pp'/\Pp) = \IA^+(w) w \IA^-(w)\) of \(w\) in \(L^+\Pp'/L^+\Pp\).
The lemma above immediately gives the following corollary.

\begin{cor}
	For any \(\delta=[\delta_0,\ldots,\delta_r]\in \Gamma(\gamma_\mu)\), there is a natural open immersion 
	\[\IA^+(\delta_0) \delta_0 \IA^-(\delta_0) \times \ldots \times \IA^+(\delta_r) \delta_r \IA^-(\delta_r)\to \Sigma(\gamma_\mu).\]
	Denoting the corresponding open subscheme by \(\IA(\delta)\subseteq \Sigma(\gamma_\mu)\), we have \(\delta\in C_\delta \subseteq \IA(\delta)\) and \(\Sigma(\gamma_\mu) = \bigcup_{\delta\in \Gamma(\gamma_\mu)} \IA(\delta)\).
\end{cor}

\begin{rmk}\thlabel{index sets agree}
	Recall that since \(G\) is residually split, its affine root system is reduced.
	Hence, for any \(w\in W_{\Pp'}/W_{\Pp}\), the affine roots in \(R^+(w)\sqcup (-R^-(w))\) are exactly those positive affine roots whose corresponding reflection hyperplane contain \(F_{\Pp'}\) but not \(wF_{\Pp}\).
\end{rmk}

Now, let us get back to understanding the open subsets \((C_\delta\cap \Fl_\mu(\mathbf{f}_0,\mathbf{f}_\mu)) \subseteq C_\delta\), for \(\delta\in \Gamma^+(\gamma_\mu)\).
We will start with the whole \(C_\delta\)'s, for which we need some more combinatorics.

\begin{nota}
	Let \(\delta=[\delta_0,\ldots, \delta_r] = (\mathbf{f}_0 \prec \Sigma_0 \succ \Sigma'_1 \prec \ldots \prec \Sigma_r\succ \Sigma'_{r+1}) \in \Gamma^+(\gamma_\mu)\) be a positively folded combinatorial gallery.
	By \thref{index sets agree} above, the set of reflection hyperplanes in \(\app\) containing \(\Sigma_j'\) but not \(\delta_j \Sigma_j\) is naturally in bijection with \(R(\delta_j):=R^+(\delta_j)\sqcup R^-(\delta_j)\); we fix this identification in what follows.
	Then, we denote by \(J_{-\infty}(\delta)\subseteq \bigsqcup_{j=0,\ldots, r} R(\delta_j)\) those hyperplanes that are load-bearing at \(\Sigma_j'\), and let \(J_{-\infty}^\pm(\delta):= J_{-\infty}(\delta) \cap (\bigsqcup_j R^\pm(\delta_j))\).
	Since \(\delta\) is positively folded, we have \(J_{-\infty}^-(\delta) = \bigsqcup_{j} R^-(\delta_j)\).
	Moreover, we the number of elements in \(J_{-\infty}\) is exactly \(\dim \delta\).
\end{nota}

\begin{prop}\thlabel{description of Cdelta}
	Let \(\delta = [\delta_0, \ldots, \delta_r]\in \Gamma^+(\gamma_\mu)\).
	Then the inclusion \(C_\delta \subseteq \IA(\delta)\) can be identified with
	\[ \prod_{j=0,\ldots, r} \left(\left(\prod_{\psi\in R^+(\delta_j)\cap J_{-\infty}(\delta)} L^+\Uu_{\psi}/L^+\Uu_{\psi+} \right) \delta_j \left( \prod_{\psi\in R^-(\delta_j)} L^+\Uu_{\psi}/L^+\Uu_{\psi+} \right) \right)\]
	\[\subseteq \prod_{j=0,\ldots, r} \left(\left(\prod_{\psi \in R^+(\delta_j)} L^+\Uu_{\psi}/L^+\Uu_{\psi+} \right) \delta_j \left( \prod_{\psi\in R^-(\delta_j)} L^+\Uu_{\psi}/L^+\Uu_{\psi+} \right) \right).\]
	Consequently, we have \(C_\delta\cong \IA_k^{\dim \delta, \perf}\).
\end{prop}
\begin{proof}
	We may assume that \(F=\breve{F}\), i.e., that \(k=\overline{k}\), and proceed as in \cite[Lemma 13]{GaussentLittelmann:LS}.
	Let \(g=[g_0,\ldots, g_r]\) be a gallery in \(\buil(G,F)\) with \(g_j\in \IA^+(\delta_j) \delta_j \IA^-(\delta_j)\). 
	We need to determine the conditions under which \(g\) retracts to \(\delta\).
	We write \(\delta=[\delta_0,\ldots, \delta_r] = (\mathbf{f}_0 \prec \Sigma_0 \succ \Sigma'_1 \prec \ldots \prec \Sigma_r\succ \Sigma'_{r+1})\).
	
	We start with \(g_0\), i.e., we need to determine when \(r_{-\infty}(g_0\Gamma_{\Pp_0}) = \delta_0(\Gamma_{\Pp_0})\).
	Recall that \(\Pp_0\) was the standard parahoric corresponding to the type \(t_0\) from \eqref{standard gallery of types}, and \(\Gamma_{\Pp_0}\) is the corresponding facet.
	Consider the maximal reductive quotient \(\mathsf{G}\) of the special fiber of \(\Gg\), and the parabolic \(\mathsf{P}\subset \mathsf{G}\) corresponding to \(\Pp_0\subset \Gg\); in particular we have \(\mathsf{G}/\mathsf{P} \cong L^+\Gg/L^+\Pp_0\), and we can view \(\delta_0\) as an element of \(W_0/W_0'\cong W_{\mathsf{G}}/W_{\mathsf{P}}\).
	Let \(\mathsf{S}\subseteq \mathsf{G}\) be the maximal torus corresponding to \(S\subseteq G\), and \(\mathsf{P}^-\supseteq \mathsf{S}\) the parabolic opposite to \(\mathsf{P}\).
	In this case, we can describe the retraction at \(-\infty\) via \(r_{-\infty}(g_0\Gamma_{\Pp_0}) = r_{w_0\Delta_f, \Aa}(g_0\Gamma_{\Pp_0})\), where \(w_0\) is the longest element in the finite Weyl group \(W_{\mathsf{G}}\) of \(\mathsf{G}\).
	By \thref{identification of apartments} and the isomorphism \(L^+\Gg/L^+\Pp_0\cong \mathsf{G}/\mathsf{P}\), we are reduced to determining the spherical facets which retract to \(\delta_0(\FF_{\mathsf{P}})\) under \(r_{-\CC_f}=r_{w_0\CC_f, \app^s}\).
	But under the identification of the facets of the same type as \(\mathsf{P}^-\) in \(\buil^s(\mathsf{G},k)\) with \(\mathsf{G}/\mathsf{P}^-(k)\), \thref{example spherical retraction} tells us that \(r_{-\CC_f}^{-1}(\delta_0 \Gamma) = \mathsf(B)^- \delta_0 w_0\) in \(G/\mathsf{P}^-\).
	Note also that \[\mathsf{B}^-\delta_0 w_0 \cong \left(\prod_{\psi <0, (\delta_0 w_0)^{-1}(\psi)>0} L^+\Uu_{\psi}/L^+\Uu_{\psi+} \right)\delta_0 w_0 = \delta_0 \left(\prod_{\psi <0, \delta_0(\psi)>0} L^+\Uu_{\psi}/L^+\Uu_{\psi+}\right) w_0,\]
	and that right-multiplication by \(w_0\) gives an isomorphism \(\mathsf{G}/\mathsf{P} \cong \mathsf{G}/\mathsf{P}^-\).
	We can then conclude that \(g_0\Gamma_{\Pp_0}\) retracts onto \(\delta_0(\Gamma_{\Pp_0})\) exactly when
	\[g_0\in \delta_0 \prod_{\psi\in R^-(\delta_0)} L^+\Uu_{\psi}/L^+\Uu_{\psi+},\]
	again using \(\mathsf{G}/\mathsf{P} \cong L^+\Gg/L^+\Pp_0\).
	Since \(R^+(\delta_0)\cap J_{-\infty}(\delta) = \varnothing\), this gives the desired condition for \(g_0\).
	
	For \(j>0\), we can retract \(g\) step by step. 
	So let 
	\[g'=[\delta_0,\ldots, \delta_{j-1},g_j,\ldots, g_r] = (\mathbf{f}_0\prec \ldots \prec\Sigma_{j-1} \succ \Sigma_j' \prec \Xi_j \succ \Xi_{j+1}' \prec \ldots),\]
	where \(\Xi_j = \delta_0 \ldots \delta_{j-1} g_j F_{t_j}\) and \(g_j\in \IA^+(\delta_j)\delta_j \IA^-(\delta_j)\).
	Consider the simple affine roots \(\zeta_1,\ldots,\zeta_{l(\delta_j)}\) corresponding to the simple reflections in a reduced decomposition \(s_{\zeta_1}\ldots s_{\zeta_j}\) of \(\delta_j\).
	Then we have 
	\[R^+(\delta_j) = \{\zeta_1,s_{\zeta_1}(\zeta_2),\ldots, s_{\zeta_1}\ldots s_{\zeta_{l(\delta_j)-1}}(\zeta_{l(\delta_j)})\}.\]
	In particular, there exist elements \(a_1,\ldots,a_{l(\delta_j)}\in \breve{F}\) and \(b_\psi\in \breve{F}\) such that 
	\[g_j = p_{\zeta_1}(a_1) \cdot s_{\zeta_1} \cdot \ldots \cdot p_{\zeta_{l(\delta_j)}}(a_{l(\delta_j)}) \cdot s_{\zeta_{l(\delta_j)}} \cdot \prod_{\psi\in R^-(\delta_j)} p_{\psi}(b_\psi),\]
	where \(p_{\psi}\colon L^+\Uu_{\psi}/L^+\Uu_{\psi+}\to L^+\Pp_j'/ L^+\Pp_j\) denotes the locally closed immersion.
	
	To determine the necessary conditions on these \(a_i\) and \(b_\psi\) such that \(g'\) retracts to \(\delta\), let us first assume all the \(b_{\psi}\) are zero.
	Let \[\Gamma_i = \delta_0 \ldots \delta_{j-1} p_{\zeta_1}(a_1) s_{\zeta_1} \ldots p_{\zeta_{i}}(a_{i}) s_{\zeta_{i}} \Delta_f\]
	for any \(i\).
	Then by \cite[Proposition 2.1.9]{BruhatTits:Groupes1}, \((\Gamma_0,\ldots,\Gamma_{l(\delta_j)})\) is a minimal gallery of alcoves joining \(\Gamma_0\succ \Sigma_{j-1}\) and \(\Gamma_{l(\delta_j)}\succ \Xi_j\).
	By minimality, \(\Xi_j\) retracts to \(\Sigma_j\) exactly when each \(\Gamma_i\) retracts to \(\Theta_i:=\delta_0\ldots \delta_{j-1} s_{\zeta_1}\ldots s_{\zeta_i}\Delta_f\), for \(1\leq i\leq l(\delta_j)\).
	We can again apply this retraction step by step, so we need to determine when \(\Upsilon_i:=\delta_0\ldots \delta_{j-1} s_{\zeta_1}\ldots s_{\zeta_{i-1}} p_{\zeta_i}(a_i)s_{\zeta_i}\Delta_f\) retracts to \(\Theta_i\).
	
	We may assume \(\Upsilon_i\neq \Theta_i\), in which case \((\Theta_{i-1}\succ \Theta_i'\prec \Upsilon_i)\) and \((s_{\wall}\Theta_{i-1} \succ \Theta_i' \prec \Upsilon_i )\) are minimal galleries in any apartment containing them.
	Here, \(\Theta_i'\) is the obvious codimension 1 face of \(\Theta_i\), and \(s_{\wall}\) is the reflection corresponding to the unique hyperplane \(\wall\) containing \(\Theta_i'\).
	Let \(\app'\) be an apartment containing \(\Upsilon_i'\) and \(\wall\).
	Let us fix some alcove \(\Delta'\subset \app\cap \app'\) such that \(r_{-\infty}=r_{\Delta',\app}\), at least for all facets we will concern ourselves with.
	If \(\wall\) is not load-bearing at this place, then \(\Delta'\) and \(\Theta_i\) are not separated by \(\wall\).
	But \(\Upsilon_i\) and \(\Delta'\) are separated by \(\wall\), so that the properties of \(r_{-\infty}\) imply that \(\Upsilon_i\) can only retract to \(\Theta_i\) if they are already equal; this contradiction shows that \(a_i\) must be \(0\).
	On the other hand, if \(\wall\) is load-bearing at this place, a similar argument shows that \(a_i\in \breve{F}\) can be arbitrary, and we have determined the possible values for the \(a_i\) under the assumption that each \(b_\psi=0\).
	
	However, since the reflection hyperplanes corresponding to \(\psi\in R^-(\delta_j)\) are automatically load-bearing at \(\Sigma_j'\), we can show in the same way that \(b_\psi\) can be arbitrary, independently of the values of \(a_i\).
	This gives the desired description for the \(g_j\)'s, and hence concludes the proof.
\end{proof}

\begin{cor}\thlabel{intersection with Weyl group conjugate}
	For any \(w\in W_{\Gg}\), we have \(\IA^{\langle\rho,\mu-w(\mu)\rangle, \perf} \cong \Fl_\mu(\ff_0,\ff_\mu)\cap \Ss_{w(\mu)}^-\subseteq \Fl(\ff_0,\ff_\mu)\).
\end{cor}
It will follow from \thref{reduce to simply connected} that this a similar statement also holds without the assumption that \(G\) is semisimple or simply connected.
\begin{proof}
	By \thref{Nonemptyness of intersections}, we have \(\Fl_\mu(\ff_0,\ff_\mu)\cap \Ss_{w(\mu)}^- = \Fl_{\leq \mu}(\ff_0,\ff_\mu)\cap \Ss_{w(\mu)}^-\), so that this intersection moreover agrees with \(\bigsqcup_{\delta\in \Gamma^+(\gamma_\mu,w(\mu))} C_{\delta}\) by \thref{Cdelta gives semiinfinite orbits}.
	Now, \(\Gamma^+(\gamma_\mu,w(\mu))\) is a singleton, and consists of the gallery obtained by facetwise applying \(w\) to \(\gamma_\mu\).
	We conclude by \thref{description of Cdelta} and \thref{examples of dimension of galleries}.
\end{proof}

\subsection{Intersections of Schubert cells and semi-infinite orbits}

Finally, we will use the results above to deduce \thref{thm.intro-decomposition}, generalizing \thref{intersection with Weyl group conjugate} above.
By \thref{Cdelta gives semiinfinite orbits}, we need to understand the intersection of each \(C_\delta\) with \(\Fl_\mu(\mathbf{f}_0, \mathbf{f}_\mu)\subseteq \Sigma(\gamma_\mu)\), i.e., which galleries in \(C_\delta\) are minimal.
This will be done by cutting \(\gamma_\mu\) (and hence all galleries of the same type) into smaller \emph{triple galleries}.

\begin{dfn}\thlabel{defi-triple galleries}
	A \emph{triple gallery} is a sequence \((\Upsilon\succ \Xi' \prec \Xi)\) of facets of \(\buil(G,F)\), such that \(\Xi'\) is a codimension 1 face of both \(\Upsilon\) and \(\Xi\).
	It is called \emph{minimal} if \(\Mm_{\app'}(\Upsilon,\Xi)\) consists exactly of those walls containing \(\Xi'\) but not \(\Xi\), for any apartment \(\app'\) containing \(\Upsilon\) and \(\Xi\). 
	Such an apartment always exists, and minimality does not depend on the choice of such an apartment.
\end{dfn}

Clearly, if \(\gamma = (\Gamma'_0 \prec \Gamma_0 \succ \Gamma'_1 \prec \ldots \succ \Gamma'_r \prec \Gamma_r\succ \Gamma'_{r+1})\) is a minimal gallery in \(\buil(G,F)\), then each \((\Gamma_{j-1} \succ \Gamma'_j \prec \Gamma_j)\) is a minimal triple gallery of faces.
Moreover, as in \cite[Remark 7]{GaussentLittelmann:LS}, \((\Upsilon\succ \Xi' \prec \Xi)\) is minimal exactly when for any alcove \(\Delta\succ \Upsilon\) in \(\app'\) at maximal distance from \(\Xi\), the length of any minimal gallery of alcoves joining \(\Delta\) and \(\proj_F(\Delta)\) is \(|\Mm_{\app'}(\Xi',\Xi)|\).

Now, let \(\epsilon = (t_{j-1} \subset t_j'\supset t_j)\) be a triple gallery of types appearing in \(t_{\gamma_\mu}\), and consider the standard parahorics \(\Pp_j'\supset \Pp_j\) of types \(t_j'\supset t_j\) respectively.
Denote by \(\tau\in W_\aff\) the shortest representative of the longest class in \(W_j'/W_j\).
Then we can explicitly describe the minimal triple galleries of type \(\epsilon\), as in \cite[Lemma 12, Proposition 8]{GaussentLittelmann:LS}.

\begin{lem}\thlabel{description of minimal triple galleries}
	Let \(\rho = (\Upsilon\succ \Xi' \prec \Xi)\) be a triple gallery of faces in \(\buil(G,F)\) of type \(\epsilon\), with \(\Upsilon \preceq \Delta_f\).
	\begin{enumerate}
		\item If \(\Xi = x \Xi_{\Pp_j}\) for some \(x\in \IA^+(\tau)\tau\), then \(\rho\) is minimal.
		\item Conversely, if \(\rho\) is minimal, we can find \(x\in \IA^+(\tau)\tau\) and \(y\in \Stab_E(\Pp_j')\) such that \(\Xi = yx\Xi_{\Pp_j}\).
	\end{enumerate}
\end{lem} 
\begin{proof}
	(1) By \cite[Proposition 2.1.9]{BruhatTits:Groupes1}, there exists a minimal gallery of alcoves of length \(l(\tau)\) between \(\Delta_f\succeq \Upsilon\) and \(x\Delta_f = \proj_{\Xi}(\Delta_f)\), so that \(\Delta_f\) is at maximal distance from \(\Xi\).
	But then \(\rho\) can obtained from such a gallery above using the action of the stabilizer of \(\Upsilon \cup \Xi\), so that the lemma follows from \cite[Lemma 3]{GaussentLittelmann:LS}.
	
	(2) Let \(\app'\) be an apartment containing \(\rho\), and let \(r:=|\Mm_{\app'}(\Xi',\Xi)|\).
	Choose an alcove \(\Delta \succeq \Upsilon\) in \(\app'\) at maximal distance from \(\Xi\).
	Then minimality of \(\rho\) implies that any minimal gallery of alcoves joining \(\Delta\) and \(\proj_{\Xi}(\Delta)\) has length \(r\); let \(\xi = (\Delta, \Delta_1, \ldots, \Delta_r = \proj_{\Xi}(\Delta))\) be such a gallery of alcoves.
	Then there exists \(y\in \Stab_E(\Pp_j')\) such that the minimal gallery \(y\rho = (y\Delta = \Delta_f, y\Delta_1, \ldots, y\Delta_r)\) starts at \(\Delta_f\).
	By \cite[Proposition 2.1.9]{BruhatTits:Groupes1} and \cite[Remark 16 (2)]{GaussentLittelmann:LS}, we can find \(x\in \IA^+(\tau)\tau\) such that \(y\proj_{\Xi}(\Delta) = x\Delta_f\).
	This implies \(\Xi = y^{-1} x \Xi_{\Pp_j}\), as desired.
\end{proof}

\begin{rmk}
	As in \cite[Proposition 10]{GaussentLittelmann:LS}, we could already conclude that for regular \(\mu\in X_*(T_{\adj})_I^+\) and \(\delta\in\Gamma(\gamma_\mu)\), the intersection \(\Fl_\mu(\ff_0,\ff_\mu)\cap C_\delta\) is isomorphic to some \(\IA^r\times \IG_m^s\) (not just up to some filtrable decomposition).
	Indeed, using \thref{description of minimal triple galleries}, we can describe which galleries in \(C_\delta\) are minimal in terms of the identification from \thref{description of Cdelta}.
	However, for general \(\mu\), we need some extra arguments, as in \cite[Proposition 9]{GaussentLittelmann:LS}.
\end{rmk}

Consider \(\Pp'\supset \Pp\) standard parahorics and \(w,\tau\in W_{\Pp'}/W_{\Pp}\) as above, where as usual we identify an equivalence class in \(W_{\Pp'}/W_{\Pp}\) with its shortest representative in \(W_{\Pp}\).

\begin{lem}\thlabel{stratification for triples}
	The intersection
	\[\IA^+(w)w\IA^-(w)\cap \IA^+(\tau)\tau\]
	in \(\Pp'/\Pp\) admits a filtrable decomposition into perfect cells.
\end{lem}
\begin{proof}
	As in \cite[Proposition 9]{GaussentLittelmann:LS}, we reduce to showing the claim for the intersection \(\Ii v^{-1}\cap \Ii^- \Pp'/\Pp\) in \(\Pp'/\Pp\).
	Using the isomorphism \(\Pp'/\Pp\cong \mathsf{P}'/\mathsf{P}\), it suffices to show the claim for the intersection of Schubert varieties in a finite flag varieties of \(\mathsf{G}\).
	But \(\mathsf{G}\) is split as \(G\) was assumed residually split, so that \cite[Corollary 1.2]{Deodhar:Geometric} gives us a decomposition into cells.
	This decomposition is moreover filtrable by \cite[Lemma 2.5]{Dudas:DeligneLusztigRestriction}.
\end{proof}

From now on, we remove the assumption that \(G\) is semisimple or simply connected.
The next lemma explains how we can still use the results of this section.

\begin{lem}\thlabel{reduce to simply connected}
	Let \(G_{\adj}\) be the adjoint quotient of \(G\), and \(G_{\sico}\) its simply connected cover.
	Let \(\Gg_{\adj}\) and \(\Gg_{\sico}\) be the respective parahoric models corresponding to \(\Gg\) under isomorphisms \(\buil(G_{\sico},F)\cong \buil(G,F) \cong \buil(G_{\adj},F)\).
	Fix some \(\mu\in X_*(T)_I^+\); by composing with \(T\to T_{\adj}\) we can view it as a cocharacter in \(X_*(T_{\adj})_I\).
	Then there are natural isomorphisms
	\[\Fl_{G,\leq \mu}(\ff_0,\ff_\mu) \cong \Fl_{G_{\adj},\leq \mu}(\ff_0,\ff_\mu) \cong \Fl_{G_{\sico},\leq \mu}(\ff_0,\ff_\mu),\]
	which are equivariant for the \(L^+\Gg\)- and \(L^+\Gg_{\sico}\)-actions respectively, and compatible for the intersections with the semi-infinite orbits.
\end{lem}
Note that we have a canonical isomorphism \(\Fl_{G,\leq \mu}(\ff_0,\ff_\mu) \cong \Gr_{\Gg,\leq \mu}\), and similarly for \(G_{\adj}\).
For \(G_{\sico}\), this only holds if \(\mu\in X_*(T_{\adj})_I^+\) is induced by a cocharacter of \(G_{\sico}\).
\begin{proof}
	Consider the morphism \(\Fl_{\Gg}\to \Fl_{\Gg_{\adj}}\) induced by the quotient \(G\to G_{\adj}\). 
	It is clearly \(LG\)-equivariant, which will imply the combatiblity for the intersections with the semi-infinite orbits.
	It restricts to a morphism \(\Fl_{G,\leq \mu}(\ff_0,\ff_\mu) \to \Fl_{G_{\adj},\leq \mu}(\ff_0,\ff_\mu)\) as both schemes are defined as orbit closures.
	By \cite[Lemma 3.8]{BhattScholze:Projectivity}, it suffices to show this map is a universal homeomorphism, i.e., surjective, radicial, and universally closed.
	As Schubert varieties are perfections of proper schemes, any morphism between them is universally closed.
	For the other properties, it suffices to show the same properties on the stratifications by \(\Ii\)-orbits.
	This is shown in \cite[Proposition 3.5]{HainesRicharz:Normality} in equal characteristic, but the proof also works in mixed characteristic.
	(We note that the use of \cite[Proposition 3.1]{HainesRicharz:Normality} becomes superfluous, as perfected Schubert varieties are always normal by \cite[Proposition 3.7]{AGLR:Local} and \cite[Lemma 2.8]{CassXu:Geometrization}.)
	
	Next, the quotient \(LG_{\sico}\to LG_{\adj}\) realizes any connected component of \(\Fl_{\Gg_{\adj}}\) as the quotient of \(LG_{\sico}\) by a very special parahoric.
	Indeed, for the neutral component this follows as in the previous paragraph, and in general by conjugating \(\Gg\) by a suitable element in \(LG_{\adj}\).
	The same argument as in the previous paragraph then concludes the proof.
\end{proof}

\begin{rmk}\thlabel{remark pu3}
	Consider the group \(G=\PU_3\) corresponding to a ramified quadratic extension.
	Then although there are two very special standard parahorics, which are not conjugate to each other, we also have \(\pi_0(LG)=\pi_1(\PU_3)_I = 1\), so that we only need to consider the neutral connected component.
\end{rmk}

\begin{thm}\thlabel{Intersections of Schubert cells and semi-infinite orbits}
	For any \(\mu\in X_*(T)_I^+\) and \(\nu\in X_*(T)_I\), the intersection \(\Gr_{\Gg,\mu}\cap \Ss_\nu^-\) admits a filtrable decomposition by subschemes of the form \(\IA_k^{r,\perf}\times_k \IG_m^{s,\perf}\).
\end{thm}

\begin{rmk}\thlabel{intersections for positive orbits}
	Since we have opted to follow the methods of \cite{GaussentLittelmann:LS}, the semi-infinite orbits \(\Ss_\nu^-\) for the negative Borel arose. 
	It follows however easily from the theorem that \(\Gr_{\Gg,\mu}\cap \Ss_\nu^+\) also admits a filtrable decomposition into products of \(\IA^{1,\perf}\)'s and \(\IG_m^{\perf}\)'s, by changing the choice of Borel.
\end{rmk}

\begin{proof}
	By \thref{Cdelta gives semiinfinite orbits}, it suffices to stratify the intersections \(C_\delta\cap \Gr_{\Gg,\mu}\), for \(\delta\in \Gamma^+(\gamma_\mu)\).
	By \thref{Schubert cells are the minimal galleries}, this amounts to determining which galleries in \(C_\delta\) are minimal.
	Breaking up a (not necessarily combinatorial) gallery \(\gamma \in C_\delta\) into triple galleries as in \thref{defi-triple galleries}, we see as in \cite[Remark 8]{GaussentLittelmann:LS} that \(\gamma\) is minimal if and only if each triple gallery appearing is minimal.
	
	Recall that \(\Sigma(\gamma_\mu)\) was defined as an iterated fibration with partial flag varieties as fibers.
	By \thref{description of minimal triple galleries}, the minimal galleries in \(C_\delta\subseteq \Sigma(\gamma_\mu)\) are those contained in the subscheme given by an iterated fibration with fibers \(\IA^+(\delta_j)\delta_j \IA^-(\delta_j)\cap \IA^+(\tau_j)\tau_j\), where \(\tau_j\) is as in \thref{description of minimal triple galleries}.
	But we know these fibers are stratified by perfect cells by \thref{stratification for triples}.
	Hence, \(C_\delta\cap \Gr_{\Gg,\mu}\) is an iterated fibration whose fibers admit stratifications by perfect cells, so that the same holds for \(C_\delta\cap \Gr_{\Gg,\mu}\) as well, concluding the proof.
\end{proof}

\begin{ex}
	Let us describe the intersections \(\Gr_{\Gg,\mu}\cap \Ss^\pm_\nu\) explicitly for the group \(\SU_3\) associated to a ramified quadratic extension \(\widetilde{F}/F\).
	This is similar to the case of the split group \(\PGL_2\), for which we refer to \cite[Example 3.38]{CassvdHScholbach:Geometric}.
	Although this \(\SU_3\) has two conjugacy classes of very special parahorics, the discussion below works uniformly for both cases.
	
	We identify the injections \(X_*(S)\subseteq X_*(T)_I\subseteq X_*(S)\otimes_{\IZ} \IR\) with \(\IZ\subseteq \frac{1}{2}\IZ \subseteq \IR\), where the dominant cocharacters (for the usual choice of Borel \(B\)) correspond to positive numbers.
	Let \(\mu\in X_*(T)_I^+\).
	Then there is a unique choice of \(\gamma_\mu\), namely the gallery of length \(4\mu\) which goes straight from \(0\) to \(\mu\).
	For each \(\nu\in X_*(T)_I\) satisfying \(|\nu|\leq |\mu|\) and \(\mu-\nu\in \IZ\), there may be multiple combinatorial galleries in \(\Gamma(\gamma_\mu,\nu)\), but there is a unique positively folded one.
	This is the gallery having at most one fold, in which case it changes from moving to \(\CC_{-\infty}\) to the opposite direction; we call this gallery \(\delta_\nu\).
	If \(\nu\) does not satisfy the conditions above, then \(\Gamma(\gamma_\mu,\nu)=\varnothing\).
	
	Then \thref{description of Cdelta} gives us \(C_{\delta_\nu}\cong \IA_k^{\dim \delta_\nu,\perf}\).
	To understand \(C_{\delta_\nu}\cap \Gr_{\Gg, \mu}\), we need to see which points correspond to minimal galleries.
	As in \cite[Proposition 10]{GaussentLittelmann:LS}, the only potential obstruction to minimality arises at the fold. 
	We deduce that \(C_{\delta_\nu}\cap \Gr_{\Gg, \mu}\cong \IA_k^{\dim \delta,\perf}\) if \(\nu=\pm \mu\), as \(\delta_\nu\) has no folds in this case, and \(C_{\delta_\nu}\cap \Gr_{\Gg, \mu}\cong \IG_{m,k}^{\perf} \times \IA_k^{\dim \delta-1,\perf}\) otherwise.
	It is also clear that \(\dim \delta_\nu\) is the number of steps where \(\delta_\nu\) moves away from \(\CC_{-\infty}\), so that \(\dim \delta_\nu = 2(\mu+\nu)\).
	This completely determines the intersections \(\Gr_{\Gg, \mu}\cap \Ss_{\nu}^{\pm}\), by \thref{Intersections of Schubert cells and semi-infinite orbits}.
\end{ex}

To end this section, we record some corollaries that will be used later on.

\begin{lem}\thlabel{triviality of torsor}
	Consider \(n\gge 0\) for which the action of \(L^+\Gg\) factors through \(L^n\Gg\), and let \(\Pp^n_{w_0(\mu)}\subseteq L^n\Gg\) be the stabilizer of \(\varpi^{w_0(\mu)}\).
	Then for any \(\delta \in \Gamma^+(\gamma_\mu)\), the \(\Pp^n_{w_0(\mu)}\)-torsor \(L^n\Gg\to \Gr_{\Gg,\mu}\) is trivial over the locally closed subscheme \(C_{\delta}\cap \Gr_{\Gg,\mu}\).
\end{lem}
\begin{proof}
Consider the following diagram with cartesian square:
\[\begin{tikzcd}
	L^n\Gg \arrow[r, "g"] & L^n\Gg/(L^{>0}\Gg\cap \Pp^n_{w_0(\mu)}) \arrow[rr, "f"] \arrow[d] && \Gr_{\Gg,\mu} \arrow[d]\\
	& L^0\Gg \arrow[r]&\mathsf{G} \arrow[r] & \mathsf{G}/\mathsf{P}_{w_0(\mu)}.
\end{tikzcd}\]
Here, \(\mathsf{G}\) is the reductive quotient of the special fiber \(L^0\Gg\) of \(\Gg\), and the parabolic subgroup \(\mathsf{P}_{w_0(\mu)}\subseteq \mathsf{G}\) is the image of \(\Pp^0_{w_0(\mu)}\).
By \cite[II.1.10 (5)]{Jantzen:Representations}, the projection \(\mathsf{G}\to \mathsf{G}/\mathsf{P}_{w_0(\mu)}\) has sections over the attractors for the \(\IG_{m,k}\)-action on \(\mathsf{G}/\mathsf{P}_{w_0(\mu)}\) induced by an anti-dominant regular cocharacter. 
Note that these attractors are affine spaces.
Since \(L^0\Gg\to \mathsf{G}\) is a vector bundle, and any vector bundle over an affine space is trivial, we see that \(L^0\Gg\to \mathsf{G}/\mathsf{P}_{(w_0(\mu))}\) has a section over each such attractor.
In particular, \(f\) has sections over the preimage in \(\Gr_{\Gg,\mu}\) of any such attractor.
On the other hand, since \(\Gr_{\Gg,\mu}\to \mathsf{G}/\mathsf{P}_{w_0(\mu)}\) is an affine morphism, \(L^n\Gg/(L^{>0}\Gg\cap \Pp^n_{w_0(\mu)})\) is also affine.
And since \(g\) is a torsor under a unipotent group scheme by \thref{cellularity of stabilizer} below, it is a trivial torsor by \cite[Proposition A.6]{RicharzScholbach:Intersection}.
So we are left to show that each \(C_{\delta}\cap \Gr_{\Gg,\mu}\) is contained in the preimage of some attractor in \(\mathsf{G}/\mathsf{P}_{w_0(\mu)}\).

Consider the Bott-Samelson scheme \(\Sigma(\gamma_\mu)\), along with its projection onto the first factor \(\Pp_{\ff_0}/\Pp_0\), which is clearly \(L^+\Gg\)-equivariant. 
Moreover, there is an equivariant identification \(\Pp_{\ff_0}/\Pp_0 \cong \mathsf{G}/\mathsf{P}_{w_0(\mu)}\), by \thref{first facet is fixed}, under which the restriction of \(\Sigma(\gamma_\mu) \to \Pp_{\ff_0}/\Pp_0\) agrees with \(\Gr_{\Gg,\mu} \to \mathsf{G}/\mathsf{P}_{w_0(\mu)}\).
We conclude by observing that \(\Sigma(\gamma_\mu) \to \Pp_{\ff_0}/\Pp_0\) preserves the attractors for the \(\IG_{m,k}\)-action induced by a regular anti-dominant cocharacter, and that the \(C_{\delta}\)'s are exactly these attractors.
\end{proof}

The following cellularity result was used in the proof above, and is a generalization of \cite[Remark 4.2.8]{RicharzScholbach:Intersection} to the residually split case.

\begin{lem}\thlabel{cellularity of stabilizer}
	For each \(n\geq m\geq 0\), the kernel \(\ker(\Pp_{w_0(\mu)}^n\to \Pp_{w_0(\mu)}^m)\) is a perfected vector group.
	In particular, \(\Pp_{w_0(\mu)}^n\) is perfectly cellular.
\end{lem}
\begin{proof}
	The first assertion follows from \cite[Proposition A.9]{RicharzScholbach:Intersection}.
	For the second one, since \(\Pp_{w_0(\mu)}^0\) is affine, it suffices by \cite[Proposition A.6]{RicharzScholbach:Intersection} to show that \(\Pp_{w_0(\mu)}^0\) is cellular.
	But it is a torsor over \(\mathsf{P}_{w_0(\mu)}\) under the unipotent radical of \(L^0\Gg\), which is also a vector group, so we are left to show cellularity of \(\mathsf{P}_{w_0(\mu)}\).
	Since we assumed \(G\) residually split, this is a parabolic of the split reductive group \(\mathsf{G}\). 
	In particular, \(\mathsf{P}_{w_0(\mu)}\) is the semi-direct product of a (split) Levi subgroup, which is cellular by the Bruhat decomposition, and its unipotent radical, which is a vector group.
\end{proof}

Let \(a\colon L^n\Gg\times (\Ss_{\nu}^-\cap \Gr_{\Gg,\mu})\to \Gr_{\Gg,\mu}\) be the action map, for some \(\nu\in X_*(T)_I\).

\begin{cor}\thlabel{fibres of action map are cellular}
	For any \(\lambda\in X_*(T)_I\), the preimage \(a^{-1}(\Ss_{\lambda}^+\cap \Gr_{\Gg,\mu})\) admits a filtrable decomposition into perfect cells.
	Moreover, it has dimension \(\dim L^n\Gg + \langle \rho,\lambda-\nu\rangle\).
\end{cor}
\begin{proof}
	Filtering \(\Ss_\lambda^+\cap \Gr_{\Gg,\mu}\) by perfect cells as in \thref{intersections for positive orbits}, it suffices to consider the preimage under \(a\) of such a cell \(X\).
	By \thref{triviality of torsor}, the map \(L^n\Gg\to \Gr_{\Gg,\mu}\colon g\mapsto g\cdot \varpi^\mu\) has a section \(s\colon X\to L^n\Gg\) over \(X\subseteq \Gr_{\Gg,\mu}\).
	This gives an isomorphism \(X\times a^{-1}(\varpi^\mu) \cong a^{-1}(X)\colon (x,f)\mapsto s(x)\cdot f\), where we used the left action of \(L^n\Gg\) on \(L^n\Gg\times (\Ss_{\nu}^-\cap \Gr_{\Gg,\mu})\) via the first factor.
	Since \(X\) is a cell, it suffices to decompose \(a^{-1}(\varpi^{\mu})\), which is isomorphic to \(a^{-1}(\varpi^{w_0\mu})\) as \(L^n\Gg\) acts transitively on \(\Gr_{\Gg,\mu}\).
	
	It is also enough to decompose the preimages of any cell \(Y\) from \thref{Intersections of Schubert cells and semi-infinite orbits} under the projection \(a^{-1}(\varpi^{w_0(\mu)}) \subseteq L^n\Gg\times (\Ss_{\nu}^-\cap \Gr_{\Gg,\mu}) \to \Ss_{\nu}^-\cap \Gr_{\Gg,\mu}\).
	But this preimage is isomorphic to \(\Pp_{w_0(\mu)}^n\times Y\), by sending \((p,y)\in \Pp_{w_0(\mu)}^n\times Y\) to \((p\cdot s'(x^{-1}),x)\), where \(s'\colon Y\to L^n\Gg\) is a section of \(L^n\Gg\mapsto \Gr_{\Gg,\mu}\colon g\mapsto g\cdot \varpi^{w_0(\mu)}\) obtained from \thref{triviality of torsor}.
	Thus, we are left to decompose \(\Pp_{w_0(\mu)}\), which was covered in \thref{cellularity of stabilizer}.
	
	Finally, the dimension follows from \(\dim \Pp_{w_0(\mu)}^n = \dim L^n\Gg - \langle 2\rho,\mu\rangle\), as well as \(\dim \Ss_\lambda^+\cap \Gr_{\Gg,\mu} = \langle \rho,\mu+\lambda\rangle\) and \(\Ss_\nu^-\cap \Gr_{\Gg,\mu} = \langle \rho,\mu-\nu\rangle\).
\end{proof}

\section{The convolution product}\label{sect-convolution}

We now go back to an arbitrary connected reductive group \(G/F\).
The goal of this section is to construct a convolution product for equivariant motives on affine flag varieties, and show it preserves stratified Tate motives.
At very special level, we will moreover show it is t-exact, up to possible issues related to the failure of exactness of the tensor product.
Our approach is a generalization of \cite[§3]{RicharzScholbach:Satake} and \cite[§4.3]{RicharzScholbach:Witt} to the case of ramified groups.

From now on, we fix an extension \(k'/k\), either finite or an algebraic closure, corresponding to a complete unramified extension \(F'/F\) such that \(G\) becomes residually split over \(F'\).
We then use this fixed \(k'\) to define categories of Artin--Tate motives as in \thref{defi Artin-Tate}.

\subsection{Convolution affine flag varieties}

We start by recalling the twisted products of affine flag varieties (and semi-infinite orbits) and record their basic properties.
These are not very surprising, but the case of ramified groups has not yet appeared in the literature with enough details.

\begin{dfn}
	Let \(\Gg_0,\Gg_1,\ldots,\Gg_n/\Oo\) be (not necessarily very special) parahoric integral models of \(G\).
	The \emph{convolution affine flag variety} is the twisted product \(\Fl_{\Gg_1} \widetilde{\times} \Fl_{\Gg_2} \widetilde{\times} \ldots \widetilde{\times}\Fl_{\Gg_n}\), which is by definition the contracted product \(LG \overset{L^+\Gg_1}{\times} LG \overset{L^+\Gg_2}{\times} \ldots \overset{L^+\Gg_{n-1}}{\times} \Fl_{\Gg_n}\).
\end{dfn}

For \(i=1,\ldots,n\), one can multiply the first \(i\) factors of this contracted product to get a map \(m_i\colon \Fl_{\Gg_1} \widetilde{\times} \Fl_{\Gg_2} \widetilde{\times} \ldots \widetilde{\times}\Fl_{\Gg_i} \to \Fl_{\Gg_i}\); in particular \(m_1\) is the projection onto the first factor.
This induces an isomorphism
\begin{equation}\label{iso of convolution Gr}
	(m_1,\ldots,m_n)\colon \Fl_{\Gg_1} \widetilde{\times} \Fl_{\Gg_2} \widetilde{\times} \ldots \widetilde{\times}\Fl_{\Gg_n} \cong \Fl_{\Gg_1} \times \Fl_{\Gg_2} \times \ldots \times\Fl_{\Gg_n},
\end{equation}
so that convolution affine flag varieties are representable by ind-(perfect projective) ind-schemes.

By replacing \(\Fl_{\Gg_i}\) in the contracted product above by an \(L^+\Gg_{i-1}\)-orbit \(\Fl_{\Gg_i,w_i}\), we get a locally closed subvariety \(\Fl_{\Gg_1, w_1} \widetilde{\times} \Fl_{\Gg_2, w_2} \widetilde{\times} \ldots \widetilde{\times}\Fl_{\Gg_n, w_n}\).
This defines a stratification of the convolution flag variety, indexed by \(\prod_{i=1}^n (W_{\Gg_{i-1}}\backslash \tilde{W}/ W_{\Gg_i})\).
Similarly, we can define twisted products of the closures of such orbits, which are perfect projective schemes.

Next, assume that \(\Gg_0=\Gg_1=\ldots=\Gg_n=\Gg\) is a very special parahoric integral model, with associated twisted affine Grassmannian \(\Gr_{\Gg}\). 
As in §\ref{subsec--semi-infinite}, choose a cocharacter \(\lambda\colon \IG_{m,\Oo}\to \Ss\).
Let \(M\subseteq P^\pm\subseteq G\) be the fixed points and attractor/repeller of the associated \(\IG_{m,F}\)-action on \(G\), and consider the associated stratification \(\Gr_{\Gg} = \bigsqcup_{\nu\in \pi_1(M)_{\Gal(\overline{F}/F)}} \Ss_{P,\nu}^\pm\).

If \(\lambda\) is regular dominant, then \(P=B\) is a Borel with Levi factor \(M=T\), and the \(\Ss_{P,\nu}^\pm= \Ss_\nu^\pm\) are the semi-infinite orbits.
If furthermore \(k=k'\), then for each \(\nu\) we have an \(L^+\Uu^\pm\)-torsor \(LU^\pm \to \Ss_{\nu}^\pm\).
We can use these to form the twisted product \(\Ss_{\nu_1}^\pm \widetilde{\times} \Ss_{\nu_2}^\pm \widetilde{\times} \ldots \widetilde{\times} \Ss_{\nu_n}^\pm\), for any tuple \((\nu_1,\ldots,\nu_n)\) of elements in \(\pi_1(T)_{\Gal(\overline{F}/F)}\).
This time, the partial multiplication maps induce an isomorphism
\[\Ss_{\nu_1}^\pm \widetilde{\times} \Ss_{\nu_2}^\pm \widetilde{\times} \ldots \widetilde{\times} \Ss_{\nu_n}^\pm \cong \Ss_{\nu_1}^\pm \times \Ss_{\nu_1+\nu_2}^\pm \times \ldots \times \Ss_{\nu_1+\ldots +\nu_n}^\pm.\]
Under the isomorphism \((\Gr_{\Gg})^{\widetilde{\times}n} \cong (\Gr_{\Gg})^{\times n}\), the diagonal \(\IG_{m,k}\)-action on the product induced by \(\lambda\) induces a \(\IG_{m,k}\)-action on the convolution affine Grassmannian.
Then it is clear that the attractors and repellers of this \(\IG_{m,k}\)-action are exactly the twisted products \(\Ss_{\nu_1}^\pm \widetilde{\times} \Ss_{\nu_2}^\pm \widetilde{\times} \ldots \widetilde{\times} \Ss_{\nu_n}^\pm\) above, compare \cite[§6]{Yu:Integral}.
In particular, we can consider the twisted products of \(\Gr_{\Bb^\pm}\) with itself along the \(L^+\Bb^\pm\)-torsor \(LB^\pm\to \Gr_{\Bb^\pm}\). This yields a coproduct
\[\Gr_{\Bb^\pm} \widetilde{\times} \Gr_{\Bb^\pm} \widetilde{\times} \ldots \widetilde{\times} \Gr_{\Bb^\pm} = \coprod_{\nu_i\in \pi_1(T)_{\Gal(\overline{F}/F)}} \Ss_{\nu_1}^\pm \widetilde{\times} \Ss_{\nu_2}^\pm \widetilde{\times} \ldots \widetilde{\times} \Ss_{\nu_n}^\pm,\]
which determines a stratification of \((\Gr_{\Gg})^{\widetilde{\times} n}\).

If \(\lambda\) is not regular dominant, then we can no longer view the \(\Ss_{P,\nu}^\pm\) as the quotient of a loop group by a positive loop group, so we cannot define their twisted products in a naive way.
However, we can define them in analogy to the \(P=B\) case, by considering the isomorphism
\[\Gr_{\Pp^\pm} \widetilde{\times} \Gr_{\Pp^\pm} \widetilde{\times} \ldots \widetilde{\times} \Gr_{\Pp^\pm} \cong \Gr_{\Pp^\pm} \times \Gr_{\Pp^\pm} \times \ldots \times \Gr_{\Pp^\pm}.\]
Then for any tuple \((\nu_1,\ldots,\nu_n)\in (\pi_1(M)_{\Gal(\overline{F}/F)})^n\) we \emph{define} the twisted product 
\[\Ss_{P,\nu_1}^\pm \widetilde{\times} \Ss_{P,\nu_2}^\pm \widetilde{\times} \ldots \widetilde{\times} \Ss_{P,\nu_n}^\pm\] 
as the preimage of \(\Ss_{P,\nu_1}^\pm \times \Ss_{P,\nu_1+\nu_2}^\pm \times \ldots \times \Ss_{P,\nu_1+\ldots +\nu_n}^\pm\) under the above isomorphism.
These are the connected components of the iterated twisted product \((\Gr_{\Pp^\pm})^{\widetilde{\times} n}\), and define a stratification of \((\Gr_{\Gg})^{\widetilde{\times}n}\).
They are also the attractor/repeller of the \(\IG_{m,k}\)-action on \((\Gr_{\Gg})^{\widetilde{\times} n}\) induced by \(\lambda\) via the diagonal action on \((\Gr_{\Gg})^{\times n}\).

\subsection{Preservation of Artin-Tate motives}

Let \(\Gg,\Gg',\Gg''/\Oo\) be three parahoric models of \(G\), which we assume to contain a common Iwahori \(\Ii\) for simplicity. 
Consider the maps
\begin{equation}\label{convolution diagram}
	L^+\Gg' \backslash LG / L^+\Gg \times L^+\Gg \backslash LG / L^+\Gg'' \xleftarrow{\overline{q}} L^+\Gg' \backslash LG \overset{L^+\Gg}{\times}LG/ L^+\Gg'' \xrightarrow{\overline{m}} L^+\Gg' \backslash LG / L^+\Gg''
\end{equation}
induced by the projection and multiplication on \(LG\).

\begin{dfn}
	The \emph{(derived) convolution product} is the functor
	 \[\star:=\overline{m}_! \overline{q}^! (-\boxtimes-)\colon \DM(L^+\Gg'\backslash LG/L^+\Gg) \times \DM(L^+\Gg\backslash LG/L^+\Gg'') \to \DM(L^+\Gg'\backslash LG/L^+\Gg'').\]
\end{dfn}

\begin{rmk}\thlabel{remarks about derived convolution}
	\begin{enumerate}
		\item By descent, we have natural equivalences \[\DM(L^+\Gg' \backslash LG / L^+\Gg) \cong \DM(L^+\Gg'\backslash \Fl_{\Gg}),\] and similarly for the other parahorics.
		Hence, we will also consider \(\star\) as a functor \[\DM(L^+\Gg'\backslash \Fl_{\Gg}) \times \DM(L^+\Gg\backslash \Fl_{\Gg''}) \to \DM(L^+\Gg'\backslash \Fl_{\Gg''}).\]
		\item By base change, the derived convolution product admits natural associativity isomorphisms, cf.~\cite[Lemma 3.7]{RicharzScholbach:Satake} for details.
		\item Any map of prestacks admits an upper-! functor by construction, while \(\overline{m}_!\) exists by descending \(m_!\), which exists for the morphism \(m\colon \Fl_{\Gg} \widetilde{\times} \Fl_{\Gg''} \to \Fl_{\Gg''}\) of ind-(pfp schemes).
		On the other hand, \(\boxtimes\) can be constructed as in \cite[Proposition 2.4.4]{RicharzScholbach:Intersection}.
	\end{enumerate}
\end{rmk}

The main purpose of this subsection is to show the derived convolution product preserves Artin-Tate motives.
And while we are mostly interested in the case where \(\Gg=\Gg'=\Gg''\) is very special, we will reduce to the case where \(\Gg=\Gg'=\Gg'=\Ii\) are Iwahori models, in order to use the full affine flag variety rather than the affine Grassmannian.

\begin{prop}\thlabel{Conv-Tate}
	The convolution product \(\star\colon \DM(\Gg'\backslash \Fl_{\Gg}) \times \DM(\Gg\backslash \Fl_{\Gg''}) \to \DM(\Gg'\backslash \Fl_{\Gg''})\) preserves anti-effective stratified Artin-Tate motives.
\end{prop}
Here, stratified Artin-Tate motives in \(\DM(\Gg'\backslash \Fl_{\Gg})\) are defined with respect to the stratification of \(\Fl_{\Gg}\) by \(L^+\Gg'\)-orbits, and similar for the other parahorics.
\begin{proof}
	We may assume \(G\) is residually split and set \(k=k'\), in which case we need to show \(\star\) preserves stratified Tate motives.
	The proof is then similar to \cite[Theorem 3.17]{RicharzScholbach:Satake}, cf.~also \cite[Theorem 4.8]{RicharzScholbach:Witt}, and \cite[Definition and Lemma 4.11]{CassvdHScholbach:Geometric} for anti-effectivity; let us recall it.
	It suffices to show that for any \(w\in W_{\Gg'}\backslash \tilde{W} /W_\Gg\) and \(w'\in W_\Gg\backslash \tilde{W}/ W_{\Gg''}\), the non-equivariant motive underlying \(\iota_{w,*}\unit \star \iota_{w',*}\unit\in \DTM_{L^+\Gg'}(\Fl_{\Gg''})\) is (anti-effective) stratified Tate; note that both \(\iota_{w,*}\unit\) and \(\iota_{w',*}\unit\) are naturally equivariant.
	
	\textbf{(I) First, assume \(\Gg=\Gg'=\Gg''=\Ii\) are Iwahori models.}
	This case contains the key geometric computations, and we analyse different cases further.
	
	\emph{(I.a) \(w=w'=s\) is a simple reflection.} 
	In this case \(\Fl_{\Ii,s}\subset \Fl_{\Ii,\leq s}\) can be identified with the inclusion \(\IA_k^{1,\perf}\subseteq \IP_k^{1,\perf}\).
	Moreover, \eqref{iso of convolution Gr} restricts to an isomorphism \(\tau\colon \Fl_{\Ii,\leq s}\ \widetilde{\times} \Fl_{\Ii,\leq s} \cong \Fl_{\Ii,\leq s} \times \Fl_{\Ii,\leq s}\).
	More specifically, we have \(\tau(\Fl_{\Ii,e}\widetilde{\times} \Fl_{\Ii,e}) = \Fl_{\Ii,e}\times \Fl_{\Ii,e}\), \(\tau(\Fl_{\Ii,e} \widetilde{\times} \Fl_{\Ii,s}) = \Fl_{\Ii,e} \times \Fl_{\Ii,s}\), and \(\tau(\Fl_{\Ii,s} \widetilde{\times} \Fl_{\Ii,e}) = \Delta(\Fl_{\Ii,s})\), where \(\Delta\) is the diagonal of \(\Fl_{\Ii,\leq s}\).
	Thus, the diagram \cite[p.~1619]{RicharzScholbach:Satake} works in our setting as well, and we are reduced to showing that \(f_*(\unit) = \iota_{s,*}\unit \star \iota_{s,*}\unit \in \DM(\IP_k^{1,\perf})\) lies in \(\DTM(\IP_k^{1,\perf},\IA_k^{1,\perf}\coprod \{\infty\})^{(\anti)}\), where \(f\colon (\IA_k^{1,\perf}\times \IP_k^{1,\perf})\setminus \Delta(\IA_k^{1,\perf}) \to \IP_k^{1,\perf}\) is the projection onto the second factor.
	(We use \(f_*\), instead of \(f_!\) as in loc.~cit., in order to handle anti-effective motives as well.
	To show *-pushforwards commute with exterior, and hence twisted products, we can use \cite[Proposition 2.1.20]{JinYang:Kunneth}.)
	This follows from the localization sequence
	\begin{equation}\label{localization for convolution}
		\iota_{\infty,*}\iota_{\infty}^!f_*(\unit) \cong \iota_{\infty,*}(\unit(-1)[-2]) \to f_*(\unit) \to \iota_{\IA^1_k,*}\iota_{\IA^1_k}^!f_*(\unit)\cong \iota_{\IA^1_k,*} (\unit(-1)[-1] \oplus \unit),
	\end{equation}
	since \(q^{-1}(\{\infty\})\cong \IA^1_k\) and \(q^{-1}(\IA^1_k)\cong \IA^1\times \IG_{m,k}\).
	
	\emph{(I.b) \(w,w'\in \tilde{W}\) satisfy \(l(ww') = l(w)+l(w')\).} Then the multiplication map restricts to an isomorphism \(\Fl_{\Ii,w}\widetilde{\times} \Fl_{\Ii,w'} \cong \Fl_{\Ii,ww'}\), so that \(\iota_{w,*}\unit \star \iota_{w',*}\unit \cong \iota_{ww',*}\unit\) is (anti-effective) stratified Tate by \thref{Schubert stratification is WT}.
	
	\emph{(I.c) \(w'=s\) is a simple reflection.} 
	The case \(l(ws) = l(w)+1\) is already handled, so we can assume \(l(ws) = l(w)-1\).
	Then the previous case shows \(\iota_{w,*}\unit \cong \iota_{ws,*}\unit \star \iota_{s,*}\unit\).
	Convolving \eqref{localization for convolution} on the left with \(\iota_{ws,*}\unit\) gives an exact triangle
	\[\iota_{ws,*}\unit (-1)[-2] \to \iota_{ws,*}\unit \star (\iota_{s,*}\unit \star \iota_{s,*}\unit) \to (\iota_{w,*}\unit (-1)[-2] \oplus \iota_{w,*}\unit[-1]).\]
	We conclude this case by associativity of \(\star\), \thref{remarks about derived convolution}.
	
	\emph{(I.d) \(w,w'\in \tilde{W}\) are arbitrary.} 
	Fixing a reduced expression for \(w'\), this follows by using the second and third cases.
	
	\textbf{(II) Next, assume \(\Gg'=\Gg''=\Ii\) are Iwahori's, and \(\Gg\supseteq \Ii\).}
	Let \(w\in \tilde{W}/W_{\Gg}\) and \(w'\in W_{\Gg}\backslash \tilde{W}\).
	Note that the action map \(L^+\Ii\to \Fl_{\Gg,w}\) is already surjective when restricted to the pro-unipotent radical of \(L^+\Ii\).
	Since \(\Fl_{\Gg,w}\) is itself an affine space, this action map admits a section, and hence \((L^+\Ii w L^+\Gg) \overset{L^+\Ii}{\times} (L^+\Gg w'L^+\Ii)\to (L^+\Ii w L^+\Gg) \overset{L^+\Gg}{\times} (L^+\Gg w'L^+\Ii)\) admits a section as well.
	As in \cite[Proposition 3.26]{RicharzScholbach:Satake}, this allows us to reduce to the case \(\Gg=\Ii\), which was already covered.
	
	\textbf{(III) Finally, let \(\Gg,\Gg',\Gg''\) be arbitrary} (but still containing a common Iwahori \(\Ii\).)
	Since the stratifications involved are given by orbits, we can use \cite[Proposition 3.1.23]{RicharzScholbach:Intersection} to reduce to the case \(\Gg'=\Ii\).
	Using \thref{fix asymmetry}, we can similarly assume \(\Gg''=\Ii\).
	Since we already handled this case, we can conclude the proof.
\end{proof}

\subsection{t-exactness of the convolution product}

Now, let us specialize to the situation where \(\Gg=\Gg'=\Gg''\) is a very special parahoric, with associated (twisted) affine Grassmannian \(\Gr_\Gg=\Fl_\Gg\).
Using \thref{Conv-Tate}, we view the convolution as a functor \(\star\colon \DATM_{L^+\Gg}(\Gr_{\Gg}) \times \DATM_{L^+\Gg}(\Gr_{\Gg}) \to \DATM_{L^+\Gg}(\Gr_{\Gg})\).
As usual, this has no hope to be t-exact (for \(\IZ[\frac{1}{p}]\)-coefficients): already if \(\Gg\) is the trivial group, then convolution agrees with the tensor product, which is not t-exact.
However, this is the only obstruction.

\begin{prop}\thlabel{convolution t-exact}
	The functor \[\overline{m}_! \overline{q}^! \pH^0 (-\boxtimes -)\colon \DATM_{L^+\Gg}(\Gr_{\Gg})\times \DATM_{L^+\Gg}(\Gr_{\Gg}) \to \DATM_{L^+\Gg}(\Gr_{\Gg})\] is t-exact, and hence preserves mixed Artin-Tate motives.
\end{prop}
\begin{proof}
	It suffices to check this for bounded objects.
	As the \(\IZ_\ell\)-étale realizations are t-exact by \thref{texactness of realization}, and jointly conservative by \thref{conservativity of etale realization}, it suffices to show the statement for étale \(\IZ_\ell\)-cohomology, with \(\ell\neq p\).
	In this situation t-exactness follows from semi-smallness of the convolution morphism, \thref{semismallness of convolution}, as in \cite[Proposition 4.2]{MirkovicVilonen:Geometric}.
\end{proof}

\begin{lem}\thlabel{semismallness of convolution}
	Equip \(\Gr_{\Gg}\) with the stratification by \(L^+\Gg\)-orbits, and \(\Gr_{\Gg} \widetilde{\times} \Gr_{\Gg}\) with the stratification given by twisted products of Schubert cells.
	Then the convolution morphism \(m\colon \Gr_{\Gg} \widetilde{\times} \Gr_{\Gg}\) is stratified semi-small.
\end{lem}
\begin{proof}
	Using that the intersections \(\Ss_\nu^+\cap \Gr_{\Gg,\mu}\) are equidimensional of dimension \(\langle \rho,\mu+\nu \rangle\), \thref{Nonemptyness of intersections}, the proofs of \cite[Corollary 3.4, Lemma 4.4]{MirkovicVilonen:Geometric} carry over verbatim.
\end{proof}

This allows us to turn \(\MATM_{L^+\Gg}(\Gr_{\Gg})\) into a (not yet symmetric) monoidal abelian category.

\begin{dfn}
	The (truncated) convolution product on \(\MATM_{L^+\Gg}(\Gr_{\Gg})\) is given by 
	\[\conv:=\overline{m}_! \overline{q}^! \pH^0 (-\boxtimes -)\colon \MATM_{L^+\Gg}(\Gr_{\Gg})\times \MATM_{L^+\Gg}(\Gr_{\Gg}) \to \MATM_{L^+\Gg}(\Gr_{\Gg}).\]
\end{dfn}

\begin{prop}
	The convolution product \(\conv\) induces a monoidal structure on \(\MATM_{L^+\Gg}(\Gr_{\Gg})\).
\end{prop}
\begin{proof}
	We need to construct an associativity constraint.
	This can be obtained by base change, using iterated twisted products of \(\Gr_{\Gg}\), along with the associativity constraint of the (truncated) tensor product \(\otimes\) on \(\MATM_{L^+\Gg\times L^+\Gg \times L^+\Gg}(\Gr_{\Gg} \times \Gr_{\Gg} \times \Gr_{\Gg})\).
	The required coherences then also follow from the similar conditions for \(\otimes\).
\end{proof}

\section{Constant terms and the fiber functor}\label{sect-CT and F}

Our next goal is to introduce and study the constant term functors in the situation at hand. 
These functors are well known in the geometric Langlands program, and will help us reduce certain questions to the case of tori (or smaller Levi's), which are easier to handle.
Throughout this section, \(\Gg/\Oo\) denotes a very special parahoric model of a connected reductive group \(G/F\).

\subsection{Constant term functors}

Let \(\lambda\colon \IG_{m,\Oo}\to \Ss\) be a dominant cocharacter defined over \(\Oo\), and consider the induced \(\IG_{m,F}\)-action on \(G\). 
As before, the attractor, repeller and fixed points of this action are given by \(P^+\), \(P^-\) and \(M = P^+\cap P^-\) respectively, where \(P^\pm\) are opposite parabolics with Levi \(M\).
Since \(\lambda\) is dominant, we have \(B^\pm \subseteq P^{\pm}\), with equality exactly when \(\lambda\) is regular dominant.
By \thref{flag varieties of fixed points and attractor}, the (twisted) affine Grassmannians of their smooth \(\Oo\)-models \(\Pp^\pm\) and \(\Mm\) can be identified with the attractor, repeller and fixed points respectively of the \(\IG_{m,k}\)-action on \(\Gr_\Gg\).
These affine Grassmannians moreover all admit a natural \(L^+\Mm\)-action.
We denote the corresponding hyperbolic localization diagram as follows, including the induced maps on prestacks obtained by quotienting out this \(L^+\Mm\)-action
\begin{equation}\label{hyperbolic localization diagram}
	\begin{tikzcd}[row sep=small]
	\Gr_{\Mm} \arrow[d] &\Gr_{\Pp^\pm} \arrow[l, "q_P^\pm"'] \arrow[r, "p_P^\pm"] \arrow[d]&\Gr_\Gg \arrow[d]\\
	L^+\Mm\backslash \Gr_{\Mm} &L^+\Mm\backslash\Gr_{\Pp^\pm} \arrow[l, "\overline{q}_P^\pm"'] \arrow[r, "\overline{p}_P^\pm"]&L^+\Mm\backslash\Gr_\Gg
	\end{tikzcd}
\end{equation}

As in \cite[Construction 2.2]{Richarz:Spaces}, we obtain a natural transformation of functors \((q_P^-)_*(p_P^-)^! \to (q_P^+)_! (p_P^+)^*\), which is an equivalence when restricted to \(\IG_{m,k}\)-monodromic objects (i.e., objects generated under colimits by the image of \(\DM_{\IG_{m,k}}(\Gr_\Gg)\to \DM(\Gr_\Gg)\), cf.~\cite[Proposition 2.5]{CassvdHScholbach:Geometric} for more details about the motivic situation).
Since \(L^+\Mm\) is pro-(perfectly smooth), this descends to a natural transformation 
\[(\overline{q}_P^-)_*(\overline{p}_P^-)^! \to (\overline{q}_P^+)_! (\overline{p}_P^+)^*.\]
Moreover, since forgetting the equivariance is conservative and the \(\IG_{m,k}\)-action factors through \(L^+\Mm\), the above natural transformation is already an equivalence.
In \cite{FarguesScholze:Geometrization, AGLR:Local}, this is defined to be the constant term functor.
However, we prefer to include a shift into the definition of the constant term functor in order to make it t-exact, as in \cite{CassvdHScholbach:Geometric}.
This yields the following definitions, and we refer to \thref{rmk abuse notation} for the slight abuse of notation.

\begin{dfn}
	Let \(\rho_{M}\) denote the half-sum of the positive roots of \(M\). Then the \emph{degree} associated to \(P\) is the locally constant function 
	\[\deg_P:\Gr_{\Mm}\to \Gr_{\Mm/\Mm_{\der}} \cong X_*(M/M_{\der})_{\Gal(\overline{F}/F)} \xrightarrow{\langle2\rho_G-2\rho_{M},- \rangle} \IZ.\]
\end{dfn}

\begin{dfn}\thlabel{dfn--CT}
	The \emph{constant term functor} associated to \(P\) is the functor
	\begin{equation}\label{equation CT}
		\CT_P:=(\overline{q}_P^-)_*(\overline{p}_P^-)^![\deg_P] \cong (\overline{q}_P^+)_! (\overline{p}_P^+)^*[\deg_P] \colon \DM_{L^+\Gg}(\Gr_\Gg) \to \DM_{L^+\Mm}(\Gr_{\Mm}),
	\end{equation}
	where we implicitly restrict part of the equivariance by applying the forgetful functor \(\DM_{L^+\Gg}(\Gr_\Gg)\to \DM_{L^+\Mm}(\Gr_\Gg)\).
	
	For \(\nu\in \pi_0(\Gr_{\Mm})\), we denote the functor obtained by restricting \(\CT_P\) to this connected component by \(\CT_{P,\nu}\).
\end{dfn}

These constant term functors are very well-behaved, and most of their properties are already well known in different settings.
So we recall certain of these properties for convenience of the reader, but we will be brief and give references for their proofs.
Recall that we assumed \(\Gg\) is very special, so that \(G\) is automatically quasi-split by \cite[Lemma 6.1]{Zhu:Ramified}.
In particular, if \(\lambda\) is regular, \(P^\pm=B^\pm\) are opposite Borels and \(M=T\) is a maximal torus.

\begin{rmk}\thlabel{rmk abuse notation}
	The whole story above really depends only on the parabolics \(P^\pm\), rather than on the cocharacter \(\lambda\), which is why we have used a subscript \(P\) in \(\deg_P\) and \(\CT_P\).
	This is also consistent with \cite[Proposition VI.7.13]{FarguesScholze:Geometrization} and \cite[Definition 5.3]{CassvdHScholbach:Geometric}.
\end{rmk}

\begin{lem}\thlabel{composition of CTs}
	Let \(P'\subseteq P\subseteq G\) be parabolics with Levi quotients \(M'\subseteq M\), and denote \(Q:=\im(P'\to M)\). Then there is a natural equivalence of functors \(\CT_{P'}\cong \CT_Q\circ \CT_P\).
\end{lem}
\begin{proof}
	As in \cite[Lemma 5.6]{CassvdHScholbach:Geometric}, this follows from base change and the compatibility of taking affine Grassmannians with fiber products along quotient maps.
\end{proof}

Recall that an object in \(\DM_{L^+\Gg}(\Gr_\Gg)\) is bounded if, after forgetting the equivariance, it is supported on some scheme.

\begin{lem}\thlabel{CT conservative}
	The constant term functor \(\CT_P\) is conservative when restricted to bounded objects.
\end{lem}
\begin{proof}
	As in \cite[Lemma 5.8]{CassvdHScholbach:Geometric}, we use \thref{composition of CTs} to reduce to the case where \(B\) is a Borel.
	Then we proceed by induction on the strata, using \thref{Nonemptyness of intersections} to see that for any Schubert variety, there is a semi-infinite orbit intersecting it in a single point.
\end{proof}

\begin{prop}\thlabel{CT preserves Artin-Tate}
	The constant term functor \(\CT_P\) preserves Artin-Tate motives, i.e., restricts to a functor
	\[\CT_P\colon \DATM_{L^+\Gg}(\Gr_\Gg)\to \DATM_{L^+\Mm}(\Gr_\Mm).\]
\end{prop}
\begin{proof}
	By definition, an object in \(\DM_{L^+\Gg}(\Gr_{\Gg})\) is stratified Artin--Tate exactly when it becomes stratified Tate after pullback along \(\Spec k'\to \Spec k\), where \(k'/k\) is the field extension fixed in the beginning of §\ref{sect-convolution}.
	Moreover, the constant term functor \(\CT_P\) commutes with pullback along \(\Spec k'\to \Spec k\).
	Thus, we may assume \(k'=k\), so that \(G\) is residually split, in which case we need to show that \(\CT_P\) preserves stratified Tate motives. 
	
	We now proceed as in \cite[Proposition 5.7]{CassvdHScholbach:Geometric}. 
	If \(P\) is a Borel, this follows from \thref{Intersections of Schubert cells and semi-infinite orbits}. 
	For the general case, it suffices by \thref{composition of CTs} to show \(\CT_B\) reflects Tateness for bounded objects.
	This can be proven by an induction on the strata.
	Indeed, as before we see that for every Schubert variety there is a semi-infinite intersecting it in a single point.
	Then we use that for equivariant motives on a scheme which is acted on transitively, Tateness can be checked after pullback to a point by \cite[Proposition 3.1.23]{RicharzScholbach:Intersection}.
\end{proof}

In order to consider properties related to t-structures, we will sometimes consider the constant terms as a functor \(\CT_P\colon \DATM_{L^+\Gg}(\Gr_\Gg)\to \DATM_{L^+\Mm}(\Gr_\Mm)\).

\begin{prop}\thlabel{CT texact}
	The constant term functor \(\CT_P\) is t-exact.
	In particular, if \(\Ff\in \DATM_{L^+\Gg}(\Gr_\Gg)\) is bounded, then \(\Ff\in \MATM_{L^+\Gg}(\Gr_\Gg)\) if and only if \(\CT_P(\Ff)\in \MATM_{L^+\Mm}(\Gr_\Mm)\).
\end{prop}
\begin{proof}
Pullback along \(\Spec k'\to \Spec k\) is conservative by descent, and t-exact by definition (when restricted to stratified Artin--Tate motives).
Thus, as in the proof of \thref{CT preserves Artin-Tate} above, we may assume \(k'=k\), so that \(G\) is residually split.
Then the proposition can be proven similar to \cite[Proposition 5.10]{CassvdHScholbach:Geometric}, using \thref{Intersections of Schubert cells and semi-infinite orbits} and \thref{dimension estimate}.
Alternatively, we can use the conservativity and t-exactness of the étale realization functors, \thref{conservativity of etale realization} and \thref{texactness of realization}, to deduce the proposition from \cite[Proposition 4.6]{ALRR:Modular} (whose proof also works in mixed characteristic).
\end{proof}

In order to relate the constant term functors to the convolution product, we introduce a twisted version of the constant term functor, in the sense that it is a functor between categories of motives on twisted products of affine Grassmannians.
Consider the \(\IG_{m,k}\)-action, induced by \(\lambda\), on \(\Gr_\Gg\widetilde{\times} \Gr_\Gg\) corresponding to the diagonal action under the isomorphism \eqref{iso of convolution Gr}.
As in \cite[§6]{Yu:Integral}, the fixed points under this action can be identified with \(\Gr_\Mm\widetilde{\times} \Gr_\Mm\), the attractor with \(\Gr_{\Pp}\widetilde{\times} \Gr_\Pp\) (which is isomorphic to \(\coprod_{\nu_1,\nu_2\in \pi_0(\Gr_{\Mm})} \Ss_{P,\nu_1}^+ \widetilde{\times} \Ss_{P,\nu_2}^+\)), and similarly for the repeller.
Moreover, the natural maps make the following diagram commute, where the vertical maps are the restrictions of the multiplication map:
\begin{equation}\label{eq--twisted vs multiplication}
\begin{tikzcd}
	\Gr_\Mm \widetilde{\times} \Gr_\Mm \arrow[d] & \Gr_{\Pp^\pm} \widetilde{\times} \Gr_{\Pp^\pm} \arrow[l, "\widetilde{q}_P^\pm"'] \arrow[r, "\widetilde{p}_P^\pm"] \arrow[d] & \Gr_\Gg \widetilde{\times} \Gr_\Gg \arrow[d] \\
	\Gr_\Mm & \Gr_{\Pp^\pm} \arrow[l, "q_P^\pm"] \arrow[r, "p_P^\pm"'] & \Gr_\Gg.
\end{tikzcd}
\end{equation}

\begin{dfn}
	The \emph{twisted constant term functor} \[\widetilde{\CT}_P\colon \DM_{L^+\Gg}(\Gr_\Gg \widetilde{\times} \Gr_\Gg) \to \DM_{L^+\Mm}(\Gr_\Mm \widetilde{\times} \Gr_\Mm)\] is defined similarly to \thref{dfn--CT}, i.e., by taking the quotient of \((\widetilde{q}_P^-)_*(\widetilde{p}_P^-)^![\widetilde{\deg}_P]\) by \(L^+\Mm\). 
	Equivalently, it can also be described by the quotient of \((\widetilde{q}_P^+)_! (\widetilde{p}_P^+)^*[\widetilde{\deg}_P]\) by \(L^+\Mm\).
	Here, \(\widetilde{\deg}_P\) is defined as \[\widetilde{\deg}_P\colon \Gr_{\Mm} \widetilde{\times} \Gr_{\Mm} \cong \Gr_{\Mm} \times \Gr_{\Mm}\xrightarrow{\deg_P \times \deg_P} \IZ\times \IZ \xrightarrow{+} \IZ.\]
\end{dfn}

\begin{rmk}\thlabel{twisted constant term and product}
	In particular, \(\widetilde{\CT}_P\) can be identified with the constant term functor \(\CT_{P\times P}\) for the group \(G\times G\) and \(\Oo\)-model \(\Gg\times \Gg\).
\end{rmk}

The lemma below will later be used to show constant terms are monoidal.
In order to state it, we recall the twisted exterior product.

\begin{nota}\thlabel{notation twisted exterior product}
	For \(\Ff_1,\Ff_2\in \MATM_{L^+\Gg}(\Gr_{\Gg})\), we denote
	\[\Ff_1\twext \Ff_2 := \overline{q}^! \pH^0(\Ff_1\boxtimes \Ff_2)\in \DM(L^+\Gg \backslash LG \overset{L^+\Gg}{\times}LG/ L^+\Gg).\]
\end{nota}

Here, \(\overline{q}\) is defined as in \eqref{convolution diagram}.
Using the same notation, we then have \(\Ff_1\conv \Ff_2 = \overline{m}_!(\Ff_1\twext \Ff_2)\), and \(\twext\) agrees with the classical twisted exterior product when taking étale realizations, up to truncating and forgetting the equivariance.
Note that \(\Gr_{\Gg}\times \Gr_{\Gg}\cong \Gr_{\Gg\times \Gg}\) is Artin--Whitney--Tate stratified by its Schubert cells by \thref{Schubert stratification is WT}, so that the truncation indeed makes sense.

\begin{lem}\thlabel{CT and twisted boxtimes}
	There is a natural equivalence of functors \(\CT_P(-) \twext \CT_P(-) \cong \widetilde{\CT}_P(-\twext-)\).
\end{lem}
\begin{proof}
	Consider the commutative diagram
	\[\begin{tikzcd}[column sep=tiny, font=\small]
		L^+\Mm \backslash \Gr_\Mm \times L^+\Mm \backslash \Gr_\Mm & L^+\Mm \backslash \Gr_{\Pp^-} \times L^+\Mm \backslash \Gr_{\Pp^-} \arrow[l] \arrow[r] &  L^+\Mm \backslash \Gr_{\Pp^-} \times L^+{\Pp^-}\backslash \Gr_{\Pp^-} \arrow[r] & L^+\Gg \backslash \Gr_\Gg \times L^+\Gg \backslash \Gr_\Gg \\
		L^+\Mm \backslash \Gr_\Mm \widetilde{\times} \Gr_\Mm \arrow[u, "q_\Mm"] & L^+\Mm \backslash \Gr_{\Pp^-} \widetilde{\times} \Gr_{\Pp^-} \arrow[rr] \arrow[l] \arrow[ur, "q_\Pp"] \arrow[u] && L^+\Gg \backslash \Gr_\Gg \widetilde{\times} \Gr_\Gg \arrow[u, "q_\Gg"],
	\end{tikzcd}\]
	where the left square is cartesian. 
	Then the proposition follows from the commutation of \(\CT_P\) with exterior products, base change, and t-exactness of \(\CT_P\times \CT_P\) (\thref{CT texact}).
\end{proof}

\subsection{The fiber functor}
Next, we construct a fiber functor, and relate it to the constant term functors from the previous subsection.
In the next subsection, we will show that all these functors are monoidal.
This will allow us to use a Tannakian approach to relate \(\MATM_{L^+\Gg}(\Gr_\Gg)\) to the category of comodules under some bialgebra in \(\MATM(\Spec k)\).

By \thref{Cellularity of Schubert cells}, pushforward \(\pi_{\Gg,!}\) along the structure map preserves Artin-Tate motives.
Let \(u\colon \Gr_\Gg\to L^+\Gg\backslash \Gr_\Gg\) be the quotient map, so that \(u^!\) is given by forgetting the equivariance.
\begin{dfn}
	The \emph{fiber functor} is \(\FibFunctor:=\bigoplus_{n\in \IZ} \pH^n \pi_{\Gg,!}u^!\colon \MATM_{L^+\Gg}(\Gr_\Gg)\to \MATM(\Spec k)\).
\end{dfn}

To justify the name of this functor, we want to show it is exact, conservative, and monoidal.
This will be done by relating it to the constant term functor \(\CT_B\).
Recall that the maximal torus \(T\subseteq G\) has a unique parahoric integral model \(\Tt\subseteq \Gg\).

\begin{prop}\thlabel{Fiber functor is constant terms}
	There is a natural isomorphism
	\[\FibFunctor\cong \pi_{\Tt,!} u^! \CT_B\]
	of functors \(\MATM_{L^+\Gg}(\Gr_\Gg)\to \MATM(\Spec k)\).
\end{prop}
\begin{proof}
	This can be proven verbatim as \cite[Proposition 5.12]{CassvdHScholbach:Geometric}.
\end{proof}

\begin{cor}\thlabel{F is fiber functor}
	The fiber functor \(\FibFunctor\) is exact, conservative, and faithful.
\end{cor}
\begin{proof}
	Exactness follows from \thref{Fiber functor is constant terms}, \thref{CT texact}, and the fact that \(\Gr_{\Tt}\) is a coproduct of points.
	Conservativity for bounded objects was covered in \thref{CT conservative}, and it implies faithfulness for bounded objects as we are working with abelian categories.
	Taking the colimit, this implies conservativity and faithfulness for all of \(\MATM_{L^+\Gg}(\Gr_{\Gg})\).
\end{proof}

The following generalization of \thref{Fiber functor is constant terms} is a weaker version of \cite[Corollary 4.10]{ALRR:Modular}.

\begin{cor}\thlabel{generalization of fiber functor vs CT}
	Assume \(k=k'\), and let \(M\in \DTM(\Spec k)\) be such that each \(\pH^n(M)\) is a direct sum of Tate twists, and moreover trivial for odd \(n\).
	Then for each \(\Ff\in \MTM_{L^+\Gg}(\Gr_{\Gg})\), there is a canonical isomorphism
	\[\bigoplus_{n\in \IZ} \pH^n \pi_{\Gg,!} u^!(\Ff\otimes M) \cong \bigoplus_{n\in \IZ} \pH^n \pi_{\Tt,!} u^! \CT_B(\Ff\otimes M)\]
\end{cor}
\begin{proof}
	We may assume \(\Ff\) is supported on a single connected component of \(\Gr_{\Gg}\), so that \(\pH^n \pi_{\Gg,!}u^!\Ff\) can only be nontrivial for \(n\) of a single parity, say of even parity.
	Then, by our assumption on \(M\), we have
	\[\pH^n \pi_{\Gg,!} u^!(\Ff\otimes M) \cong \bigoplus_{i+j=n} \pH^i (\pi_{\Gg,!}u^!\Ff) \otimes \pH^j(M),\]
	which vanishes if either \(i\) or \(j\) is odd.
	Similarly, we get
	\[\pH^n \pi_{\Tt,!} u^! \CT_B(\Ff\otimes M) \cong \bigoplus_{i+j=n} \pH^i (\pi_{\Tt,!} u^! \CT_B\Ff) \otimes \pH^j(M),\]
	so we conclude by \thref{Fiber functor is constant terms}.
\end{proof}

The following result will be used when constructing integral Satake isomorphisms.
Recall that when working with anti-effective motives, we only consider motives with rational coefficients.

\begin{prop}\thlabel{anti-effective motives and the fiber functor}
	The fiber functor \(\FibFunctor\colon \MATM_{L^+\Gg}(\Gr_{\Gg}) \to \MATM(\Spec k)\) preserves and reflects anti-effective motives.
\end{prop}
\begin{proof}
	Since pullback along \(\Spec k'\to \Spec k\) preserves and reflects anti-effectivity, we may assume \(k=k'\) and consider Tate motives only.
	Then, the proposition can be proven as in \cite[Proposition 6.31]{CassvdHScholbach:Geometric}: using \thref{Intersections of Schubert cells and semi-infinite orbits} one sees that constant terms, and hence the fiber functor, can only lower the Tate twists, and hence preserve anti-effectivity.
	But since there is always a semi-infinite orbit intersecting a given Schubert variety in a point, these functors also detect anti-effective motives.
\end{proof}

\subsection{Intersection and (co)standard motives}

In order to get some control on \(\MATM_{L^+\Gg}(\Gr_{\Gg})\), we introduce specific examples of objects in this category.
These are the motivic analogues of the objects defined in \cite[End of §2]{MirkovicVilonen:Geometric}.
These motives will moreover be used in the proof of \thref{thm.intro-Satake}.
For the rest of this section, we assume \(G\) is residually split.

\begin{dfn}
	For \(\mu\in X_*(T)_I^+\), let \(\iota_\mu\colon L^+\Gg\backslash\Gr_{\Gg,\mu}\to L^+\Gg\backslash\Gr_{\Gg}\) be the inclusion, and consider the structure map \(p_\mu\colon L^+\Gg\backslash \Gr_{\Gg,\mu}\to \Spec k\). 
	Moreover, we denote by \(p_\mu^*:=p_\mu^!(-\langle2\rho,\mu\rangle)[-\langle 4\rho,\mu\rangle]\) the functor which agrees with *-pullback after forgetting the equivariance.
	\begin{enumerate}
		\item The \emph{standard functor} associated to \(\mu\) is
		\[\Jj_!(\mu,-) := \pH^0\iota_{\mu,!} p_\mu^*(-)[\langle 2\rho,\mu\rangle] \colon \MTM(\Spec k)\to \MTM_{L^+\Gg}(\Gr_{\Gg}).\]
		\item The \emph{costandard functor} associted to \(\mu\) is 
		\[\Jj_*(\mu,-) := \pH^0\iota_{\mu,*} p_\mu^*(-)[\langle 2\rho,\mu\rangle] \colon \MTM(\Spec k)\to \MTM_{L^+\Gg}(\Gr_{\Gg}).\]
		\item The \emph{IC-functor} associated to \(\mu\) is
		\[\IC_\mu(-) := \im(\Jj_!(\mu,-)\to \Jj_*(\mu,-)) \colon \MTM(\Spec k)\to \MTM_{L^+\Gg}(\Gr_{\Gg}).\]
	\end{enumerate}
\end{dfn}

\begin{lem}\thlabel{fiber functor of standard}
	For \(\mu\in X_*(T)_I^+\) and \(\nu\in X_*(T)_I\), there is an equivalence
	\[\CT_{B,\nu}(\Jj_!(\mu,\unit)) \cong \bigoplus_{\Irr(\Gr_{\Gg,\leq \mu}\cap \Ss_\nu^+)} \unit(-\langle \rho,\mu+\nu\rangle).\]
\end{lem}
\begin{proof}
	Since \(\CT_{B,\nu}\) is t-exact, we have
	\[\CT_{B,\nu}(\Jj_!(\mu,\unit)) \cong \pH^{\langle 2\rho,\nu\rangle} (q_\nu^+)_! (p_\nu^+)^* \iota_{\mu,!} p_\mu^*(\unit)[\langle 2\rho,\mu\rangle] \cong \pH^{\langle 2\rho,\mu+\nu\rangle} f_!f^*(\unit),\]
	where \(f\colon \Gr_{\Gg,\mu}\cap \Ss_\nu^+\to \Spec k\) is the structure map.
	Since \(\Gr_{\Gg\mu}\cap \Ss_\nu^+\) has dimension \(\langle \rho,\mu+\nu\) by \thref{Nonemptyness of intersections} and it admits a filtrable decomposition by perfect cells by \thref{Intersections of Schubert cells and semi-infinite orbits}, we conclude by \cite[Lemma 2.20]{CassvdHScholbach:Geometric}.
\end{proof}

By general properties of t-structures, along with \cite[Proposition 3.2.22]{RicharzScholbach:Intersection}, the essential images of the IC-functors generate \(\MTM(\Gr_{\Gg})\) under colimits and extensions.
Moreover, for rational coefficients we can say something even stronger.
Note that \(\MTM_{L^+\Gg}(\Spec k,\IQ) \cong \MTM(\Spec k,\IQ) \cong \grQVect\) is itself semisimple.

\begin{prop}\thlabel{semisimplicity}
	The abelian category \(\MTM_{L^+\Gg}(\Gr_{\Gg},\IQ)\) is semisimple, with simple generators given by \(\IC_\mu(\Ff)\) for \(\mu\in X_*(T)_I^+\) and simple \(\Ff\in \MTM(\Spec k,\IQ)\).
\end{prop}
\begin{proof}
	By \thref{forgetting equivariance is ff} below, it suffices to show \(\MTM(\Gr_{\Gg},\IQ)\) is semisimple; the enumeration of the simple generators then follow from generalities about t-structures and \cite[Proposition 3.2.22]{RicharzScholbach:Intersection}.
	As in \cite[Proposition 5.3]{RicharzScholbach:Witt} this semisimplicity follows from the parity vanishing condition for \(\ell\)-adic sheaves on \(\Gr_{\Gg}\), which in turn can be deduced exactly as in \cite[Lemma 1.1]{Zhu:Ramified} or \cite[Lemma 2.1]{Zhu:Affine}.
\end{proof}

\begin{lem}\thlabel{forgetting equivariance is ff}
	The functor \(u^! \colon \MTM(\Gr_{\Gg}) \to \MTM_{L^+\Gg}(\Gr_{\Gg})\), given by forgetting the equivariance, is fully faithful, and its image is closed under subquotients.
\end{lem}
\begin{proof}
	The proof of \cite[Proposition 4.30]{CassvdHScholbach:Geometric} carries over verbatim to the current situation.
\end{proof}

In fact, the functor \(u^!\) above is even an equivalence. Although one can likely give a direct motivic proof of this fact by slightly modifying \cite[Appendix A]{ALRR:Modular}, we give a shorter proof, by reducing to the result of loc.~cit.

\begin{prop}\thlabel{forgetting equivariance is equivalence}
	The functor \(u^! \colon \MTM(\Gr_{\Gg}) \to \MTM_{L^+\Gg}(\Gr_{\Gg})\) is an equivalence.
\end{prop}
\begin{proof}
	We may replace \(\Gr_{\Gg}\) by an \(L^+\Gg\)-stable closed subscheme \(X\), and \(L^+\Gg\) by a perfectly smooth quotient \(L^n\Gg\) through which the action on \(X\) factors.
	The forgetful functor \(u^!\colon \DM(L^n\Gg\backslash X) \to \DM(X)\) admits a left adjoint \(\coav\) by \cite[Lemma 2.24]{CassvdHScholbach:Geometric}.
	The construction of loc.~cit.~is purely in terms of the six functor formalism, so that \(\coav\) is compatible with the left adjoint of the forgetful functor on étale sheaves, under the étale realization functor.
	By \cite[Lemma 2.24 (2)]{CassvdHScholbach:Geometric} and since positive loop groups are pcs, motivic coaveraging maps \(\DTM(X)\) to \(\DTM_{L^n\Gg}(X)\).
	Hence, the forgetful functor \(\MTM_{L^n\Gg}(X) \to \MTM(X)\) admits a left adjoint \(\pH^0\circ \coav\), which is again compatible with the similar functor on étale perverse sheaves.
	It thus suffices to show that the unit and counit of this adjunction are equivalences.
	Since the étale realization functors are t-exact and jointly conservative when restricted to \(\DTM\) (\thref{texactness of realization} and \thref{conservativity of etale realization}), this can be checked on étale sheaves, where it follows from \cite[Proposition A.3]{ALRR:Modular}.
\end{proof}

\section{Tannakian reconstruction}\label{sect-tannakian}

As usual, we let \(\Gg/\Oo\) be a very special parahoric model of the reductive group \(G/F\). 
Moreover, \(S\subseteq G\) is a maximal \(\breve{F}\)-split torus, which splits over an unramified extension \(F'/F\) with residue field \(k'/k\).
In the following section, we apply a generalized Tannakian formalism to show \(\MATM_{L^+\Gg}(\Gr_{\Gg})\) is equivalent to the category of comodules under some bialgebra in \(\MATM(\Spec k)\). 
Since this Tannakian approach is well known, we will give references when proofs already exist in the literature.

\subsection{An adjoint to the fiber functor}

Fix some finite subset \(W\subseteq X_*(T)_I^+\) closed under the Bruhat order, and let \(\Gr_{\Gg,W}:= \bigcup_{\mu\in W} \Gr_{\Gg,\mu}\).
We denote by \(i_W\colon \Gr_{\Gg,W}\into \Gr_\Gg\) the inclusion, and by \(\FibFunctor_W:=\FibFunctor\circ (i_W)_*\colon \MATM_{L^+\Gg}(\Gr_{\Gg,W})\to \MATM(\Spec k)\) the restriction of the fiber functor.
Similarly, we denote by \(p_W^\pm, q_W^\pm\) the restrictions of the maps from \eqref{hyperbolic localization diagram} in case \(P=B\) is the Borel, and by \(\pi_{\Gg,W}\) the restriction of \(\pi_\Gg\).

\begin{prop}\thlabel{adjoints of fiber functor}
	The restricted fiber functor \(\FibFunctor_W\colon \MATM_{L^+\Gg}(\Gr_{\Gg,W})\to \MATM(\Spec k)\) admits a left adjoint \(L_W\). 
\end{prop}
\begin{proof}
	We proceed as in \cite[§6.1.1]{CassvdHScholbach:Geometric}.
	Recall the motivic coaveraging functor from \cite[Lemma 2.24]{CassvdHScholbach:Geometric}, and
	consider the functor 
	\[\coav (p_W^-)_! (q_W^-)^* \pi_{\Tt,W}^* [-\deg_B] \colon \MATM(\Spec k)\to \DM_{L^+\Gg}(\Gr_{\Gg,W})\]
	It suffices to show this preserves Artin-Tate motives, as perverse truncating will then give the desired left adjoint.
	As usual, we can assume \(G\) is residually split, and consider Tate motives only.
	
	As \((p_W^-)_! (q_W^-)^* \pi_{\Tt,W}^*\) maps \(\MTM(\Spec k)\) to motives which are stratified Tate for the stratification given by \(\Ss^-_\nu\cap \Gr_{\Gg,\mu}\) for \(\mu\in W\) and \(\nu\in X_*(T)_I\), we have to show coaveraging takes values in Tate motives for the stratification by Schubert cells.
	In other words, if \(i_\nu\colon \Ss^-_\nu\cap \Gr_{\Gg,\mu}\to \Gr_{\Gg,\mu}\) is the inclusion, we must show \(u^!\coav i_{\nu,!}(\unit)\) is Tate.
	As coaveraging gives equivariant motives, it suffices by \cite[Proposition 3.1.23]{RicharzScholbach:Intersection} to show \((\iota_{w_0(\mu)})^!u^!\coav i_{\nu,!}(\unit)\in \DTM(\Spec k)\), where \(\iota_{w_0(\mu)}\colon \varpi^{w_0(\mu)}\into \Gr_{\Gg,\mu}\) is the inclusion of the basepoint.
	By \cite[Lemma 2.24 (2)]{CassvdHScholbach:Geometric}, this motive is isomorphic to \((\iota_{w_0(\mu)})^! a_! p^! i_{\nu,!}(\unit)\), where \(a,p\colon L^n\Gg\times \Gr_{\Gg,\mu}\to \Gr_{\Gg,\mu}\) are the action and projection maps respectively, for \(n\gge 0\).
	It thus suffices to show \(f_*(\unit)\) is Tate, where \(f\) is the structure map of \((a')^{-1}(\varpi^{w_0(\mu)})\), and \(a'\colon L^n\Gg\times (\Ss^-_\nu\cap \Gr_{\Gg,\mu})\to \Gr_{\Gg,\mu}\) is the restriction of the action map.
	This follows from \thref{fibres of action map are cellular} applied to \(\lambda = w_0(\mu)\).
\end{proof}

\begin{cor}
	For \(W'\subseteq W\) and the corresponding inclusion \(i_{W'}^W\colon \Gr_{\Gg,W'}\into \Gr_{\Gg,W}\), the adjoints constructed above are related by \((i_{W'}^W)^*L_W\cong L_{W'}\).
\end{cor}
\begin{proof}
	This follows from the identity \(\FibFunctor_{W'}= \FibFunctor_W \circ (i_{W'}^W)_*\).
\end{proof}

Unwinding the definition of \(L_W\), we see that \(L_W(\unit)\) corresponds to \(A_Z(\unit)\) in \cite[§11]{MirkovicVilonen:Geometric}.
The advantage of our formulation is that we can make use of monadicity theorems.

\begin{prop}\thlabel{bounded equivalences}
	The adjunction \((L_W,\FibFunctor_W)\) is monadic, so that there is an equivalence
	\[\MATM_{L^+\Gg}(\Gr_{\Gg,W}) \cong \Mod_{T_W}(\MATM(\Spec k)),\]
	where the right hand side denotes the category of modules under the monad \(T_W:=\FibFunctor_W\circ L_W\) of the adjunction.
\end{prop}
\begin{proof}
	Since \(\FibFunctor_W\) is exact, it preserves finite colimits. 
	As it is also conservative, the equivalence follows from the Barr-Beck monadicity theorem.
\end{proof}

Here are some more properties of the left adjoints \(L_W\).

\begin{lem}\thlabel{tensoring with adjoint}
	For any \(\Ff\in \MATM(\Spec k)\), there is a canonical isomorphism
	\[L_W(\Ff) \cong L_W(\unit) \otimes \Ff.\]
\end{lem}
\begin{proof}
	Consider the functor \(\coav (p_W^-)_! (q_W^-)^* \pi_{\Tt,W}^* [-\deg_B]\), which gives the adjoint \(L_W\) after truncating.
	By \cite[Lemma 2.24 (6)]{CassvdHScholbach:Geometric}, it commutes with exterior products.
	Hence, the desired isomorphism will follow by applying \(\pH^0\) to 
	\[\coav (p_W^-)_! (q_W^-)^* \pi_{\Tt,W}^*(\unit)[-\deg_B] \otimes \Ff \cong \coav (p_W^-)_! (q_W^-)^* \pi_{\Tt,W}^*(\Ff)[\deg_B],\] 
	as long as we can show this lies in \(\DATM_{L^+\Gg}^{\leq 0}(\Gr_{\Gg})\).
	But this holds since \(\coav (p_W^-)_! (q_W^-)^* \pi_{\Tt,W}^* [-\deg_B]\) is right t-exact, as the left adjoint of the t-exact constant term functor \(\CT_B\).
\end{proof}

\begin{lem}\thlabel{algebras are dualizable}
	Assume \(k=k'\).
	Then \(T_W(\unit)\) is the image of a free module under the faithful functor \(\grZMod\to \MTM(\Spec k)\).
	In particular, \(T_W(\unit)\) is dualizing, and tensoring with \(T_W(\unit)\) is t-exact.
\end{lem}
\begin{proof}
	Consider the functor \(\coav (p_W^-)_! (q_W^-)^* \pi_{\Tt,W}^* [-\deg_B]\), as in the proof of \thref{tensoring with adjoint}.
	Restricting to the connected components of \(\Gr_{\Tt}\), we can decompose it as \(\bigoplus_{\nu\in X_*(T)_I} L_\nu\), where only finitely many \(L_\nu\colon \MTM(\Spec k)\to \DTM_{L^+\Gg}(\Gr_{\Gg})\) are nontrivial.
	Let \(\widetilde{a}\colon L^n\Gg\times (\Ss_\nu^-\cap \Gr_{\Gg,W})\to \Gr_{\Gg,W}\) be the action map, for \(n\gge 0\). 
	For \(\lambda\in X_*(T)_I\), let \(f\colon \widetilde{a}^{-1}(\Ss_\lambda^+\cap \Gr_{\Gg,W})\to \Spec k\) be the structure map.
	Then by \cite[Lemma 2.24 (2)]{CassvdHScholbach:Geometric}, we have
	\[u^!\CT_{B,\lambda}(L_{W,\nu}(\unit)) = f_!\unit(\dim L^n\Gg)[2\dim L^n\Gg + \langle 2\rho,\lambda-\nu\rangle].\]
	By \thref{fibres of action map are cellular}, \(\widetilde{a}^{-1}(\Ss_\lambda^+\cap \Gr_{\Gg,W})\) admits a filtrable decomposition by cells, and has dimension \(\dim L^n\Gg + \langle\rho,\lambda-\nu\rangle\).
	As \(\CT_{B,\lambda}\) is t-exact, \(T_W(\unit)\) computes the top cohomology of a cellular scheme, which lies in the image of \(\grZMod\to \MTM(\Spec k)\) by \cite[Lemma 2.20]{CassvdHScholbach:Geometric}.
\end{proof}

\subsection{Monoidality of the fiber functor}

We can now finally show the constant term functor \(\CT_B\) and the fiber functor \(\FibFunctor\) are monoidal, following \cite[Remark 4.19]{ALRR:Modular}.

\begin{lem}\thlabel{Surjection from projectives}
	Assume \(k=k'\). Every bounded object \(\Ff\in \MTM_{L^+\Gg}(\Gr_{\Gg})\) admits a surjection from some \(\Ff' \in \MTM_{L^+\Gg}(\Gr_{\Gg})\), such that \(\FibFunctor(\Ff')\) is a direct sum of Tate twists. 
\end{lem}
\begin{proof}
	First, note that for any \(M\in \MTM(\Spec k)\), the natural map
	\[\bigoplus_{n\in \IZ} \unit(n)^{\Hom_{\DM(\Spec k)}(\unit(n),M)}\to M\]
	is an epimorphism.
	Indeed, it suffices to see the fiber of this map (in \(\DTM(\Spec k)\)) is concentrated in degree 0, which follows from \thref{BS-vanishing} and the fact that the Tate twists generate \(\MTM(\Spec k)\) (\thref{defi-t-structure}).
	
	Now, let \(W\subseteq X_*(T)_I^+\) be a finite subset such that \(\Gr_{\Gg,W}\) is a closed subscheme containing the support of \(\Ff\), and let \(M\to \FibFunctor_W(\Ff)\) be an epimorphism, where \(M\) is a direct sum of Tate twists.
	Then the adjoint map \(L_W(M) \to \Ff\) is an epimorphism (as this can be checked after applying \(\FibFunctor_W\)), and \(L_W(M)\) satisfies the desired condition by \thref{algebras are dualizable}.
\end{proof}

\begin{prop}\thlabel{reduction to ALRR}
	Let \(\Ff_1,\Ff_2\in \MATM_{L^+\Gg}(\Gr_{\Gg})\).
	Then there exists a canonical equivalence
	\[\CT_B m_{\Gg,!}(\Ff_1\twext\Ff_2) \cong m_{\Tt,!}\widetilde{\CT}_B(\Ff_1 \twext \Ff_2).\]
\end{prop}
\begin{proof}
	We start by some reduction steps.
	Since the desired equivalence will be canonical, we can assume by descent that \(G\) is residually split and \(k=k'\).
	Since every object in \(\MTM_{L^+\Gg}(\Gr_{\Gg})\) is a colimit of bounded objects, we may moreover assume that \(\Ff_1,\Ff_2\) are bounded.
	Finally, by choosing a presentation \(M'\to M\to \Ff_2\to 0\) with \(M,M'\) as in \thref{Surjection from projectives}, and since the functors \(\CT_B m_{\Gg,!}(\Ff_1\twext-)\) and \(m_{\Tt,!}\widetilde{\CT}_B(\Ff_1 \twext -)\) are right exact, we may assume \(\FibFunctor_W(\Ff_2)\) is a direct sum of Tate twists.
	
	In this situation, we can follow the strategy from \cite{ALRR:Modular}.
	Namely, denote by \(\pr_{1,\Gg}\colon \Gr_{\Gg}\widetilde{\times} \Gr_{\Gg}\to \Gr_{\Gg}\) the projection onto the first factor, and by \(\overline{\pr}_{1,\Gg}\colon L^+\Gg\backslash \Gr_{\Gg}\widetilde{\times} \Gr_{\Gg}\to L^+\Gg \backslash \Gr_{\Gg}\) the quotient map.
	Then we can use the same proof as \cite[Corollary 4.15]{ALRR:Modular} to get a canonical equivalence
	\[\bigoplus_{n\in \IZ} \pH^n \pi_{\Tt,!} u^! \CT_B \overline{\pr}_{1,\Gg,!}(\Ff_1\twext \Ff_2) \cong \bigoplus_{n\in \IZ} \pH^n \pi_{\Tt,!} u^! \overline{\pr}_{1,\Tt,!} \widetilde{\CT}_B(\Ff_1\twext \Ff_2).\]
	Combining this with \thref{generalization of fiber functor vs CT} (instead of \cite[Corollary 4.10]{ALRR:Modular}), the proofs of \cite[Corollary 4.16~and~Proposition 4.17]{ALRR:Modular} carry over to give the desired equivalence.
\end{proof}

\begin{cor}\thlabel{CT monoidal}
	The constant term functor \(\CT_B\colon \MATM_{L^+\Gg}(\Gr_{\Gg})\to \MATM_{L^+\Tt}(\Gr_{\Tt})\) admits a natural monoidal structure.
\end{cor}
\begin{proof}
	The monoidality isomorphisms are constructed by combining \thref{CT and twisted boxtimes} and \thref{reduction to ALRR}.
	By considering triple twisted products, it is then a standard check that this really defines a monoidal structure.
\end{proof}

\begin{cor}\thlabel{fiber functor monoidal}
	The fiber functor \(\FibFunctor\) admits a natural monoidal structure.
\end{cor}
\begin{proof}
	Since \(\FibFunctor \cong \pi_{\Tt,!} u^! \CT_B\) and we know \(\CT_B\) is monoidal, it suffices to construct a monoidal structure on \(\pi_{\Tt,!}\).
	But this can easily be done by identifying the convolution map \(\Gr_\Tt \widetilde{\times} \Gr_\Tt\to \Gr_\Tt\) with the addition map \(X_*(T)_I \times X_*(T)_I \to X_*(T)_I\).
\end{proof}

\begin{cor}\thlabel{cor bounded equivalences}
	The monad \(T_W\) from \thref{bounded equivalences} is given by tensoring with \(T_W(\unit)\).
\end{cor}
\begin{proof}
	Since the fiber functor \(\FibFunctor\) (which restricts to \(\FibFunctor_W\)) is monoidal by \thref{fiber functor monoidal}, the corollary follows from Lemmas \ref{tensoring with adjoint} and \ref{algebras are dualizable}.
\end{proof}

\subsection{The Hopf algebra}

We can now prove the main theorem of this section.

\begin{thm}\thlabel{Hopf algebra thm}
	The fiber functor \(\FibFunctor\colon \MATM_{L^+\Gg}(\Gr_{\Gg})\to \MATM(\Spec k)\) is comonadic, and the comonad is given by tensoring with the coalgebra
	\[H_{\Gg}:= \varinjlim_W T_W(\unit)^\vee.\]
	Moreover, \(H_{\Gg}\) is a commutative Hopf algebra, and we obtain a monoidal equivalence
	\[(\MATM_{L^+\Gg}(\Gr_{\Gg}),\conv) \cong (\coMod_{H_\Gg}(\MATM(\Spec k)),\otimes).\]
\end{thm}
Note that this equivalence equips \(\MATM_{L^+\Gg}(\Gr_{\Gg})\) with a symmetric monoidal structure for the convolution.
But since we did not directly construct a commutativity constraint for the convolution, it does not make sense to ask for the above equivalence to be symmetric monoidal.
\begin{proof}
	By \thref{bounded equivalences} and \thref{cor bounded equivalences}, we have equivalences \[\MATM_{L^+\Gg}(\Gr_{\Gg,W}) \cong \Mod_{T_W(\unit)}(\MATM(\Spec k))\] for any finite \(W\).
	Since each \(T_W(\unit)\) is dualizable by \thref{algebras are dualizable}, its dual \(H_{\Gg,W}:= \underline{\operatorname{Hom}}(T_W(\unit),\unit)\) is a coalgebra, and we have
	\[\MATM_{L^+\Gg}(\Gr_{\Gg,W}) \cong \Mod_{T_W(\unit)}(\MATM(\Spec k)) \cong \coMod_{H_{\Gg,W}} (\MATM(\Spec k)).\]
	Thus, \(\FibFunctor_W\) has a right adjoint, given by tensoring with the dualizable \(H_{\Gg,W}\).
	Passing to colimits then gives an equivalence
	\begin{equation}\label{full equivalence}
		\MATM_{L^+\Gg}(\Gr_{\Gg}) \cong \coMod_{H_\Gg}(\MATM(\Spec k))
	\end{equation}
	where \(H_{\Gg} := \varinjlim_W H_{\Gg,W}\).
	We are left to show \(H_{\Gg}\) is a commutative Hopf algebra, and that this equivalence is monoidal.
	
	Since \(\FibFunctor\) is monoidal, its right adjoint is lax monoidal.
	This gives a morphism \(H_{\Gg}\otimes H_{\Gg} \to H_{\Gg}\), and by a classical argument this makes \(H_{\Gg}\) into a bialgebra object.
	
	The commutativity of \(H_{\Gg}\) is more subtle, since we do not yet know \(\MATM_{L^+G}(\Gr_{\Gg})\) is symmetric monoidal.
	Instead, we can check this after applying the étale realizations \(\rho_{\ell}\) to \(\IZ_\ell\)-étale sheaves, which are jointly faithful on mixed Artin-Tate motives.
	Since \(H_{\Gg}\) is dualizable, and hence flat, we can even check this after rationalizing to \(\IQ_\ell\)-étale sheaves, and we postpone this to \thref{bialgebra is commutative}.
	
	To give \(H_{\Gg}\) the structure of a Hopf algebra, we can construct an antipode as in \cite[Proposition VI.10.2]{FarguesScholze:Geometrization}.
	Finally, since \(\FibFunctor\) is monoidal, the equivalence \eqref{full equivalence} is monoidal as well.
\end{proof}

In order to relate \(H_\Gg\) to the inertia-invariants of the Langlands dual group of \(G\), we will make use of the following two properties of \(H_\Gg\).
First, let \(\Lambda\) be either the field \(\IQ\), or the finite field \(\IF_\ell\), for a prime \(\ell\neq p\).

\begin{prop}\thlabel{change of scalars}
	There is an equivalence of symmetric monoidal categories
	\[(\MATM_{L^+\Gg}(\Gr_{\Gg};\Lambda), \star) \cong (\coMod_{H_\Gg\otimes \Lambda}(\MATM(\Spec k; \Lambda)), \otimes).\]
	Moreover, the functor \(\pH^0(-\otimes \Lambda)\colon \MATM_{L^+\Gg}(\Gr_{\Gg})\to \MATM_{L^+\Gg}(\Gr_\Gg;\Lambda)\) is symmetric monoidal.
\end{prop}
For rational coefficients, the t-structure on \(\DATM\) can be defined similarly as for integral coefficients.
On the other hand, for torsion coefficients, we use the t-structure induced by the equivalence between étale motives and étale cohomology \cite[Corollary 5.5.4]{CisinskiDeglise:Etale}.
In both cases, we have \(\MATM(-,\Lambda)\subseteq \MATM(-)\).
Note that since we are working with field coefficients, there is no need to truncate the convolution or tensor product.
\begin{proof}
	For \(\Lambda= \IQ\), this follows easily from the fact that \(\Ff\otimes \IQ = \varinjlim_{n>0} (\ldots \to \Ff \xrightarrow{n} \Ff \to \ldots)\) and the fact that all our functors are additive and commute with filtered colimits.
	
	On the other hand, if \(\Lambda = \IF_\ell\), we denote \(M/\ell:=\coker(M\xrightarrow{\ell} M)\in \MATM_{L^+\Gg}(\Gr_{\Gg},\Lambda)\) for any \(M\in \MATM_{L^+\Gg}(\Gr_{\Gg})\).
	Then we have
	\[\MATM_{L^+\Gg}(\Gr_{\Gg},\Lambda) = \{M\in \MATM_{L^+\Gg}(\Gr_{\Gg})\mid M=M/\ell\}.\]
	Since the convolution product \(\conv\) is right t-exact, we get \((M_1{\conv} M_2)/\ell \cong (M_1/\ell){\conv} (M_2/\ell)\).
	This shows that \(\MATM_{L^+\Gg}(\Gr_{\Gg},\Lambda)\) is controlled by the Hopf algebra \(H_{\Gg}/\ell = \pH^0(H_{\Gg}\otimes \Lambda) = H_{\Gg}\otimes \Lambda\), where the second isomorphism follows from \thref{algebras are dualizable}.
\end{proof}

\begin{rmk}\thlabel{Hopf algebra is reduced}
	Recall that by étale descent we have \(\MATM(\Spec k) \cong \MTM_{\Gamma}(\Spec k')\), where \(\Gamma = \Gal(k'/k)\).
	By \thref{graded to motives}, we have a faithful functor \(\grZMod\to \MTM(\Spec k')\), which is fully faithful when restricted to ind-free modules.
	Hence, by \thref{algebras are dualizable}, we can consider \(H_{\Gg}\) as a graded Hopf algebra over \(\IZ[\frac{1}{p}]\) with a \(\Gamma\)-action, or equivalently, a Hopf algebra with \(\IG_m\times \Gamma\)-action.
	Consequently, we can also view \(H_\Gg\otimes \Lambda\) as a Hopf algebra over \(\Lambda\) with \(\IG_m\times \Gamma\)-action, even for torsion coefficients.
\end{rmk}

\section{Some \(p\)-adic geometry and nearby cycles}\label{Sect--padic}

Let us digress from the world of motives and (perfect) schemes, and move on to the realm of \(p\)-adic geometry and diamonds.
This will allow us to consider a nearby cycles functor, constructed in \cite[§6]{AGLR:Local}.
Although it is likely possible to construct a ramified Satake equivalence in mixed characteristic without using diamonds (at least with rational coefficients, cf.~also \cite[Remark 2.37]{Zhu:Affine}), the approach using nearby cycles simplifies certain aspects and is more conceptual.
At the time of writing this article, a motivic version of the Satake equivalence from \cite[Chapter VI]{FarguesScholze:Geometrization} was not yet available. 
For this reason, we only use \(\ell\)-adic étale sheaves in this section, and reduce  everything to this case in later sections.
However, we note that during the refereeing process, such a motivic Satake equivalence has appeared in \cite[§5]{Scholze:Geometrization}.

Since we only use the nearby cycles functor from \cite{AGLR:Local} and the Satake equivalence from \cite{FarguesScholze:Geometrization} as an input, we will not need to go deep into the theory of perfectoid spaces and diamonds.
As such, we will not recall them, and instead refer to \cite{Scholze:Perfectoid, ScholzeWeinstein:Berkeley, Scholze:Etale, FarguesScholze:Geometrization, AGLR:Local} for background and details. 
We note that while \cite{ScholzeWeinstein:Berkeley,AGLR:Local} work in mixed characteristic, everything we will need also works in equal characteristic.
Alternatively, one could also use more classical nearby cycles functors in equal characteristic, as in \cite{Zhu:Ramified,Richarz:Affine,ALRR:Modular}.

As usual, we let \(\Gg/\Oo\) be a parahoric model of \(G\); we will specialize to the case where \(\Gg\) is very special later in this section.
For this section only, we assume \(k=\overline{k}\) is already algebraically closed, so that \(F=\breve{F}\).
Moreover, we let \(C/F\) be a completed algebraic closure, with ring of integers \(\Oo_C\) and residue field \(\overline{k}\).
Throughout this section, we denote by \(\Perfd\) the category of perfectoid spaces in characteristic \(p\).

\subsection{The \(\BdR^+\)-affine Grassmannian}

Rather than power series or Witt vectors, the affine Grassmannian from \cite{ScholzeWeinstein:Berkeley} is defined via the \emph{de Rham period rings}.

\begin{dfn}
	Let \((R,R^+)\) be a perfectoid Tate--Huber pair, and consider the canonical surjection \(\theta\colon W_{\Oo}(R^{\flat+}) \to R^+\), whose kernel is generated by a non-zero-divisor \(\xi\).
	Consider a pseudouniformizer \(\varpi^\flat\in R^\flat\) such that \(\varpi^p\mid p\) for \(\varpi = (\varpi^\flat)^\sharp\).
	Then the ring \(\BdR^+(R)\) of \emph{de Rham periods} is defined as the \(\xi\)-adic completion of \(W_{\Oo}(R^{\flat+})[[\varpi^\flat]^{-1}]\), and we set \(\BdR(R):=\BdR^+(R)[\xi^{-1}]\).
\end{dfn}

\begin{dfn}(\cite[§20.2]{ScholzeWeinstein:Berkeley})
	The \emph{\(\BdR^+\)-loop group} is the group functor over \(\Perfd_{/\Spd F}\), which on affinoid perfectoids is given by
	\[L_{\dR}G\colon \Spa(R,R^+)\mapsto G(\BdR(R^\sharp)),\]
	where \(R^\sharp\) is the untilt of \(R\) corresponding to the structure morphism \(\Spa(R,R^+)\to \Spd F\).
	Similarly, the \emph{positive \(\BdR^+\)-loop group} is given by
	\[L_{\dR}^+G\colon \Spa(R,R^+)\mapsto G(\BdR^+(R^\sharp)).\]
	Both these functors are v-sheaves, and hence extend to functors on all perfectoid spaces.
	
	Next, the \emph{\(\BdR^+\)-affine Grassmannian} \(\Gr^{\dR}_G\) is the étale sheafification of \(L_{\dR}G/L_{\dR}^+G\).
	By \cite[Proposition 20.2.3]{ScholzeWeinstein:Berkeley}, the map \(\Gr_G^{\dR}\to \Spd F\) is ind-proper and ind-representable in spatial diamonds.
	
	Finally, we define the Hecke stack \(\Hck_G^{\dR}\) as the v-stack quotient \(L_{\dR}^+G\backslash \Gr_G^{\dR}\).
\end{dfn}

We will denote the base change of \(\Gr_G^{\dR}\) to \(\Spd C\) by \(\Gr_{G,C}^{\dR}\), and similarly for \(\Hck_{G,C}^{\dR}\).
Since \(F=\breve{F}\) by assumption, this yields an action of \(I\) on both these base changes.

Fix a prime \(\ell\neq p\). As in \cite[§6.1]{AGLR:Local}, based on \cite{Scholze:Etale}, we can consider the stable \(\infty\)-category \(D_{\et}(X,\IQ_\ell)\) of \(\ell\)-adic sheaves on any small v-stack \(X\).
It comes equipped with a six-functor formalism.

Next, we recall the geometric Satake equivalence for the \(\BdR^+\)-affine Grassmannian from \cite[§VI]{FarguesScholze:Geometrization}.
Consider the t-structure on \(D_{\et}(\Hck_{G,C}^{\dR},\IQ_\ell)\) given by \cite[Definition/Proposition VI.7.1]{FarguesScholze:Geometrization} and \cite[(6.6)]{AGLR:Local}.
Recall also the notion of ULA (= universal local acyclicity) from \cite[§4.2]{FarguesScholze:Geometrization} and \cite[§6.1]{AGLR:Local}.
Let \(\Sat_{G,C}^{\dR}\subseteq D_{\et}(\Hck_{G,C}^{\dR},\IQ_\ell)\) be the full subcategory consisting of bounded ULA perverse sheaves (since we are working with \(\IQ_\ell\)-coefficients, flat perversity is automatic).

The convolution product from \cite[§VI.8]{FarguesScholze:Geometrization} makes \(\Sat_{G,C}^{\dR}\) a monoidal category, and the fusion product from \cite[§VI.9]{FarguesScholze:Geometrization} gives it a symmetric monoidal structure.
A variation of \cite[Theorem VI.0.2]{FarguesScholze:Geometrization} then gives the following theorem, cf.~also \cite[Theorem 6.3]{AGLR:Local}.
In this section, we consider the Langlands dual group \(\widehat{G}\) defined over \(\IQ_\ell\).

\begin{thm}\thlabel{Fargues-Scholze Satake}
	Up to identifying the root groups of \(\widehat{G}\) with their Tate twists, there is a canonical symmetric monoidal equivalence
	\[\Sat_{G,C}^{\dR} \cong \Rep(\widehat{G})^{\fd}.\]
\end{thm}

In particular, choosing a trivialization of the Tate twist \(\IQ_\ell(1)\cong\IQ_\ell\) induces a canonical equivalence \(\Sat_{G,C}^{\dR}\cong \Rep(\widehat{G})^{\fd}\).

\subsection{Nearby cycles}

In order to define a nearby cycles functor, we need a family of affine Grassmannians over \(\Spd \Oo_C\).

\begin{dfn}(\cite[§20.3]{ScholzeWeinstein:Berkeley})
	The \emph{Beilinson-Drinfeld Grassmannian} is the étale sheafification \(\Gr_{\Gg}^{\BD}\) of the functor
	\[\Perfd_{/\Spd \Oo_C}\to \Set\colon \Spa(R,R^+)\mapsto \Gg(\BdR(R^\sharp))/\Gg(\BdR^+(R^\sharp)).\]
	By \cite[Theorem 21.2.1]{ScholzeWeinstein:Berkeley}, the morphism \(\Gr_{\Gg}^{\BD}\to \Spd \Oo_C\) is ind-proper and ind-representable in spatial diamonds.
	By \cite[Lemma 4.10]{AGLR:Local}, its generic fiber is \(\Gr_{G,C}^{\dR}\), while its special fiber is identified with \(\Fl_{\Gg}^\diamondsuit\), where \((-)^\diamondsuit\) is the diamond functor as defined as in \cite[§27]{Scholze:Etale}.
	
	Similarly, we define the Beilinson-Drinfeld Hecke stack \(\Hck_{\Gg}^{\BD}\) as the étale sheafification of
	\[\Spa(R,R^+)\mapsto \Gg(\BdR^+(R^\sharp))\backslash \Gg(\BdR(R^\sharp))/\Gg(\BdR^+(R^\sharp)),\]
	which is a small v-stack.
\end{dfn}

This gives us a \emph{nearby cycles} functor
\[\Psi_\Gg\colon D_{\et}(\Hck_{G,C}^{\dR}, \IQ_\ell)^{\bd} \to D_{\et}(L^+\Gg\backslash LG /L^+\Gg,\IQ_\ell)^{\bd}\]
as in \cite[(6.27) and Proposition A.5]{AGLR:Local}, where we restrict to sheaves with bounded support.
By \cite[Theorem 1.8]{AGLR:Local}, \(\Psi_{\Gg}\) preserves constructibility and universal local acyclicity.

Let us now assume \(\Gg\) is very special.
Then \(\Psi_\Gg\) is t-exact by \cite[Corollary 6.14]{AGLR:Local}.
It is also compatible with the monoidal structures and fiber functors.
In order to show this, we consider the constant term functors
\[\CT_P^{\et}\colon \Perv_{\et}(L^+\Gg\backslash LG/L^+\Gg,\IQ_\ell)^{\bd} \to \Perv_{\et}(L^+\Mm \backslash LM /L^+\Mm,\IQ_\ell)^{\bd}\]
for \(\ell\)-adic sheaves on the special fiber as in \cite[(6.10)]{AGLR:Local}, and
\[\CT_P^{\dR}\colon \Sat_{G,C}^{\dR} \to \Sat_{M,C}^{\dR}\]
for \(\ell\)-adic sheaves on the generic fiber as in \cite[§VI.9]{FarguesScholze:Geometrization}.
Here, \(P\subseteq G\) is any standard parabolic with Levi factor \(M\), and \(\Mm\) is the parahoric model of \(M\) corresponding to \(\Gg\).
We also note that we have already included the degree shift in the definition of the constant term functors, and that \(\CT_P^{\et}\) is the \(\ell\)-adic realization of the motivic \(\CT_P\).
Composing with the obvious fiber functors in case \(G=T\) is a torus, we obtain fiber functors, which we denote \(\FibFunctor_{\et}\) and \(\FibFunctor_{\dR}\) respectively.

\begin{prop}\thlabel{monoidality of nearby cycles}
	The nearby cycles functor \(\Psi_{\Gg}\) admits a monoidal structure, for which there exists a monoidal equivalence
	\[\FibFunctor_{\et} \circ \Psi_{\Gg} \cong \FibFunctor_{\dR}\colon \Sat_{G,C}^{\dR} \to \QlVect.\]
\end{prop}
\begin{proof}
	First, we show \(\Psi_{\Gg}\) admits a monoidal structure.
	In the notation of \cite{AGLR:Local}, we have \(\Psi_{\Gg} = i^*Rj_*\), where \(i\) and \(j\) are the inclusions of the special and generic fibers of \(\Hck_{\Gg}^{\BD}\).
	By base change, both \(i^*\) and \(j^*\) are monoidal with respect to the convolution product.
	Since on ULA objects, \(j^*\) is an equivalence with inverse \(Rj_*\) by \cite[Proposition 6.12]{AGLR:Local}, this yields a natural monoidal structure on \(\Psi_{\Gg}\).
	(Compare also with \cite[Proposition 4.7]{ALWY:Gaitsgory}.)
	
	The equivalence \(\FibFunctor_{\et} \circ \Psi_{\Gg} \cong \FibFunctor_{\dR}\) follows from commutation of nearby cycles and proper pushforward, using the interpretation of fiber functors as total cohomology.
	It remains to show this equivalence is monoidal.
	In case \(G=T\) is a torus, this can be checked by hand, so it suffices to show that the equivalence
	\[\CT_B^{\et} \circ \Psi_{\Gg} \cong \Psi_{\Tt} \circ \CT_B^{\dR}\]
	from \cite[Proposition 6.13]{AGLR:Local} is monoidal.
	The equal characteristic analogue has been checked in \cite[§7.2,~§8.6]{ALRR:Modular}, and the proof directly extends to our situation.
\end{proof}

As in \thref{Hopf algebra thm}, we have the Tannakian bialgebra \(H_{\Gg,\et}\) corresponding to the abelian category \(\Perv_{\et}(L^+\Gg\backslash LG /L^+\Gg,\IQ_\ell)^{\bd}\) of constructible perverse \(L^+\Gg\)-equivariant étale sheaves on \(\Gr_{\Gg}\) with bounded support. 
This is the \(\ell\)-adic realization of the bialgebra \(H_{\Gg}\) from \thref{Hopf algebra thm}.
By \thref{monoidality of nearby cycles}, we get a morphism
\(\IQ_\ell[\widehat{G}] \to H_{\Gg,\et}\) of bialgebras.
This allows us to deduce that \(H_{\Gg,\et}\) and \(H_{\Gg}\) are commutative Hopf algebras, similarly to \cite[Proposition 9.1]{ALRR:Modular}.

\begin{lem}\thlabel{bialgebra is commutative}
	The bialgebra \(H_{\Gg,\et}\) is commutative.
	Furthermore, if the map \(X_*(T)^+ \to X_*(T)_I^+\) is surjective, then so is the morphism \(\IQ_\ell[\widehat{G}] \to H_{\Gg,\et}\) constructed above.
\end{lem}
\begin{proof}
	Assume first that the map \(X_*(T)^+\to X_*(T)_I^+\) is surjective, which is the case for adjoint groups by \cite[Lemma 2.6 (2)]{ALRR:Modular}.
	In that case we will show that the morphism \(\IQ_\ell[\widehat{G}] \to H_{\Gg,\et}\) is surjective, which will in particular imply that \(H_{\Gg,\et}\) is commutative.
	As in \thref{semisimplicity}, we can show that \(\Perv_{\et}(L^+\Gg\backslash LG /L^+\Gg,\IQ_\ell)^{\bd}\) is semisimple with simple objects the intersection complexes \(\IC_{\mu}^{\et}\), which are the \(\ell\)-adic realizations of \(\IC_{\mu}(\unit)\) for \(\mu\in X_*(T)_I^+\).
	By \cite[Proposition 2.21]{DeligneMilne:Tannakian}, it suffices to show each \(\IC_{\mu}^{\et}\) is a subquotient of an object in the essential image of \(\Psi_\Gg\).
	Let \(V\in \Sat_{G,C}^{\dR}\) be a sheaf which, under the Satake equivalence, corresponds to an irreducible representation with highest weight a lift of \(\mu\) to \(X_*(T)^+\).
	Since \(\Gr_{\Gg,\leq \mu}\cap \Ss_{\mu}^- \cong \Spec k\), it suffices to show that the restriction of \(\CT_B^{\et}(\Psi_\Gg(V))\) to \(\Gr_{\Tt,\mu}\) is nontrivial.
	Indeed, by \cite[Theorem 6.16]{AGLR:Local} the special fiber of the closure of the Schubert diamond corresponding to \(V\) is \(\Gr_{\Gg,\leq \mu}\), so that \(\mu\) is the only element that can contribute to the restriction of \(\CT_B^{\et}(\Psi_\Gg(V))\) to \(\Gr_{\Tt,\mu}\).
	Hence, we are done by \cite[(6.32)]{AGLR:Local}.
	Note that the proof of \thref{Hopf algebra thm} shows that \(H_{\Gg,\et}\) is even a commutative Hopf algebra.
	
	For general \(G\), consider the quotient map \(G\to G_{\adj}\), as well as the induced map \(\Gg\to \Gg_{\adj}\), where \(\Gg_{\adj}\) is the very special parahoric corresponding to the same facet as \(\Gg\) under the isomorphism \(\buil(G,F)\xrightarrow{\cong} \buil(G_{\adj},F)\) of reduced buildings.
	Recall that \(\Gr_{\Gg} \to \Gr_{\Gg_{\adj}}\) restricts to an isomorphism on each connected component of \(\Gr_{\Gg}\).
	Hence, by \thref{connected components of affine flag varieties} and \thref{forgetting equivariance is equivalence}, \(\Perv_{\et}(L^+\Gg\backslash LG / L^+\Gg,\IQ_\ell)^{\bd}\) is equivalent to the category of \(A\in\Perv_{\et}(L^+\Gg_{\adj}\backslash LG_{\adj} / L^+\Gg_{\adj},\IQ_\ell)^{\bd}\) equipped with a refinement of its \(\pi_1(G_{\adj})_I\)-grading to a \(\pi_1(G)_I\)-grading (with the obvious notion of morphisms).
	Consequently, \(H_{\Gg,\et}\) agrees with the global sections of \[D(\pi_1(G)_I) \overset{D(\pi_1(G_{\adj})_I)}{\times} \Spec H_{\Gg_{\adj},\et},\] where \(D(-)\) denotes the diagonalizable \(\IQ_\ell\)-group scheme with given character group, which is in particular commutative.
\end{proof}

This concludes the proof of \thref{Hopf algebra thm}, and we see that \(H_{\Gg,\et}\) and \(H_{\Gg}\) are commutative Hopf algebras in general.
The final goal of this section is to describe it explicitly, as this will help us to identify \(H_{\Gg}\) integrally and motivically in the next section.
As in \cite{FarguesScholze:Geometrization, CassvdHScholbach:Geometric}, we want to use constant term functors for general parabolics.
However, we have not yet shown these are monoidal.

\begin{lem}\thlabel{cheat for monoidality of CT}
	Let \(P\subseteq G\) be any standard parabolic.
	Then the constant term functor \(\CT_P^{\et}\) induces a morphism \(H_{\Gg,\et}\to H_{\Mm,\et}\) of Hopf algebras, which is surjective if \(G\) is adjoint.
\end{lem}
\begin{proof}
First assume \(G\) is adjoint. 
The commutative diagram
\[\begin{tikzcd}
	\Sat_{M,C}^{\dR} \arrow[d, "\Psi_\Mm"'] & \Sat_{G,C}^{\dR} \arrow[d, "\Psi_\Gg"] \arrow[l, "\CT_P^{\dR}"']\\
	\Perv_{\et}(L^+\Mm\backslash LM /L^+\Mm)^{\bd} & \Perv_{\et}(L^+\Gg\backslash LG/L^+\Gg)^{\bd} \arrow[l, "\CT_P^{\et}"]
\end{tikzcd}\]
from \cite[Proposition 6.13]{AGLR:Local} induces a commutative diagram
\[\begin{tikzcd}
	\IQ_\ell[\widehat{M}] \arrow[d] & \IQ_\ell[\widehat{G}] \arrow[d] \arrow[l]\\
	H_{\Mm,\et} & H_{\Gg,\et} \arrow[l],
\end{tikzcd}\]
where the arrows are a priori just morphisms of \(\IQ_\ell\)-vector spaces.
We have already seen in \thref{monoidality of nearby cycles} that \(\Psi_{\Gg}\) and \(\Psi_\Mm\) are monoidal, and \(\CT_P^{\dR}\) is monoidal by \cite[Proposition IV.9.6]{FarguesScholze:Geometrization}, so that all but the lower arrow are morphisms of Hopf algebras.
(The compatibility of the resulting Hopf algebra structure similarly follows from \thref{monoidality of nearby cycles}.)
The fact that \(H_{\Gg,\et}\to H_{\Mm,\et}\) is also a morphism of Hopf algebras then follows from surjectivity of \(\IQ_\ell[\widehat{G}] \to H_{\Gg,\et}\), \thref{bialgebra is commutative}.
Now, since \(G\) is adjoint, the center of any Levi subgroup \(M\subseteq G\) is a torus; this follows from \cite[Proposition 21.8]{Milne:Algebraic}, as the simple roots form a basis of the character lattice for adjoint groups.
Thus, \(\IQ_\ell[\widehat{M}] \to H_{\Mm,\et}\) is surjective by \cite[Lemma 2.6 (2)]{ALRR:Modular} and \thref{bialgebra is commutative}.
Since \(\IQ_\ell[\widehat{G}]\to \IQ_\ell[\widehat{M}]\) is also surjective, we conclude that \(H_{\Gg,\et}\to H_{\Mm,\et}\) is surjective as well.

The fact that \(H_{\Gg,\et}\to H_{\Mm,\et}\) is a morphism of Hopf algebras in general follows from the adjoint case, as in \cite[Proposition 9.2]{ALRR:Modular}.
\end{proof}

Now we can determine \(H_{\Gg,\et}\).
Since \(\Perv_{\et}(L^+\Gg\backslash LG /L^+\Gg,\IQ_\ell)^{\bd}\) is semisimple and has a finite number of generators as a monoidal category (since \(X_*(T)_I^+\) is finitely generated), we already know that \(\Spec H_{\Gg,\et}\) has reductive neutral component \cite[Proposition 2.20, Remark 2.28]{DeligneMilne:Tannakian}.

\begin{prop}\thlabel{Hopf algebra for ladic sheaves}
	The morphism \(\IQ_\ell[\widehat{G}]\to H_{\Gg,\et}\) from the previous lemma factors through an isomorphism 
	\[\IQ_\ell[\widehat{G}^I] \cong H_{\Gg,\et}.\]
\end{prop}
\begin{proof}
	First, we show that \(\IQ_\ell[\widehat{G}]\to H_{\Gg,\et}\) factors through \(\IQ_\ell[\widehat{G}^I]\), where the \(I\)-action on \(\widehat{G}\) is induced by the \(I\)-action on \(\Sat_{G,C}^{\dR}\).
	By \cite[Lemma 4.5]{Zhu:Ramified}, it suffices to construct equivalences \(\Psi_{\Gg}(\gamma \cdot -) \cong \Psi_{\Gg}(-)\) for any \(\gamma\in I\).
	Since the \(I\)-action on \(C\) preserves \(\Oo_C\) and induces the trivial action on \(\overline{k}\), the \(I\)-action on \(\Gr_{G,C}^{\dR}\) extends to an \(I\)-action on \(\Gr_{\Gg}^{\BD}\) which acts trivially on the special fiber.
	This gives the desired equivalences, and hence a surjection \(\IQ_\ell[\widehat{G}^I] \to H_{\Gg,\et}\).
	Passing to group schemes, we get a closed immersion \(\phi\colon \Spec  H_{\Gg,\et} \to \widehat{G}^I\) of (possibly disconnected) reductive groups, as both representation categories are semisimple.
	
	We show \(\phi\) is an isomorphism in several steps.
	By \cite[Lemma 9.4]{ALRR:Modular} (which also works in mixed characteristic), we may assume \(G\) is adjoint.
	In particular, the proposition holds for tori.
	
	Assume first that \(G\) has a single nondivisible relative root, i.e., is of \(F\)-rank 1.
	Then \(\Spec H_{\Gg,\et}\) cannot be a torus, since \(\Perv_{\et}(L^+\Gg\backslash LG /L^+\Gg,\IQ_\ell)^{\bd}\) contains simple objects which have dimension \(>1\) when applying the fiber functor.
	Hence \(\Spec H_{\Gg,\et}\subseteq \widehat{G}^I\) must contain a root group.
	By compatibility of the constant terms with nearby cycles from \cite[Proposition 6.13]{AGLR:Local}, \(\Spec H_{\Gg,\et}\) must also contain the maximal torus \(\widehat{T}^I\subseteq \widehat{G}^I\).
	Hence \(\Spec H_{\Gg,\et} \subseteq \widehat{G}^I\) is an equality, by reductivity and bijectivity of \(\pi_0(\widehat{T}^I)\to \pi_0(\widehat{G}^I)\) \cite[Lemma 4.6]{Zhu:Ramified}.
	
	Finally, by \thref{cheat for monoidality of CT} and our assumption that \(G\) is adjoint, we have surjective morphisms \(H_{\Gg,\et}\to H_{\Mm,\et}\to H_{\Tt,\et}\) for any standard Levi \(M\subseteq G\).
	Hence, we can use \cite[Corollary 6.6]{ALRR:Fixed} to deduce the general case, by the compatibility of nearby cycles and constant term functors from \cite[Proposition 6.13]{AGLR:Local}, and the case where \(G\) has semisimple \(F\)-rank 1.
\end{proof}

This gives us a mixed characteristic ramified Satake equivalence for étale \(\IQ_\ell\)-sheaves:

\begin{cor}
	After fixing a compatible system of \(\ell^n\)-roots of unity in \(k\), there is a canonical monoidal equivalence
	\[\Perv_{\et}(L^+\Gg\backslash LG /L^+\Gg,\IQ_\ell)^{\bd} \cong \Rep(\widehat{G}^I)^{\fd}.\]
\end{cor}

\section{Identification of the dual group}\label{sect-identification}

By \thref{Hopf algebra thm}, we have an equivalence \(\MATM_{L^+\Gg}(\Gr_\Gg)\cong \coMod_{H_\Gg}(\MATM(\Spec k))\) of monoidal categories, where \(\Gg/\Oo\) is very special parahoric.
In order to finish the Satake equivalence, it hence remains to relate the Hopf algebra \(H_\Gg\) to the Langlands dual group of \(G\); this will be the goal of this section.
By \thref{Hopf algebra is reduced}, we can view \(H_\Gg\) as a \(\IZ\)-graded Hopf algebra over \(\IZ[\frac{1}{p}]\) with a \(\Gal(k'/k)=\Gamma\)-action, or in other words, since the \(\Gamma\)-action preserves the grading, as a module with a \(\IG_m\times \Gamma\)-action.

Let us start with some notation.
Recall that \(I \subseteq \Gal(\overline{F}/F)\) was the inertia subgroup, and that \(k'/k\) was the residue field of some unramified extension \(F'/F\) splitting the maximal \(\breve{F}\)-split torus \(S\subseteq G\).
Let \((\widehat{G}, \widehat{B}, \widehat{T}, \widehat{e})\) be the pinned dual group of \(G\) over \(\IZ[\frac{1}{p}]\).
The Galois group \(\Gal(\overline{F}/F)\) acts on \(\widehat{G}\) by pinning-preserving automorphisms, and we can consider the group \(\widehat{G}^I\) of inertia-invariants.
In particular, since \(I\) acts trivially on \(\widehat{G}^I\), the \(\Gal(\overline{F}/F)\)-action on \(\widehat{G}^I\) factors through a \(\Gal(\overline{k}/k)\)-action, and even through a \(\Gal(k'/k)=\Gamma\)-action.

On the other hand, consider \(2\rho\in X_*(\widehat{T})=X^*(T)\) as the sum of the positive coroots of \(\widehat{G}\).
This is clearly \(I\)-invariant, and at least after passing to the adjoint quotient \(\widehat{G}\to \widehat{G}_{\adj}\), it admits a unique square root \(\rho_{\adj}\colon \IG_m\to \widehat{T}_{\adj}\).
This induces a \(\IG_m\)-action on \(\widehat{G}^I\) via \(\IG_m\xrightarrow{\rho_{\adj}} \widehat{T}_{\adj}^I \xrightarrow{\operatorname{Ad}} \Aut(\widehat{G}^I)\).
Then the above \(\Gamma\)- and \(\IG_m\)-actions on \(\widehat{G}^I\) commute, so they induce an \(\IG_m\times \Gamma\)-action on \(\widehat{G}^I\).
The main theorem of this section is the following.

\begin{thm}\thlabel{identification hopf algebra}
	The Hopf algebra \(H_\Gg\in \MATM(\Spec k)\), viewed as an \(\IZ[\frac{1}{p}]\)-module with \(\IG_m\times \Gamma\)-action, is isomorphic to the Hopf algebra corresponding to \(\widehat{G}^I\) equipped with the above \(\IG_m\times \Gamma\)-action.
	More precisely, such an isomorphism is canonical, up to the choice of a basis element of the Tate twist \(\IZ[\frac{1}{p}](1)\).
\end{thm}

For étale \(\IQ_\ell\)-sheaves, this was already shown in \thref{Hopf algebra for ladic sheaves}.

\subsection{Proof of \thref{identification hopf algebra}}

Let \(\widetilde{\Gg} = \Spec H_\Gg\) denote the Tannakian dual group; it is flat over \(\IZ[\frac{1}{p}]\) by \thref{algebras are dualizable}.
Our goal will be to construct a canonical isomorphism \(\widehat{G}^I \cong \widetilde{\Gg}\), compatible with the \(\IG_m\times \Gamma\)-action.
We also fix a basis element of the Tate twist \(\IZ[\frac{1}{p}](1)\).

\begin{lem}\thlabel{identification - torus}
	\thref{identification hopf algebra} holds if \(G=T\) is a torus.
\end{lem}
\begin{proof}
	Assume first that \(G=T\) is residually split.
	In this case, \(T\) has a unique parahoric \(\Oo\)-model \(\Tt\), for which \(\Gr_\Tt \cong \coprod_{\nu\in X_*(T)_I} \Spec k\).
	Since there is a canonical isomorphism \(X_*(T)_I\cong X^*(\widehat{T}^I)\), the desired isomorphism \(\widetilde{\Tt}\cong \widehat{T}^I\) follows from \(\MATM_{L^+\Tt}(\Spec k) \cong \MATM(\Spec k)\), which in turns follows from \cite[Proposition 3.2.20]{RicharzScholbach:Intersection} and the fact that \(L^+\Tt\) acts trivially on \(\Gr_\Tt\).
	Moreover, \(\IG_m\) acts trivially on both \(\widetilde{\Tt}\) and \(\widehat{T}^I\).
	
	The case of general \(T\) follows by descent, observing that \(\Gamma\) acts on both \(\widetilde{\Tt}\) and \(\widehat{T}^I\) via the natural \(\Gamma\)-action on \(X_*(T)_I\cong X^*(\widehat{T}^I)\).
	(Note that the isomorphism is independent of a trivialization of the Tate twist in this case).
\end{proof}

As the constant term functor \(\CT_B\colon \MATM_{L^+\Gg}(\Gr_\Gg)\to \MATM_{L^+\Tt}(\Gr_\Tt)\) is monoidal and commutes with the fiber functors, it induces a morphism \(H_\Gg\to H_\Tt\) of Hopf algebras, or equivalently a group morphism \(\widetilde{\Tt}\to \widetilde{\Gg}\). 
Let \(\nu\in X_*(T)_I\), and \(\nu^+ \in X_*(T)_I^+\) be its dominant representative. 
Then the motive in \(\MATM_{L^+\Tt}(\Gr_{\Tt})\) supported at \(\nu\) with value \(\Ff\in \MATM(\Spec k)\) is a quotient of \(\CT_B(\IC_{\nu^+}(\Ff))\) by \thref{intersection with Weyl group conjugate}.
Hence, \(\widetilde{\Tt}\to \widetilde{\Gg}\) is a closed immersion.

Let us look at the generic fiber.
Since the base change of \(H_\Gg\) to \(\IQ_\ell\) agrees, after forgetting the \(\IG_m\times \Gamma\)-action, with \(H_{\Gg,\et}\) from Section \ref{Sect--padic}, we see that \(H_{\Gg,\IQ}\) is (possibly disconnected) reductive.
Moreover, denoting the neutral components by \((-)^{\circ}\), the inclusion \(\widetilde{\Tt}_{\IQ}^{\circ} \to \widetilde{\Gg}_{\IQ}^{\circ}\) is a split maximal torus, so that \(\widetilde{\Gg}_{\IQ}^{\circ}\) is split reductive.

Now, consider the quotient \(H_\Gg\to K\) that stabilizes the filtration \(\bigoplus_{i\geq n} \pH^i \pi_{\Gg,!}u^!\) of the fiber functor \(F\).
This corresponds to a subgroup \(\widetilde{\Bb}\subseteq \widetilde{\Gg}\), such that \(\widetilde{\Tt}\subseteq \widetilde{\Bb}\).
It will follow from the proof of \thref{identification hopf algebra} that \(\widetilde{\Bb}_{\IQ}^{\circ}\subseteq \widetilde{\Gg}_{\IQ}^{\circ}\) is a Borel; we omit details as we will not need this explicitly.

Next, we consider two explicit examples: the split group \(\PGL_2\), and a ramified \(\PU_3\).
In fact, we even allow groups \(G\) which become one of these two groups after an unramified base change.
The case of \(\PGL_2\) can be handled as in \cite[Proposition 6.22]{CassvdHScholbach:Geometric}, cf.~also \cite[§VI.11]{FarguesScholze:Geometrization}, and we refer to loc.~cit.~for details; note that \(I\) acts trivially on \(\widehat{\PGL_2} = \SL_2\) in that case.

\begin{lem}
	\thref{identification hopf algebra} holds if \(G_{F'}\cong \PU_3\) corresponding to a ramified quadratic extension.
\end{lem}
There are two conjugacy classes of very special parahorics \(\Gg\), and the geometry of the corresponding twisted affine Grassmannians is different \cite[p.411]{Zhu:Ramified}.
However, it turns out that \(\widetilde{\Gg}\) is independent of the choice of very special parahoric.
\begin{proof}
	Note that in this situation, we have \(I\cong \IZ/2\IZ\) and \(\widehat{G} \cong \SL_3\).
	We first assume that \(G\) is residually split, i.e., that \(G\cong \PU_3\).
	We will remove this assumption at the end of the proof.
	Let \(0\neq \mu\in X_*(T)_I^+\) be the image of the unique quasi-minuscule dominant cocharacter, and consider the corresponding Schubert variety \(\Gr_{\Gg,\leq \mu}\).
	Then \(\widetilde{\Gg}\) acts on \(\FibFunctor(\Jj_!(\mu,\unit)) \cong \unit \oplus \unit(-1) \oplus \unit(-2)\), cf.~\thref{fiber functor of standard}, which gives a homomorphism \(\widetilde{\Gg}\to \GL(\unit \oplus \unit(-1) \oplus \unit(-2))\).
	Using the fixed trivialization of the Tate twist, we will identify \(\GL(\FibFunctor(\Jj_!(\mu,\unit))) \cong \GL_3\), and similarly for the special linear subgroup.
	Note that \(\widetilde{\Tt}\) acts on \(\unit\) by weight \(-2\), on \(\unit(-1)\) by weight \(0\), and on \(\unit(-2)\) by weight \(2\).
	Hence, the composition \(\widetilde{\Tt}\to \widetilde{\Gg} \to \GL_3\) factors through \(\SL_3\) rationally, and then integrally by flatness.
	The (a priori possibly disconnected) reductive group \(\widetilde{\Gg}_{\IQ}\) is connected and of rank 1, as this holds after base change to \(\IQ_\ell\) by \thref{Hopf algebra for ladic sheaves}.
	Hence, the map \(\widetilde{\Gg}\to \GL_3\) factors through \(\SL_3\) rationally, and even integrally by flatness.
	
	This special linear group comes equipped with a natural maximal torus and Borel, which contain the images of \(\widetilde{\Tt}\subseteq \widetilde{\Gg}\) and \(\widetilde{\Bb}\subseteq \widetilde{\Gg}\) respectively.
	For this Borel pair, the two simple root groups can be identified with \(\IHom(\unit,\unit(-1))\cong \unit(-1)\) and \(\IHom(\unit(-1),\unit(-2))\cong \unit(-1)\) respectively.
	Viewing everything as graded modules, we can forget the grading, in which case the choice of a trivialization of the Tate twist induces a pinning of \(\SL_3\).
	Consider the unique non-trivial action of \(\IZ/2\IZ\) on \(\SL_3\) preserving this pinning, cf.~\cite[§2.3.2]{ALRR:Fixed}. 
	We note that while the pinning depends on the choice of trivialization, the \(\IZ/2\IZ\)-action does not.
	However, such a trivialization does identify the base change of \(\SL_3\) to \(\IQ_\ell\) with the dual group of \(\PGL_3\) arising from \thref{Fargues-Scholze Satake}.
	Hence, \thref{Hopf algebra for ladic sheaves} tells us that \(\widetilde{\Gg}\to \SL(\FibFunctor(\IC_\mu(\unit)))\) factors through the \(\IZ/2\IZ\)-invariants after base change to \(\IQ_\ell\).
	Generically, we hence also get a factorization \(\widetilde{\Gg}\to \SL_{3,\IQ}^I\), which is an isomorphism by \thref{Hopf algebra for ladic sheaves}.
	Since all schemes involved are flat over \(\IZ[\frac{1}{p}]\), where we use \cite[Theorem 1.1 (1)]{ALRR:Fixed} for the invariants, we get a map \(\widetilde{\Gg}\to \widehat{G}^I\) integrally as well.
 	
	To show this map is an isomorphism, it suffices by \cite[Lemma 6.20]{CassvdHScholbach:Geometric} and \cite[Tag 056A]{stacks-project} to show that this map is schematically dominant, fiberwise over \(\Spec \IZ[\frac{1}{p}]\).
	By \cite[Example 5.9 (1) and Proposition 6.9]{ALRR:Fixed}, it is enough to show that this map is surjective.
	Since by loc.~cit., the reduced fibers of \(\SL_3^I\) are isomorphic to \(\SL_2\) if \(\ell=2\), and to \(\PGL_2\) otherwise, surjectivity follows from \cite[Lemma VI.11.2]{FarguesScholze:Geometrization} (while this lemma is only stated for \(\SL_2\), the proof also works for \(\PGL_2\)).
	
	We are left to identify the \(\IG_m\times \Gamma\)-action, for which we remove the assumption that \(G\) was residually split.
	The \(\IG_m\)-actions can be computed explicitly, or alternatively it suffices to check the compatibility after base change to \(\IQ_\ell\).
	Then the Tate twists were explicitly identified in \cite[§VI.11]{FarguesScholze:Geometrization} and \cite[Theorem 6.18]{CassvdHScholbach:Geometric}, and the two resulting \(\IG_m\)-actions agree (but note \cite[Remark 6.23]{CassvdHScholbach:Geometric}).
	Hence, \(\widetilde{\Gg} \to \SL_3\) is \(\IG_m\)-equivariant, where \(\SL_3 = \widehat{\PGL_3}\) is equipped with the \(\IG_m\)-action from \cite[(6.11)]{CassvdHScholbach:Geometric}.
	On the other hand, since the isomorphism for residually split groups is canonical, compare \cite{FarguesScholze:Geometrization, CassvdHScholbach:Geometric}, it automatically descends to a similar isomorphism for general quasi-split groups, compatibly with the \(\Gamma\)-actions.
\end{proof}

Next, we look at Weil restriction of scalars.

\begin{lem}
	Let \(F/K\) be a finite separable totally ramified extension.
	If \thref{identification hopf algebra} holds for \(G\), then it also holds for the Weil restriction \(\Res_{F/K} G\).
\end{lem}
\begin{proof}
	Using that the Bruhat-Tits buildings of \(G\) and \(G':=\Res_{F/K} G\) are canonically identified, let us fix corresponding very special parahorics \(\Gg\) and \(\Gg'\).
	Since \(F/K\) is totally ramified, they have the same residue field, and the corresponding (positive) loop groups, and hence also affine flag varieties, are isomorphic, cf.~\cite[Lemma 3.2]{HainesRicharz:Smoothness}.
	In particular, there is a canonical equivariant isomorphism \(\widetilde{\Gg} \cong \widetilde{\Gg'}\).
	
	Now, let \(I'\) be the inertia group of \(K\), which contains \(I\) as a subgroup; more precisely we have \(\Gal(\overline{F}/F)\cap I' = I\) inside \(\Gal(\overline{K}/K)\).
	Then by \cite[§5]{Borel:Automorphic}, there is a Galois-equivariant isomorphism
	\[\widehat{G'} \cong \operatorname{Ind}_{\Gal(\overline{F}/F)}^{\Gal(\overline{K}/K)}\widehat{G}.\]
	It follows that there is a natural \(\IG_m\times \Gamma\)-equivariant isomorphism \(\widehat{G'}^{I'}\cong \widehat{G}^I\), which gives the required equivariant isomorphisms
	\[\widetilde{\Gg'} \cong \widetilde{\Gg} \cong \widehat{G}^I \cong \widehat{G'}^{I'}.\]
\end{proof}

This finishes all the cases where \(G\) is adjoint and has at most one nondivisible relative root. 
We can also remove the assumption that \(G\) is adjoint.

\begin{lem}\thlabel{identification - single root}
	\thref{identification hopf algebra} holds if \(G\) has semisimple \(\breve{F}\)-rank at most 1.
\end{lem}
\begin{proof}
	First, assume \(G\) is residually split.
	Consider the adjoint quotient \(G_{\adj}\) of \(G\), and the (automatically very special) parahoric \(\Gg_{\adj}\) corresponding to \(\Gg\) under the isomorphism \(\buil(G,F)\to \buil(G_{\adj},F)\) of reduced Bruhat-Tits buildings.
	Then, as we are working with perfect schemes, the natural map \(\Gr_{\Gg}\to \Gr_{\Gg_{\adj}}\) restricts to an isomorphism on each connected component, so that \(\Gr_{\Gg} \cong \Gr_{\Gg_{\adj}} \times_{\pi_1(G_{\adj})_I} \pi_1(G)_I\).
	So by \thref{forgetting equivariance is equivalence}, an object in \(\MTM_{L^+\Gg}(\Gr_{\Gg})\) is equivalent to an object in \(\MTM_{L^+\Gg_{\adj}}(\Gr_{\Gg_{\adj}})\) equipped with a refinement of the \(\pi_1(G_{\adj})_I\)-grading to a \(\pi_1(G)_I\)-grading.
	This gives a canonical equivariant isomorphism \[\widetilde{\Gg} \cong \widetilde{\Gg_{\adj}} \overset{D(\pi_1(G_{\adj})_I)}{\times} D(\pi_1(G)_I),\] where \(D(-)\) denotes the diagonalizable \(\IZ[\frac{1}{p}]\)-group scheme with given character group.
	We deduce an isomorphism \(\widetilde{\Gg}\cong \widehat{G}^I\), satisfying the required properties.
	
	Since this isomorphism is canonical, it moreover descends to the desired equivariant isomorphism for general quasi-split groups.
\end{proof}

Finally, we can finish the proof of the main theorem.

\begin{proof}[Proof of \thref{identification hopf algebra}]
Now we consider the case where \(G\) is a general group.
For \(N\gge 0\), choose a representation \(\phi\colon \widetilde{\Gg}\to \GL_N\) induced by some element of \(\MATM_{L^+\Gg}(\Gr_{\Gg})\), and a faithful representation \(\psi\colon \widehat{G}\to \GL_N\), such that, after base change to \(\IQ_\ell\), \(\phi\) factors through \(\psi\); this is possible by \thref{Hopf algebra for ladic sheaves}.
Then \(\phi\) also factors through \(\psi\) after base change to \(\IQ\), and hence integrally by flatness, giving a map \(\widetilde{\Gg} \to \widehat{G}\).

Now, let \(a\) be a nondivisible simple relative coroot, and consider the absolute coroots of \(G\) which restrict to relative coroots given by scalar multiples of \(a\).
The corresponding roots in \(\widehat{G}\) then induce a Levi \(\widehat{M}_a\subseteq \widehat{G}\).
This Levi can moreover be identified with the Langlands dual group of a Levi subgroup \(M_a\subseteq G\) of semisimple \(\breve{F}\)-rank 1.
The parahoric model \(\Gg\) of \(G\) also gives rise to a parahoric model \(\Mm_a\) of \(M_a\).
Then by \thref{identification - single root}, we obtain a canonical \(\IG_m\)-equivariant isomorphism \(\widetilde{\Mm}_a\cong \widehat{M}_a^I\).

Next, let \(P_a\supseteq B\) be a standard parabolic with Levi quotient \(M_a\).
Then the constant term functor \(\CT_{P_a}\) induces a map \(H_{\Gg} \to H_{\Mm_a}\), and we claim this is a morphism of Hopf algebras.
As \(H_{\Gg}\) and \(H_{\Mm_a}\) are dualizable, this can be checked after rationalizing.
Since the rational \(\ell\)-adic realization functor is faithful when restricted to mixed Artin-Tate motives by \cite[Lemma 3.2.8]{RicharzScholbach:Intersection}, the claim follows from \thref{cheat for monoidality of CT}.
By compatibility of étale constant terms and nearby cycles, the resulting composition \(\widetilde{\Mm}_a\to \widetilde{\Gg} \to \widehat{G}\) factors through \(\widehat{M}_a^I\subseteq \widehat{M}_a\subseteq \widehat{G}\): first \(\ell\)-adically, then rationally, and hence integrally by flatness.
In particular, \(\widetilde{\Gg}\to \widehat{G}\) factors through \(\widehat{G}^I\).
Moreover, \(\widetilde{\Gg}\to \widehat{G}^I\) is an isomorphism after base change to \(\IQ_\ell\) by \thref{Hopf algebra for ladic sheaves}, and hence after base change to \(\IQ\) by faithfully flat descent.
To show it is an isomorphism integrally, it suffices to show it is surjective by \cite[Lemma 6.20]{CassvdHScholbach:Geometric} and \cite[Tag 056A]{stacks-project}.
We then conclude by the case of semisimple \(\breve{F}\)-rank 1 groups, as well as \cite[Corollary 6.6]{ALRR:Fixed}.

The fact that the \(\IG_m\times \Gamma\)-action on \(H_{\Gg}\) (viewed as a \(\IZ[\frac{1}{p}]\)-module) is the correct one then also follows from \thref{identification - single root}.
\end{proof}

\subsection{Variants}

By \thref{change of scalars}, \thref{identification hopf algebra} also holds for coefficient rings such as \(\IQ\) and \(\IF_\ell\).
In fact, for rational coefficients, we can rephrase the theorem as follows.

\begin{cor}\thlabel{Cgroup equiv}
	For rational coefficients, there is a canonical symmetric monoidal equivalence
	\[(\MATM_{L^+\Gg}(\Gr_\Gg; \IQ), \star) \cong (\Rep_{\widehat{G}^I\rtimes (\IG_m \times \Gamma)}(\QVect), \otimes).\]
\end{cor}
\begin{proof}
	This follows from the fact that \(\MATM(\Spec k; \IQ)\) is equivalent to \(\IZ\)-graded vector spaces with a \(\Gamma\)-action (cf.~\thref{graded to motives}).
\end{proof}

Let us recall the C-group, as it will also appear in the next section.
First introduced in \cite{BuzzardGee:Conjectural}, it will be more useful to consider a variant as in \cite{Zhu:Integral}, where not necessarily the full Galois group appears.

\begin{dfn}\thlabel{Cgroup}
	Let \(\Gal(\overline{F}/F)\onto \Gamma'\) be a (not necessarily finite) quotient through which the \(\Gal(\overline{F}/F)\)-action on \(\widehat{G}\) factors.
	Then the \emph{C-group} is the semi-direct product \(\CG:= \widehat{G} \rtimes (\IG_m\times \Gamma')\), where \(\IG_m\) acts via \(\Ad \rho_{\adj}\colon \IG_m \xrightarrow{\rho_{\adj}} \widehat{T}_{\adj} \xrightarrow{\Ad} \Aut(\widehat{G})\).
\end{dfn}

We can then make the following observations.

\begin{rmk}
	\begin{enumerate}
		\item In the unramified case, i.e., when \(\Gg\) is hyperspecial, the inertia \(I\) acts trivially on \(G\).
		In particular, in that case the group \(\widehat{G}^I\rtimes (\IG_m \times \Gamma)\) is the C-group defined above.
		In the split case, we can moreover choose \(F'=F\), so that \(k'=k\).
		Then the Galois group vanishes, so that the C-group becomes Deligne's modification of the Langlands dual group as in \cite{Deligne:Letter2007, FrenkelGross:Irregular}, and we recover the main results of \cite{RicharzScholbach:Satake,RicharzScholbach:Witt}.
		\item To construct an equivalence as in \thref{Cgroup equiv} for general coefficients, one would need to use the reduced motives from \cite{EberhardtScholbach:Integral}, as in \cite{CassvdHScholbach:Geometric}. 
		However, unlike Nisnevich motives, one can only reduce étale motives with rational coefficients, which is why we only state the corollary in this case.
		(Note also that over \(\Spec k\), rational regular motives are already reduced, so there is no need to consider reduced motives explicitly.)
		\item In the appendix of \cite{Zhu:Ramified}, two different \(\Gamma\)-actions on \(\widehat{G}\) are considered:~an algebraic action, arising from the definition of the Langlands dual group, and a geometric action, arising from the \(\Gamma\)-action on the affine Grassmannian.
		By \cite[Proposition A.6]{Zhu:Ramified}, they differ exactly by a cyclotomic twist, i.e., by a Tate twist.
		Thus, since the \(\widehat{G}\) and \(\widetilde{\Gg}\) already incorporate the information coming from the Tate twist, they implicitly relate the algebraic \(\Gamma\)-action on \(\widehat{G}\) to the geometric \(\Gamma\)-action on \(\widetilde{\Gg}\).
	\end{enumerate}
\end{rmk}

Although we only stated and proved \thref{identification hopf algebra} in the motivic setting, the proof also works (and can be simplified) for étale sheaves.
Let us record this explicitly.

\begin{thm}\thlabel{Satake for etale cohomology}
	Up to trivializing the Tate twist, there is a canonical isomorphism
	\[\Perv_{\et}(L^+\Gg\backslash LG /L^+\Gg,\IZ_\ell)^{\bd} \cong \Rep(\widehat{G}^I)^{\fg},\]
	where the Langlands dual group \(\widehat{G}\) is defined over \(\IZ_\ell\), and we only consider representations on finitely generated modules.
\end{thm}

In particular, if \(F\) is of equal characteristic, we recover the main theorem of \cite{ALRR:Modular}.
We note that while \(\widehat{G}^I\) is generically reductive (possible disconnected), this is not true integrally.
Already in the example of \(G=\SU_3\) corresponding to a ramified quadratic extension, the dual group \(\SL_3^{\IZ/2\IZ}\) is isomorphic to \(\PGL_2\) over \(\IZ[\frac{1}{2p}]\), but is nonreduced over \(\IZ/2\IZ\).
Moreover, the reduction of the fiber modulo 2 is isomorphic to \(\SL_2\).
We refer to \cite{ALRR:Fixed,ALRR:Modular} for more details.

\section{Hecke algebras and the Vinberg monoid}\label{sect-Hecke}

In this final section, we will decategorify our Satake equivalence, using Grothendieck's sheaf-function dictionary.
This will also highlight certain advantages of motives over e.g.~\(\ell\)-adic cohomology.
Since we will take Grothendieck rings and consider anti-effective motives, we will work with \(\IQ\)-coefficients.
This has the advantage that \(\MTM(\Spec k,\IQ) \cong \grQVect\), for which the Grothendieck ring \(K_0(\grQVect^{\fd})\cong \IZ[t,t^{-1}]\) is known.
Hence, throughout this section we will always use motives with \(\IQ\)-coefficients, and write \(\DM(-)=\DM(-,\IQ)\).
As usual, \(\Gg/\Oo\) denotes a very special parahoric integral model with reductive generic fiber \(G\), but we now assume that \(k=\IF_q\) is a finite field.
For simplicity, we will also take \(k'/k\) to be a finite extension.
Then \(\Gamma=\Gal(k'/k)\) is a finite cyclic group, generated by the geometric \(q\)-Frobenius \(\sigma\).

\subsection{The Vinberg monoid}

The object representing the Langlands dual side in \cite{Zhu:Integral} is the Vinberg monoid.
We will recall its definition and basic properties, and explain how it behaves under taking invariants.
We will follow the approach of loc.~cit., and refer to \cite{Vinberg:Reductive, XiaoZhu:Vector} for more details.

Let \((\widehat{G},\widehat{B},\widehat{T},\widehat{e})\) be the pinned dual group of \(G\) as before, but which we now consider as a group over \(\Spec \IZ\).
As in the previous section, the group \(\widehat{G}^I\) of inertia-invariants admits an action of \(\IG_m\times \Gamma\), preserving \(\widehat{B}\) and \(\widehat{T}\).
Although most results below will hold for more general pinning-preserving actions by finite groups, we restrict ourselves to the situation at hand for simplicity.

\begin{nota}\thlabel{dual Weyl groups}
	We denote by \(W=N_{\widehat{G}}(\widehat{T})/\widehat{T}\) the Weyl group of \(\widehat{G}\), and write \(W_0 := W^{\Gal(\overline{F}/F)} = (W^I)^\Gamma\).
	Recall from \cite[Lemma 6.3]{ALRR:Fixed} that \(N_{\widehat{G}^I}(\widehat{T}^I) = N_{\widehat{G}}(\widehat{T})^I\), and that \(W^I\) can be seen as the Weyl group of \(\widehat{G}^I\).
	Using this, we let \(\widehat{N}_0\) be the preimage of \(W_0\) in \(N_{\widehat{G}^I}(\widehat{T}^I)\).
\end{nota}

Let \(\widehat{G}_{\adj}\) be the adjoint quotient of \(\widehat{G}\), and \(\widehat{T}_{\adj}\) the adjoint torus, i.e., the image of \(\widehat{T}\) in \(\widehat{G}_{\adj}\).
Then we have inclusions \(X^*(\widehat{T}_{\adj})\subseteq X^*(\widehat{T})\), which can be identified with the root lattice inside the character lattice of \(\widehat{T}\).
Let \(X^*(\widehat{T}_{\adj})_{\pos} \subseteq X^*(\widehat{T})_{\pos}^+ \subseteq X^*(\widehat{T})\) be the submonoid generated by the simple roots of \(\widehat{G}\), respectively the simple roots and the dominant characters, corresponding to the Borel \(\widehat{B}\).

Now, \(\IZ[\widehat{G}]\) admits a \(\widehat{G} \times \widehat{G}\)-action induced by the left and right multiplication of \(\widehat{G}\) on itself.
Moreover, \(\IZ[\widehat{G}]\) admits a multifiltration (in the sense of \cite[§2]{XiaoZhu:Vector}) by \(\widehat{G}\times \widehat{G}\)-submodules
\begin{equation}\label{filtration dual group}
	\IZ[\widehat{G}] = \bigcup_{\mu\in X^*(\widehat{T})_{\pos}^+} \fil_{\mu}\IZ[\widehat{G}],
\end{equation}
where \(\fil_{\mu}\IZ[\widehat{G}]\) is the largest submodule for which any weight \((\lambda,\lambda')\in X^*(\widehat{T}) \times X^*(\widehat{T})\) satisfies \(\lambda\leq -w_0(\mu)\) and \(\lambda'\leq \mu\), where the partial order is defined by \(\lambda_1\leq \lambda_2\) if \(\lambda_1-\lambda_2\in X^*(\widehat{T}_{\adj})_{\pos}\) and \(w_0\in W\) is the longest element.
The associated graded of this filtration is given by 
\begin{equation}\label{associated graded}
	\gr \IZ[\widehat{G}] = \bigoplus_{\mu\in X^*(\widehat{T})} S_{-w_0(\mu)} \otimes S_\mu,
\end{equation}

where \(S_\mu\) denotes the Schur module of highest weight \(\mu\).

\begin{dfn}
	The \emph{universal Vinberg monoid} \(V_{\widehat{G}}\) of \(\widehat{G}\) is the spectrum of the Rees algebra \(R_{X^*(\widehat{T})_{\pos}^+} \IZ[\widehat{G}] = \bigoplus_{\mu \in X^*(\widehat{T})_{\pos}^+} \fil_{\mu} \IZ[\widehat{G}]\).
	It is a finite type affine monoid scheme, equipped with a faithfully flat monoid morphism
	\[d\colon V_{\widehat{G}} \to \widehat{T}_{\adj}^+ := \Spec \IZ[ X^*(\widehat{T}_{\adj})_{\pos}].\]
	The \(\Gal(\overline{F}/F)\)-action on \(\widehat{G}\) extends to actions on both \(V_{\widehat{G}}\) and \(\widehat{T}_{\adj}^+\), for which \(d\) is equivariant.
\end{dfn}

For example, if \(\widehat{G} = \widehat{T}\) is a torus, then \(V_{\widehat{T}} = \widehat{T}\) and \(\widehat{T}_{\adj}^+ = \Spec \IZ\).
Note that any dominant cocharacter \(\widehat{\lambda}\colon \IG_{m,\IZ}\to \widehat{T}_{\adj}\) extends to a monoid morphism \(\widehat{\lambda}^+\colon \IA^1_\IZ\to \widehat{T}_{\adj}^+\).
Hence, given any such dominant cocharacter, we can specialize the universal monoid by defining
\[d_{\widehat{\lambda}}\colon V_{\widehat{G},\widehat{\lambda}} := \IA^1_{\IZ} \times_{\widehat{\lambda},\widehat{T}_{\adj}^+} V_{\widehat{G}} \to \IA^1_{\IZ}.\]
We will mostly be interested in the cocharacter \(\widehat{\lambda} = \rho_{\adj}\), which is the unique square root of \(\IG_{m,\IZ}\xrightarrow{2\rho} \widehat{T} \to \widehat{T}_{\adj}\), where \(2\rho\) is the sum of the positive coroots of \(\widehat{G}\).
In that case, we call \(\vinberg\) simply the \emph{Vinberg monoid}.
Since \(\rho_{\adj}\) is \(\Gal(\overline{F}/F)\)-invariant, the \(\Gal(\overline{F}/F)\)-action on \(V_{\widehat{G}}\) restricts to \(\vinberg\), and \(d_{\rho_{\adj}}\) extends to a morphism
\[\widetilde{d}_{\rho_{\adj}}\colon \vinberg \rtimes \Gal(\overline{F}/F) \to \IA^1_{\IZ} \times \Gal(\overline{F}/F).\]
By \cite[(1.7)]{Zhu:Integral} there is an isomorphism \(\widetilde{d}_{\rho_{\adj}}^{-1}(\IG_{m,\IZ}\times \Gal(\overline{F}/F)) \cong \CG\), which coincides with the group of units of \(\vinberg \rtimes \Gal(\overline{F}/F)\).

\begin{ex}
	In case \(G=\PGL_2\), we have \(\CG=\widehat{G}\rtimes \IG_m\cong \GL_2\).
	Then \(\vinberg\) is isomorphic to the monoid of \(2\times 2\)-matrices, and \(d_{\rho_{\adj}}\) is identified with the determinant map.
\end{ex}

Now, restricting the \(\Gal(\overline{F}/F)\)-action on \(V_{\widehat{G}}\) and \(\vinberg\), we can consider the inertia-invariants \(V_{\widehat{G}}^I\) and \(\vinberg^I\).
These invariants admit an action of \(\Gal(\overline{k}/k)\), which clearly factors through \(\Gamma\).

\begin{rmk}\thlabel{rmk-description of invariants}
	Let us describe \(V_{\widehat{G}}^I\) explicitly.
	As \(I\) preserves \(\widehat{T}\subseteq \widehat{B}\subseteq \widehat{G}\), it induces an action on \(X^*(\widehat{T})_{\pos}^+\) and \(X^*(\widehat{T}_{\adj})_{\pos}\).
	This in turn defines a filtration 
	\begin{equation}\label{filtration invariants}
		\IZ[\widehat{G}^I] = \bigcup_{\mu\in (X^*(\widehat{T})_{\pos}^+)_I} \fil_{\mu}\IZ[\widehat{G}^I],
	\end{equation}
	induced by \eqref{filtration dual group}.
	Moreover, the global sections of invariants are given by coinvariants, so that
	\[\IZ[V_{\widehat{G}}^I] = R_{(X^*(\widehat{T})_{\pos}^+)_I} \IZ[\widehat{G}^I] = \bigoplus_{\mu \in (X^*(\widehat{T})_{\pos}^+)_I} \fil_{\mu} \IZ[\widehat{G}^I].\]
\end{rmk}

\begin{rmk}\thlabel{units of invariants}
	In fact, we claim that there is a natural bijection \((X^*(\widehat{T}_{\adj})_{\pos})_I \cong X^*(\widehat{T}_{\adj}^I)_{\pos}\), where the latter denotes the submonoid of \(X^*(\widehat{T}^I)\) generated by the positive roots of \(\widehat{G}^I\) with respect to \(\widehat{T}^I\). 
	Note that invariants of adjoint reductive groups are connected and adjoint by \cite[Example 5.3 and Lemma 6.7]{ALRR:Fixed}.
	Then the claim follows from \cite[Proposition 6.1]{ALRR:Fixed}.
	Similarly, we can define \(X^*(\widehat{T}^I)_{\pos}^+\), and there is natural bijection \((X^*(\widehat{T})_{\pos}^+)_I \cong X^*(\widehat{T}^I)_{\pos}^+\).
	
	Recall from \cite[Proposition 3.2.2]{XiaoZhu:Vector} that the group of units of \(V_{\widehat{G}}\) is given by the quotient \(\widehat{G}\overset{Z_{\widehat{G}}}{\times}\widehat{T}\) of the inclusion \(Z_{\widehat{G}}\subseteq \widehat{G}\times \widehat{T}\colon z\mapsto (z,z)\).
	Using the description of \(V_{\widehat{G}}^I\) above, we can similarly see that the group of units of \(V_{\widehat{G}}^I\), which is given by \((\widehat{G}\overset{Z_{\widehat{G}}}{\times} \widehat{T})^I\), agrees with \(\widehat{G}^I\overset{Z_{\widehat{G}^I}}{\times} \widehat{T}^I\).
	Here we are again using the fact that \(Z_{\widehat{G}^I} \cong Z_{\widehat{G}}^I\) \cite[Lemma 6.7]{ALRR:Fixed}.
\end{rmk}

At least for the generic fiber, this allows us to give a Tannakian interpretation of \(\vinberg^I\)

\begin{thm}\thlabel{Anti-effective equivalence}
	The equivalence from \thref{Cgroup equiv} restricts to a monoidal equivalence
	\[(\MATM_{L^+\Gg}(\Gr_{\Gg})^{\anti},\star)\cong (\Rep_{\vinberg^I\rtimes \Gamma}(\IQ\text{-}\operatorname{Vect}),\otimes).\]
\end{thm}
\begin{proof}
	We saw in \thref{units of invariants} that \(\widehat{G}^I\overset{Z_{\widehat{G}^I}}{\times} \widehat{T}^I\) identifies with the group of units of \(V_{\widehat{G}}^I\). 
	Since we are working in characteristic 0, this group is (possibly disconnected) reductive \cite[Theorem 1.1]{ALRR:Fixed}; in particular its representation category is semisimple.
	Moreover, this inclusion is open dense, so that the restriction of representations is a fully faithful functor.
	It thus suffices to identify the representations corresponding to anti-effective motives.
	The theorem will then follow since anti-effective motives are preserved under convolution by \thref{Conv-Tate}.
	Since the property of being anti-effective can be checked after base change along \(\Spec k'\to \Spec k\), we may assume \(k=k'\) and consider Tate motives only; this amounts to forgetting the \(\Gamma\)-action.
	We then proceed as in \cite[Lemma 21]{Zhu:Integral}.
	In the rest of the proof, all our schemes will live over \(\Spec \IQ\).
	
	Let \(M\) be an irreducible \(\widehat{G}^I\overset{Z_{\widehat{G}^I}}{\times} \widehat{T}^I\)-representation.
	Then \(M_{\mid 1\times \widehat{T}^I}\) has a single weight \(\mu\in X^*(\widehat{T}^I)\).
	By \thref{rmk-description of invariants}, \(M\) extends to \(V_{\widehat{G}}^I\) if and only if we have \(\mu\in X^*(\widehat{T}^I)_{\pos}^+\subseteq X^*(\widehat{T}^I)\) and the coaction map sends \(M\) to \(\fil_\mu\IQ[\widehat{G}^I]\otimes_{\IQ} M\).
	This second condition can be rephrased as \(\lambda \leq -w_0(\mu)\) for any weight \(\lambda\) of \(M_{\mid \widehat{G}^I\times 1}\).
	Or equivalently, \(\mu+\lambda_-\in X^*(\widehat{T}_{\adj}^I)_{\pos} \subseteq X^*(\widehat{T}^I)\) for any weight \(\lambda\) as above, where \(\lambda_-\) is the unique anti-dominant element in \(W^I\lambda\).
	
	In order to pass to \(\vinberg^I\), consider the composition \(\widehat{G}^I \times \IG_m \to (\widehat{G}^I\times \IG_m)/(2\rho\times \identity)(\mu_2) \into \vinberg^I\), where the middle term is isomorphic to \(\widehat{G}^I\rtimes \IG_m\), and identifies with the group of units of \(\vinberg^I\).
	By the previous paragraph, an irreducible representation \(M\) of \(\widehat{G}^I\times \IG_m\) extends to \(\vinberg^I\) if and only if for any weight \((\lambda,n)\) of \(M\), there exists \(\mu\in X^*(\widehat{T}^I)_{\pos^+}\) with \(\langle 2\rho,\mu\rangle = n\) and \(\mu+\lambda_-\in X^*(\widehat{T}_{\adj}^I)_{\pos}\).
	This is in turn equivalent to having \(\langle 2\rho,\lambda_-\rangle \geq -n\) and \(\langle 2\rho,\lambda\rangle \equiv n \mod 2\).
	Since the isomorphism \((\widehat{G}^I\times \IG_m)/(2\rho\times \identity)(\mu_2) \cong \widehat{G}^I\rtimes \IG_m\) is given by \((g,t)\mapsto (g2\rho(t)^{-1},t^2)\), the weight \((\lambda,n)\) of \(\widehat{G}^I\rtimes \IG_m\) induces the weight \((\lambda, 2n-\langle 2\rho,\lambda,\rangle)\) of \(\widehat{G}\times \IG_m\).
	Thus, the equivalence from \thref{Cgroup equiv} identifies the anti-effective Tate motives with those representations whose \(\IG_m\)-weights are \(\geq 0\).
	Since anti-effectivity can be checked after applying the fiber functor by \thref{anti-effective motives and the fiber functor}, which corresponds to pullback along \(1\times \IG_m\into \widehat{G}^I\rtimes \IG_m\), this concludes the proof.
\end{proof}

Now, since \(\Gamma = \langle \sigma \rangle\) acts on \(V_{\widehat{G}}^I\) and \eqref{filtration invariants} is a filtration by \(\widehat{G}^I\times \widehat{G}^I\)-modules, we can consider the twisted conjugation action
\[c_{\sigma}\colon \widehat{G}^I\times V_{\widehat{G}}^I \to V_{\widehat{G}}^I\colon (g,x)\mapsto g x \sigma(g)^{-1}.\]
Moreover, the regular conjugation action on \(V_{\widehat{G}}^I\rtimes \langle \sigma \rangle\) preserves the subscheme \(V_{\widehat{G}\mid d=\rho_{\adj}(q)}^I\sigma\), and we have
\[\IZ[\restrvinberg^I\sigma]^{\widehat{G}^I} = \IZ[\restrvinberg^I]^{c_\sigma(\widehat{G}^I)}.\]

We denote by \(V_{\widehat{T}}\) the closure of \(\widehat{T}\overset{Z_{\widehat{G}}}{\times} \widehat{T} \subseteq \widehat{G}\overset{Z_{\widehat{G}}}{\times} \widehat{T}\) in \(V_{\widehat{G}}\).
This is stable under the inertia action, and the commutative diagram
\[\begin{tikzcd}
	\widehat{T} \overset{Z_{\widehat{G}}}{\times} \widehat{T} \arrow[r] & V_{\widehat{T}} \\
	(\widehat{T} \overset{Z_{\widehat{G}}}{\times} \widehat{T})^I \arrow[u] \arrow[r] & V_{\widehat{T}}^I \arrow[u]
\end{tikzcd}\]
corresponds to the following diagram by taking global sections:
\[\begin{tikzcd}
	\bigoplus_{\lambda,\nu\in X^*(\widehat{T})\colon \lambda+\nu \in X^*(\widehat{T}_{\adj})} \IZ(e_1^\lambda \otimes e_2^\nu) \arrow[d] & \bigoplus_{\lambda,\nu\in X^*(\widehat{T})\colon \nu+\lambda_-\in X^*(\widehat{T}_{\adj})_{\pos}} \IZ(e_1^\lambda \otimes e_2^\nu) \arrow[d] \arrow[l]\\
	\bigoplus_{\lambda,\nu\in X^*(\widehat{T}^I)\colon \lambda +\nu \in X^*(\widehat{T}_{\adj}^I)} \IZ(e_1^\lambda\otimes e_2^\nu) & \bigoplus_{\lambda,\nu\in X^*(\widehat{T}^I)\colon \nu+\lambda_-\in (X^*(\widehat{T}_{\adj})_{\pos})_I} \IZ(e_1^\lambda + e_2^\nu), \arrow[l]
\end{tikzcd}\]
where for \(\lambda\in X^*(\widehat{T})\) (resp.~\(\lambda\in X^*(\widehat{T}^I)\)) we denote by \(\lambda_-\in X^*(\widehat{T}^-)\) (resp.~\(\lambda_-\in X^*(\widehat{T}^I)^-\)) the unique antidominant representative of \(W\cdot \lambda\) (resp.~\(W^I\cdot \lambda\)).
Indeed, this follows from \cite[§3.2]{XiaoZhu:Vector} (although we use the same sign convention as in \cite[§1.3]{Zhu:Integral}), the observation that \(X^*(\widehat{T})_I \cong X^*(\widehat{T}^I)\), and similarly for \(\widehat{T}_{\adj}\).

The morphism \(V_{\widehat{T}}^I\to (\widehat{T}_{\adj}^+)^I\) can be explicitly described as \[\IZ[(X^*(\widehat{T}_{\adj})_{\pos})_I]\to \IZ[V_{\widehat{T}}^I]\colon e^{\lambda}\mapsto e_1^0\otimes e_2^\lambda,\]
so that it admits a section \(s\colon (\widehat{T}_{\adj}^+)^I \to V_{\widehat{T}}^I\) given by
\[\IZ[V_{\widehat{T}}^I] \to \IZ[(X^*(\widehat{T}_{\adj})_{\pos})_I] \colon e_1^\lambda \otimes e_2^\nu \mapsto e^{\lambda+\nu}.\]
Now, consider the closed immersion \(i\colon \widehat{T}^I\to (\widehat{T}\overset{Z_{\widehat{G}}}{\times} \widehat{T})^I \to V_{\widehat{T}}^I\) given by the inclusion into the first factor.
Then the product \((i,s)\colon \widehat{T}^I \times (\widehat{T}_{\adj}^+)^I\to V_{\widehat{T}}^I\) is given by
\[\IZ[V_{\widehat{T}}^I] \to \IZ[X^*(\widehat{T}^I)]\otimes_{\IZ} \IZ[(X^*(\widehat{T}_{\adj})_{\pos})_I] \colon e_1^\lambda \otimes e_2^\nu \mapsto e^\lambda \otimes e^{\lambda+\nu},\]
and hence induces an injection
\begin{equation}
	\IZ[V_{\widehat{T}\mid d=\rho_{\adj}(q)}^I] \into \IZ[X^*(\widehat{T}^I)]\colon e_1^\lambda \otimes e_2^\nu\mapsto q^{\langle \rho_{\adj},\lambda+\nu\rangle} e^\lambda.
\end{equation}

Recall the group \(\widehat{N}_0\) from \thref{dual Weyl groups}.
The twisted conjugation action \(c_\sigma\colon \widehat{G}^I\times V_{\widehat{G}}^I\to V_{\widehat{G}}^I\) restricts to an action \(c_\sigma\colon \widehat{N}_0\times V_{\widehat{T}}^I\to V_{\widehat{T}}^I\).
Taking invariants under this action, we obtain the following:
\begin{lem}\thlabel{lemma invariants of vinberg}
	The composition
	\begin{equation}\label{description invariants of vinberg}
		\IZ[V_{\widehat{T}\mid d=\rho_{\adj}(q)}^I]^{c_\sigma(\widehat{N}_0)} \into	\IZ[V_{\widehat{T}\mid d=\rho_{\adj}(q)}^I] \into \IZ[X^*(\widehat{T}^I)]
	\end{equation}
	factors through \(\IZ[X^*(\widehat{T}^I)^{\sigma}]\subseteq \IZ[X^*(\widehat{T}^I)]\).
	Moreover, the elements \(\sum_{\lambda'\in W_0\lambda} q^{\langle \rho_{\adj},\lambda'-\lambda \rangle} e^{\lambda'}\) for \(\lambda\in X^*(\widehat{T}^I)^{\sigma}\cap X^*(\widehat{T}^I)^-\) form a \(\IZ\)-basis for its image.
\end{lem}

The following generalization of \cite[Proposition 4.2.3]{XiaoZhu:Vector} is the twisted Chevalley restriction isomorphism in the current setting.

\begin{prop}\thlabel{Chevalley restriction}
	The inclusion \(V_{\widehat{T}}^I\subseteq V_{\widehat{G}}^I\) induces an isomorphism
	\[\IZ[V_{\widehat{G}}^I]^{c_\sigma(\widehat{G}^I)} \xrightarrow{\cong} \IZ[V_{\widehat{T}}^I]^{c_\sigma(\widehat{N}_0)},\]
	which restricts to an isomorphism
	\[\Res\colon \IZ[\restrvinberg^I]^{c_\sigma(\widehat{G}^I)} \xrightarrow{\cong} \IZ[V_{\widehat{T}\mid d=\rho_{\adj}(q)}^I]^{c_\sigma(\widehat{N}_0)}\]
\end{prop}
\begin{proof}
	Consider the \((X^*(\widehat{T})^+_{\pos})_I\)-filtration on \(\IZ[\widehat{G}^I]\) from \eqref{filtration invariants}.
	Via the embedding \(\widehat{T}^I\subseteq \widehat{G}^I\), it induces a \((X^*(\widehat{T})^+_{\pos})_I\)-filtration on \(\IZ[\widehat{T}^I]\), given by
	\[\fil_\mu \IZ[\widehat{T}^I] = \bigoplus_{\lambda \in X^*(\widehat{T}^I) \colon\lambda_{\dom}\leq \nu} \IZ\cdot e^\lambda,\]
	where \(\lambda_{\dom}\in X^*(\widehat{T}^I)^+\) is the unique dominant representative in \(W^I\cdot \lambda\).
	
	Then the twisted conjugation actions, by \(\widehat{G}^I\) and \(\widehat{N}_0\) respectively, preserve these filtrations, and we have
	\[(\fil_\mu \IZ[\widehat{T}^I])^{c_\sigma(\widehat{N}_0)} = \bigoplus_{\lambda\in X^*(\widehat{T}^I)^{+,\sigma}\colon \mu-\lambda \in X^*(\widehat{T}_{\adj}^I)_{\pos}} \IZ \cdot (\sum_{\nu\in W_0\lambda} e^\nu).\]
	In particular, each graded piece of this multi-filtration is either trivial or free of rank 1.
	Using \eqref{associated graded}, we see that the same holds for the graded pieces of \((\fil_\mu \IZ[\widehat{G}^I])^{c_\sigma(\widehat{G}^I)}\), and that we get an isomorphism \(\gr \IZ[\widehat{G}^I]^{c_\sigma(\widehat{G}^I)} \cong \IZ[\widehat{T}^I]^{c_\sigma(\widehat{N}_0)}\).
	Since both \(V_{\widehat{G}}^I\) and \(V_{\widehat{T}}^I\) can be described as the Rees algebra associated to the filtrations above, this implies that \(\IZ[V_{\widehat{G}}^I]^{c_\sigma(\widehat{G}^I)} \to \IZ[V_{\widehat{T}}^I]^{c_\sigma(\widehat{N}_0)}\) is an isomorphism.
	
	This isomorphism is moreover clearly compatible with the \(\IZ[X^*(\widehat{T}_{\adj}^I)_{\pos}]\)-structures on both sides.
	Hence, we get the desired isomorphism \(\Res\) by tensoring along \(\IZ[X^*(\widehat{T}_{\adj}^I)_{\pos}]\to \IZ\), via the map defined by \(q^{\langle\rho_{\adj},-\rangle}\).
\end{proof}

In order to construct an integral Satake isomorphism in the next subsection, we will take traces of representations.
This is similar to \cite[Lemma 22]{Zhu:Integral}, but simpler since we know our traces take values in \(\IQ\) rather than \(\IQ_\ell\).
For an abelian category \(\Cc\), we will denote by \(K_0(\Cc)\) the Grothendieck group of the category of compact objects of \(\Cc\), and by \([-]\) the class of an object in \(\Cc\).
Since the notation \(V_{\widehat{T}}\) was already used above, we will use \(\widehat{T}\times \IA^1\) to denote the analogue of \(\vinberg\) for \(\widehat{T}\).
There is a Galois-equivariant embedding \(\widehat{T}\times \IA^1 \subseteq \vinberg\), and we denote by \(\Res\) the functor given by restriction of representations.

\begin{lem}\thlabel{trace of representations}
	There exists a (necessarily unique) surjective morphism \(\tr\), such that the diagram below commutes:
	\[\begin{tikzcd}[column sep=huge]
		K_0(\Rep_{\IQ}(\vinberg^I\rtimes \Gamma)) \arrow[rrr, "\tr"] \arrow[d, "K_0(\Res)"'] &&& \IZ[\restrvinberg^I]^{c_{\sigma}(\widehat{G}^I)} \arrow[d, "\eqref{description invariants of vinberg} \circ \Res"] \\
		K_0(\Rep_{\IQ}((\widehat{T}^I\rtimes \Gamma)\times \IA^1)) \arrow[rrr, "{[}V{]}\mapsto \sum_{\lambda\in X^*(\widehat{T}^I)^{\sigma}} \tr((q{,}\sigma)\mid V(\lambda)) e^\lambda"'] &&& \IZ[X^*(\widehat{T}^I)^{\sigma}].
	\end{tikzcd}\]
\end{lem}
\begin{proof}
	The map \(\tr\) is defined by taking the trace of representations, as in \cite[Lemma 22]{Zhu:Integral}.
	To see that it is well-defined, it suffices to show these traces take integral values.
	First, note that these traces are always rational numbers, since we are using rational representations.
	The integrality then follows since the \(\IG_m\)-representations obtained via restriction along \(\IG_m\into \vinberg^I\rtimes \Gamma\) have positive weights, and the fact that \(\sigma\) has finite order.
	Finally, surjectivity follows by considering the objects corresponding to \(\IC_\mu(\unit)\) under \thref{Anti-effective equivalence}, as well as the description of \(\IZ[\restrvinberg^I]^{c_{\sigma}(\widehat{G}^I)} \cong  \IZ[V_{\widehat{T}\mid d=\rho_{\adj}(q)}^I]^{c_\sigma(\widehat{N}_0)}\) given in \thref{lemma invariants of vinberg}.
\end{proof}

\subsection{The integral Satake isomorphism}

We can now use the Vinberg monoid to construct integral Satake isomorphisms, involving Hecke algebras at very special level.
This generalizes \cite{Zhu:Integral} to the case of ramified groups, and \cite{HainesRostami:Satake} to integral coefficients.

Recall that the (very special) Hecke algebra of \(\Gg\) is defined as the ring \[\Hh_{\Gg}:=C_c(\Gg(\Oo)\backslash G(F)/\Gg(\Oo),\IZ)\] of compactly supported locally constant bi-\(\Gg(\Oo)\)-invariant \(\IZ\)-valued functions on \(G(F)\).
It is equipped with the convolution product, for which we fix a Haar measure on \(G(F)\) such that \(\Gg(\Oo)\) has measure 1.
Recall also that we have bijections of sets \(\Gg(\Oo)\backslash G(F) /\Gg(\Oo) \cong (X^*(\widehat{T}^I)^+)^{\sigma} \cong (X_*(T)_I^+)^\sigma\).
In particular, if \(G=T\) is a torus, the Hecke algebra \(\Hh_{\Tt}\) is isomorphic to the group algebra \(\IZ[X^*(\widehat{T}^I)^{\sigma}]\).
In order to construct an injection \(\Hh_\Gg\to \Hh_{\Tt}\), we define the \emph{Satake transform}
\[\CT^{\cl}\colon \Hh_{\Gg} \to \IZ[X^*(\widehat{T}^I)^{\sigma}] \colon f\mapsto \sum_{\lambda \in X^*(\widehat{T}^I)^{\sigma}} \left( \sum_{u\in U(F)/\Uu(\Oo)} f(\lambda(\varpi)u) e^\lambda\right).\]
Here, \(U\) denotes the unipotent radical of the Borel \(B\) of \(G\) and \(\Uu\) its natural \(\Oo\)-model, although \(\CT^{\cl}\) is independent of the choice of Borel and uniformizer.
This is clearly a well-defined ring homomorphism, and it is injective by the Iwasawa decomposition.

On the other hand, the Hecke algebra arises geometrically via the sheaf-function dictionary as follows.

\begin{lem}\thlabel{trace of Satake}
	The trace of geometric Frobenius induces a surjective ring morphism
	\[\tr\colon K_0(\MATM_{L^+\Gg}(\Gr_{\Gg})^{\anti}) \to \Hh_\Gg.\]
\end{lem}
\begin{proof}
	Recall that the trace of Frobenius on the Tate twist \(\unit(n)\in \DM(\Spec \IF_q,\IZ[\frac{1}{p}])\) is given by \(q^{-n}\) \cite[(1.2.5) (iv)]{Deligne:Weil2}.
	Hence, as in the proof of \thref{trace of representations}, the trace of Frobenius function of any object in \(\MATM_{L^+\Gg}(\Gr_\Gg)^{\anti}\) takes values in \(\IZ\).
	This gives a map \(\tr\colon K_0(\MATM_{L^+\Gg}(\Gr_{\Gg})^{\anti}) \to \Hh_\Gg\), which is easily seen to be a ring morphism, e.g.~as in \cite[Lemma 5.6.1]{Zhu:Introduction}.
	To show surjectivity, note that for dominant \(\mu\in X^*(\widehat{T}^I)^{\sigma}\) we have
	\[\tr(\IC_\mu(\unit)) = \pm 1_{\Gg(\Oo)\mu(\varpi)\Gg(\Oo)} + \sum_{\mu'<\mu} c_{\mu',\mu} 1_{\Gg(\Oo)\mu'(\varpi)\Gg(\Oo)},\]
	where \(1_{-}\) denotes the characteristic function.
	Indeed, this follows from the fact that \(\IC_{\mu}(\unit)\) restricts to a shifted constant sheaf on \(\Gr_{\Gg,\mu}\), and is supported on \(\Gr_{\Gg,\leq \mu}\).
	Since these characteristic functions form a basis of \(H_{\Gg}\), we are finished.
\end{proof}

\begin{thm}\thlabel{Integral Satake isomorphism}
	There is a canonical isomorphism
	\[\IZ[\restrvinberg^I]^{c_{\sigma}(\widehat{G}^I)} \cong \Hh_{\Gg}.\]
	In particular, \(\Hh_{\Gg}\) is commutative.
\end{thm}
\begin{proof}
	Consider the diagram
	\[\begin{tikzcd}
		K_0(\Rep_{\vinberg^I\rtimes \Gamma}(\IQ\text{-}\operatorname{Vect})) \arrow[d, "\cong"'] \arrow[r, "\tr"] & \IZ[\restrvinberg^I]^{c_{\sigma}(\widehat{G}^I)} \arrow[r, "\Res"] & \IZ[V_{\widehat{T}\mid d=\rho_{\adj}(q)}^I]^{c_{\sigma}(\widehat{N}_0)} \arrow[d, "\eqref{description invariants of vinberg}"]\\
		K_0(\MATM_{L^+\Gg}(\Gr_{\Gg})^{\anti}) \arrow[r, "\tr"] & \Hh_\Gg \arrow[r, "\CT^{\cl}"] & \IZ[X^*(\widehat{T}^I)^{\sigma}],
	\end{tikzcd}\]
	which commutes by the motivic Grothendieck-Lefschetz trace formula from \cite[Theorem 3.4.2.25]{Cisinski:Cohomological}.
	Since both maps \(\tr\) are surjective by Lemmas \ref{trace of representations} and \ref{trace of Satake}, \(\Res\) is the isomorphism from \thref{Chevalley restriction}, and \(\CT^{\cl}\) and \eqref{description invariants of vinberg} are injective, there is a unique isomorphism \(\IZ[\restrvinberg^I]^{c_{\sigma}(\widehat{G}^I)} \cong \Hh_{\Gg}\) making the diagram commute.
\end{proof}

We note that Satake isomorphisms with integral coefficients have also appeared in \cite{HenniartVigneras:Satake}.
However, contrary to loc.~cit., the Langlands dual side becomes apparent in the above theorem. 

After inverting and adding square roots of \(q\), we can get more familiar forms of the Satake isomorphism, generalizing \cite{HainesRostami:Satake} from complex to \(\IZ[q^{\pm \frac{1}{2}}]\)-coefficients.

\begin{cor}
	Fix a square root \(q^{\frac{1}{2}}\) of \(q\).
	The isomorphism from \thref{Integral Satake isomorphism} induces isomorphisms
	\[\IZ[q^{\pm \frac{1}{2}}] [X^*(\widehat{T})_I^{\sigma}]^{W_0} \cong \IZ[q^{\pm \frac{1}{2}}][\widehat{G}^I\sigma]^{\widehat{G}^I} \cong \Hh_{\Gg} \otimes_{\IZ} \IZ[q^{\pm \frac{1}{2}}].\]
\end{cor}
\begin{proof}
	The first isomorphism is obtained by taking fibers of the twisted Chevalley restriction isomorphism as in \thref{Chevalley restriction}.
	For the second isomorphism, recall that the preimage \(\widetilde{d}_{\rho_{\adj}}^{-1}(\IG_m\times \Gamma)\) under \(\widetilde{d}_{\rho_{\adj}} \colon \vinberg^I \rtimes \Gamma\to \IA^1\times \Gamma\) is exactly \(\widehat{G}^I\rtimes (\IG_m\times \Gamma)\).
	Hence \thref{Integral Satake isomorphism} induces an isomorphism
	\[\IZ[q^{-1}][(\widehat{G}^I\rtimes (\IG_m\times \Gamma))_{\mid \widetilde{d}_{\rho_{\adj}}=(q,\sigma)}]^{\widehat{G}^I} \cong \Hh_{\Gg} \otimes_{\IZ} \IZ[q^{-1}].\]
	On the other hand, after fixing \(q^{\frac{1}{2}}\), we get \[(\widehat{G}^I\rtimes (\IG_m\times \Gamma))_{\widetilde{d}_{\rho_{\adj}}=(q,\sigma)} \cong \widehat{G}^I\sigma\colon (g,(q,\sigma))\mapsto g 2\rho(q^{-\frac{1}{2}})\sigma,\]
	giving the required isomorphism.
\end{proof}

On the other hand, the fiber of \(d\colon V_{\widehat{G}}\to \widehat{T}_{\adj}^+\) over the origin is the \emph{asymptotic cone} \(\As_{\widehat{G}} := \Spec \gr \IZ[\widehat{G}]\) of \(\widehat{G}\) \cite[§3.2]{XiaoZhu:Vector}.
Hence we get \(\IF_p[\restrvinberg] = \IF_p[\As_{\widehat{G}}]\), and taking invariants also \(\IF_p[\restrvinberg^I] = \IF_p[\As_{\widehat{G}}^I]\).
As in \cite[Corollary 7]{Zhu:Integral}, this gives a mod \(p\) Satake isomorphism involving the Langlands dual side.

\begin{cor}
	The isomorphism from \thref{Integral Satake isomorphism} induces an isomorphism
	\[\IF_p[\As_{\widehat{G}}^I]^{c_\sigma(\widehat{G}^I)} \cong \Hh_{\Gg}\otimes_{\IZ} \IF_p.\]
\end{cor}

Similarly to \cite[§1.4]{Zhu:Integral}, this recovers the mod \(p\) Satake isomorphism from \cite{HenniartVigneras:Satake}.
Namely, there are isomorphisms
\[\IF_p[\As_{\widehat{G}}^I]^{c_{\sigma}(\widehat{G}^I)} \cong \bigoplus_{\mu\in X^*(\widehat{T}^I)^+} (S_{-w_0(\mu)} \otimes S_{\mu})^{c_{\sigma}(\widehat{G}^I)} \cong \IF_p[X_*(T)_I^{\sigma,-}],\]
where this time \(S_{\mu}\) denotes an \(\IF_p\)-linear Schur module for \(\widehat{G}^I\), and \(w_0\) is the longest element in \(W^I\).
Moreover, at least mod \(p\), our Satake transform coincides with the one from \cite[Proposition 2.7]{HenniartVigneras:Satake}.

\subsection{Generic Hecke algebras}

In this final section, we explain how to define generic Hecke algebras at very special level, and how they fit into generic Satake and Bernstein isomorphisms.
Such a generic Hecke algebra should be a \(\IZ[\qq]\)-algebra, which recovers usual Hecke algebras by specializing the variable \(\qq\) to some prime power \(q\).
Moreover, setting \(p=\qq=0\) for some prime \(p\) should recover mod \(p\) Hecke algebras.

\begin{rmk}\thlabel{Rmk residually split}
	At Iwahori-level, generic Hecke algebras usually depend on multiple parameters, indexed by the relative Iwahori-Weyl group.
	As in \cite[§2]{Vigneras:ProI}, the values of these parameters are supposed to represent the number of rational points of Iwahori-orbits in the full affine flag variety.
	This is moreover multiplicative for reduced decompositions.
	For residually split groups and simple reflections, these Iwahori-orbits always have \(q\) rational points, so we do not lose too much by using only a single parameter, meant to represent all simple reflections at the same time.
\end{rmk}

For this reason, we restrict ourselves to the case where \(G\) is residually split, and let \(k=k'\).
Note that for split groups, generic spherical Hecke algebras with a single parameter have already appeared in \cite{PepinSchmidt:Generic, CassvdHScholbach:Geometric}.
Such groups are already defined over \(\Spec \IZ\), and hence over any finite field.
The generic spherical Hecke algebras then interpolates the different Hecke algebras arising by varying the finite field.

Our goal here is to construct a suitable candidate for very special generic Hecke algebras.
This should open a path to study mod \(p\) Hecke algebras at very special level, by studying regular Hecke algebras and varying the parameter.
Since \(G\) is residually split, we can use the representation ring of the Vinberg monoid to describe the very special Hecke algebra.
For a group or monoid \(M\),  we will denote by \(R(M)=K_0(\Rep M)\) its representation ring.

\begin{prop}
	\thref{Anti-effective equivalence} induces an isomorphism
	\[\Hh_{\Gg} \cong R(\vinberg^I)/([\widetilde{d}_{\rho_{\adj}}]-q).\]
\end{prop}
\begin{proof}
	Let \(K\) denote the Grothendieck ring of the category of compact objects in \(\MTM_{L^+\Gg}(\Gr_{\Gg})^{\anti}\).
	First, we construct a surjective morphism \(K\to \Hh_{\Gg}\), as in \cite[(6.16)]{RicharzScholbach:Satake}.
	By taking the trace of geometric Frobenius as in \cite{Cisinski:Cohomological}, we can associate to each compact object \(M\in \MTM_{L^+\Gg}(\Gr_{\Gg})^{\anti}\) a function \(f_M\) on \(\Gg(\Oo)\backslash G(F) /\Gg(\Oo)\).
	Now, *-pulling back \(M\) to a Schubert cell \(\Gr_{\Gg,\mu}\) gives a compact motive, which is a finite colimit of shifted Tate twists.
	Since the trace of Frobenius of the Tate twist \(\unit(1)\) is \(q^{-1}\), we see that \(f_M\) takes values in \(\IZ[q^{-1}]\), and even in \(\IZ\) as \(M\) was assumed anti-effective.
	In other words, \(f_M\in \Hh_{\Gg}\), which gives the desired ring morphism.
	
	Since \(k=k'\), there is a bijection \(\Gg(\Oo)\backslash G(F) /\Gg(\Oo) \cong X_*(T)_I^+\), so \(\Hh_{\Gg}\) is a free abelian group with basis the characteristic functions \(\{1_\mu\mid \mu\in X_*(T)_I^+\}\)
	Hence, surjectivity follows from the observation that for any \(\mu\in X_*(T)_I\), the pullback \(\iota_\mu^* \IC_\mu(\unit)\) agrees with \(\unit\) up to a shift, so that their trace of Frobenii agree up to multiplication by \(\pm 1\).
	
	Now, \thref{Anti-effective equivalence} induces an isomorphism \(K\cong R(\vinberg^I)\), so that we get a surjection \[\psi\colon R(\vinberg^I)\to \Hh_{\Gg}.\]
	Both \(\IC_0(\unit(-1))\) and \(\bigoplus_q \IC_0(\unit)\) get mapped to \(q\cdot 1_0\in \Hh_{\Gg}\), so that \([\widetilde{d}_{\rho_{\adj}}]-q\) lies in the kernel of \(\phi\).
	On the other hand, 
	\[[\widetilde{d}_{\rho_{\adj}}]^n-q^n = ([\widetilde{d}_{\rho_{\adj}}]-q)\cdot \left(\sum_{i=1}^{n-1} [\widetilde{d}_{\rho_{\adj}}]^{i} \cdot q^{n-1-i}\right)\]
	lies in the ideal generated by \([\widetilde{d}_{\rho_{\adj}}]-q\) for any \(n\geq 0\).
	Since all the \(\IC_\mu(\unit)\) give linearly independent elements in \(K\), we see that the kernel of \(\psi\) is generated by \([\widetilde{d}_{\rho_{\adj}}]-q\), which concludes the proof.
\end{proof}

\begin{rmk}\thlabel{raison d'etre of generic algebra}
	The map \(\IZ[\qq]\to R(\vinberg^I)\colon \qq\mapsto [\widetilde{d}_{\rho_{\adj}}]\) exhibits \(R(\vinberg^I)\) as a \(\IZ[\qq]\)-algebra.
	Moreover, by the proposition above, specializing \(\qq\mapsto q\) recovers an honest very special Hecke algebra.
	Now, by the classification of reductive groups as in \cite[§4]{Tits:Reductive}, for any nonarchimedean local field \(F\) we can find a residually split reductive group \(G\) such that its dual group and inertia action give rise to the same Vinberg monoid \(\vinberg^I\) as above.
	Hence, the proposition above actually implies that for \emph{every} prime power \(q\), the base change \(R(\vinberg^I)\otimes_{\IZ[\qq],\qq\mapsto q} \IZ\) is a very special Hecke algebra.
\end{rmk}

Thus, it makes sense to define the generic Hecke algebra as the representation ring of \(\vinberg^I\).
However, using \thref{Anti-effective equivalence}, we give preference to the following definition, to make the appearance of \(\Gg\) itself more prominant.

\begin{dfn}
	The \emph{generic (spherical) Hecke algebra} associated to \(\Gg\) is
	\[\Hh_{\Gg}(\qq):=K_0(\MTM_{L^+\Gg}(\Gr_{\Gg})^{\anti}),\]
	where the \(\IZ[\qq]\)-algebra structure is given by the negative Tate twist \(\IC_0(\unit(-1))\).
\end{dfn}

\thref{Anti-effective equivalence} then gives an isomorphism \(\Hh_{\Gg}(\qq)\cong R(\vinberg^I)\), which can be viewed as a \emph{generic Satake isomorphism}.
In particular, the definition above agrees with \cite[Definition 6.34]{CassvdHScholbach:Geometric} for split groups.
Now, let \(\Ii\subseteq \Gg\) be an Iwahori model of \(G\).
\begin{dfn}
	The generic Iwahori-Hecke algebra \(\Hh_{\Ii}(\qq)\) is the free \(\IZ[\qq]\)-module with basis \((T_w)_{w\in \tilde{W}}\) indexed by the Iwahori-Weyl group of \(G\), and the unique \(\IZ[\qq]\)-algebra structure given by
	\begin{itemize}
		\item (braid relations) \(T_w T_{w'} = T_{ww'}\) for \(w,w'\in \tilde{W}\) satisfying \(l(ww') = l(w)+l(w')\), and
		\item (quadratic relations) \(T_s^2 = \qq + (\qq-1)T_s\) for any simple reflection \(s\).
	\end{itemize}
\end{dfn}
By \cite{Vigneras:ProI}, specializing \(\qq\) to \(q\) recovers the usual Iwahori-Hecke algebra: \(\Hh_{\Ii}(\qq) \otimes_{\IZ[\qq],\qq\mapsto q} \IZ \cong \Hh_{\Ii}\).
Again, this definition is only sensible when \(G\) is residually split.
In general, one has to consider multiple parameters as in \cite{Vigneras:ProI}.

\begin{rmk}
	As in \cite[Proposition 6.3]{CassvdHScholbach:Central}, one can show \(\Hh_{\Ii}(\qq)\cong K_0(\DTM_{L^+\Ii}(\Fl_{\Ii})^{\anti})\).
	More generally, for any parahoric model \(\Gg'\) of \(G\), the \(\IZ[\qq]\)-algebra \(K_0(\MTM_{L^+\Gg'}(\Fl_{\Gg'})^{\anti})\) is a suitable candidate for a generic parahoric Hecke algebra. 
	However, if \(\Gg\) is neither very special nor an Iwahori, it is not clear that specializing \(\qq\) to an arbitrary prime power yields a classical parahoric Hecke algebra.
\end{rmk}

Now, consider the \emph{generic group algebra} \(\IZ[\qq][X_*(T)_I] \cong \IZ[\qq][X^*(\widehat{T}^I)]\).
We define a \emph{generic Satake transform}
\[\CT^{\cl}(\qq)\colon \Hh_{\Gg}(\qq) \cong R(\vinberg^I) \to R(\widehat{T}^I\times \IA^1) \cong \IZ[\qq][X^*(\widehat{T}^I)],\]
where the middle morphism is given by the restriction of representations as in \thref{trace of representations}.

Next, we want to define a morphism from the generic group algebra to the generic Iwahori-Hecke algebra, at least after inverting \(\qq\).
Recall the generic Bernstein elements \(E(\nu)\) from e.g.~\cite[Corollary 5.47]{Vigneras:ProI}.
Let \(\mu\in X_*(T)_I^+\) be such that \(\mu-\nu\in X_*(T)_I^+\).
Then we define a morphism \(\IZ[\qq][X_*(T)_I] \to \Hh_{\Ii}(\qq)[\qq^{-1}]\) of \(\IZ[\qq]\)-algebras by sending \(\nu\in X_*(T)_I\) to \(\qq^{-\langle 2\rho,\mu\rangle}\cdot E(\mu)\cdot E(\nu-\mu)\) (note that our choice of very special parahoric \(\Gg\) already leads to a preferred choice of orientation, in the terminology of loc.~cit.). 
By the product formula from \cite[Corollary 5.47]{Vigneras:ProI}, this is a ring morphism, and independent of the choice of \(\mu\).
Moreover, by \thref{fiber functor of standard}, the constant term functor adds sufficiently many negative twists, so that the composition
\[\Hh_{\Gg}(\qq) \to \IZ[\qq][X_*(T)_I] \to \Hh_\Ii(\qq)[\qq^{-1}]\]
has its image contained in \(\Hh_{\Ii}(\qq) \subseteq \Hh_{\Ii}(\qq)[\qq^{-1}]\). 
The following can then be seen as a \emph{generic Bernstein isomorphism} for residually split groups.
\begin{thm}\thlabel{generic Bernstein}
	The map \(\Hh_{\Gg}(\qq)\to \Hh_{\Ii}(\qq)\) defined above induces an isomorphism of the generic spherical Hecke algebra with the center of the generic Iwahori-Hecke algebra.
\end{thm}
\begin{proof}
	First, consider the diagram
	\[\begin{tikzcd}
		\Hh_{\Gg}(\qq) \arrow[d, "\qq \mapsto q"'] \arrow[r] & \Hh_{\Ii}(\qq) \arrow[d, "\qq\mapsto q"] \\
		\Hh_{\Gg} \arrow[r] & \Hh_{\Ii},
	\end{tikzcd}\]
	where the bottom map is the unique map making the diagram commute.
	Moreover, by \cite[Theorem 1.2]{Vigneras:ProII}, this map realizes \(\Hh_{\Gg}\) as the center of \(\Hh_{\Ii}\).
	
	Now, for any prime power \(q'\), let us write \(\Hh_{\Gg}(q')\) for the specialization \(\Hh_{\Gg}(\qq) \otimes_{\IZ[\qq],\qq\mapsto q'} \IZ\).
	Then we have an injective map \(\Hh_{\Gg}(\qq)\to \prod_{q'} \Hh_{\Gg}(q')\).
	The same holds for Iwahori-Hecke algebras.
	This gives a cartesian diagram
	\[\begin{tikzcd}
		\Hh_{\Gg}(\qq) \arrow[d, "\qq\mapsto q'"'] \arrow[r] & \Hh_{\Ii}(\qq) \arrow[d, "\qq\mapsto q'"]\\
		\prod_{q'} \Hh_{\Gg}(q')\arrow[r] &  \prod_{q'} \Hh_{\Ii}(q').
	\end{tikzcd}\]
	In particular, all vertical maps are injective, so the fact that each \(\Hh_{\Gg}(q')\) is the center of \(\Hh_{\Ii}(q')\) (by the previous paragraph and \thref{raison d'etre of generic algebra}) implies that \(\Hh_{\Gg}(\qq)\to \Hh_{\Ii}(\qq)\) is injective and has central image.
	On the other hand, each \(\Hh_{\Gg}(\qq)\to \Hh_{\Gg}(q')\) and \(\Hh_{\Ii}(\qq)\to \Hh_{\Ii}(q')\) is surjective.
	Thus cartesianness implies that \(\Hh_{\Gg}(\qq)\) surjects onto the center of \(\Hh_{\Ii}(\qq)\), which concludes the proof.
\end{proof}

For split groups in equal characteristic, a geometric interpretation of this isomorphism can be found in \cite{CassvdHScholbach:Central}, using motivic nearby cycles.
Finally, we mention that the generic Satake and Bernstein isomorphisms should open the door to generalizing \cite{PepinSchmidt:Generic} to ramified (but still residually split) reductive groups.
	
\bibliographystyle{alphaurl}
\bibliography{bib}

@misc{stacks-project,
	shorthand    = {Stacks},
	author       = {The {Stacks Project Authors}},
	title        = {\textit{Stacks Project}},
	howpublished = {\url{https://stacks.math.columbia.edu}},
}

@article {Zhu:Ramified,
	AUTHOR = {Zhu, Xinwen},
	TITLE = {The geometric {S}atake correspondence for ramified groups},
	NOTE = {With an appendix by T. Richarz and X. Zhu},
	JOURNAL = {Ann. Sci. \'{E}c. Norm. Sup\'{e}r. (4)},
	FJOURNAL = {Annales Scientifiques de l'\'{E}cole Normale Sup\'{e}rieure. Quatri\`eme
	S\'{e}rie},
	VOLUME = {48},
	YEAR = {2015},
	NUMBER = {2},
	PAGES = {409--451},
	ISSN = {0012-9593},
	MRCLASS = {20G25 (14G35)},
	MRNUMBER = {3346175},
	MRREVIEWER = {David-Alexandre Guiraud},
	DOI = {10.24033/asens.2248},
	URL = {https://doi.org/10.24033/asens.2248},
}

@article{Zhu:Affine,
	AUTHOR = {Zhu, Xinwen},
	TITLE = {Affine {G}rassmannians and the geometric {S}atake in mixed
	characteristic},
	JOURNAL = {Ann. of Math. (2)},
	FJOURNAL = {Annals of Mathematics. Second Series},
	VOLUME = {185},
	YEAR = {2017},
	NUMBER = {2},
	PAGES = {403--492},
	ISSN = {0003-486X},
	MRCLASS = {14D24 (14L35 14M15 20G25)},
	MRNUMBER = {3612002},
	MRREVIEWER = {Rolf Berndt},
	DOI = {10.4007/annals.2017.185.2.2},
	URL = {https://doi.org/10.4007/annals.2017.185.2.2},
}

@incollection {XiaoZhu:Vector,
	AUTHOR = {Xiao, Liang and Zhu, Xinwen},
	TITLE = {On vector-valued twisted conjugation invariant functions on a
	group},
	BOOKTITLE = {Representations of reductive groups},
	SERIES = {Proc. Sympos. Pure Math.},
	VOLUME = {101},
	PAGES = {361--425},
	NOTE = {With an appendix by Stephen Donkin},
	PUBLISHER = {Amer. Math. Soc., Providence, RI},
	YEAR = {2019},
	ISBN = {978-1-4704-4284-2},
	MRCLASS = {20G05 (20G10)},
	MRNUMBER = {3930024},
	MRREVIEWER = {Stuart\ Martin},
	DOI = {10.1090/pspum/101/14},
	URL = {https://doi.org/10.1090/pspum/101/14},
}

@incollection {Zhu:Integral,
	AUTHOR = {Zhu, Xinwen},
	TITLE = {A note on integral {S}atake isomorphisms},
	BOOKTITLE = {Arithmetic geometry},
	SERIES = {Tata Inst. Fundam. Res. Stud. Math.},
	VOLUME = {41},
	PAGES = {469--489},
	PUBLISHER = {Tata Inst. Fund. Res., Mumbai},
	YEAR = {2024},
	ISBN = {978-81-957829-7-0},
	MRCLASS = {20C08 (14D24 22E50 22E57)},
	MRNUMBER = {4812712},
}

@incollection {Zhu:Introduction,
	AUTHOR = {Zhu, Xinwen},
	TITLE = {An introduction to affine {G}rassmannians and the geometric
	{S}atake equivalence},
	BOOKTITLE = {Geometry of moduli spaces and representation theory},
	SERIES = {IAS/Park City Math. Ser.},
	VOLUME = {24},
	PAGES = {59--154},
	PUBLISHER = {Amer. Math. Soc., Providence, RI},
	YEAR = {2017},
	ISBN = {978-1-4704-3574-5},
	MRCLASS = {14M15 (14D24 20F65 22E57)},
	MRNUMBER = {3752460},
	MRREVIEWER = {Felipe\ Zald\'{\i}var},
}

@article {Scholze:Perfectoid,
	AUTHOR = {Scholze, Peter},
	TITLE = {Perfectoid spaces},
	JOURNAL = {Publ. Math. Inst. Hautes \'{E}tudes Sci.},
	FJOURNAL = {Publications Math\'{e}matiques. Institut de Hautes \'{E}tudes
	Scientifiques},
	VOLUME = {116},
	YEAR = {2012},
	PAGES = {245--313},
	ISSN = {0073-8301,1618-1913},
	MRCLASS = {14G99},
	MRNUMBER = {3090258},
	MRREVIEWER = {Jean-Marc\ Fontaine},
	DOI = {10.1007/s10240-012-0042-x},
	URL = {https://doi.org/10.1007/s10240-012-0042-x},
}

@article {BhattScholze:Proetale,
	AUTHOR = {Bhatt, Bhargav and Scholze, Peter},
	TITLE = {The pro-\'{e}tale topology for schemes},
	JOURNAL = {Ast\'{e}risque},
	FJOURNAL = {Ast\'{e}risque},
	NUMBER = {369},
	YEAR = {2015},
	PAGES = {99--201},
	ISSN = {0303-1179,2492-5926},
	ISBN = {978-2-85629-805-3},
	MRCLASS = {14F05 (14F20 14F35 14H30 18B25)},
	MRNUMBER = {3379634},
	MRREVIEWER = {Pieter\ Belmans},
}

@article{BhattScholze:Projectivity,
	AUTHOR = {Bhatt, Bhargav and Scholze, Peter},
	TITLE = {Projectivity of the {W}itt vector affine {G}rassmannian},
	JOURNAL = {Invent. Math.},
	FJOURNAL = {Inventiones Mathematicae},
	VOLUME = {209},
	YEAR = {2017},
	NUMBER = {2},
	PAGES = {329--423},
	ISSN = {0020-9910},
	MRCLASS = {14F05 (14M15 19G12)},
	MRNUMBER = {3674218},
	MRREVIEWER = {Marc-Hubert Nicole},
	DOI = {10.1007/s00222-016-0710-4},
	URL = {https://doi.org/10.1007/s00222-016-0710-4},
}

@book {ScholzeWeinstein:Berkeley,
	AUTHOR = {Scholze, Peter and Weinstein, Jared},
	TITLE = {Berkeley lectures on {$p$}-adic geometry},
	SERIES = {Annals of Mathematics Studies},
	VOLUME = {207},
	PUBLISHER = {Princeton University Press, Princeton, NJ},
	YEAR = {2020},
	PAGES = {x+250},
	ISBN = {978-0-691-20209-9; 978-0-691-20208-2; 978-0-691-20215-0},
	MRCLASS = {14G45 (14A15 14G22 14G35 14M15)},
	MRNUMBER = {4446467},
}

@article {HansenScholze:Relative,
	AUTHOR = {Hansen, David and Scholze, Peter},
	TITLE = {Relative perversity},
	JOURNAL = {Comm. Amer. Math. Soc.},
	FJOURNAL = {Communications of the American Mathematical Society},
	VOLUME = {3},
	YEAR = {2023},
	PAGES = {631--668},
	ISSN = {2692-3688},
	MRCLASS = {14F43 (14F08 14F20)},
	MRNUMBER = {4630128},
	DOI = {10.1090/cams/21},
	URL = {https://doi.org/10.1090/cams/21},
}

@article{CassvdHScholbach:Geometric,
	AUTHOR = {Cass, Robert and van den Hove, Thibaud and Scholbach, Jakob},
	TITLE = {The geometric {S}atake equivalence for integral motives},
	JOURNAL = {Compos. Math.},
	FJOURNAL = {Compositio Mathematica},
	VOLUME = {161},
	YEAR = {2025},
	NUMBER = {11},
	PAGES = {2755--2851},
	ISSN = {0010-437X,1570-5846},
	MRCLASS = {20G05 (14F42 14M15 19E15 20C08)},
	MRNUMBER = {5003140},
	DOI = {10.1112/S0010437X25102480},
	URL = {https://doi.org/10.1112/S0010437X25102480},
}

@article {Richarz:Schubert,
	AUTHOR = {Richarz, Timo},
	TITLE = {Schubert varieties in twisted affine flag varieties and local
	models},
	JOURNAL = {J. Algebra},
	FJOURNAL = {Journal of Algebra},
	VOLUME = {375},
	YEAR = {2013},
	PAGES = {121--147},
	ISSN = {0021-8693},
	MRCLASS = {14M15},
	MRNUMBER = {2998951},
	MRREVIEWER = {Leonardo Constantin Mihalcea},
	DOI = {10.1016/j.jalgebra.2012.11.013},
	URL = {https://doi.org/10.1016/j.jalgebra.2012.11.013},
}

@article {Richarz:Affine,
	AUTHOR = {Richarz, Timo},
	TITLE = {Affine {G}rassmannians and geometric {S}atake equivalences},
	JOURNAL = {Int. Math. Res. Not. IMRN},
	FJOURNAL = {International Mathematics Research Notices. IMRN},
	YEAR = {2016},
	NUMBER = {12},
	PAGES = {3717--3767},
	ISSN = {1073-7928},
	MRCLASS = {14L17 (14F05 14M15 20G25 22E67)},
	MRNUMBER = {3544618},
	MRREVIEWER = {Guy Rousseau},
	DOI = {10.1093/imrn/rnv226},
	URL = {https://doi.org/10.1093/imrn/rnv226},
}

@article {Richarz:Spaces,
	AUTHOR = {Richarz, Timo},
	TITLE = {Spaces with {$\mathbb{G}_m$}-action, hyperbolic localization and
	nearby cycles},
	JOURNAL = {J. Algebraic Geom.},
	FJOURNAL = {Journal of Algebraic Geometry},
	VOLUME = {28},
	YEAR = {2019},
	NUMBER = {2},
	PAGES = {251--289},
	ISSN = {1056-3911},
	MRCLASS = {14L30 (14A20)},
	MRNUMBER = {3912059},
	MRREVIEWER = {Yiqiang Li},
	DOI = {10.1090/jag/710},
	URL = {https://doi.org/10.1090/jag/710},
}

@article {HainesRicharz:TestParahoric,
	AUTHOR = {Haines, Thomas J. and Richarz, Timo},
	TITLE = {The test function conjecture for parahoric local models},
	JOURNAL = {J. Amer. Math. Soc.},
	FJOURNAL = {Journal of the American Mathematical Society},
	VOLUME = {34},
	YEAR = {2021},
	NUMBER = {1},
	PAGES = {135--218},
	ISSN = {0894-0347},
	MRCLASS = {14G35 (14M15 20G05)},
	MRNUMBER = {4188816},
	MRREVIEWER = {Alan Koch},
	DOI = {10.1090/jams/955},
	URL = {https://doi.org/10.1090/jams/955},
}

@article {HainesRicharz:Smoothness,
	AUTHOR = {Haines, Thomas J. and Richarz, Timo},
	TITLE = {Smoothness of {S}chubert varieties in twisted affine
	{G}rassmannians},
	JOURNAL = {Duke Math. J.},
	FJOURNAL = {Duke Mathematical Journal},
	VOLUME = {169},
	YEAR = {2020},
	NUMBER = {17},
	PAGES = {3223--3260},
	ISSN = {0012-7094},
	MRCLASS = {14M15 (14G35)},
	MRNUMBER = {4173154},
	MRREVIEWER = {Subramaniam Senthamarai Kannan},
	DOI = {10.1215/00127094-2020-0025},
	URL = {https://doi.org/10.1215/00127094-2020-0025},
}

@article {HainesRicharz:Normality,
	AUTHOR = {Haines, Thomas J. and Richarz, Timo},
	TITLE = {Normality and {C}ohen-{M}acaulayness of parahoric local
	models},
	JOURNAL = {J. Eur. Math. Soc. (JEMS)},
	FJOURNAL = {Journal of the European Mathematical Society (JEMS)},
	VOLUME = {25},
	YEAR = {2023},
	NUMBER = {2},
	PAGES = {703--729},
	ISSN = {1435-9855,1435-9863},
	MRCLASS = {14G35 (11G18)},
	MRNUMBER = {4556794},
	DOI = {10.4171/jems/1192},
	URL = {https://doi.org/10.4171/jems/1192},
}

@article{RicharzScholbach:Intersection,
	AUTHOR = {Richarz, Timo and Scholbach, Jakob},
	TITLE = {The intersection motive of the moduli stack of shtukas},
	JOURNAL = {Forum Math. Sigma},
	FJOURNAL = {Forum of Mathematics. Sigma},
	VOLUME = {8},
	YEAR = {2020},
	PAGES = {Paper No. e8, 99},
	MRCLASS = {20G05 (14D23 14F42 19E15)},
	MRNUMBER = {4061978},
	MRREVIEWER = {Ilya Karzhemanov},
	DOI = {10.1017/fms.2019.32},
	URL = {https://doi.org/10.1017/fms.2019.32},
}

@article{RicharzScholbach:Satake,
	AUTHOR = {Richarz, Timo and Scholbach, Jakob},
	TITLE = {The motivic {S}atake equivalence},
	JOURNAL = {Math. Ann.},
	FJOURNAL = {Mathematische Annalen},
	VOLUME = {380},
	YEAR = {2021},
	NUMBER = {3-4},
	PAGES = {1595--1653},
	ISSN = {0025-5831},
	MRCLASS = {14L24 (14G35 14M15)},
	MRNUMBER = {4297194},
	MRREVIEWER = {Matthias Wendt},
	DOI = {10.1007/s00208-021-02176-9},
	URL = {https://doi.org/10.1007/s00208-021-02176-9},
}

@article{RicharzScholbach:Witt,
	AUTHOR = {Richarz, Timo and Scholbach, Jakob},
	TITLE = {Tate motives on {W}itt vector affine flag varieties},
	JOURNAL = {Selecta Math. (N.S.)},
	FJOURNAL = {Selecta Mathematica. New Series},
	VOLUME = {27},
	YEAR = {2021},
	NUMBER = {3},
	PAGES = {Paper No. 44, 34},
	ISSN = {1022-1824},
	MRCLASS = {14F42 (14M15 20G05)},
	MRNUMBER = {4269679},
	DOI = {10.1007/s00029-021-00665-y},
	URL = {https://doi.org/10.1007/s00029-021-00665-y},
}

@article{PappasRapoport:Twisted,
	AUTHOR = {Pappas, G. and Rapoport, M.},
	TITLE = {Twisted loop groups and their affine flag varieties},
	NOTE = {With an appendix by T. Haines and Rapoport},
	JOURNAL = {Adv. Math.},
	FJOURNAL = {Advances in Mathematics},
	VOLUME = {219},
	YEAR = {2008},
	NUMBER = {1},
	PAGES = {118--198},
	ISSN = {0001-8708},
	MRCLASS = {22E67 (14M15 17B67 20G25)},
	MRNUMBER = {2435422},
	MRREVIEWER = {Dmitry A. Timash\"{e}v},
	DOI = {10.1016/j.aim.2008.04.006},
	URL = {https://doi.org/10.1016/j.aim.2008.04.006},
}

@article{HainesRapoport:Parahoric,
	title = {Appendix: On parahoric subgroups},
	journal = {Advances in Mathematics},
	volume = {219},
	number = {1},
	pages = {188-198},
	year = {2008},
	issn = {0001-8708},
	doi = {10.1016/j.aim.2008.04.020},
	url = {https://doi.org/10.1016/j.aim.2008.04.020},
	author = {Haines, Thomas J. and Rapoport, M.},
}

@article {Deligne:Weil2,
	AUTHOR = {Deligne, Pierre},
	TITLE = {La conjecture de {W}eil. {II}},
	JOURNAL = {Inst. Hautes \'{E}tudes Sci. Publ. Math.},
	FJOURNAL = {Institut des Hautes \'{E}tudes Scientifiques. Publications
	Math\'{e}matiques},
	NUMBER = {52},
	YEAR = {1980},
	PAGES = {137--252},
	ISSN = {0073-8301,1618-1913},
	MRCLASS = {14G13 (10H10)},
	MRNUMBER = {601520},
	MRREVIEWER = {Spencer\ J.\ Bloch},
	URL = {http://www.numdam.org/item?id=PMIHES_1980__52__137_0},
}

@incollection {BBD:Faisceaux,
	AUTHOR = {Be\u{\i}linson, A. A. and Bernstein, J. and Deligne, P.},
	TITLE = {Faisceaux pervers},
	BOOKTITLE = {Analysis and topology on singular spaces, {I} ({L}uminy,
	1981)},
	SERIES = {Ast\'{e}risque},
	VOLUME = {100},
	PAGES = {5--171},
	PUBLISHER = {Soc. Math. France, Paris},
	YEAR = {1982},
	MRCLASS = {32C38},
	MRNUMBER = {751966},
	MRREVIEWER = {Zoghman\ Mebkhout},
}

@article {CisinskiDeglise:Etale,
	AUTHOR = {Cisinski, Denis-Charles and D\'{e}glise, Fr\'{e}d\'{e}ric},
	TITLE = {\'{E}tale motives},
	JOURNAL = {Compos. Math.},
	FJOURNAL = {Compositio Mathematica},
	VOLUME = {152},
	YEAR = {2016},
	NUMBER = {3},
	PAGES = {556--666},
	ISSN = {0010-437X},
	MRCLASS = {14F20 (14F42)},
	MRNUMBER = {3477640},
	MRREVIEWER = {Matthias Wendt}
}

@book {CisinskiDeglise:Triangulated,
	AUTHOR = {Cisinski, Denis-Charles and D\'{e}glise, Fr\'{e}d\'{e}ric},
	TITLE = {Triangulated categories of mixed motives},
	SERIES = {Springer Monographs in Mathematics},
	PUBLISHER = {Springer, Cham},
	YEAR = {2019},
	PAGES = {xlii+406},
	ISBN = {978-3-030-33241-9; 978-3-030-33242-6},
	MRCLASS = {14F42 (14C15 14C35 18G80 19D55)},
	MRNUMBER = {3971240},
	MRREVIEWER = {Igor\ A.\ Rapinchuk},
	DOI = {10.1007/978-3-030-33242-6},
	URL = {https://doi.org/10.1007/978-3-030-33242-6},
}

@article {Yu:Integral,
	AUTHOR = {Yu, Jize},
	TITLE = {The integral geometric {S}atake equivalence in mixed
	characteristic},
	JOURNAL = {Represent. Theory},
	FJOURNAL = {Representation Theory. An Electronic Journal of the American
	Mathematical Society},
	VOLUME = {26},
	YEAR = {2022},
	PAGES = {874--905},
	MRCLASS = {22E57 (14D24 20G05)},
	MRNUMBER = {4470196},
	DOI = {10.1090/ert/610},
	URL = {https://doi.org/10.1090/ert/610},
}

@article {HainesRostami:Satake,
	AUTHOR = {Haines, Thomas J. and Rostami, Sean},
	TITLE = {The {S}atake isomorphism for special maximal parahoric {H}ecke
	algebras},
	JOURNAL = {Represent. Theory},
	FJOURNAL = {Representation Theory. An Electronic Journal of the American
	Mathematical Society},
	VOLUME = {14},
	YEAR = {2010},
	PAGES = {264--284},
	MRCLASS = {20G25 (11E95 11G18 14G35 22E50)},
	MRNUMBER = {2602034},
	MRREVIEWER = {Rainer Schulze-Pillot},
	DOI = {10.1090/S1088-4165-10-00370-5},
	URL = {https://doi.org/10.1090/S1088-4165-10-00370-5},
}

@article {EberhardtScholbach:Integral,
	AUTHOR = {Eberhardt, Jens Niklas and Scholbach, Jakob},
	TITLE = {Integral motivic sheaves and geometric representation theory},
	JOURNAL = {Adv. Math.},
	FJOURNAL = {Advances in Mathematics},
	VOLUME = {412},
	YEAR = {2023},
	PAGES = {Paper No. 108811, 42},
	ISSN = {0001-8708},
	MRCLASS = {14C15 (22E47 32S60)},
	MRNUMBER = {4521696},
	DOI = {10.1016/j.aim.2022.108811},
	URL = {https://doi.org/10.1016/j.aim.2022.108811},
}

@article {MirkovicVilonen:Geometric,
	AUTHOR = {Mirkovi\'{c}, I. and Vilonen, K.},
	TITLE = {Geometric {L}anglands duality and representations of algebraic
	groups over commutative rings},
	JOURNAL = {Ann. of Math. (2)},
	FJOURNAL = {Annals of Mathematics. Second Series},
	VOLUME = {166},
	YEAR = {2007},
	NUMBER = {1},
	PAGES = {95--143},
	ISSN = {0003-486X},
	MRCLASS = {22E55 (11R39 20G05)},
	MRNUMBER = {2342692},
	MRREVIEWER = {Peter Fiebig},
	DOI = {10.4007/annals.2007.166.95},
	URL = {https://doi.org/10.4007/annals.2007.166.95},
}

@article {GaussentLittelmann:LS,
	AUTHOR = {Gaussent, S. and Littelmann, P.},
	TITLE = {L{S} galleries, the path model, and {MV} cycles},
	JOURNAL = {Duke Math. J.},
	FJOURNAL = {Duke Mathematical Journal},
	VOLUME = {127},
	YEAR = {2005},
	NUMBER = {1},
	PAGES = {35--88},
	ISSN = {0012-7094},
	MRCLASS = {20G05 (14M15 17B10 22E46 51E24)},
	MRNUMBER = {2126496},
	MRREVIEWER = {Guy Rousseau},
	DOI = {10.1215/S0012-7094-04-12712-5},
	URL = {https://doi.org/10.1215/S0012-7094-04-12712-5},
}

@article {JinYang:Kunneth,
	AUTHOR = {Jin, Fangzhou and Yang, Enlin},
	TITLE = {K\"{u}nneth formulas for motives and additivity of traces},
	JOURNAL = {Adv. Math.},
	FJOURNAL = {Advances in Mathematics},
	VOLUME = {376},
	YEAR = {2021},
	PAGES = {Paper No. 107446, 83},
	ISSN = {0001-8708},
	MRCLASS = {14F42 (19E15)},
	MRNUMBER = {4178918},
	MRREVIEWER = {Anand Sawant},
	DOI = {10.1016/j.aim.2020.107446},
	URL = {https://doi.org/10.1016/j.aim.2020.107446},
}

@book {Tits:Buildings,
	AUTHOR = {Tits, Jacques},
	TITLE = {Buildings of spherical type and finite {BN}-pairs},
	SERIES = {Lecture Notes in Mathematics, Vol. 386},
	PUBLISHER = {Springer-Verlag, Berlin-New York},
	YEAR = {1974},
	PAGES = {x+299},
	MRCLASS = {20G15 (50A20)},
	MRNUMBER = {0470099},
	MRREVIEWER = {Bruce Cooperstein},
}

@article {Deodhar:Geometric,
	AUTHOR = {Deodhar, Vinay V.},
	TITLE = {On some geometric aspects of {B}ruhat orderings. {I}. {A}
	finer decomposition of {B}ruhat cells},
	JOURNAL = {Invent. Math.},
	FJOURNAL = {Inventiones Mathematicae},
	VOLUME = {79},
	YEAR = {1985},
	NUMBER = {3},
	PAGES = {499--511},
	ISSN = {0020-9910},
	MRCLASS = {20G15 (20H15)},
	MRNUMBER = {782232},
	MRREVIEWER = {James E. Humphreys},
	DOI = {10.1007/BF01388520},
	URL = {https://doi.org/10.1007/BF01388520},
}

@book {Brown:Buildings,
	AUTHOR = {Brown, Kenneth S.},
	TITLE = {Buildings},
	PUBLISHER = {Springer-Verlag, New York},
	YEAR = {1989},
	PAGES = {viii+215},
	ISBN = {0-387-96876-8},
	MRCLASS = {20-02 (20E32 20G15 22E99 51B25)},
	MRNUMBER = {969123},
	MRREVIEWER = {W. M. Kantor},
	DOI = {10.1007/978-1-4612-1019-1},
	URL = {https://doi.org/10.1007/978-1-4612-1019-1},
}

@book {KalethaPrasad:BruhatTits,
	AUTHOR = {Kaletha, Tasho and Prasad, Gopal},
	TITLE = {Bruhat-{T}its theory---a new approach},
	SERIES = {New Mathematical Monographs},
	VOLUME = {44},
	PUBLISHER = {Cambridge University Press, Cambridge},
	YEAR = {2023},
	PAGES = {xxx+718},
	ISBN = {978-1-108-83196-3},
	MRCLASS = {20E42 (11F70 20G25 22E50)},
	MRNUMBER = {4520154},
}

@article {Schwer:Roots,
	AUTHOR = {Schwer, Petra},
	TITLE = {Root operators, root groups and retractions},
	JOURNAL = {J. Comb. Algebra},
	FJOURNAL = {Journal of Combinatorial Algebra},
	VOLUME = {2},
	YEAR = {2018},
	NUMBER = {3},
	PAGES = {215--230},
	ISSN = {2415-6302},
	MRCLASS = {20E42 (20G05 51E24)},
	MRNUMBER = {3845717},
	MRREVIEWER = {Guy Rousseau},
	DOI = {10.4171/JCA/2-3-1},
	URL = {https://doi.org/10.4171/JCA/2-3-1},
}

@article {BB:Theorems,
	AUTHOR = {Bialynicki-Birula, A.},
	TITLE = {Some theorems on actions of algebraic groups},
	JOURNAL = {Ann. of Math. (2)},
	FJOURNAL = {Annals of Mathematics. Second Series},
	VOLUME = {98},
	YEAR = {1973},
	PAGES = {480--497},
	ISSN = {0003-486X},
	MRCLASS = {14M15 (14L99)},
	MRNUMBER = {366940},
	MRREVIEWER = {Birger Iversen},
	DOI = {10.2307/1970915},
	URL = {https://doi.org/10.2307/1970915},
}

@article {BB:Properties,
	AUTHOR = {Bialynicki-Birula, A.},
	TITLE = {Some properties of the decompositions of algebraic varieties
	determined by actions of a torus},
	JOURNAL = {Bull. Acad. Polon. Sci. S\'{e}r. Sci. Math. Astronom. Phys.},
	FJOURNAL = {Bulletin de l'Acad\'{e}mie Polonaise des Sciences. S\'{e}rie des
	Sciences Math\'{e}matiques, Astronomiques et Physiques},
	VOLUME = {24},
	YEAR = {1976},
	NUMBER = {9},
	PAGES = {667--674},
	ISSN = {0001-4117},
	MRCLASS = {14M99 (20G20 58F10)},
	MRNUMBER = {453766},
	MRREVIEWER = {Vladimir L. Popov},
}

@book {Jantzen:Representations,
	AUTHOR = {Jantzen, Jens Carsten},
	TITLE = {Representations of algebraic groups},
	SERIES = {Mathematical Surveys and Monographs},
	VOLUME = {107},
	EDITION = {Second},
	PUBLISHER = {American Mathematical Society, Providence, RI},
	YEAR = {2003},
	PAGES = {xiv+576},
	ISBN = {0-8218-3527-0},
	MRCLASS = {20G05 (17B10)},
	MRNUMBER = {2015057},
}

@incollection {BuzzardGee:Conjectural,
	AUTHOR = {Buzzard, Kevin and Gee, Toby},
	TITLE = {The conjectural connections between automorphic
	representations and {G}alois representations},
	BOOKTITLE = {Automorphic forms and {G}alois representations. {V}ol. 1},
	SERIES = {London Math. Soc. Lecture Note Ser.},
	VOLUME = {414},
	PAGES = {135--187},
	PUBLISHER = {Cambridge Univ. Press, Cambridge},
	YEAR = {2014},
	MRCLASS = {11F33},
	MRNUMBER = {3444225},
	MRREVIEWER = {Jeremy A. Rouse},
	DOI = {10.1017/CBO9781107446335.006},
	URL = {https://doi.org/10.1017/CBO9781107446335.006},
}

@article {FrenkelGross:Irregular,
	AUTHOR = {Frenkel, Edward and Gross, Benedict},
	TITLE = {A rigid irregular connection on the projective line},
	JOURNAL = {Ann. of Math. (2)},
	FJOURNAL = {Annals of Mathematics. Second Series},
	VOLUME = {170},
	YEAR = {2009},
	NUMBER = {3},
	PAGES = {1469--1512},
	ISSN = {0003-486X},
	MRCLASS = {14D24 (14F40)},
	MRNUMBER = {2600880},
	DOI = {10.4007/annals.2009.170.1469},
	URL = {https://doi.org/10.4007/annals.2009.170.1469},
}

@article {BruhatTits:Groupes1,
	AUTHOR = {Bruhat, F. and Tits, J.},
	TITLE = {Groupes r\'{e}ductifs sur un corps local. {I}. {D}onn\'{e}es radicielles valu\'{e}es},
	JOURNAL = {Inst. Hautes \'{E}tudes Sci. Publ. Math.},
	FJOURNAL = {Institut des Hautes \'{E}tudes Scientifiques. Publications
	Math\'{e}matiques},
	NUMBER = {41},
	YEAR = {1972},
	PAGES = {5--251},
	ISSN = {0073-8301},
	MRCLASS = {20G25 (22E20)},
	MRNUMBER = {327923},
	MRREVIEWER = {M. Krusemeyer},
	URL = {http://www.numdam.org/item?id=PMIHES_1972__41__5_0},
}

@incollection {Cisinski:Cohomological,
	AUTHOR = {Cisinski, Denis-Charles},
	TITLE = {Cohomological methods in intersection theory},
	BOOKTITLE = {Homotopy theory and arithmetic geometry---motivic and
	{D}iophantine aspects},
	SERIES = {Lecture Notes in Math.},
	VOLUME = {2292},
	PAGES = {49--105},
	PUBLISHER = {Springer, Cham},
	YEAR = {2021},
	ISBN = {978-3-030-78976-3; 978-3-030-78977-0},
	MRCLASS = {14C17 (14F20 14F42)},
	MRNUMBER = {4419310},
	DOI = {10.1007/978-3-030-78977-0_3},
	URL = {https://doi.org/10.1007/978-3-030-78977-0_3},
}

@article {ElmantoKhan:Perfection,
	AUTHOR = {Elmanto, Elden and Khan, Adeel A.},
	TITLE = {Perfection in motivic homotopy theory},
	JOURNAL = {Proc. Lond. Math. Soc. (3)},
	FJOURNAL = {Proceedings of the London Mathematical Society. Third Series},
	VOLUME = {120},
	YEAR = {2020},
	NUMBER = {1},
	PAGES = {28--38},
	ISSN = {0024-6115,1460-244X},
	MRCLASS = {14F42 (19E08 19E15)},
	MRNUMBER = {3999675},
	MRREVIEWER = {Federico\ Binda},
	DOI = {10.1112/plms.12280},
	URL = {https://doi.org/10.1112/plms.12280},
}

@article {SoergelWendt:Perverse,
	AUTHOR = {Soergel, Wolfgang and Wendt, Matthias},
	TITLE = {Perverse motives and graded derived category {$\mathcal{O}$}},
	JOURNAL = {J. Inst. Math. Jussieu},
	FJOURNAL = {Journal of the Institute of Mathematics of Jussieu. JIMJ.
	Journal de l'Institut de Math\'{e}matiques de Jussieu},
	VOLUME = {17},
	YEAR = {2018},
	NUMBER = {2},
	PAGES = {347--395},
	ISSN = {1474-7480,1475-3030},
	MRCLASS = {14C15 (14F42 14M15 16S37 17B10 18E10 18E30 22E47)},
	MRNUMBER = {3773272},
	MRREVIEWER = {Mikhail\ V.\ Bondarko},
	DOI = {10.1017/S1474748016000013},
	URL = {https://doi.org/10.1017/S1474748016000013},
}

@article {Satake:Theory,
	AUTHOR = {Satake, Ichirô},
	TITLE = {Theory of spherical functions on reductive algebraic groups
	over {$p$}-adic fields},
	JOURNAL = {Inst. Hautes \'{E}tudes Sci. Publ. Math.},
	FJOURNAL = {Institut des Hautes \'{E}tudes Scientifiques. Publications
	Math\'{e}matiques},
	YEAR = {1963},
	NUMBER = {18},
	PAGES = {5--69},
	ISSN = {0073-8301,1618-1913},
	MRCLASS = {22.60 (14.50)},
	MRNUMBER = {195863},
	MRREVIEWER = {T.\ Ono},
	URL = {http://www.numdam.org/item?id=PMIHES_1963__18__5_0},
}

@article {Littelmann:Paths,
	AUTHOR = {Littelmann, Peter},
	TITLE = {Paths and root operators in representation theory},
	JOURNAL = {Ann. of Math. (2)},
	FJOURNAL = {Annals of Mathematics. Second Series},
	VOLUME = {142},
	YEAR = {1995},
	NUMBER = {3},
	PAGES = {499--525},
	ISSN = {0003-486X,1939-8980},
	MRCLASS = {17B10 (17B67)},
	MRNUMBER = {1356780},
	MRREVIEWER = {Arun\ Ram},
	DOI = {10.2307/2118553},
	URL = {https://doi.org/10.2307/2118553},
}

@incollection {Tits:Reductive,
	AUTHOR = {Tits, J.},
	TITLE = {Reductive groups over local fields},
	BOOKTITLE = {Automorphic forms, representations and {$L$}-functions
	({P}roc. {S}ympos. {P}ure {M}ath., {O}regon {S}tate {U}niv.,
	{C}orvallis, {O}re., 1977), {P}art 1},
	SERIES = {Proc. Sympos. Pure Math},
	VOLUME = {XXXIII},
	PAGES = {pp 29--69},
	PUBLISHER = {Amer. Math. Soc., Providence, R.I.},
	YEAR = {1979},
	MRCLASS = {20G25 (20G10)},
	MRNUMBER = {546588},
}

@article {HenniartVigneras:Satake,
	AUTHOR = {Henniart, Guy and Vign\'{e}ras, Marie-France},
	TITLE = {A {S}atake isomorphism for representations modulo {$p$} of
	reductive groups over local fields},
	JOURNAL = {J. Reine Angew. Math.},
	FJOURNAL = {Journal f\"{u}r die Reine und Angewandte Mathematik. [Crelle's
	Journal]},
	VOLUME = {701},
	YEAR = {2015},
	PAGES = {33--75},
	ISSN = {0075-4102,1435-5345},
	MRCLASS = {20G25 (22E50)},
	MRNUMBER = {3331726},
	MRREVIEWER = {Maarten\ Sander\ Solleveld},
	DOI = {10.1515/crelle-2013-0021},
	URL = {https://doi.org/10.1515/crelle-2013-0021},
}

@incollection {Vinberg:Reductive,
	AUTHOR = {Vinberg, E. B.},
	TITLE = {On reductive algebraic semigroups},
	BOOKTITLE = {Lie groups and {L}ie algebras: {E}. {B}. {D}ynkin's {S}eminar},
	SERIES = {Amer. Math. Soc. Transl. Ser. 2},
	VOLUME = {169},
	PAGES = {145--182},
	PUBLISHER = {Amer. Math. Soc., Providence, RI},
	YEAR = {1995},
	ISBN = {0-8218-0454-5},
	MRCLASS = {20G15 (20M20)},
	MRNUMBER = {1364458},
	MRREVIEWER = {Lex\ Renner},
	DOI = {10.1090/trans2/169/10},
	URL = {https://doi.org/10.1090/trans2/169/10},
}

@incollection {Borel:Automorphic,
	AUTHOR = {Borel, Armand},
	TITLE = {Automorphic {$L$}-functions},
	BOOKTITLE = {Automorphic forms, representations and {$L$}-functions
	({P}roc. {S}ympos. {P}ure {M}ath., {O}regon {S}tate {U}niv.,
	{C}orvallis, {O}re., 1977), {P}art 2},
	SERIES = {Proc. Sympos. Pure Math.},
	VOLUME = {XXXIII},
	PAGES = {27--61},
	PUBLISHER = {Amer. Math. Soc., Providence, RI},
	YEAR = {1979},
	ISBN = {0-8218-1437-0},
	MRCLASS = {10D40 (12A67 22E50)},
	MRNUMBER = {546608},
	MRREVIEWER = {Yasuhiro\ Asoo},
}

@article {Herzig:Satake,
	AUTHOR = {Herzig, Florian},
	TITLE = {A {S}atake isomorphism in characteristic {$p$}},
	JOURNAL = {Compos. Math.},
	FJOURNAL = {Compositio Mathematica},
	VOLUME = {147},
	YEAR = {2011},
	NUMBER = {1},
	PAGES = {263--283},
	ISSN = {0010-437X,1570-5846},
	MRCLASS = {22E50 (20C08 20G25)},
	MRNUMBER = {2771132},
	MRREVIEWER = {Michael\ M.\ Schein},
	DOI = {10.1112/S0010437X10004951},
	URL = {https://doi.org/10.1112/S0010437X10004951},
}

@article {Quillen:Cohomology,
	AUTHOR = {Quillen, Daniel},
	TITLE = {On the cohomology and {$K$}-theory of the general linear
	groups over a finite field},
	JOURNAL = {Ann. of Math. (2)},
	FJOURNAL = {Annals of Mathematics. Second Series},
	VOLUME = {96},
	YEAR = {1972},
	PAGES = {552--586},
	ISSN = {0003-486X},
	MRCLASS = {20J05 (18F25 55B15)},
	MRNUMBER = {315016},
	DOI = {10.2307/1970825},
	URL = {https://doi.org/10.2307/1970825},
}

@article {Vigneras:ProI,
	AUTHOR = {Vigneras, Marie-France},
	TITLE = {The pro-{$p$}-{I}wahori {H}ecke algebra of a reductive
	{$p$}-adic group {I}},
	JOURNAL = {Compos. Math.},
	FJOURNAL = {Compositio Mathematica},
	VOLUME = {152},
	YEAR = {2016},
	NUMBER = {4},
	PAGES = {693--753},
	ISSN = {0010-437X,1570-5846},
	MRCLASS = {11E95 (20G25)},
	MRNUMBER = {3484112},
	MRREVIEWER = {Neven\ Grbac},
	DOI = {10.1112/S0010437X15007666},
	URL = {https://doi.org/10.1112/S0010437X15007666},
}

@article {Vigneras:ProII,
	AUTHOR = {Vign\'{e}ras, Marie-France},
	TITLE = {The pro-{$p$}-{I}wahori-{H}ecke algebra of a reductive
	{$p$}-adic group, {II}},
	JOURNAL = {M\"{u}nster J. Math.},
	FJOURNAL = {M\"{u}nster Journal of Mathematics},
	VOLUME = {7},
	YEAR = {2014},
	NUMBER = {1},
	PAGES = {363--379},
	ISSN = {1867-5778,1867-5786},
	MRCLASS = {22E50 (16P40 20C08)},
	MRNUMBER = {3271250},
	MRREVIEWER = {Ju-Lee\ Kim},
}

@incollection {Gross:Satake,
	AUTHOR = {Gross, Benedict H.},
	TITLE = {On the {S}atake isomorphism},
	BOOKTITLE = {Galois representations in arithmetic algebraic geometry
	({D}urham, 1996)},
	SERIES = {London Math. Soc. Lecture Note Ser.},
	VOLUME = {254},
	PAGES = {223--237},
	PUBLISHER = {Cambridge Univ. Press, Cambridge},
	YEAR = {1998},
	ISBN = {0-521-64419-4},
	MRCLASS = {22E50 (11F70)},
	MRNUMBER = {1696481},
	MRREVIEWER = {David\ Manderscheid},
	DOI = {10.1017/CBO9780511662010.006},
	URL = {https://doi.org/10.1017/CBO9780511662010.006},
}

@article{PepinSchmidt:Generic,
	AUTHOR = {Pepin, C\'{e}dric and Schmidt, Tobias},
	TITLE = {Generic and mod {$p$} {K}azhdan-{L}usztig {T}heory for
	{$GL_2$}},
	JOURNAL = {Represent. Theory},
	FJOURNAL = {Representation Theory. An Electronic Journal of the American
	Mathematical Society},
	VOLUME = {27},
	YEAR = {2023},
	PAGES = {1142--1193},
	ISSN = {1088-4165},
	MRCLASS = {11S37 (20C08)},
	MRNUMBER = {4672123},
	DOI = {10.1090/ert/656},
	URL = {https://doi.org/10.1090/ert/656},
}

@article {Dudas:DeligneLusztigRestriction,
	AUTHOR = {Dudas, Olivier},
	TITLE = {Deligne-{L}usztig restriction of a {G}elfand-{G}raev module},
	JOURNAL = {Ann. Sci. \'{E}c. Norm. Sup\'{e}r. (4)},
	FJOURNAL = {Annales Scientifiques de l'\'{E}cole Normale Sup\'{e}rieure.
	Quatri\`eme S\'{e}rie},
	VOLUME = {42},
	YEAR = {2009},
	NUMBER = {4},
	PAGES = {653--674},
	ISSN = {0012-9593,1873-2151},
	MRCLASS = {20C33 (16E35 18E30)},
	MRNUMBER = {2568878},
	MRREVIEWER = {Fran\c{c}ois\ Digne},
	DOI = {10.24033/asens.2105},
	URL = {https://doi.org/10.24033/asens.2105},
}

@article {Lourenco:Grassmanniennes,
	AUTHOR = {Lourenço, Jo\~{a}o},
	TITLE = {Grassmanniennes affines tordues sur les entiers},
	JOURNAL = {Forum Math. Sigma},
	FJOURNAL = {Forum of Mathematics. Sigma},
	VOLUME = {11},
	YEAR = {2023},
	PAGES = {Paper No. e12, 65},
	ISSN = {2050-5094},
	MRCLASS = {14M15 (11G18 14L15 20G15 20G25 20G44)},
	MRNUMBER = {4554780},
	MRREVIEWER = {Aigli\ Papantonopoulou},
	DOI = {10.1017/fms.2023.4},
	URL = {https://doi.org/10.1017/fms.2023.4},
}

@book {Milne:Algebraic,
	AUTHOR = {Milne, J. S.},
	TITLE = {Algebraic groups},
	SERIES = {Cambridge Studies in Advanced Mathematics},
	VOLUME = {170},
	NOTE = {The theory of group schemes of finite type over a field},
	PUBLISHER = {Cambridge University Press, Cambridge},
	YEAR = {2017},
	PAGES = {xvi+644},
	ISBN = {978-1-107-16748-3},
	MRCLASS = {14L15 (14-01 17B45 20-01 20G15)},
	MRNUMBER = {3729270},
	MRREVIEWER = {Boris\ \`E.\ Kunyavski\u i},
	DOI = {10.1017/9781316711736},
	URL = {https://doi.org/10.1017/9781316711736},
}

@article{CassXu:Geometrization,
	AUTHOR = {Cass, Robert and Xu, Yujie},
	TITLE = {Geometrization of the {S}atake transform for {${\rm mod}\, p$}
	{H}ecke algebras},
	JOURNAL = {Forum Math. Sigma},
	FJOURNAL = {Forum of Mathematics. Sigma},
	VOLUME = {13},
	YEAR = {2025},
	PAGES = {Paper No. e26, 22},
	ISSN = {2050-5094},
	MRCLASS = {14D24 (14M15 20C08)},
	MRNUMBER = {4859471},
	DOI = {10.1017/fms.2024.130},
	URL = {https://doi.org/10.1017/fms.2024.130},
}

@unpublished{Deligne:Letter2007,
	title={{Letter to Serre}},
	year={2007},
	author={Deligne, Pierre},
	url={https://publications.ias.edu/sites/default/files/2007\%20Serre_0.pdf},
}

@incollection{DeligneMilne:Tannakian,
	title={Tannakian categories},
	author={Deligne, Pierre and Milne, James S.},
	booktitle={{Hodge cycles, motives, and Shimura varieties}},
	pages={101--228},
	year={1982},
	publisher={Springer}
}

@phdthesis{Preis:Motivic,
	title={Motivic nearby cycles functors, local
	monodromy and universal local acyclicity},
	author={Preis, Benedikt},
	school={Universität Regensburg},
	year={2023}
}

@phdthesis{Robalo:Theorie,
	title={Th{\'e}orie homotopique motivique des espaces noncommutatifs},
	author={Robalo, Marco},
	school={Universit{\'e} de Montpellier 2},
	year={2014}
}

@phdthesis{Khan:Motivic,
	title={Motivic homotopy theory in derived algebraic geometry},
	author={Khan, Adeel},
	school={Universit{\"a}t Duisburg-Essen},
	year={2016}
}

@misc{Lurie:HigherAlgebra,
	title={Higher algebra},
	author={Lurie, Jacob},
	year={2017}
}

@article{XiaoZhu:Cycles,
	title={{Cycles on Shimura varieties via geometric Satake}},
	author={Xiao, Liang and Zhu, Xinwen},
	journal={arXiv preprint arXiv:1707.05700},
	year={2017},
}

@article{ALWY:Gaitsgory,
	title={{Gaitsgory's central functor and the Arkhipov-Bezrukavnikov equivalence in mixed characteristic}},
	author={Anschütz, Johannes and Louren{\c{c}}o, Jo{\~a}o and Wu, Zhiyou and Yu, Jize},
	journal={arXiv preprint arXiv:2311.04043},
	year={2023},
}

@article{ALRR:Fixed,
	title={Fixed points under pinning-preserving automorphisms of reductive group schemes},
	author={Achar, Pramod N and Louren{\c{c}}o, Jo{\~a}o and Richarz, Timo and Riche, Simon},
	journal={arXiv preprint arXiv:2212.10182},
	year={2022},
}

@article{ALRR:Modular,
	title={{A modular ramified geometric Satake equivalence}},
	author = {Achar, Pramod N and Louren{\c{c}}o, Jo{\~a}o and Richarz, Timo and Riche, Simon},
	journal={arXiv preprint arXiv:2403.10651},
	year={2024},
}

@article{Ruimy:IntegralI,
	title={{Integral Artin motives I: Smooth objects and the ordinary t-structure}},
	author={Ruimy, Rapha{\"e}l},
	journal={arXiv preprint arXiv:2211.02505},
	year={2025},
}

@article{Bando:Derived,
	title={{Derived Satake category and Affine Hecke category in mixed characteristics}},
	author={Bando, Katsuyuki},
	journal={arXiv preprint arXiv:2310.16244},
	year={2023}
}

@article{Scholze:Geometrization,
	title={{Geometrization of the local Langlands correspondence, motivically}},
	author={Scholze, Peter},
	journal={arXiv preprint arXiv:2501.07944},
	year={2025},
	note={To appear in \emph{J. Reine Angew. Math.}}
}

@article{AGLR:Local,
	title={{On the $p$-adic theory of local models}},
	author={Anschütz, Johannes and Gleason, Ian and Lourenço, Jo{\~a}o and Richarz, Timo},
	journal={arXiv preprint arXiv:2201.01234},
	year={2022},
	note={To appear in \emph{Annals of Mathematics}}
}

@article{CassvdHScholbach:Central,
	title={{Central motives on parahoric flag varieties}},
	author={Cass, Robert and van den Hove, Thibaud and Scholbach, Jakob},
	JOURNAL = {J. Eur. Math. Soc.},
	year={2025},
	note={published online first}
}

@article{FarguesScholze:Geometrization,
	title={Geometrization of the local Langlands correspondence},
	author={Fargues, Laurent and Scholze, Peter},
	journal={arXiv preprint arXiv:2102.13459},
	year={2021},
	note={To appear in \emph{Astérisque}}
}

@article{Scholze:Etale,
	title={{\'E}tale cohomology of diamonds},
	author={Scholze, Peter},
	journal={arXiv preprint arXiv:1709.07343},
	year={2017},
	note={To appear in \emph{Astérisque}}
}
	
\end{document}